\pdfoutput=1
\documentclass[11pt]{amsart}
\usepackage[margin=1.15in]{geometry}
\usepackage{amsmath,amssymb,amsthm,mathtools}
\usepackage{enumitem}
\usepackage{array}
\usepackage[T1]{fontenc}
\usepackage{xcolor}
\usepackage[colorlinks=true,linkcolor=blue!50!black,citecolor=green!35!black,urlcolor=blue!50!black]{hyperref}
\usepackage[capitalise,nameinlink]{cleveref}

\newcommand{\N}{\mathbb{N}}
\newcommand{\Z}{\mathbb{Z}}
\newcommand{\Q}{\mathbb{Q}}
\newcommand{\R}{\mathbb{R}}
\newcommand{\T}{\mathbb{T}}

\renewcommand{\P}{\mathbb{P}}            
\newcommand{\one}{\mathbf{1}}
\newcommand{\dens}{\mathsf{d}}           
\newcommand{\udens}{\overline{\mathsf{d}}}
\newcommand{\ldens}{\underline{\mathsf{d}}}
\newcommand{\ubd}{\mathsf{d}^*}          
\newcommand{\udlog}{\overline{\mathsf{d}}_{\log}}
\newcommand{\dlog}{\mathsf{d}_{\log}}
\newcommand{\FS}{\mathrm{FS}}

\newcommand{\eps}{\varepsilon}
\newcommand{\e}{e}                        
\newcommand{\Fol}{F{\o}lner}
\DeclareMathOperator{\supp}{supp}

\newcommand{\leanfile}[1]{\nolinkurl{Kmrr25/#1.lean}}
\newcommand{\leanname}[1]{\nolinkurl{#1}}
\newcommand{\formalization}[1]{\par\smallskip\noindent\textit{Formalization.} #1\par\smallskip}

\theoremstyle{plain}
\newtheorem{theorem}{Theorem}[section]
\newtheorem{proposition}[theorem]{Proposition}
\newtheorem{lemma}[theorem]{Lemma}
\newtheorem{corollary}[theorem]{Corollary}
\newtheorem*{theorem*}{Theorem}
\newtheorem*{mainthm}{Main Theorem}
\theoremstyle{definition}
\newtheorem{definition}[theorem]{Definition}
\newtheorem{surveyqinner}{Question}
\newtheorem{surveyconjinner}{Conjecture}
\newenvironment{question}[1]{\renewcommand{\thesurveyqinner}{#1}%
  \begin{surveyqinner}[{\cite{kra2023problems}}]}{\end{surveyqinner}}
\newenvironment{surveyconj}[1]{\renewcommand{\thesurveyconjinner}{#1}%
  \begin{surveyconjinner}[{\cite{kra2023problems}}]}{\end{surveyconjinner}}
\newtheorem{example}[theorem]{Example}
\crefname{surveyqinner}{Question}{Questions}
\crefname{surveyconjinner}{Conjecture}{Conjectures}
\theoremstyle{remark}
\newtheorem{remark}[theorem]{Remark}

\title[Solutions to problems on infinite sumset configurations]{Solutions to several problems on\\ infinite sumset configurations}

\author{Shisheng Li}
\author{Yong Su}

\date{\today}


\begin{document}

\begin{abstract}
We resolve a number of the problems collected by Kra, Moreira, Richter and Robertson in their survey
\emph{Problems on infinite sumset configurations in the integers and beyond}.
Our main result is that in every countably infinite abelian group, every set of positive upper Banach
density contains a sumset $B_1+\dots+B_k$ of $k$ infinite sets, for every $k$; the proof uses the
Kronecker factor, but no higher-order Host--Kra structure theory, nilsystems of step greater than one,
or nilspaces.
We also answer, positively or negatively, questions on monochromatic sets $B+mB$, on sumsets in
sets of logarithmic density, on polynomial sumset configurations, on sums of the form
$B_{i_1}+\dots+B_{i_m}$ with $B_i$ drawn from a density family, on random sets, on ordered sumsets in
shifted primes, on sums and products in $\N$ and $\Q$, and on sumsets in abelian groups. Several of
these answers concern the questions as literally stated, and two follow directly from known
examples; in each case we say so and discuss the natural corrected question.
Finally, answering a question of Kousek and Radi\'c, we show that the supremum of the lower Banach
densities of sets $A\subset\N$ such that $B+B+t\not\subset A$ for every infinite $B\subset\N$ and
every $t\ge0$ (respectively $B+B\not\subset A$ for every infinite $B$) is $1/2$ (respectively $3/4$).
The main theorem and most of the other results are formalized in Lean~4, with some published
theorems entering as explicit hypotheses.
\end{abstract}

\maketitle
\setcounter{tocdepth}{1}
\tableofcontents

\section{Introduction}\label{I:sec:intro}

A well-known theorem of Moreira, Richter and Robertson~\cite{moreira2019proof}, answering a question of
Erd\H{o}s, says that every set $A\subset\N$ of positive upper Banach density contains a sumset
$B+C=\{b+c:b\in B,\ c\in C\}$ with $B$ and $C$ infinite. In the survey~\cite{kra2023problems},
Kra, Moreira, Richter and Robertson collected a large number of open problems around this theorem:
which other infinite configurations must occur in sets of positive density, which density
hypotheses are needed, and what happens in the primes, in random sets, in multiplicative settings and
in other groups.

In this paper we treat twenty of the questions and conjectures posed in~\cite{kra2023problems}
(nineteen questions and Conjecture~4.5); for two of them the answer follows directly from known
results, as explained below. Our main result
concerns sums of more than two infinite sets in abelian groups.

\begin{mainthm}
Let $\Gamma$ be a countably infinite abelian group and let $k\ge1$. Every set $A\subset\Gamma$ of
positive upper Banach density contains a sumset $B_1+\dots+B_k$, where $B_1,\dots,B_k\subset\Gamma$
are infinite.
\end{mainthm}

This is \cref{K:thm:main}; it answers Question~5.14 of~\cite{kra2023problems}. The proof uses the
Kronecker factor, but no higher-order Host--Kra structure theory, nilsystems of step greater than
one, or nilspaces: it extends the short proof of the case $k=2$
by Kra, Moreira, Richter and Robertson~\cite{kra2026short} to an arbitrary number of summands,
building the required configuration one level at a time and keeping track only of the supports of
suitable measures on the Kronecker factor.

The other answers are listed in \cref{I:tab:results}, in the order in which they appear in the
paper. They are of very different depth. Several are short, and some rely in an essential way on
recent work of other authors: Question~5.16 on a theorem of Charamaras and
Mountakis~\cite{charamaras2024finding}, the positive part of Question~\S3.3(a) on a theorem of
Kousek~\cite{kousek2025asymmetric}, Questions~3.31 and~\S3.6(a) on the dynamical part of the density
finite sums theorem of Kra, Moreira, Richter and Robertson~\cite{kra2025density}, and Conjecture~4.5
on a theorem of Tao and Ziegler~\cite{tao2023infinite}. The negative answer to Question~\S3.3(e)
follows directly from known counterexamples (one from~\cite{kra2023problems} itself and one due to
Ackelsberg~\cite{ackelsberg2024counterexamples}), and Question~3.24 is already answered by an example
of Kousek and Radi\'c~\cite{kousek2024infinite}, as we check; we record these for completeness. In
each case we say precisely what is new.

In addition, in \cref{KR:sec:main} we answer a question of Kousek and
Radi\'c~\cite{kousek2024infinite} that is not in~\cite{kra2023problems}: the supremum of the lower
Banach densities of sets $A\subset\N$ such that $B+B+t\not\subset A$ for every infinite $B\subset\N$
and every $t\ge0$ is $1/2$, and the supremum for sets such that $B+B\not\subset A$ for every
infinite $B$ is $3/4$.

\subsection*{Conventions}
We use the numbering of the second arXiv version of~\cite{kra2023problems}. Some questions in the
survey are posed in the running text without a number. We refer to them by the subsection in which
they appear followed by a letter, in order of appearance: for instance, Question~\S3.2(a) is the
first unnumbered question in Subsection~3.2 of~\cite{kra2023problems}. Before answering each
question we restate it in our notation.

Throughout, $\N=\{1,2,3,\dots\}$, $\P$ denotes the set of primes, and for $A\subset\N$
\[
\udens(A)=\limsup_{N\to\infty}\frac{|A\cap[N]|}{N},\qquad
\ubd(A)=\lim_{N\to\infty}\sup_{M\ge0}\frac{|A\cap[M+1,M+N]|}{N},
\]
\[
\udlog(A)=\limsup_{N\to\infty}\frac{1}{\log N}\sum_{n\in A\cap[N]}\frac1n,
\]
where $[N]=\{1,\dots,N\}$. The lower densities $\ldens$ and $\underline{\mathsf d}_{\log}$ are defined
with $\liminf$ in place of $\limsup$, and the natural and logarithmic densities $\dens$ and $\dlog$
are the common values when upper and lower densities agree. The lower Banach density is defined in
\cref{KR:sec:main}. For a \Fol{} sequence $\Phi=(\Phi_N)$ in a countable abelian group we write
$\udens_\Phi(A)=\limsup_N|A\cap\Phi_N|/|\Phi_N|$, and the upper Banach density $\ubd(A)$ is the
supremum of $\udens_\Phi(A)$ over all \Fol{} sequences $\Phi$.
A set is \emph{thick} if it contains arbitrarily long intervals (in an abelian group: translates of
every finite set). Conventions for nonabelian groups are given where they are used. For sets $B,C$ we write $B+C=\{b+c:b\in B,c\in C\}$, $B+t=\{b+t:b\in B\}$,
$mB=\{mb:b\in B\}$ and $B\oplus B=\{b+b':b,b'\in B,\ b\ne b'\}$; thus $B+B$ is the unrestricted
sumset, which contains the elements $2b$.

\subsection*{Questions whose answer depends on the formulation}
For several questions the answer we give concerns the question exactly as it is stated in the survey,
and the literal statement differs from what the authors of the survey probably intended. We include
these answers because the difference is not always obvious and the literal statements have been
repeated in the literature, but in each case we point out the issue explicitly, give the natural
corrected formulation, and say whether it remains open. This applies to
Questions~5.23 and~\S5.4(b) (where the fibre conditions constrain only one coordinate),
Question~\S3.2(a) (where the density condition can be read in two ways; we answer both readings),
Question~3.20 (where $m=1$ is coprime to every $q$; we also answer the question for $m\ge2$),
Question~3.23 (where $B+mB$ contains the diagonal terms $(m+1)b$) and Question~5.6 (where the
sign of $x$ gives an obstruction at the real place).

\begin{table}[ht]
\centering
\small
\renewcommand{\arraystretch}{1.25}
\begin{tabular}{@{}>{\raggedright\arraybackslash}p{0.17\textwidth}>{\raggedright\arraybackslash}p{0.6\textwidth}>{\raggedright\arraybackslash}p{0.12\textwidth}@{}}
\hline
Question & Answer & Section\\
\hline
5.14 & Yes: $B_1+\dots+B_k$ in every countably infinite abelian group & \ref{K:sec:main}\\
5.16, \S5.3(a) & Exactly which sets of translates suffice when $[G:2G]<\infty$ & \ref{G:sec:translates}\\
5.10 & No: fibrewise lower density does not suffice in $\N^2$ & \ref{G:sec:q510}\\
5.23, \S5.4(b) & No, as stated, and also for a repaired version in $\Z^d$; a second repair is open & \ref{G:sec:q523}\\
\S3.2(a) & No, for both readings & \ref{IA:sec:u32a}\\
\S3.3(a) & Ordered polynomial sums: exactly the linear case & \ref{IA:sec:u33a}\\
\S3.3(e) & No (direct consequence of known examples) & \ref{IA:sec:u33e}\\
3.20 & No, for every number $r\ge2$ of colours, also when only $m\ge2$ is considered & \ref{IA:sec:q320}\\
3.23 & No, as stated & \ref{IA:sec:q323}\\
3.24 & No (already implied by Kousek--Radi\'c) & \ref{IA:sec:q324}\\
3.31, \S3.6(a) & Yes, in the strongest $k$-fold form & \ref{IB:sec:q331}\\
4.13 & For $\gamma>0$: yes iff $\gamma\notin\Z$ (for Lebesgue-almost every $x\in\R$) & \ref{IB:sec:q413}\\
\S4.4(a) & Random sets: criterion for monotone $p_n$; in general the set of good $p$ is not Borel & \ref{IB:sec:u44a}\\
4.5 & Yes (a consequence of a theorem of Tao and Ziegler) & \ref{IB:sec:c45}\\
\S5.1(a) & Yes for both examples & \ref{SP:sec:u51a}\\
5.4 & No & \ref{SP:sec:q54}\\
5.6 & No, as stated & \ref{SP:sec:q56}\\
\hline
\end{tabular}
\caption{Questions from~\cite{kra2023problems} answered in this paper.}
\label{I:tab:results}
\end{table}

\subsection*{Formalization}
The main theorem and most of the other results have been formalized in Lean~4 with Mathlib; the
formalization is available at \url{https://github.com/daizisheng/kmrr25}, in the directory
\nolinkurl{lean/Kmrr25}. No file contains \texttt{sorry} or declares an axiom; every main theorem
depends only on the standard axioms of Lean. Where one of our proofs relies on a published theorem
that is not formalized, that theorem is stated as a proposition and appears as an explicit
hypothesis of the formal statement, so that what is formalized is the deduction from it. \Cref{I:tab:lean} gives an overview, and
after each formalized result a paragraph headed \emph{Formalization} names the file and the
main formal statements and says what is not covered.

\begin{table}[ht]
\centering
\small
\renewcommand{\arraystretch}{1.2}
\begin{tabular}{@{}>{\raggedright\arraybackslash}p{0.16\textwidth}>{\raggedright\arraybackslash}p{0.22\textwidth}>{\raggedright\arraybackslash}p{0.54\textwidth}@{}}
\hline
Question & File & Status\\
\hline
5.14 & \nolinkurl{ThmK*} (14 files) & conditional on four classical theorems (see the end of \cref{K:ssec:proof})\\
5.16, \S5.3(a) & \nolinkurl{Q516} & conditional on \cite[Corollary~1.13]{charamaras2024finding}\\
5.10 & \nolinkurl{Q510} & complete\\
5.23, \S5.4(b) & \nolinkurl{Q523} & answer to the questions as stated: complete; repaired version
(\cref{G:thm:repairA}): not formalized\\
\S3.2(a) & \nolinkurl{S32a} & complete\\
KR, Q.~5.2 & \nolinkurl{KR52} & \cref{KR:thm:main} and the lower bounds in \cref{KR:cor:main}: complete;
upper bounds (\cite[Theorems~1.2 and~1.3]{kousek2024infinite}): not formalized\\
\S3.3(a) & \nolinkurl{S33a} & negative part: complete; positive part: conditional on
\cite[Theorem~1.2]{kousek2025asymmetric}\\
\S3.3(e) & \nolinkurl{S33e} & complete\\
3.20 & \nolinkurl{Q320} & complete\\
3.23 & \nolinkurl{Q323} & complete\\
3.24 & \nolinkurl{Q324} & complete\\
3.31, \S3.6(a) & \nolinkurl{Q331} & conditional on the dynamical input from~\cite{kra2025density}
(\cref{IB:prop:dyn})\\
4.13 & \nolinkurl{Q413}, \nolinkurl{Q413a}, \nolinkurl{Q413SecondMoment} & complete\\
\S4.4(a) & \nolinkurl{S44a}, \nolinkurl{S44b} & criteria and examples: complete; non-Borel results:
conditional on \cite[Theorems~14.2 and~27.1]{kechris1995classical}\\
4.5 & \nolinkurl{C45} & conditional on \cite[Theorem~1.5]{tao2023infinite}\\
\S5.1(a) & \nolinkurl{S51a} & complete\\
5.4 & \nolinkurl{Q54} & complete\\
5.6 & \nolinkurl{Q56} & complete\\
\hline
\end{tabular}
\caption{Lean formalization. ``Complete'' means that the answer to the question proved in this
paper is formalized without hypotheses beyond the standard axioms; the file names refer to
\nolinkurl{lean/Kmrr25/*.lean}.}
\label{I:tab:lean}
\end{table}

\subsection*{Organisation}
The paper has three parts, organised by the setting. \Cref{K:sec:main}, at the beginning of the first
part, contains the proof of the main theorem; the other sections of that part concern sumsets in
abelian groups and in $\N^2$. The second part concerns sumsets in the integers: density thresholds
for $B+B$ and $B+B+t$ (\cref{IA:sec:u32a,KR:sec:main}), polynomial and dilated sumsets
(\cref{IA:sec:ordered,IA:sec:BmB}), density Ramsey families (\cref{IB:sec:q331}), sparse and random
sets (\cref{IB:sec:q413,IB:sec:u44a}) and the primes (\cref{IB:sec:c45}). The third part concerns
configurations that mix sums and products. Each section begins by restating the questions it
answers, and the sections can be read independently.

\part{Sumsets in groups}
\section{Sums of \texorpdfstring{$k$}{k} infinite sets in countable abelian groups}
\label{K:sec:main}

This section answers Question~5.14 of~\cite{kra2023problems}.
We show (\cref{K:thm:main}) that in every countably infinite abelian group, every set of
positive upper Banach density contains a sumset $B_1+\dots+B_k$ of $k$ infinite sets, for
every $k$.
For finite groups the question has a trivially negative answer (\cref{K:rem:finite}), so
``countably infinite'' is the right setting.
The proof does not use higher-order Host--Kra theory, nilsystems or nilspaces, or any
equidistribution theorem for nilsystems; besides the Kronecker factor, it uses only standard
facts of ergodic and spectral theory.
It extends to $k$ summands, and to arbitrary countable abelian groups, the short proof of
the case $k=2$ for subsets of $\N$ given by Kra, Moreira, Richter and
Robertson~\cite{kra2026short}.
The cube is built one level at a time.
At each level we re-base on the Kronecker factor of the current measure, and from one level
to the next we keep only a condition on topological supports.

\subsection{The question and the result}\label{K:ssec:result}

Let $\Gamma$ be a countable abelian group, written additively.
A \emph{\Fol{} sequence} in $\Gamma$ is a sequence $\Phi=(\Phi_N)_{N\in\N}$ of nonempty finite
subsets of $\Gamma$ such that $|(\Phi_N+\gamma)\triangle\Phi_N|/|\Phi_N|\to0$ for every
$\gamma\in\Gamma$.
The \emph{upper Banach density} of $A\subset\Gamma$ is
\[
\ubd(A)=\sup_{\Phi}\ \limsup_{N\to\infty}\frac{|A\cap\Phi_N|}{|\Phi_N|},
\]
where the supremum is over all \Fol{} sequences $\Phi$ in $\Gamma$.

\begin{question}{5.14}\label{K:q:514}
Is it true in every countable abelian group $(G,+)$ that every set $A$ of positive upper
Banach density contains $B_1+\dots+B_k$ for infinite sets $B_1,\dots,B_k\subset G$?
\end{question}

Here $k\in\N$ is arbitrary; the case $k=1$ is trivial.

\begin{theorem}\label{K:thm:main}
Let $\Gamma$ be a countably infinite abelian group and let $k\ge1$.
If $A\subset\Gamma$ satisfies $\ubd(A)>0$, then there are infinite sets
$B_1,\dots,B_k\subset\Gamma$ with
\[
B_1+\dots+B_k=\{b_1+\dots+b_k:\ b_i\in B_i\}\subset A .
\]
\end{theorem}

\begin{remark}[Finite groups]\label{K:rem:finite}
Read literally, \cref{K:q:514} includes finite groups, and for them the answer is no.
A finite group has no infinite subsets, yet every nonempty $A\subset G$ has
$\ubd(A)\ge|A|/|G|>0$.
(For finite $G$, the Følner condition forces $\Phi_N=G$ for all large $N$.)
So the meaningful version of the question concerns countably \emph{infinite} groups, and
\cref{K:thm:main} answers that version positively for every $k$.
\end{remark}

\begin{remark}[Previously known cases]\label{K:rem:prior}
\leavevmode
\begin{enumerate}[label=(\alph*)]
\item For $k=2$ the conclusion holds in every countably infinite amenable group, with products in
place of sums, for sets of positive density along a two-sided \Fol{} sequence. This is due
to Moreira, Richter and Robertson~\cite[Theorem~1.3]{moreira2019proof}. In an abelian group
every \Fol{} sequence is two-sided.
\item For $\Gamma=\Z$, or rather for $A\subset\N$, the case of general $k$ is
\cite[Theorem~1.1]{kra2022infinite} (Theorem~2.8 of~\cite{kra2023problems}). Here and below,
theorem numbers in~\cite{kra2022infinite} refer to arXiv:2206.01786v2.
Its proof relies on pronilfactors: Kra, Moreira, Richter and Robertson construct Erdős
cubes in an extension of the system with continuous factor maps to its pronilfactors
\cite[Section~5]{kra2022infinite}.
\item For $k=3$ and $[\Gamma:6\Gamma]<\infty$, the conclusion follows from a theorem of Ackelsberg
and Jamneshan~\cite[Theorem~1.9]{ackelsberg2025equidistribution}. They find $t\in\Gamma$ and an
infinite $B\subset\Gamma$ such that $A-t$ contains every sum of three distinct elements of
$B$. Splitting $B$ into three disjoint infinite sets $B'_1,B'_2,B'_3$ then gives
$(B'_1+t)+B'_2+B'_3\subset A$.
\end{enumerate}
After stating \cref{K:q:514}, the survey~\cite{kra2023problems} remarks that the techniques
of~\cite{kra2022infinite}, together with work of Griesmer~\cite{griesmer2009ergodic}, may
suffice for $\Gamma=\Z^d$. In other groups, it says, ``one cannot
proceed in the same way, because the analogous ergodic structure theory is not currently
available.''
The proof below uses no such structure theory.
\end{remark}

\begin{remark}[The integers]\label{K:rem:N}
Our proof specialises to a proof of \cite[Theorem~1.1]{kra2022infinite} that uses no
pronilfactors.
To see this, let $A\subset\N$ have positive upper Banach density.
Choose intervals $I_n=[M_n+1,M_n+N_n]$ with $N_n\to\infty$ and
$|A\cap I_n|/N_n\to\ubd(A)$; they form a \Fol{} sequence in $\Z$, so $\ubd(A)>0$ also when $A$
is viewed as a subset of $\Z$.
\cref{K:thm:main} then gives infinite $B_i\subset\Z$ with $B_1+\dots+B_k\subset A\subset\N$.
Fix $b_j\in B_j$ for $j\neq i$. Then every element of $B_i$ is at least
$1-\sum_{j\ne i}b_j$, so each $B_i$ is bounded below.
Hence $B_i\cap\N$ is still infinite, and $(B_1\cap\N)+\dots+(B_k\cap\N)\subset A$.
\end{remark}

\subsection{Strategy of the proof}\label{K:ssec:strategy}

\subsubsection*{Erdős cubes}
Following \cite{kra2022infinite}, we look for a dynamical configuration called an Erdős cube.
The correspondence principle (\cref{K:thm:corr}) gives the following data:
\begin{itemize}
\item an ergodic system $(X,\mu,T)$;
\item an open set $E\subset X$ with $\mu(E)>0$;
\item a point $a\in X$, generic for $\mu$, such that $A=\{\gamma:T^\gamma a\in E\}$.
\end{itemize}
An \emph{Erdős cube} is a tuple $y=(y_\eps)_{\eps\in\{0,1\}^k}$ of points of $X$ with the
following property. For every direction $i$, the lower $i$-face
$(y_\eps)_{\eps_i=0}$ is moved arbitrarily close to the upper $i$-face
$(y_{\eps+e_i})_{\eps_i=0}$ by infinitely many group elements $\beta$, acting diagonally;
here $\eps+e_i$ is $\eps$ with its $i$-th digit changed to $1$.
If $y_{0\cdots0}=a$ and $y_{1\cdots1}\in E$, a greedy argument produces
$B_1+\dots+B_k\subset A$ (\cref{K:lem:cube-sumset}).
So the task is to construct such a cube.

\subsubsection*{The case $k=2$}
The argument of~\cite{kra2026short} can be read as constructing a $2$-dimensional cube
whose bottom face is $(a,y)$ and whose top face $(x,x')$ is random; we use our own
notation. A point $y\in\supp\mu$ is chosen so that $(a,y)$
is generic for some joining $\lambda$ of $\mu$ with itself.
The key claim (Claim~1 in the proof of \cite[Theorem~3.1]{kra2026short}) produces
infinitely many $\beta$ with two properties:
\begin{itemize}
\item $T^\beta a$ is close to $y$;
\item the top face $(x,x')$, which is distributed according to $\lambda$, still has positive
measure after the direction-$1$ constraint imposed by $\beta$.
\end{itemize}
The proof decomposes $1_E=f_c+f_{wm}$ into its Kronecker and weak mixing parts.
\begin{itemize}
\item The weak mixing part averages out along every \Fol{} sequence.
\item The Kronecker part is handled by approximating $f_c$ by \emph{continuous}
eigenfunctions. This uses an enhanced correspondence principle in which every measurable
eigenfunction is topological \cite[Theorem~2.5(4)]{kra2026short}. Continuity gives the
equicontinuity estimate~(4.2) of~\cite{kra2026short}, which turns ``$T^\beta a$ is close to
$y$'' into ``$f_c\circ T^\beta$ is close to $f_c$ along the orbit''.
\item The positive frequency of the $\beta$ with $T^\beta a$ near $y$ comes from
$y\in\supp\mu$.
\end{itemize}

\subsubsection*{Where the difficulty lies for $k\ge3$}
Suppose a $(k-1)$-dimensional bottom face $w=(w_\eps)$ with $w_{0\cdots0}=a$ has been
built, generic for an ergodic measure $\lambda$ on $X^{\{0,1\}^{k-1}}$. We try to add a
random top face distributed according to $\lambda$.
\begin{itemize}
\item The direction $k$ (bottom face to top face) behaves exactly as in the case $k=2$.
\item In each of the other directions $i$, the relevant system is the face measure $\nu_i$,
the distribution under $\lambda$ of the upper $i$-face. Its Kronecker factor carries
information about $\mu$ beyond the Kronecker factor of $\mu$. For instance, for the skew
product $(s,u)\mapsto(s+\theta,u+s)$ on $\T^2$, the measures on pairs that arise already
have eigenfunctions of the form $e(u_1-u_0)$, with eigenvalues that depend on the
parameter of the pair measure. For a general system there is no reason for such
eigenfunctions to be continuous in the topological model in which $a$ lives.
\end{itemize}
So the equicontinuity step of~\cite{kra2026short} is not available on faces.
The proof of~\cite{kra2022infinite} for $\Z$ proceeds differently. It disintegrates the
Host--Kra cubic measure $\mu^{[k]}$ over the first coordinate and needs the conditional
measure at the prescribed point $a$. To define that conditional measure at every point,
it passes to an extension in which the factor maps onto the pronilfactors are continuous,
and it proves that the ergodic decomposition of $\mu^{[k]}$ is then continuous.

\subsubsection*{The three ideas}
\begin{enumerate}[label=(\roman*)]
\item \textbf{Phase returns replace continuity} (\cref{K:prop:criterion}).
Let $\iota_i$ be the Kronecker map of $\nu_i$, with rotation $\alpha_i$.
\begin{itemize}
\item By ergodicity, $\iota_i(\text{upper face})-\iota_i(\text{lower face})$ is a
$\lambda$-a.e.\ constant $r_i$.
\item By continuity of translation in $L^1$ of the Kronecker group, the Kronecker part
is controlled for all $\beta$ whose ``phase'' $\alpha_i(\beta)$ is close to $r_i$.
\item So it suffices to have many $\beta$ for which $T^\beta(\text{lower face of }w)$ is close
to the upper face of $w$ \emph{and} $\alpha_i(\beta)$ is close to $r_i$. Here ``many'' means
a set of positive upper Banach density.
\end{itemize}
\item \textbf{Phase returns come from supports} (\cref{K:lem:phase}).
Such phase returns exist as soon as the upper face of $w$ lies in the topological support
of $\nu_i$. The proof re-bases $w$: we change the \Fol{} sequence so that $w$, decorated
with a point of the Kronecker group, becomes generic for the graph measure. A stabiliser
argument for the support of the graph measure then forces the decorated upper face into
that support (\cref{K:lem:stab}).
\item \textbf{The support condition survives one more level} (\cref{K:prop:step}).
We build the cube level by level. At level $m$ we have a point $w\in X^{\{0,1\}^m}$ that
is generic for an ergodic, face-symmetric measure $\lambda_m$ and whose upper faces lie in
the supports of the face measures. To pass to level $m+1$:
\begin{itemize}
\item re-base $w$ on the Kronecker factor of $\lambda_m$;
\item choose the new top face at random, so that the new cube is generic for an ergodic
component of the relative product of $\lambda_m$ with itself over its Kronecker factor;
\item check that the support condition persists. This is the content of an elementary
convolution lemma (\cref{K:lem:PS}). The lemma holds at \emph{every} support point, not
only at typical ones.
\end{itemize}
\end{enumerate}
From one level to the next, the induction carries only genericity, ergodicity, face
symmetry, the vertex marginals and the support condition; the Kronecker data are rebuilt at
each level.
The only prescribed point is $a$. Every vertex added at an induction step is a typical
point for the current measure, chosen after all the data it must avoid have been fixed, and
the final top face is constructed by nested boxes. Apart from genericity, no statement
about typical points is ever applied at a prescribed point, and this is why no
equidistribution at a prescribed point is needed.

\subsection{Notation and standard facts}\label{K:ssec:prelim}

Throughout, $\Gamma$ is a countably infinite abelian group.
For a finite set $F$ and a function $u$ on $F$ we write
$\mathop{\mathbb E}_{\gamma\in F}u(\gamma)=|F|^{-1}\sum_{\gamma\in F}u(\gamma)$.

\begin{lemma}\label{K:lem:folner-grow}
If $\Phi$ is a \Fol{} sequence in $\Gamma$, then $|\Phi_N|\to\infty$.
\end{lemma}

\begin{proof}
Let $F\subset\Gamma$ be finite with $|F|=n$. The number of pairs $(x,g)\in\Phi_N\times F$
with $x+g\in\Phi_N$ is at most $|\Phi_N|^2$. It also equals
$\sum_{g\in F}|\Phi_N\cap(\Phi_N-g)|$, which is $n|\Phi_N|(1-o(1))$ by the \Fol{}
property. Hence $|\Phi_N|\ge n(1-o(1))$. Since $\Gamma$ is infinite, $n$ is arbitrary.
\end{proof}

\subsubsection*{Topological systems and generic points}
A \emph{topological $\Gamma$-system} $(Y,T)$ consists of a compact metric space $Y$ and an
action of $\Gamma$ by homeomorphisms $T^\gamma$, with $T^{\gamma+\delta}=T^\gamma T^\delta$.
All measures are Borel probability measures.
A measure $\nu$ is \emph{invariant} if $T^\gamma_*\nu=\nu$ for all $\gamma$, and
\emph{ergodic} if moreover every invariant Borel set has measure $0$ or $1$.
Since $\Gamma$ is countable, this is equivalent to asking that every Borel set $S$ with
$\nu(S\triangle T^{-\gamma}S)=0$ for all $\gamma$ has measure $0$ or $1$ (replace $S$ by
$\bigcap_{\gamma}T^{-\gamma}S$); in particular a Borel map into a compact metric space
which is $\nu$-a.e.\ invariant under each $T^\gamma$ is $\nu$-a.e.\ constant.
A point $y\in Y$ is \emph{generic for $\nu$ along $\Phi$}, written $y\in\mathrm{gen}(\nu,\Phi)$,
if $\mathop{\mathbb E}_{\gamma\in\Phi_N}f(T^\gamma y)\to\int f\,d\nu$ for every $f\in C(Y)$.
The (topological) support $\supp\nu$ is the smallest closed set of full measure. It is
invariant whenever $\nu$ is.
We use the following facts freely.
\begin{enumerate}[label=(G\arabic*)]
\item\label{K:G1} If $y\in\mathrm{gen}(\nu,\Phi)$ and $O\subset Y$ is open, then
$\liminf_N|\{\gamma\in\Phi_N:T^\gamma y\in O\}|/|\Phi_N|\ge\nu(O)$.
This follows by approximating $1_O$ from below by continuous functions.
In particular, if $O$ meets $\supp\nu$, then the set of return times
$\{\gamma:T^\gamma y\in O\}$ has positive upper Banach density. By \cref{K:lem:folner-grow}
it is then infinite.
\item\label{K:G2} If $y\in\mathrm{gen}(\nu,\Phi)$, then $\supp\nu$ is contained in the orbit closure
$\overline{\{T^\gamma y:\gamma\in\Gamma\}}$. This is immediate from \ref{K:G1}.
\item\label{K:G3} If $P\colon Y\to Y'$ is continuous and equivariant and
$y\in\mathrm{gen}(\nu,\Phi)$, then $P(y)\in\mathrm{gen}(P_*\nu,\Phi)$.
\item\label{K:G4} If $\nu$ is ergodic, then for every \Fol{} sequence $\Phi$ there is a
subsequence $\Psi$ such that $\nu$-almost every point is generic for $\nu$ along $\Psi$
\cite[Lemma~2.7]{charamaras2024finding}. Since only a subsequence is required, the mean ergodic
theorem suffices: for $f$ in a countable dense subset of $C(Y)$ the averages of $f$ along $\Phi$
converge to $\int f\,d\nu$ in $L^2(\nu)$, hence $\nu$-almost everywhere along a subsequence,
and a diagonal argument gives one subsequence $\Psi$ that works for all of them.
(\cite{charamaras2024finding} uses Lindenstrauss' pointwise ergodic theorem instead.)
\end{enumerate}

\subsubsection*{Cubes}
For $m\ge0$ let $[m]=\{0,1\}^m$, with elements written $\eps=\eps_1\cdots\eps_m$.
Let $\vec0=0\cdots0$ and $\vec1=1\cdots1$.
$\Gamma$ acts on $X^{[m]}$ diagonally, $(T^\gamma x)_\eps=T^\gamma x_\eps$; we use the same
letter $T$ for this action.
We write $e_i\in[m]$ for the vertex with a single digit $1$, in position $i$, and
$\eps+e_i$ for the vertex obtained from $\eps$ (with $\eps_i=0$) by changing its $i$-th
digit to $1$.
For $1\le i\le m$, $\eta\in[m-1]$ and $j\in\{0,1\}$, let $\eta^{(i,j)}\in[m]$ be the vertex
obtained by inserting the digit $j$ at position $i$.
The \emph{lower} and \emph{upper $i$-faces} are the continuous equivariant maps
$L_i,U_i\colon X^{[m]}\to X^{[m-1]}$ given by
\[
(L_ix)_\eta=x_{\eta^{(i,0)}},\qquad (U_ix)_\eta=x_{\eta^{(i,1)}} .
\]
We identify $X^{[m+1]}$ with $X^{[m]}\times X^{[m]}$ by $y\mapsto(y',y'')$, where
$y'_\eps=y_{\eps0}$ and $y''_\eps=y_{\eps1}$ (the last digit is split off).
Under this identification:
\begin{itemize}
\item $L_i(y)=(L_iy',L_iy'')$ and $U_i(y)=(U_iy',U_iy'')$ for $i\le m$;
\item $L_{m+1}(y)=y'$ and $U_{m+1}(y)=y''$;
\item $y_{\vec0}=y'_{\vec0}$ and $y_{\vec1}=y''_{\vec1}$.
\end{itemize}
For a measure $\lambda$ on $X^{[m]}$ and $\eps\in[m]$, the \emph{vertex marginal} at $\eps$ is
the image of $\lambda$ under $x\mapsto x_\eps$.

\subsubsection*{Spectral measures and Wiener's lemma}
Let $\widehat\Gamma$ be the (compact metrizable) dual group of $\Gamma$; its elements are
homomorphisms $\psi\colon\Gamma\to S^1=\{z\in\mathbb C:|z|=1\}$.
For an invariant measure $\nu$, write $\mathcal U_\gamma f=f\circ T^\gamma$ for the Koopman
representation on $L^2(\nu)$.
For $f\in L^2(\nu)$, the spectral measure $\sigma_f$ is the finite positive measure on
$\widehat\Gamma$ with $\langle\mathcal U_\gamma f,f\rangle=\int\psi(\gamma)\,d\sigma_f(\psi)$
(Bochner--Herglotz); note $\sigma_f(\widehat\Gamma)=\|f\|_2^2$.
An \emph{eigenfunction with eigenvalue $\psi\in\widehat\Gamma$} is a nonzero $f$ with
$\mathcal U_\gamma f=\psi(\gamma)f$ for all $\gamma$. Write $H_\psi$ for the corresponding
eigenspace (including $0$).

\begin{lemma}\label{K:lem:wiener}
Let $\nu$ be invariant and let $\Phi$, $\Theta$ be \Fol{} sequences.
\begin{enumerate}[label=(\alph*)]
\item\label{K:wa} For $\psi\in\widehat\Gamma$, $\mathop{\mathbb E}_{\gamma\in\Phi_N}\psi(\gamma)\to1$ if
$\psi=1$, and $\to0$ otherwise.
\item\label{K:wb} For every unitary representation $V$ of $\Gamma$ on a Hilbert space and every
vector $v$, $\mathop{\mathbb E}_{\gamma\in\Phi_N}V_\gamma v$ converges to the orthogonal
projection of $v$ onto the $V$-invariant vectors.
\item\label{K:wc} For $f\in L^2(\nu)$ and $\psi\in\widehat\Gamma$, $\sigma_f(\{\psi\})=\|P_\psi f\|^2$,
where $P_\psi$ is the orthogonal projection onto $H_\psi$. In particular, $\sigma_f$ has no
atoms if and only if $f$ is orthogonal to every eigenfunction.
\item\label{K:wd} If $\sigma_f$ has no atoms, then for every $g\in L^2(\nu)$,
$\mathop{\mathbb E}_{\beta\in\Theta_M}|\langle\mathcal U_\beta f,g\rangle|\to0$ as $M\to\infty$.
\end{enumerate}
\end{lemma}

\begin{proof}
\ref{K:wa}: If $\psi(g)\ne1$, then
$|1-\psi(g)|\cdot|\mathop{\mathbb E}_{\Phi_N}\psi|\le|\Phi_N\triangle(\Phi_N+g)|/|\Phi_N|\to0$.

\ref{K:wb}: The space splits orthogonally into the invariant vectors and the closed span of
$\{u-V_gu\}$. On a vector $u-V_gu$ the averages are bounded by
$2\|u\|\,|\Phi_N\triangle(\Phi_N+g)|/|\Phi_N|$.

\ref{K:wc}: Apply \ref{K:wb} to $V_\gamma=\overline{\psi(\gamma)}\mathcal U_\gamma$, whose invariant
vectors form $H_\psi$. Then
\[
\|P_\psi f\|^2=\lim_N\mathop{\mathbb E}_{\gamma\in\Phi_N}\overline{\psi(\gamma)}\langle\mathcal U_\gamma f,f\rangle
=\lim_N\int\mathop{\mathbb E}_{\gamma\in\Phi_N}(\psi'\overline\psi)(\gamma)\,d\sigma_f(\psi')
=\sigma_f(\{\psi\})
\]
by \ref{K:wa} and dominated convergence.

\ref{K:wd}: Let $Q$ be the orthogonal projection onto the cyclic subspace $Z(f)$, the closed
span of $\{\mathcal U_\gamma f\}$.
The map $\mathcal U_\gamma f\mapsto(\psi\mapsto\psi(\gamma))$ extends to a unitary isomorphism
$Z(f)\to L^2(\sigma_f)$. Let $\phi$ be the image of $Qg$.
Then $\langle\mathcal U_\beta f,g\rangle=\langle\mathcal U_\beta f,Qg\rangle=\int\psi(\beta)\,d\tau(\psi)$, where
$\tau=\overline\phi\,\sigma_f$ is a complex measure without atoms.
By Fubini and \ref{K:wa},
\[
\mathop{\mathbb E}_{\beta\in\Theta_M}\Big|\int\psi(\beta)\,d\tau\Big|^2
=\iint\mathop{\mathbb E}_{\beta\in\Theta_M}(\psi\overline{\psi'})(\beta)\,d\tau(\psi)\,d\overline\tau(\psi')
\longrightarrow\sum_{\psi}|\tau(\{\psi\})|^2=0 .
\]
The Cauchy--Schwarz inequality finishes the proof.
\end{proof}

\subsubsection*{Kronecker maps}
Let $(Y,T)$ be a topological $\Gamma$-system and $\lambda$ an ergodic measure on $Y$.
The \emph{Kronecker subspace} $H_c(\lambda)\subset L^2(\lambda)$ is the closed linear span of
all eigenfunctions.
A \emph{rotation} is a pair $(K,\alpha)$, where $K$ is a compact metrizable abelian group
(written additively, with Haar measure $m_K$) and $\alpha\colon\Gamma\to K$ is a homomorphism
with dense image; $\Gamma$ acts on $K$ by $z\mapsto z+\alpha(\gamma)$.

\begin{definition}\label{K:def:kron}
A \emph{Kronecker map} for $(Y,\lambda)$ is a Borel map $\kappa\colon Y\to K$ to a rotation
$(K,\alpha)$ with the following properties:
\begin{itemize}
\item $\kappa_*\lambda=m_K$;
\item for every $\gamma$, $\kappa(T^\gamma y)=\kappa(y)+\alpha(\gamma)$ for $\lambda$-a.e.\ $y$;
\item $L^2(Y,\sigma(\kappa),\lambda)=H_c(\lambda)$.
\end{itemize}
We write $\mathbb E_\lambda(f\mid\kappa)$ for the orthogonal projection onto $H_c(\lambda)$,
that is, the conditional expectation onto $\sigma(\kappa)$.
\end{definition}

\begin{lemma}\label{K:lem:kron}
Let $\lambda$ be ergodic.
\begin{enumerate}[label=(\alph*)]
\item\label{K:ka} A Kronecker map for $(Y,\lambda)$ exists.
\item\label{K:kb} Let $\kappa\colon Y\to K$ be any Kronecker map, with rotation $\alpha$. Every
eigenfunction is a.e.\ equal to $c\cdot(\varphi\circ\kappa)$ for a constant $c$ and a
character $\varphi\in\widehat K$. Here $\varphi$ is the unique character with
$\varphi\circ\alpha$ equal to the eigenvalue.
\item\label{K:kc} For $f\in L^2(\lambda)$: $\mathbb E_\lambda(f\mid\kappa)=0$ if and only if
$\sigma_f$ has no atoms.
\end{enumerate}
\end{lemma}

\begin{proof}
\ref{K:ka}: This is the discrete spectrum theorem (Halmos--von Neumann; for amenable groups
see \cite[Proposition~2.12]{charamaras2024finding}). For completeness we sketch the
abelian case.
\begin{itemize}
\item The modulus of an eigenfunction is invariant, hence constant. Products of
eigenfunctions are eigenfunctions. So the set $\mathcal E\subset\widehat\Gamma$ of
eigenvalues is a subgroup.
\item Eigenfunctions with distinct eigenvalues are orthogonal and $L^2(\lambda)$ is
separable, so $\mathcal E$ is countable.
\item Let $K$ be the dual of the discrete group $\mathcal E$, which is compact and metrizable,
and let $\alpha(\gamma)=(\psi\mapsto\psi(\gamma))$. The image $\alpha(\Gamma)$ is dense, since
its annihilator in $\mathcal E$ is trivial.
\item Choose unimodular eigenfunctions $e_\psi$ for $\psi\in\mathcal E$, and set
$\Theta(y)=(e_\psi(y))_{\psi}\in(S^1)^{\mathcal E}$. Then
$\Theta(T^\gamma y)=\alpha(\gamma)\cdot\Theta(y)$ a.e., where $K$ is viewed as a closed subgroup
of $(S^1)^{\mathcal E}$.
\item Hence $\Theta_*\lambda$ is invariant under $\alpha(\Gamma)$, and so under $K$.
\item The image of $\Theta$ in $(S^1)^{\mathcal E}/K$ is invariant, so by ergodicity it is
a.e.\ a constant coset $Ku_0$.
\item Put $\kappa=\Theta\cdot u_0^{-1}$, redefined on a null set to take values in $K$.
Then $\kappa_*\lambda=m_K$ and $\kappa$ is a.e.\ equivariant.
\item By ergodicity each eigenspace is one-dimensional, spanned by $e_\psi$. The characters
of $K$ are the evaluations at elements of $\mathcal E$, so the functions $\chi\circ\kappa$,
$\chi\in\widehat K$, are constant multiples of the $e_\psi$. Hence $L^2(\sigma(\kappa))$, the
closed span of these functions, equals $H_c(\lambda)$.
\end{itemize}

\ref{K:kb}: Let $h$ be an eigenfunction with eigenvalue $\psi$. Then $h\in H_c(\lambda)$, so
$h=\tilde h\circ\kappa$ for some $\tilde h\in L^2(m_K)$, and
$\tilde h(z+\alpha(\gamma))=\psi(\gamma)\tilde h(z)$ for $m_K$-a.e.\ $z$.
Expanding $\tilde h=\sum_\varphi a_\varphi\varphi$ in characters gives
$a_\varphi(\varphi(\alpha(\gamma))-\psi(\gamma))=0$ for all $\varphi$ and $\gamma$.
Since $\alpha(\Gamma)$ is dense, at most one $\varphi$ satisfies $\varphi\circ\alpha=\psi$.

\ref{K:kc}: $H_c(\lambda)$ is the closed span of the eigenspaces $H_\psi$, so this follows
from \cref{K:lem:wiener}\ref{K:wc}.
\end{proof}

\subsubsection*{Disintegrations}
Let $\kappa\colon Y\to K$ be a Kronecker map for $(Y,\lambda)$.
There is a Borel family $z\mapsto\lambda_z$ of probability measures on $Y$ with the following
properties \cite[Theorem~2.2 and (2.1)]{charamaras2024finding}:
\begin{itemize}
\item $\lambda=\int_K\lambda_z\,dm_K(z)$;
\item $\lambda_z(\kappa^{-1}(z))=1$ for $m_K$-a.e.\ $z$;
\item for bounded Borel $f$ on $Y$,
$\mathbb E_\lambda(f\mid\kappa)=\tilde f\circ\kappa$ $\lambda$-a.e., where
\[
\tilde f(z):=\int f\,d\lambda_z .
\]
\end{itemize}
The family is unique up to $m_K$-null sets. It is a priori defined only for $m_K$-a.e.\
$z$; we extend it to all $z\in K$ by setting $\lambda_z:=\lambda$ on the exceptional null
set, which affects none of the statements below.
If $0\le f\le 1$, then $0\le\tilde f\le1$ a.e.
Since $\sigma(\kappa)$ is invariant modulo null sets, $\mathbb E_\lambda(\cdot\mid\kappa)$
commutes with $\mathcal U_\gamma$. Hence $\widetilde{f\circ T^\gamma}(z)=\tilde f(z+\alpha(\gamma))$
for $m_K$-a.e.\ $z$.

\subsection{The correspondence principle and Erdős cubes}\label{K:ssec:corr}

The following is Furstenberg's correspondence principle~\cite{furstenberg1977ergodic}, in the
form given by Charamaras and Mountakis.

\begin{theorem}[{\cite[Theorem~2.15]{charamaras2024finding}}]\label{K:thm:corr}
Let $A\subset\Gamma$, and let $\Phi$ be a \Fol{} sequence such that
$\dens_\Phi(A)=\lim_N|A\cap\Phi_N|/|\Phi_N|$ exists.
Then there exist an ergodic system $(X,\mu,T)$ on a compact metric space $X$, a clopen set
$E\subset X$, a \Fol{} sequence $\Psi$ and a point $a\in\mathrm{gen}(\mu,\Psi)$ such that
$\mu(E)\ge\dens_\Phi(A)$ and $A=\{\gamma\in\Gamma:T^\gamma a\in E\}$.
\end{theorem}

In~\cite{charamaras2024finding} this is stated for countable amenable groups. In the proof,
$X=\{0,1\}^\Gamma$ with the shift action and $a=1_A$.
We do \emph{not} need the strengthened form in which measurable eigenfunctions are
topological (property~(4) of \cite[Theorem~2.5]{kra2026short}). The re-basing lemma
(\cref{K:lem:rebase}) takes its place.

\begin{definition}[Erdős cubes; cf.\ {\cite[Definition~4.1]{kra2022infinite}}]\label{K:def:cube}
Let $(X,T)$ be a topological $\Gamma$-system and $k\ge1$.
A point $y\in X^{[k]}$ is a \emph{$k$-dimensional Erdős cube} if, for each $1\le i\le k$,
there are pairwise distinct $\beta_1,\beta_2,\ldots\in\Gamma$ with
$T^{\beta_n}L_i(y)\to U_i(y)$ in $X^{[k-1]}$.
\end{definition}

The next lemma is \cite[Lemma~4.5 and Corollary~4.6]{kra2022infinite}, transferred to
$\Gamma$. We formulate the acceptability condition so that it is meaningful while some of
the sets are still empty.

\begin{lemma}\label{K:lem:cube-sumset}
Let $y\in X^{[k]}$ be an Erdős cube, and let $E\subset X$ be open with $y_{\vec1}\in E$.
Then there are infinite sets $B_1,\dots,B_k\subset\Gamma$ such that
$T^{b_1+\dots+b_k}y_{\vec0}\in E$ for all $b_i\in B_i$.
\end{lemma}

\begin{proof}
For $\eps\in[k]$ let $Z(\eps)=\{j:\eps_j=0\}$.
Call a tuple $(B_1,\dots,B_k)$ of finite sets \emph{acceptable} if, for every $\eps\in[k]$
and every choice of $b_j\in B_j$ for $j\in Z(\eps)$,
\begin{equation}\label{K:eq:accept}
T^{\sum_{j\in Z(\eps)}b_j}\,y_\eps\in E .
\end{equation}
For $\eps=\vec1$ the sum is empty, and the condition reads $y_{\vec1}\in E$.
Hence $(\emptyset,\dots,\emptyset)$ is acceptable.

Let $(B_1,\dots,B_k)$ be acceptable, and fix $i$.
For $\eps$ with $\eps_i=0$, let
\[
H_\eps=\bigcap T^{-\sum_{j\in Z(\eps)\setminus\{i\}}b_j}E ,
\]
the intersection over all choices $b_j\in B_j$, $j\in Z(\eps)\setminus\{i\}$. (If some of
these $B_j$ is empty there is no choice and $H_\eps=X$; if $Z(\eps)\setminus\{i\}=\emptyset$
there is exactly one, empty, choice and $H_\eps=E$.)
This is a finite intersection of open sets, so $H_\eps$ is open.
Since $Z(\eps+e_i)=Z(\eps)\setminus\{i\}$, acceptability at $\eps+e_i$ says that
$y_{\eps+e_i}\in H_\eps$.

Because $y$ is an Erdős cube, there are pairwise distinct $\beta_n$ with
$T^{\beta_n}y_\eps\to y_{\eps+e_i}$ for all $\eps$ with $\eps_i=0$.
Hence for all large $n$ we have $T^{\beta_n}y_\eps\in H_\eps$ for every such $\eps$.
Choose such a $b^*=\beta_n$ with $b^*\notin B_1\cup\dots\cup B_k$, and put $B'_i=B_i\cup\{b^*\}$ and $B'_j=B_j$
for $j\ne i$.

We check~\eqref{K:eq:accept} for the new tuple. Let $\eps\in[k]$, and choose $b_j\in B'_j$
for $j\in Z(\eps)$. There are two cases.
\begin{itemize}
\item If $i\notin Z(\eps)$, or if $b_i\in B_i$, then~\eqref{K:eq:accept} is an old condition.
\item If $i\in Z(\eps)$ and $b_i=b^*$, then
$T^{\sum_{j\in Z(\eps)}b_j}y_\eps=T^{\sum_{j\in Z(\eps)\setminus\{i\}}b_j}(T^{b^*}y_\eps)\in E$,
because $T^{b^*}y_\eps\in H_\eps$. This step uses commutativity of $\Gamma$.
\end{itemize}

Now add elements in the order $i=1,2,\dots,k,1,2,\dots$.
This produces increasing acceptable tuples whose unions $B_1,\dots,B_k$ are infinite.
Each finite subfamily is contained in an acceptable stage, so~\eqref{K:eq:accept} holds for
the unions. Taking $\eps=\vec0$ gives the claim. (Since $b^*$ is always chosen outside
$B_1\cup\dots\cup B_k$, the sets $B_i$ are in addition pairwise disjoint.)
\end{proof}

\subsection{Relative products over the Kronecker factor}\label{K:ssec:relprod}

Throughout this subsection, $(Y,T)$ is a topological $\Gamma$-system, $\lambda$ is an ergodic
measure on $Y$, and $\kappa\colon Y\to K$ is a Kronecker map with rotation $\alpha$ and
disintegration $(\lambda_z)$.

\begin{definition}\label{K:def:J}
For $t\in K$, the \emph{relative product with parameter $t$} is the measure on $Y\times Y$
\[
J^\lambda_t:=\int_K\lambda_z\otimes\lambda_{z+t}\,dm_K(z).
\]
The \emph{graph measure} is $\rho_\lambda:=(\mathrm{id},\kappa)_*\lambda$ on $Y\times K$.
Here $\Gamma$ acts on $Y\times K$ by $T^\gamma\times(\cdot+\alpha(\gamma))$.
We write $\Sigma_\lambda:=\supp\rho_\lambda$.
\end{definition}

For $\Gamma=\Z$ these measures appear in \cite[(3.10) and Proposition~3.11]{kra2022infinite},
and for amenable groups in \cite[Definition~4.7 and Theorem~4.8]{charamaras2024finding}, where
$(y,y')\mapsto J^\lambda_{\kappa(y')-\kappa(y)}$ is shown to be a continuous ergodic
decomposition of $\lambda\times\lambda$ when $\kappa$ is continuous. \cref{K:lem:J} below is
the version we need; we include its short proof.

Translations preserve $m_K$-null sets. Hence $J^\lambda_t$ does not depend on the version
of the disintegration, and it is defined for \emph{every} $t\in K$.
The graph measure $\rho_\lambda$ is invariant and concentrated on the graph of $\kappa$.
The projection onto $Y$ is therefore a measure-theoretic isomorphism
$(Y\times K,\rho_\lambda)\to(Y,\lambda)$, and $\rho_\lambda$ is ergodic.
For Borel sets $S\subset Y$ and $O\subset K$ we have
\begin{equation}\label{K:eq:rho}
\rho_\lambda(S\times O)=\int_O\lambda_z(S)\,dm_K(z).
\end{equation}

\begin{lemma}\label{K:lem:J}
\leavevmode
\begin{enumerate}[label=(\alph*)]
\item\label{K:Ja} For bounded Borel $f,g$ on $Y$ and every $t\in K$,
$\int f\otimes g\,dJ^\lambda_t=\int_K\tilde f(z)\,\tilde g(z+t)\,dm_K(z)$.
\item\label{K:Jb} Each $J^\lambda_t$ is invariant under the diagonal action, has both marginals
equal to $\lambda$, and is concentrated on $\{(y,y'):\kappa(y')-\kappa(y)=t\}$.
The map $t\mapsto J^\lambda_t$ is weak$^*$ continuous.
\item\label{K:Jc} $J^\lambda_t$ is ergodic for $m_K$-almost every $t$.
\end{enumerate}
\end{lemma}

\begin{proof}
\ref{K:Ja} is the definition.

\ref{K:Jb}: The marginals are $\int\lambda_z\,dm_K=\lambda$ and $\int\lambda_{z+t}\,dm_K=\lambda$.
Concentration follows from $\lambda_z(\kappa^{-1}(z))=1$ for a.e.\ $z$.
For invariance, use \ref{K:Ja} and $\widetilde{f\circ T^\gamma}=\tilde f(\cdot+\alpha(\gamma))$:
\[
J^\lambda_t\big((f\circ T^\gamma)\otimes(g\circ T^\gamma)\big)
=\int\tilde f(z+\alpha(\gamma))\,\tilde g(z+t+\alpha(\gamma))\,dm_K(z)
=J^\lambda_t(f\otimes g).
\]
For continuity, note that $t\mapsto\int\tilde f(z)\tilde g(z+t)\,dm_K$ is continuous, by
continuity of translation in $L^2(m_K)$. Linear combinations of functions $f\otimes g$ with
$f,g\in C(Y)$ are uniformly dense in $C(Y\times Y)$.

\ref{K:Jc}: Let $D(y,y')=\kappa(y')-\kappa(y)$ and $\Lambda=\lambda\times\lambda$.

\emph{Step 1: $J$ disintegrates $\Lambda$ over $D$.}
By \ref{K:Ja} and Fubini, for bounded Borel $f,g$ on $Y$ and $h$ on $K$,
\begin{equation}\label{K:eq:disD}
\int f\otimes g\cdot(h\circ D)\,d\Lambda=\iint\tilde f(z)\,\tilde g(z')\,h(z'-z)\,dm_K\,dm_K
=\int_K h(t)\,J^\lambda_t(f\otimes g)\,dm_K(t).
\end{equation}
By a monotone class argument, \eqref{K:eq:disD} holds with $f\otimes g$ replaced by any bounded
Borel $F$ on $Y\times Y$, and $t\mapsto J^\lambda_t(F)$ is Borel.
Together with \ref{K:Jb}, this shows that $t\mapsto J^\lambda_t$ is the disintegration of
$\Lambda$ over $D$. In particular $\mathbb E_\Lambda(F\mid\sigma(D))(y,y')=J^\lambda_{D(y,y')}(F)$
for $\Lambda$-a.e.\ $(y,y')$ (recall $D_*\Lambda=m_K$).

\emph{Step 2: invariant functions of $\Lambda$ are $\sigma(D)$-measurable.}
Write $L^2(\lambda)=H_c\oplus H_w$, where $H_c=L^2(\sigma(\kappa))$ has orthonormal basis
$\{\varphi\circ\kappa:\varphi\in\widehat K\}$. Let $P_1$ be the projection onto the
invariant functions in $L^2(\Lambda)$.
\begin{itemize}
\item For $f\in H_w$ and $g\in L^2(\lambda)$, we have
$\langle\mathcal U_\gamma(f\otimes g),f\otimes g\rangle=\widehat{\sigma_f}(\gamma)\widehat{\sigma_g}(\gamma)$,
so $\sigma_{f\otimes g}=\sigma_f*\sigma_g$. This has no atoms because $\sigma_f$ has none
(\cref{K:lem:kron}\ref{K:kc}). Hence $P_1(f\otimes g)=0$ by \cref{K:lem:wiener}\ref{K:wc}.
The same holds for $g\otimes f$.
\item The function $\varphi\circ\kappa\otimes\varphi'\circ\kappa$ is an eigenfunction with
eigenvalue $(\varphi\varphi')\circ\alpha$. Since $\alpha(\Gamma)$ is dense, this eigenvalue
is trivial only if $\varphi'=\overline\varphi$.
\end{itemize}
Since $L^2(\Lambda)$ is the orthogonal sum of $H_c\otimes H_c$, $H_w\otimes L^2(\lambda)$ and
$H_c\otimes H_w$, and $P_1$ maps an eigenfunction with nontrivial eigenvalue to $0$,
$P_1L^2(\Lambda)$ is the closed span of the functions
$\varphi(\kappa(y))\overline{\varphi(\kappa(y'))}=\overline{\varphi}(D(y,y'))$, all of
which are $\sigma(D)$-measurable. Conversely $D$ is invariant.
So $P_1F=\mathbb E_\Lambda(F\mid\sigma(D))$, which is the function
$(y,y')\mapsto J^\lambda_{D(y,y')}(F)$ by Step~1.

\emph{Step 3: ergodicity of almost every $J^\lambda_t$.}
Fix a \Fol{} sequence $\Phi$, write
$A_NF=\mathop{\mathbb E}_{\gamma\in\Phi_N}F\circ T^\gamma$, and let $F\in C(Y\times Y)$.
The function $t\mapsto\|A_NF-J^\lambda_t(F)\|^2_{L^2(J^\lambda_t)}$ is continuous by \ref{K:Jb}.
Since $J^\lambda_t$ is concentrated on $D^{-1}(t)$,
\[
\int_K\big\|A_NF-J^\lambda_t(F)\big\|^2_{L^2(J^\lambda_t)}\,dm_K(t)
=\big\|A_NF-J^\lambda_{D(\cdot)}(F)\big\|^2_{L^2(\Lambda)}\xrightarrow[N\to\infty]{}0
\]
by \cref{K:lem:wiener}\ref{K:wb} and Step~2.
For each $t$, the integrand converges to $\|P^{(t)}_1F-J^\lambda_t(F)\|^2_{L^2(J^\lambda_t)}$,
where $P^{(t)}_1$ is the invariant projection in $L^2(J^\lambda_t)$.
By Fatou's lemma this limit vanishes for a.e.\ $t$.
Applying this to a countable dense subset of $C(Y\times Y)$ shows the following. For a.e.\
$t$, $P^{(t)}_1$ maps a dense subset of $L^2(J^\lambda_t)$, and hence all of it, to
constants. That is, $J^\lambda_t$ is ergodic.
\end{proof}

\begin{lemma}[Kronecker factors of factors]\label{K:lem:factor}
Let $\lambda$ be ergodic on $Y$ with Kronecker map $\kappa\colon Y\to K$ (rotation $\alpha$).
Let $P\colon Y\to Z$ be a Borel map into a topological $\Gamma$-system with
$P_*\lambda=\nu$ and, for each $\gamma$, $P\circ T^\gamma=T^\gamma\circ P$ $\lambda$-a.e., and let $\iota\colon Z\to K'$ be a Kronecker map for $\nu$ (rotation $\alpha'$).
Then:
\begin{enumerate}[label=(\alph*)]
\item\label{K:Fa} $\mathbb E_\lambda(f\circ P\mid\kappa)=\mathbb E_\nu(f\mid\iota)\circ P$ for $f\in L^2(\nu)$;
\item\label{K:Fb} there is a unique continuous homomorphism $\bar\chi\colon K\to K'$ with
$\bar\chi\circ\alpha=\alpha'$. It is surjective, and there is $c\in K'$ with
$\iota\circ P=\bar\chi\circ\kappa+c$ $\lambda$-a.e.;
\item\label{K:Fc} for \emph{every} $s\in K$, $(P\times P)_*J^\lambda_s=J^\nu_{\bar\chi(s)}$.
\end{enumerate}
\end{lemma}

\begin{proof}
\ref{K:Fa}: Write $f=f_c+f_w$ with $f_c\in H_c(\nu)$ and $f_w\perp H_c(\nu)$.
Eigenfunctions of $\nu$ pull back under $P$ to eigenfunctions of $\lambda$, so
$f_c\circ P\in H_c(\lambda)$.
Since $P$ is a factor map, $\sigma_{f_w\circ P}=\sigma_{f_w}$, which has no atoms. Hence
$f_w\circ P\perp H_c(\lambda)$ by \cref{K:lem:kron}\ref{K:kc}.

\ref{K:Fb}: Uniqueness holds because $\alpha(\Gamma)$ is dense.
For existence, take $\psi\in\widehat{K'}$. Then $\psi\circ\iota\circ P$ is an eigenfunction of
$\lambda$ with eigenvalue $\psi\circ\alpha'$.
By \cref{K:lem:kron}\ref{K:kb} there is a unique $\varphi_\psi\in\widehat K$ with
$\varphi_\psi\circ\alpha=\psi\circ\alpha'$. By uniqueness, $\psi\mapsto\varphi_\psi$ is a
homomorphism $\widehat{K'}\to\widehat K$.
By Pontryagin duality there is a continuous homomorphism $\bar\chi\colon K\to K'$ with
$\psi\circ\bar\chi=\varphi_\psi$ for all $\psi$. Then $\bar\chi\circ\alpha=\alpha'$.
The image $\bar\chi(K)$ is closed and contains the dense set $\alpha'(\Gamma)$, so $\bar\chi$
is onto.
Finally, $\iota\circ P-\bar\chi\circ\kappa$ is $\lambda$-a.e.\ invariant, since both terms
are shifted by $\alpha'(\gamma)$ under $T^\gamma$. By ergodicity it is a.e.\ equal to a
constant $c$.

\ref{K:Fc}: Let $f,g$ be bounded Borel on $Z$.
By \ref{K:Fa} and \ref{K:Fb},
$\mathbb E_\lambda(f\circ P\mid\kappa)=\tilde f(\iota\circ P)=\tilde f(\bar\chi\circ\kappa+c)$,
so $\widetilde{f\circ P}(z)=\tilde f(\bar\chi z+c)$ for $m_K$-a.e.\ $z$. Here
$\tilde f$ is formed with respect to $(\nu,\iota)$.
The surjective homomorphism $\bar\chi$ maps $m_K$ to $m_{K'}$. Hence, by \cref{K:lem:J}\ref{K:Ja},
\[
\int f\otimes g\,d(P\times P)_*J^\lambda_s=\int_K\tilde f(\bar\chi z+c)\,\tilde g(\bar\chi z+\bar\chi s+c)\,dm_K(z)
=\int_{K'}\tilde f(u)\,\tilde g(u+\bar\chi s)\,dm_{K'}(u),
\]
and the right-hand side is $J^\nu_{\bar\chi(s)}(f\otimes g)$.
Both sides are probability measures, so they agree.
\end{proof}

The next lemma is \cite[Lemma~3.18]{kra2022infinite} for $\Gamma=\Z$, and the Claim in the
proof of \cite[Theorem~4.10]{charamaras2024finding} for amenable groups. We include a short
proof, based on the van der Corput inequality rather than on the continuity of the ergodic
decomposition. In the argument of~\cite{kra2026short} its role is played by the choice of a
generic $y$ such that $(x,y)$ is generic for some (not necessarily ergodic) joining.

\begin{lemma}[Generic pairs]\label{K:lem:pairs}
Suppose in addition that $\pi\colon Y\to K$ is a \emph{continuous} Kronecker map with
$\pi(T^\gamma y)=\pi(y)+\alpha(\gamma)$ for \emph{all} $y$ and $\gamma$. Let $(J_t)$ be the
relative products formed with respect to $\pi$.
Let $y_0\in\mathrm{gen}(\lambda,\Phi)$.
Then there are a subsequence $\Phi'$ of $\Phi$ and a Borel set $Y_0\subset Y$ with
$\lambda(Y_0)=1$ such that
\[
(y_0,y)\in\mathrm{gen}\big(J_{\pi(y)-\pi(y_0)},\Phi'\big)\qquad\text{for every }y\in Y_0 .
\]
\end{lemma}

\begin{proof}
Let $f,g\in C(Y)$.
We show that, in $L^2(\lambda)$ (as functions of $y$),
\begin{equation}\label{K:eq:L2conv}
\mathop{\mathbb E}_{\gamma\in\Phi_N}f(T^\gamma y_0)\,g(T^\gamma y)\ \longrightarrow\
\Lambda_{f,g}(y):=J_{\pi(y)-\pi(y_0)}(f\otimes g)=\int\tilde f(z)\,\tilde g(z+\pi(y)-\pi(y_0))\,dm_K(z).
\end{equation}
Fix $\tau>0$. Choose a trigonometric polynomial $\tilde f'=\sum_jc_j\varphi_j$ on $K$
with $\|\tilde f-\tilde f'\|_{L^2(m_K)}<\tau$, and put $f'=\tilde f'\circ\pi$.
The function $f'$ is continuous because $\pi$ is.
Let $f''=f-f'=f^\perp+f^\sharp$, with $f^\perp=f-\mathbb E_\lambda(f\mid\pi)$ and
$f^\sharp=(\tilde f-\tilde f')\circ\pi$. Then $\|f^\sharp\|_2<\tau$, and $\sigma_{f^\perp}$ has no atoms
(\cref{K:lem:kron}\ref{K:kc}).

\emph{The part $f'$.}
We have $f'(T^\gamma y_0)=\sum_jc_j\varphi_j(\pi y_0)\,\varphi_j(\alpha(\gamma))$.
By \cref{K:lem:wiener}\ref{K:wb} applied to $V_\gamma=\varphi_j(\alpha(\gamma))\mathcal U_\gamma$, the
average $\mathop{\mathbb E}_{\Phi_N}\varphi_j(\alpha(\gamma))\mathcal U_\gamma g$ converges to the
projection of $g$ onto $H_{\overline{\varphi_j\circ\alpha}}$. That eigenspace is spanned by
$\overline{\varphi_j}\circ\pi$, so the projection is
$\big(\int\tilde g\,\varphi_j\,dm_K\big)\,\overline{\varphi_j}\circ\pi$.
Hence the averages for $f'$ converge in $L^2(\lambda)$ to
\[
\Lambda'_{f,g}(y)=\sum_jc_j\,\varphi_j(\pi y_0)\,\overline{\varphi_j(\pi y)}\int\tilde g\,\varphi_j\,dm_K .
\]
A change of variables shows that $\Lambda'_{f,g}(y)=\int\tilde f'(z)\tilde g(z+\pi(y)-\pi(y_0))\,dm_K(z)$.
It follows that $|\Lambda_{f,g}-\Lambda'_{f,g}|\le\tau\|g\|_\infty$ everywhere.

\emph{The part $f''$.}
Put $u(\gamma)=f''(T^\gamma y_0)$. By the spectral theorem,
\[
\Big\|\mathop{\mathbb E}_{\gamma\in\Phi_N}u(\gamma)\mathcal U_\gamma g\Big\|_2^2=\int_{\widehat\Gamma}\Big|\mathop{\mathbb E}_{\gamma\in\Phi_N}u(\gamma)\psi(\gamma)\Big|^2d\sigma_g(\psi).
\]
We bound the integrand uniformly in $\psi$ by a van der Corput argument.
Fix a finite $H\subset\Gamma$. By the \Fol{} property,
$\mathop{\mathbb E}_{\Phi_N}u\psi=\mathop{\mathbb E}_{\gamma\in\Phi_N}\mathop{\mathbb E}_{h\in H}(u\psi)(\gamma+h)+o(1)$.
By Cauchy--Schwarz,
\[
\Big|\mathop{\mathbb E}_{\Phi_N}u\psi\Big|^2\le\frac1{|H|^2}\sum_{h,h'\in H}\Big|\mathop{\mathbb E}_{\gamma\in\Phi_N}f''(T^{\gamma+h}y_0)\overline{f''(T^{\gamma+h'}y_0)}\Big|+o(1).
\]
The function $f''$ is continuous and $y_0$ is generic, so the inner averages converge to
$\langle\mathcal U_{h-h'}f'',f''\rangle$.
Moreover,
$|\langle\mathcal U_\delta f'',f''\rangle-\langle\mathcal U_\delta f^\perp,f^\perp\rangle|\le\tau(\|f^\perp\|_2+\|f''\|_2)\le\tau(2\|f\|_\infty+\tau)$.
Take $H=\Theta_M$ for a \Fol{} sequence $\Theta$. By Cauchy--Schwarz and the computation
in \cref{K:lem:wiener}\ref{K:wd},
\[
\frac1{|\Theta_M|^2}\sum_{h,h'\in\Theta_M}|\langle\mathcal U_{h-h'}f^\perp,f^\perp\rangle|\le\Big(\iint\Big|\mathop{\mathbb E}_{h\in\Theta_M}(\psi\overline{\psi'})(h)\Big|^2d\sigma_{f^\perp}\,d\sigma_{f^\perp}\Big)^{1/2}\to0 .
\]
Hence $\sup_\psi\limsup_N|\mathop{\mathbb E}_{\Phi_N}u\psi|^2\le\tau(2\|f\|_\infty+\tau)$.
The integrands are bounded and $\sigma_g$ is finite, so by reverse Fatou
\[
\limsup_N\Big\|\mathop{\mathbb E}_{\Phi_N}u(\gamma)\mathcal U_\gamma g\Big\|_2^2\le\tau(2\|f\|_\infty+\tau)\|g\|_2^2 .
\]
Combining the two parts,
\[
\limsup_N\Big\|\mathop{\mathbb E}_{\gamma\in\Phi_N}f(T^\gamma y_0)\,\mathcal U_\gamma g-\Lambda_{f,g}\Big\|_2
\le\tau\|g\|_\infty+\big(\tau(2\|f\|_\infty+\tau)\big)^{1/2}\|g\|_2 ,
\]
and letting $\tau\to0$ proves~\eqref{K:eq:L2conv}.

Now let $\mathcal D$ be a countable dense subset of $C(Y)$.
Convergence in $L^2$ implies a.e.\ convergence along a subsequence. A diagonal argument
gives a subsequence $\Phi'$ and a conull Borel set $Y_0$ such that the convergence
in~\eqref{K:eq:L2conv} holds at every $y\in Y_0$ along $\Phi'$, for all $f,g\in\mathcal D$.
Both sides of~\eqref{K:eq:L2conv} are bounded bilinear in $(f,g)$ with respect to the
uniform norm. The span of $\{f\otimes g\}$ is dense in $C(Y\times Y)$.
Hence $(y_0,y)\in\mathrm{gen}(J_{\pi(y)-\pi(y_0)},\Phi')$ for $y\in Y_0$.
\end{proof}

\subsection{Re-basing and two support lemmas}\label{K:ssec:support}

The following lemma is the rotation case of \cite[Lemma~3.5]{charamaras2024finding} (see
the remark following that lemma, which allows any topologically distal factor in place of
the Kronecker factor). Charamaras and Mountakis prove it for amenable groups, extending a
result of Host and Kra for $\Z$-actions \cite[Proposition~6.1]{host2009uniformity} (see also
\cite[Proposition~3.19]{kra2022infinite}); their proof is in
\cite[Appendix~A]{charamaras2024finding}.
For rotations the Ellis-semigroup step of that proof reduces to compactness of $K$, so we
include the short argument.

\begin{lemma}[Re-basing]\label{K:lem:rebase}
Let $(W,T)$ be a topological $\Gamma$-system and let $w\in W$ have dense orbit.
Let $\lambda$ be an ergodic measure on $W$.
Let $\kappa\colon W\to K$ be a Borel map into a compact metrizable abelian group, and
$\alpha\colon\Gamma\to K$ a homomorphism, such that for each $\gamma$,
$\kappa(T^\gamma x)=\kappa(x)+\alpha(\gamma)$ for $\lambda$-a.e.\ $x$.
Let $\hat\lambda=(\mathrm{id},\kappa)_*\lambda$ on $W\times K$, with the action
$T^\gamma\times(\cdot+\alpha(\gamma))$.
Then there are $d\in K$ and a \Fol{} sequence $\Psi$ such that $(w,d)\in\mathrm{gen}(\hat\lambda,\Psi)$.
\end{lemma}

\begin{proof}
The measure $\hat\lambda$ is invariant and ergodic, since it is isomorphic to $\lambda$ via the
projection.
By \ref{K:G4} there are a point $(x_1,z_1)$ and a \Fol{} sequence $\Phi$ with
$(x_1,z_1)\in\mathrm{gen}(\hat\lambda,\Phi)$.

For each $N$, the finitely many maps $T^\gamma$, $\gamma\in\Phi_N$, are continuous at $x_1$,
and the orbit of $w$ is dense. So there is $h_N\in\Gamma$ with
$\sup_{\gamma\in\Phi_N}d_W(T^{\gamma+h_N}w,T^\gamma x_1)<1/N$.

By compactness, $\alpha(h_{N_j})\to u_*$ in $K$ along a subsequence. Put $d=z_1-u_*$ and
$\Psi_j=\Phi_{N_j}+h_{N_j}$; this is again a \Fol{} sequence.
For $F\in C(W\times K)$,
\[
\mathop{\mathbb E}_{\gamma\in\Psi_j}F(T^\gamma w,d+\alpha(\gamma))
=\mathop{\mathbb E}_{\gamma\in\Phi_{N_j}}F\big(T^{\gamma+h_{N_j}}w,\ z_1+\alpha(\gamma)+(\alpha(h_{N_j})-u_*)\big).
\]
The arguments are uniformly within $o(1)$ of $(T^\gamma x_1,z_1+\alpha(\gamma))$: in the first
coordinate by the choice of $h_N$, and in the second by uniform continuity of addition on
$K$. Uniform continuity of $F$ and genericity of $(x_1,z_1)$ give the claim.
\end{proof}

\begin{lemma}[Support of the graph measure]\label{K:lem:stab}
Let $(Z,T)$ be a topological $\Gamma$-system, $\nu$ ergodic on $Z$, and $\iota\colon Z\to K$ a
Kronecker map with rotation $\alpha$. Let $\rho=\rho_\nu$ and $\Sigma=\supp\rho\subset Z\times K$.
Suppose $(p,\theta)\in Z\times K$ is generic for $\rho$ along some \Fol{} sequence, and
$p\in\supp\nu$. Then $(p,\theta)\in\Sigma$.
\end{lemma}

\begin{proof}
The projection of $\Sigma$ to $Z$ is compact and has full $\nu$-measure, so it contains
$\supp\nu$. Pick $\theta_0$ with $(p,\theta_0)\in\Sigma$ and put $\zeta=\theta-\theta_0$.

For $u\in K$ let $S_u(x,t)=(x,t+u)$. These maps commute with the action, so $S_\zeta\Sigma$
is closed and invariant. It contains $S_\zeta(p,\theta_0)=(p,\theta)$, and hence the orbit closure
of $(p,\theta)$, which contains $\Sigma$ by~\ref{K:G2}.
So $\Sigma\subset S_\zeta\Sigma$, that is, $S_{-\zeta}\Sigma\subset\Sigma$, and hence
$S_{-n\zeta}\Sigma\subset\Sigma$ for all $n\ge1$.

By compactness there are $1\le n_1<n_2<\cdots$ with $n_{l+1}-n_l\ge2$ and $n_l\zeta\to\zeta_*$.
Then $(n_{l+1}-n_l)\zeta\to0$, so with $q_l=n_{l+1}-n_l-1\ge1$ we get $-q_l\zeta\to\zeta$.
Since $\Sigma$ is closed, $S_\zeta\Sigma\subset\Sigma$, and so
$(p,\theta)=S_\zeta(p,\theta_0)\in\Sigma$.
\end{proof}

\begin{lemma}[Phase returns]\label{K:lem:phase}
Let $(W,T)$ and $(Z,T)$ be topological $\Gamma$-systems, $\lambda$ ergodic on $W$, and
$w\in\mathrm{gen}(\lambda,\Psi)$ for some \Fol{} sequence $\Psi$.
Let $L,U\colon W\to Z$ be continuous equivariant maps with $L_*\lambda=U_*\lambda=:\nu$.
Let $\iota\colon Z\to K$ be a Kronecker map for $\nu$ with rotation $\alpha$.
\begin{enumerate}[label=(\alph*)]
\item\label{K:Pa} There is $r\in K$ with $\iota\circ U=\iota\circ L+r$ $\lambda$-a.e.
\item\label{K:Pb} If $U(w)\in\supp\nu$, then for every open $V\ni U(w)$ and every open $O\ni0$ in $K$,
there is a \Fol{} sequence $\Omega$ along which the set
\[
\{\gamma\in\Gamma:\ T^\gamma L(w)\in V,\ \alpha(\gamma)\in r+O\}
\]
has positive lower density; in particular this set has positive upper Banach density.
\end{enumerate}
\end{lemma}

\begin{proof}
\ref{K:Pa}: Both $\iota\circ U$ and $\iota\circ L$ are shifted by $\alpha(\gamma)$ under
$T^\gamma$ (a.e.), so their difference is $\lambda$-a.e.\ invariant, hence a.e.\ constant.

\ref{K:Pb}: Let $W'$ be the orbit closure of $w$. Then $\lambda(W')=1$ by~\ref{K:G2}, and $w$
has dense orbit in $W'$.
The map $\kappa'=\iota\circ L$ satisfies the hypotheses of \cref{K:lem:rebase} on $W'$.
That lemma gives $z\in K$ and a \Fol{} sequence $\Omega$ with $(w,z)$ generic along $\Omega$
for $\hat\lambda=(\mathrm{id},\iota\circ L)_*\lambda$.
Consider the continuous equivariant maps $(x,t)\mapsto(Lx,t)$ and $(x,t)\mapsto(Ux,t+r)$
from $W'\times K$ to $Z\times K$.
They push $\hat\lambda$ forward to $(L,\iota\circ L)_*\lambda=\rho_\nu$ and to
$(U,\iota\circ L+r)_*\lambda=(U,\iota\circ U)_*\lambda=\rho_\nu$, respectively.
By~\ref{K:G3}, both $(L(w),z)$ and $(U(w),z+r)$ are generic for $\rho_\nu$ along $\Omega$.
By \cref{K:lem:stab}, $(U(w),z+r)\in\supp\rho_\nu$.
So $V\times(z+r+O)$ is an open set meeting $\supp\rho_\nu$, and by~\ref{K:G1} the set
\[
\{\gamma:(T^\gamma L(w),z+\alpha(\gamma))\in V\times(z+r+O)\}=\{\gamma:T^\gamma L(w)\in V,\ \alpha(\gamma)\in r+O\}
\]
has positive lower density along $\Omega$.
\end{proof}

\begin{lemma}[Pointwise support]\label{K:lem:PS}
Let $\nu$, $\iota\colon Z\to K$, $\rho=\rho_\nu$, $\Sigma=\supp\rho$ and $J_t=J^\nu_t$ be as above.
If $(p,\theta)\in\Sigma$, then for $\rho$-almost every $(q,\varphi)\in Z\times K$,
\[
(p,q)\in\supp J_{\varphi-\theta}.
\]
More precisely, the set of $(q,\varphi)$ with $(p,q)\notin\supp J_{\varphi-\theta}$ is
contained in a $\rho$-null Borel set.
\end{lemma}

\begin{proof}
Fix open sets $O_1\ni p$ and $O_2\subset Z$.
Define Borel functions $f(\xi)=\nu_\xi(O_1)$ and $g(\xi)=\nu_\xi(O_2)$ on $K$, with values in
$[0,1]$, and let $F(t)=J_t(O_1\times O_2)=\int f(\xi)g(\xi+t)\,dm_K(\xi)$.
The function $F$ is continuous.

By~\eqref{K:eq:rho}, $\int_Nf\,dm_K=\rho(O_1\times N)>0$ for every open $N\ni\theta$, because
$O_1\times N$ is a neighbourhood of $(p,\theta)\in\Sigma$.

Let $C=\{\varphi\in K:F(\varphi-\theta)=0\}$, a closed set. We claim $\int_Cg\,dm_K=0$.
Suppose not.
Put $R(k)=\int_Cg(\varphi)\,g(k+\varphi-\theta)\,dm_K(\varphi)$.
Then $R\ge0$, and $R$ is continuous since
$|R(k)-R(k')|\le\|g(\cdot+k)-g(\cdot+k')\|_{L^1(m_K)}$.
Moreover $R(\theta)=\int_Cg^2\,dm_K>0$. So $R>R(\theta)/2$ on some open $N\ni\theta$, and
\[
\int f\,R\,dm_K\ \ge\ \tfrac12R(\theta)\int_Nf\,dm_K>0 .
\]
On the other hand, by Fubini,
\[
\begin{aligned}
\int f(k)R(k)\,dm_K(k)&=\int_Cg(\varphi)\int f(k)\,g(k+\varphi-\theta)\,dm_K(k)\,dm_K(\varphi)\\
&=\int_Cg(\varphi)F(\varphi-\theta)\,dm_K(\varphi)=0,
\end{aligned}
\]
a contradiction.

By~\eqref{K:eq:rho}, $\int_Cg\,dm_K=\rho\big(O_2\times C\big)$. So
\[
\mathcal N_{O_1,O_2}:=\{(q,\varphi):\ q\in O_2,\ J_{\varphi-\theta}(O_1\times O_2)=0\}=O_2\times C
\]
is a $\rho$-null Borel set.
Let $\mathcal B$ be a countable base of $Z$.
If $(p,q)\notin\supp J_{\varphi-\theta}$, then some $O_1\times O_2$ with $O_1,O_2\in\mathcal B$,
$p\in O_1$ and $q\in O_2$ satisfies $J_{\varphi-\theta}(O_1\times O_2)=0$, so
$(q,\varphi)\in\mathcal N_{O_1,O_2}$. The union of these countably many $\rho$-null sets is
$\rho$-null.
\end{proof}

\begin{remark}\label{K:rem:PS}
For a fixed point $p$, the proof uses only that $\theta\in\supp(f_{O_1}m_K)$ for every open
$O_1\ni p$, where $f_{O_1}(\xi)=\nu_\xi(O_1)$. By~\eqref{K:eq:rho} this is equivalent to
$(p,\theta)\in\Sigma$.
\end{remark}

\subsection{Cube data and the cube criterion}\label{K:ssec:criterion}

From now on, $(X,\mu,T)$ is an ergodic system on a compact metric space, $a\in X$, and
$E\subset X$ is open with $\mu(E)>0$.

\begin{definition}\label{K:def:datum}
Let $m\ge0$. An \emph{$m$-dimensional cube datum} is a triple $(w,\lambda,\Psi)$ with the
following properties:
\begin{enumerate}[label=(D\arabic*)]
\item\label{K:D1} $\lambda$ is an ergodic measure on $X^{[m]}$ (for the diagonal action),
$\Psi$ is a \Fol{} sequence, $w\in\mathrm{gen}(\lambda,\Psi)$, and $w_{\vec0}=a$;
\item\label{K:D2} every vertex marginal of $\lambda$ equals $\mu$;
\item\label{K:D3} (face symmetry) $(L_i)_*\lambda=(U_i)_*\lambda=:\nu_i$ for $1\le i\le m$;
\item\label{K:D4} (upper faces in the support) $U_i(w)\in\supp\nu_i$ for $1\le i\le m$.
\end{enumerate}
\end{definition}

If $a\in\mathrm{gen}(\mu,\Psi)$, then $(a,\mu,\Psi)$ is a $0$-dimensional cube datum, since
\ref{K:D3} and \ref{K:D4} are empty.
Each face measure $\nu_i$ is ergodic, being a factor of $\lambda$. In what follows,
$\iota_i\colon X^{[m-1]}\to K_i$ denotes a Kronecker map for $\nu_i$, with rotation $\alpha_i$.

The next proposition is our version of Claim~1 in the proof of
\cite[Theorem~3.1]{kra2026short}, together with the iteration that follows it.
The structure of the argument comes from there:
\begin{itemize}
\item the splitting into a Kronecker part and a weak mixing part;
\item the positivity remark \cite[Remark~2.2]{kra2026short};
\item the averaging of the weak mixing part along \emph{every} \Fol{} sequence.
\end{itemize}
What is new is the treatment of the Kronecker part. We use continuity of translation on
the Kronecker group of the face measure and the phase returns of \cref{K:lem:phase}. No
continuity of eigenfunctions is needed.

\begin{proposition}[Cube criterion]\label{K:prop:criterion}
If an $m$-dimensional cube datum $(w,\lambda,\Psi)$ exists, then there is $x\in X^{[m]}$ such
that $y=(w,x)\in X^{[m+1]}$ is an $(m+1)$-dimensional Erdős cube with $y_{\vec0}=a$ and
$y_{\vec1}=x_{\vec1}\in E$.
\end{proposition}

\begin{proof}
Fix a metric on $X$ and use the maximum metric on products.
An \emph{open box} is a set $V=\prod_{\eps\in[m]}V_\eps\subset X^{[m]}$ with each $V_\eps$
open.
For $1\le i\le m$ and $\beta\in\Gamma$ put $U_i(V)=\prod_{\eta\in[m-1]}V_{\eta^{(i,1)}}$ and
\[
W_i(V,\beta)=\{x\in V:\ T^\beta L_i(x)\in U_i(V)\}
=\{x\in V:\ T^\beta x_{\eta^{(i,0)}}\in V_{\eta^{(i,1)}}\ \forall\eta\in[m-1]\}.
\]
This set is open, and $V\subset U_i^{-1}(U_i(V))$.

\smallskip\noindent\emph{Claim A (direction $m+1$).} If $V$ is open with $\lambda(V)>0$, there are
infinitely many $\beta$ with $T^\beta w\in V$.

This is~\ref{K:G1} together with~\ref{K:D1}.

\smallskip\noindent\emph{Claim B (directions $i\le m$).} If $V$ is an open box with $\lambda(V)>0$,
$1\le i\le m$ and $\delta>0$, then there are infinitely many $\beta$ such that
\[
d\big(T^\beta L_i(w),U_i(w)\big)<\delta\qquad\text{and}\qquad\lambda\big(W_i(V,\beta)\big)>0 .
\]

\emph{Proof of Claim B.}
Let $F=1_{U_i(V)}$ on $X^{[m-1]}$. Since $L_i$ is equivariant,
\begin{equation}\label{K:eq:Wi}
\lambda(W_i(V,\beta))=\int 1_V(x)\,F(L_iT^\beta x)\,d\lambda(x)=\langle F\circ L_i\circ T^\beta,1_V\rangle_\lambda .
\end{equation}
In $L^2(\nu_i)$, write $F=F_c+F_w$, where $F_c=\mathbb E_{\nu_i}(F\mid\iota_i)=G\circ\iota_i$
with $G\colon K_i\to[0,1]$ Borel, and $F_w\perp H_c(\nu_i)$.
Let $r_i$ be the constant of \cref{K:lem:phase}\ref{K:Pa} for $L=L_i$ and $U=U_i$, so that
$\iota_i\circ U_i=\iota_i\circ L_i+r_i$ $\lambda$-a.e.

\emph{Positivity.}
Let $\Delta=\{F_c>0\}$, a $\sigma(\iota_i)$-measurable set. Then
$\int_{\Delta^c}F\,d\nu_i=\int_{\Delta^c}F_c\,d\nu_i=0$, so $\nu_i(U_i(V)\setminus\Delta)=0$
\cite[Remark~2.2]{kra2026short}.
Since $(U_i)_*\lambda=\nu_i$ and $V\subset U_i^{-1}(U_i(V))$, we get $F_c(U_ix)>0$ for
$\lambda$-a.e.\ $x\in V$. Hence
\[
\eta_V:=\int_VF_c(U_ix)\,d\lambda(x)>0 .
\]

\emph{The Kronecker part.}
For every $\beta$ and $\lambda$-a.e.\ $x$ we have
$F_c(T^\beta L_ix)=G(\iota_i(L_ix)+\alpha_i(\beta))$ and
$F_c(U_ix)=G(\iota_i(L_ix)+r_i)$. Here we use that $L_i$ and $U_i$ push $\lambda$ to $\nu_i$,
that $\iota_i$ is a.e.\ equivariant, and the choice of $r_i$.
Since $(\iota_i\circ L_i)_*\lambda=m_{K_i}$,
\[
\big|\langle F_c\circ L_i\circ T^\beta,1_V\rangle_\lambda-\eta_V\big|
\le\int_{K_i}\big|G(z+\alpha_i(\beta))-G(z+r_i)\big|\,dm_{K_i}(z)
=\big\|G(\cdot+\alpha_i(\beta)-r_i)-G\big\|_{L^1(m_{K_i})}.
\]
By continuity of translation in $L^1(m_{K_i})$, there is an open $O\ni0$ in $K_i$ such that
the right-hand side is $<\eta_V/4$ whenever $\alpha_i(\beta)\in r_i+O$.

\emph{The weak mixing part.}
Let $h=F_w\circ L_i\in L^2(\lambda)$, so that $F_w\circ L_i\circ T^\beta=\mathcal U_\beta h$.
Since $L_i$ is a factor map,
$\langle\mathcal U_\beta h,h\rangle_\lambda=\langle\mathcal U_\beta F_w,F_w\rangle_{\nu_i}$, so
$\sigma_h=\sigma_{F_w}$, which has no atoms by \cref{K:lem:kron}\ref{K:kc}.
By \cref{K:lem:wiener}\ref{K:wd}, for every \Fol{} sequence $\Theta$,
\[
\mathop{\mathbb E}_{\beta\in\Theta_M}|\langle\mathcal U_\beta h,1_V\rangle_\lambda|\to0 .
\]

\emph{Assembly.}
Let $\mathcal N_\delta$ be the open $\delta$-ball around $U_i(w)$, and let
\[
R_\delta=\{\beta:\ T^\beta L_i(w)\in\mathcal N_\delta,\ \alpha_i(\beta)\in r_i+O\}.
\]
By~\ref{K:D4} and \cref{K:lem:phase}\ref{K:Pb}, applied with $L=L_i$, $U=U_i$ and
$\nu=\nu_i$, there are a \Fol{} sequence $\Omega$ and $c_0>0$ with
$|R_\delta\cap\Omega_M|\ge c_0|\Omega_M|$ for all large $M$.
Let $R'_\delta=\{\beta\in R_\delta:|\langle\mathcal U_\beta h,1_V\rangle|\ge\eta_V/4\}$. By the weak
mixing part,
$|R'_\delta\cap\Omega_M|\le(4/\eta_V)\sum_{\beta\in\Omega_M}|\langle\mathcal U_\beta h,1_V\rangle|=o(|\Omega_M|)$.
Hence $|(R_\delta\setminus R'_\delta)\cap\Omega_M|\ge(c_0/2)|\Omega_M|$ for all large $M$, and
$R_\delta\setminus R'_\delta$ is infinite by \cref{K:lem:folner-grow}.
For $\beta\in R_\delta\setminus R'_\delta$, by~\eqref{K:eq:Wi},
\[
\lambda(W_i(V,\beta))=\langle F_c\circ L_i\circ T^\beta,1_V\rangle+\langle\mathcal U_\beta h,1_V\rangle
\ge\eta_V-\tfrac{\eta_V}4-\tfrac{\eta_V}4>0 .
\]
This proves Claim~B.

\smallskip\noindent\emph{Nested boxes.}
Every open subset of $X^{[m]}$ is a countable union of open boxes whose closures lie in it
and whose coordinates have arbitrarily small diameter. So every open set of positive
$\lambda$-measure contains such a box of positive measure.
Let $V_0=\{x\in X^{[m]}:x_{\vec1}\in E\}$. This is an open box, and
$\lambda(V_0)=\mu(E)>0$ by~\ref{K:D2}.
Let $i_n\in\{1,\dots,m+1\}$ with $i_n\equiv n\pmod{m+1}$.

Suppose an open box $V_n$ with $\lambda(V_n)>0$ has been chosen. Choose
$\beta_n\notin\{\beta_0,\dots,\beta_{n-1}\}$ as follows.
\begin{itemize}
\item If $i_n=m+1$, let $\beta_n$ satisfy $T^{\beta_n}w\in V_n$ (Claim~A), and put $W^*=V_n$.
\item If $i_n=i\le m$, let $\beta_n$ satisfy $d(T^{\beta_n}L_i(w),U_i(w))<1/(n+1)$ and
$\lambda(W_i(V_n,\beta_n))>0$ (Claim~B), and put $W^*=W_i(V_n,\beta_n)$.
\end{itemize}
Then choose an open box $V_{n+1}$ with $\overline{V_{n+1}}\subset W^*$, $\lambda(V_{n+1})>0$, and all
coordinate diameters $<1/(n+1)$.

The compact sets $\overline{V_n}$ are nested and their diameters tend to $0$. So
$\bigcap_n\overline{V_n}=\{x\}$, and $x\in V_n$ for all $n$; in particular $x_{\vec1}\in E$.
Let $y=(w,x)$.
\begin{itemize}
\item If $i_n=m+1$, then $T^{\beta_n}w$ and $x$ both lie in $V_n$, so
$d(T^{\beta_n}L_{m+1}(y),U_{m+1}(y))=d(T^{\beta_n}w,x)\le\operatorname{diam}V_n$.
\item If $i_n=i\le m$, then $x\in W_i(V_n,\beta_n)$. So for every $\eta\in[m-1]$,
both $T^{\beta_n}x_{\eta^{(i,0)}}$ and $x_{\eta^{(i,1)}}$ lie in $V_{n,\eta^{(i,1)}}$, and
\[
d\big(T^{\beta_n}L_i(x),U_i(x)\big)\le\operatorname{diam}V_n .
\]
Together with $d(T^{\beta_n}L_i(w),U_i(w))<1/(n+1)$, and $L_i(y)=(L_iw,L_ix)$,
$U_i(y)=(U_iw,U_ix)$, this gives $T^{\beta_n}L_i(y)\to U_i(y)$ along these $n$.
\end{itemize}
The $\beta_n$ are pairwise distinct, so $y$ is an Erdős cube. Moreover $y_{\vec0}=w_{\vec0}=a$
and $y_{\vec1}=x_{\vec1}\in E$.
\end{proof}

\subsection{Building the cube level by level}\label{K:ssec:induction}

\begin{proposition}[Induction step]\label{K:prop:step}
If an $m$-dimensional cube datum exists, then an $(m+1)$-dimensional cube datum exists.
\end{proposition}

\begin{proof}
Let $(w,\lambda,\Psi)$ be an $m$-dimensional cube datum, with face measures $\nu_i$ and
Kronecker maps $\iota_i\colon X^{[m-1]}\to K_i$ (rotations $\alpha_i$), $1\le i\le m$.
Let $\kappa\colon X^{[m]}\to K$ be a Kronecker map for $\lambda$, with rotation $\alpha$.
For $i\le m$, apply \cref{K:lem:factor} to $P=U_i$ and to $P=L_i$. Both push $\lambda$ to
$\nu_i$, so they give the same homomorphism $\bar\chi_i\colon K\to K_i$, the unique one with
$\bar\chi_i\circ\alpha=\alpha_i$, together with constants $c_i^U,c_i^L\in K_i$ such that
\begin{equation}\label{K:eq:affine}
\iota_i\circ U_i=\bar\chi_i\circ\kappa+c_i^U,\qquad\iota_i\circ L_i=\bar\chi_i\circ\kappa+c_i^L\qquad\lambda\text{-a.e.}
\end{equation}
Moreover, for every $s\in K$,
\begin{equation}\label{K:eq:facesym}
(U_i\times U_i)_*J^\lambda_s=(L_i\times L_i)_*J^\lambda_s=J^{\nu_i}_{\bar\chi_i(s)} .
\end{equation}

\emph{Step 1: re-basing.}
This is the graph construction of \cite[Proposition~3.20]{kra2022infinite} and
\cite[Proposition~3.7]{charamaras2024finding}, applied to the current measure $\lambda$.
Let $W$ be the orbit closure of $w$ in $X^{[m]}$, so that $\lambda(W)=1$ by~\ref{K:G2}.
By \cref{K:lem:rebase} there are $d\in K$ and a \Fol{} sequence $\Psi'$ with
$(w,d)\in\mathrm{gen}(\hat\lambda,\Psi')$, where $\hat\lambda=(\mathrm{id},\kappa)_*\lambda$.
Let $\widehat W\subset W\times K$ be the orbit closure of $(w,d)$; then $\hat\lambda(\widehat W)=1$.
The projection $\mathrm{pr}\colon\widehat W\to X^{[m]}$ is a measure-theoretic isomorphism
from $\hat\lambda$ to $\lambda$, with inverse $x\mapsto(x,\kappa(x))$.
So $\pi(x,t)=t$ is a \emph{continuous} Kronecker map for $\hat\lambda$, it is equivariant
everywhere, and $\kappa\circ\mathrm{pr}=\pi$ $\hat\lambda$-a.e.
By \cref{K:lem:factor}\ref{K:Fc} (with $\bar\chi=\mathrm{id}$ and $c=0$),
\begin{equation}\label{K:eq:prJ}
(\mathrm{pr}\times\mathrm{pr})_*J^{\hat\lambda}_s=J^\lambda_s\qquad\text{for every }s\in K .
\end{equation}

\emph{Step 2: the new top face.}
Apply \cref{K:lem:pairs} to $(\widehat W,\hat\lambda,\pi)$ and $y_0=(w,d)$.
This gives a subsequence $\Psi''$ of $\Psi'$ and a $\hat\lambda$-conull set $\widehat W_0$
such that, for $\hat x\in\widehat W_0$,
\[
((w,d),\hat x)\in\mathrm{gen}\big(J^{\hat\lambda}_{\pi(\hat x)-d},\Psi''\big).
\]
For $x\in X^{[m]}$ put $\hat x=(x,\kappa(x))$ and $s(x)=\kappa(x)-d$.
Since $\hat\lambda=(\mathrm{id},\kappa)_*\lambda$, we have $\hat x\in\widehat W_0$ for
$\lambda$-a.e.\ $x$.
For such $x$, projecting with~\ref{K:G3} and~\eqref{K:eq:prJ} gives
\begin{equation}\label{K:eq:newgen}
w':=(w,x)\in\mathrm{gen}\big(\lambda':=J^\lambda_{s(x)},\ \Psi''\big)\qquad\text{in }X^{[m]}\times X^{[m]}=X^{[m+1]} .
\end{equation}
Note that $s(x)$ is $m_K$-distributed when $x\sim\lambda$.
By \cref{K:lem:J}\ref{K:Jc}, $\lambda'=J^\lambda_{s(x)}$ is ergodic for $\lambda$-a.e.\ $x$.

\emph{Step 3: \ref{K:D1}--\ref{K:D3} for $(w',\lambda',\Psi'')$.}
\begin{itemize}
\item \ref{K:D1}: $\lambda'$ is invariant and (for a.e.\ $x$) ergodic, $w'$ is generic by
\eqref{K:eq:newgen}, and $w'_{\vec0}=w_{\vec0}=a$.
\item \ref{K:D2}: by \cref{K:lem:J}\ref{K:Jb} both marginals of $\lambda'$ on $X^{[m]}$ equal
$\lambda$. So the vertex marginals of $\lambda'$ are those of $\lambda$, all equal to $\mu$.
\item \ref{K:D3}, direction $m+1$: $(L_{m+1})_*\lambda'$ and $(U_{m+1})_*\lambda'$ are the two
marginals of $\lambda'$, both equal to $\lambda$.
\item \ref{K:D3}, direction $i\le m$: $L_i(y',y'')=(L_iy',L_iy'')$, and similarly for $U_i$.
By~\eqref{K:eq:facesym}, $(L_i)_*\lambda'=(U_i)_*\lambda'=J^{\nu_i}_{t_i}$ with
$t_i:=\bar\chi_i(s(x))$.
\end{itemize}

\emph{Step 4: \ref{K:D4} for $(w',\lambda',\Psi'')$.}

\emph{Direction $m+1$.} We need $U_{m+1}(w')=x\in\supp\lambda$, which holds for
$\lambda$-a.e.\ $x$.

\emph{Direction $i\le m$.}
Let $p=U_i(w)$ and $\theta=\bar\chi_i(d)+c_i^U$.
The map $(z,u)\mapsto(U_iz,\bar\chi_i(u)+c_i^U)$ from $X^{[m]}\times K$ to $X^{[m-1]}\times K_i$
is continuous. It is equivariant because $\bar\chi_i\circ\alpha=\alpha_i$.
By~\eqref{K:eq:affine} it pushes $\hat\lambda$ to
$(U_i,\iota_i\circ U_i)_*\lambda=(\mathrm{id},\iota_i)_*\nu_i=\rho_{\nu_i}$.
By~\ref{K:G3}, $(p,\theta)$ is generic for $\rho_{\nu_i}$ along $\Psi'$.
Since $p\in\supp\nu_i$ by~\ref{K:D4}, \cref{K:lem:stab} gives $(p,\theta)\in\supp\rho_{\nu_i}$.
Now let
\[
(q,\varphi):=\big(U_ix,\ \iota_i(U_ix)\big).
\]
When $x\sim\lambda$, the pair $(q,\varphi)$ is distributed according to $\rho_{\nu_i}$.
By~\eqref{K:eq:affine}, for $\lambda$-a.e.\ $x$,
\[
\varphi-\theta=\bar\chi_i(\kappa(x))-\bar\chi_i(d)=\bar\chi_i(s(x))=t_i .
\]
By \cref{K:lem:PS}, applied to $(p,\theta)\in\supp\rho_{\nu_i}$, for $\lambda$-a.e.\ $x$ we have
\[
U_i(w')=(U_iw,U_ix)=(p,q)\in\supp J^{\nu_i}_{\varphi-\theta}=\supp J^{\nu_i}_{t_i}=\supp(U_i)_*\lambda' .
\]

\emph{Conclusion.}
The requirements on $x$ in Steps 2--4 are finitely many: $\hat x\in\widehat W_0$, ergodicity of
$J^\lambda_{s(x)}$, $x\in\supp\lambda$, and the $m$ conditions from \cref{K:lem:PS}.
Each holds for $\lambda$-almost every $x$.
The data they depend on, namely $d$, $\Psi'$, $\Psi''$, $\widehat W_0$ and the points
$(p,\theta)$, were all fixed before $x$ was chosen.
So some $x$ satisfies all of them, and $(w',\lambda',\Psi'')$ is an $(m+1)$-dimensional
cube datum.
\end{proof}

\subsection{Proof of \texorpdfstring{\cref{K:thm:main}}{the main theorem}}\label{K:ssec:proof}

Let $A\subset\Gamma$ with $\ubd(A)>0$, and let $k\ge1$.
Choose a \Fol{} sequence along which $|A\cap\Phi_N|/|\Phi_N|$ has a positive limsup, and pass
to a subsequence along which the limit $\dens_\Phi(A)>0$ exists.
\cref{K:thm:corr} gives an ergodic $(X,\mu,T)$, a clopen $E$ with $\mu(E)>0$, a \Fol{}
sequence $\Psi$, and $a\in\mathrm{gen}(\mu,\Psi)$ with $A=\{\gamma:T^\gamma a\in E\}$.
\begin{enumerate}
\item $(a,\mu,\Psi)$ is a $0$-dimensional cube datum.
\item By \cref{K:prop:step}, applied $k-1$ times, there is a $(k-1)$-dimensional cube datum.
\item By \cref{K:prop:criterion}, there is a $k$-dimensional Erdős cube $y$ with $y_{\vec0}=a$
and $y_{\vec1}\in E$.
\item By \cref{K:lem:cube-sumset}, there are infinite $B_1,\dots,B_k$ with
$T^{b_1+\dots+b_k}a\in E$, that is, $b_1+\dots+b_k\in A$, for all $b_i\in B_i$.
\end{enumerate}
\qed

\formalization{\Cref{K:thm:main} is formalized in the fourteen files
\nolinkurl{Kmrr25/ThmK*.lean}; the main statement is \leanname{thmK} in \leanfile{ThmK}. All
lemmas and propositions of this section are proved there, with four classical results stated as
explicit hypotheses (the structure \leanname{ClassicalInputs} in \leanfile{ThmKDefs}): the
correspondence principle (\cref{K:thm:corr}); the existence of generic points~\ref{K:G4}; the
discrete spectrum theorem (\cref{K:lem:kron}\ref{K:ka}); and, for vectors orthogonal to all
eigenfunctions, the decay of averaged correlations given by the spectral theorem and Wiener's
lemma (\cref{K:lem:wiener}\ref{K:wc},\ref{K:wd}). Disintegrations are taken from Mathlib. Some
arguments are organised differently in the formal proof: for instance, the Kronecker group is
realised as a closed subgroup of $(\R/\Z)^\N$, which avoids Pontryagin duality, and
\cref{K:lem:pairs} splits $g$ into its Kronecker and weak mixing parts instead of using the
spectral theorem for $g$.}

\subsection{Comments on the proof}\label{K:ssec:comments}

\subsubsection*{Structure of the induction}
The only information passed from one level to the next is the cube datum
\ref{K:D1}--\ref{K:D4}. The decoration $d$ and the Kronecker maps are used within a single
step and are rebuilt for the new measure. The face measures $J^{\nu_i}_{t_i}$ of the new
measure may have much larger Kronecker factors than the $\nu_i$, but nothing about them is
needed: at the next level they enter only through \ref{K:D3} and \ref{K:D4}.
The measures $\lambda_m$ need not be invariant under face transformations or permutations
of digits, since the cube criterion treats each direction separately. Each direction
produces its own sequence of return times, so the method gives $k$ unrelated sets $B_i$.
It says nothing about sumsets $B\oplus B+t$ built from a single set, which can fail when
$\{g+g:g\in G\}$ has infinite index (see the remark after Conjecture~5.15
in~\cite{kra2023problems}).

\subsubsection*{Ingredients}
Besides elementary facts about \Fol{} sequences, the proof uses:
\begin{itemize}
\item Furstenberg's correspondence principle;
\item the existence of generic points for ergodic measures along a subsequence (via the mean
ergodic theorem);
\item the mean ergodic theorem, Wiener's lemma and the van der Corput inequality;
\item the discrete spectrum theorem for the Kronecker factor and Pontryagin duality;
\item disintegration of measures;
\item re-basing for rotations.
\end{itemize}
It uses no characteristic factors beyond the Kronecker factor, no Host--Kra or
Conze--Lesigne structure theory, no nilsystems, and no equidistribution theorem for
nilsystems. Nor does
it use finite generation, or any assumption on the subgroups $n\Gamma$.
Commutativity of $\Gamma$ is used throughout: in the spectral theory on $\widehat\Gamma$, in
describing the Kronecker factor as a rotation of a compact abelian group, and in the fact
that translates $\Phi_N+h_N$ of a \Fol{} sequence are again \Fol{}. It enters the
combinatorics in \cref{K:lem:cube-sumset} and in the diagonal action on faces.
Infiniteness of $\Gamma$ is used to obtain infinitely many distinct return times
(\cref{K:lem:folner-grow}).

\subsubsection*{Relation to earlier work}
The relative products, their ergodicity, the generic-pairs lemma and the graph construction
used for re-basing are known in this context
\cite{kra2022infinite,charamaras2024finding}; see the references given with
\cref{K:def:J}, \cref{K:lem:pairs,K:lem:rebase} and Step~1 of \cref{K:prop:step}.
For $k=2$ our argument differs from that of~\cite{kra2026short} in three respects:
re-basing replaces topological eigenfunctions, phase returns
(\cref{K:lem:phase,K:lem:stab}) replace the equicontinuity estimate~(4.2)
of~\cite{kra2026short}, and an ergodic relative product $J^\mu_s$ replaces an unspecified
joining.
The ingredients that are new relative to these arguments are the treatment of the Kronecker part of the face measures through
phase returns (\cref{K:prop:criterion}), the support lemmas
(\cref{K:lem:stab,K:lem:phase,K:lem:PS}), and the level-by-level induction
(\cref{K:prop:step}).

\section{Admissible translates for \texorpdfstring{$B\oplus B+t$}{B⊕B+t} in abelian groups}
\label{G:sec:translates}

In this section we answer Question~5.16 of~\cite{kra2023problems} and the related
Question~\S5.3(a). For a countably infinite abelian group $G$ in which $2G=\{g+g:g\in G\}$
has finite index (by work of Charamaras--Mountakis and Ackelsberg these are exactly the
groups in which Conjecture~5.15 of~\cite{kra2023problems} holds), we show that a set
$T\subset G$ of translates suffices if and only if $T$ meets every coset of $2G$. In
particular, the transversals of $2G$ are exactly the minimal sets of translates that
suffice. We also determine exactly which sets arise as the set of admissible translates
of some $A$ of positive upper Banach density, and we give two-sided bounds for the set of
admissible translates of each individual $A$.
The new input is small: the existence of some translate is a theorem of
Charamaras--Mountakis~\cite{charamaras2024finding}, and the two elementary observations
we use have already appeared in the literature in closely related settings; see
\cref{G:rem:credit}.

\subsection{Setting and the survey's questions}
Let $(G,+)$ be a countable abelian group. A \emph{\Fol{} sequence} in $G$ is a sequence
$\Phi=(\Phi_N)_{N\in\N}$ of finite nonempty subsets of $G$ such that
$|\Phi_N\cap(g+\Phi_N)|/|\Phi_N|\to1$ for every $g\in G$. (Since $G$ is abelian, left,
right and two-sided \Fol{} sequences coincide.) Every countable abelian group is amenable,
so \Fol{} sequences exist. For $A\subset G$ we write
$\udens_\Phi(A)=\limsup_{N\to\infty}|A\cap\Phi_N|/|\Phi_N|$, and we define the
\emph{upper Banach density} as
\[
\ubd(A)=\sup_\Phi\udens_\Phi(A),
\]
the supremum being taken over all \Fol{} sequences in $G$. Thus $\ubd(A)>0$ if and only if
$\udens_\Phi(A)>0$ for some \Fol{} sequence $\Phi$; this is the notion of positive upper
Banach density used in~\cite{charamaras2024finding}. For $B\subset G$ let
$B\oplus B=\{b+b':b,b'\in B,\ b\neq b'\}$.

Conjecture~5.15 of~\cite{kra2023problems} asserts that in every countable abelian group
$G$, every set $A\subset G$ with $\ubd(A)>0$ contains $B\oplus B+t$ for some infinite
$B\subset G$ and some $t\in G$. The survey then asks:

\begin{question}{5.16}\label{G:q:516}
Assuming that Conjecture~5.15 holds, what values of the shift $t$ suffice?
\end{question}

\begin{question}{\S5.3(a)}\label{G:q:53a}
The authors write: ``We suspect that there is a direct relation between coset
representatives of the subgroup $\{g+g:g\in G\}$ and the restrictions that can be placed
on the shift $t$.'' Make this relation precise.
\end{question}

The sentence ``We would like to know if similar restrictions can be placed in other
groups'', which precedes Question~5.16 in~\cite{kra2023problems}, is the motivation for
Question~5.16 itself. By Question~\S5.3(a) we mean the suspicion quoted above, which the
survey states right after Question~5.16.

As recorded in~\cite{kra2023problems}, among countably infinite abelian groups,
Conjecture~5.15 is known to hold precisely for those with $[G:2G]<\infty$. We recall the two results behind this statement.

\begin{theorem}[{Charamaras--Mountakis~\cite[Corollary~1.13]{charamaras2024finding}}]
\label{G:thm:CM}
Let $(G,+)$ be a countably infinite abelian group such that $2G$ is a subgroup of finite index, and
let $A\subset G$ have positive upper Banach density. Then there exist an infinite set
$B\subset A$ and some $t\in G$ such that $B\oplus B\subset A-t$.
\end{theorem}

(For finite groups the statement is false, since then $G$ has no infinite subsets. The result is cited as Corollary~1.12 in~\cite{kra2023problems}
and in~\cite[Theorem~1.1]{ackelsberg2024counterexamples}, following an earlier arXiv
version of~\cite{charamaras2024finding}.)

\begin{theorem}[{Ackelsberg~\cite[Theorem~1.2]{ackelsberg2024counterexamples}}]
\label{G:thm:Ack}
Suppose $\Gamma$ is a countable abelian group and $2\Gamma$ is a subgroup of infinite
index. Then for any $\eps>0$ there exists $A\subset\Gamma$ such that $\ubd(A)>1-\eps$
and such that if $B\subset\Gamma$ and $t\in\Gamma$ satisfy $B\oplus B+t\subset A$, then
$B$ is finite.
\end{theorem}

Ackelsberg also records~\cite[footnote~2]{ackelsberg2024counterexamples} that T.~Tao
observed independently that the condition $[\Gamma:2\Gamma]<\infty$ in \cref{G:thm:CM}
is necessary.

Conjecture~5.15 fails trivially for finite groups. Hence ``assuming that Conjecture~5.15
holds'' in \cref{G:q:516} means that $G$ is a countably infinite abelian group with
$[G:2G]<\infty$. We fix such a group for the rest of this section. For $A\subset G$ we let
\[
T(A)=\{t\in G:\ B\oplus B+t\subset A\ \text{for some infinite } B\subset G\}
\]
be the set of \emph{admissible translates} for $A$. We say that a set $T\subset G$
\emph{suffices} if $T(A)\cap T\neq\emptyset$ for every $A\subset G$ with $\ubd(A)>0$.

\subsection{Three elementary lemmas}

\begin{lemma}\label{G:lem:density}
Let $G$ be a countably infinite abelian group, and let $H\le G$ be a subgroup of finite
index $\ell$.
\begin{enumerate}[label=\textup{(\roman*)}]
\item For every \Fol{} sequence $\Phi$ and every $r\in G$ we have
$|(r+H)\cap\Phi_N|/|\Phi_N|\to1/\ell$. In particular $\ubd(r+H)=1/\ell>0$.
\item For all $X,Y\subset G$ we have $\ubd(X\cup Y)\le\ubd(X)+\ubd(Y)$.
\end{enumerate}
\end{lemma}

\begin{proof}
(i) Let $C\subset G$ and $x\in G$. Since $|(x+C)\cap\Phi_N|=|C\cap(\Phi_N-x)|$, we get
\[
\big||(x+C)\cap\Phi_N|-|C\cap\Phi_N|\big|\le|(\Phi_N-x)\,\triangle\,\Phi_N|
=2\big(|\Phi_N|-|\Phi_N\cap(\Phi_N-x)|\big)=o(|\Phi_N|).
\]
Let $r_1+H,\dots,r_\ell+H$ be the cosets of $H$. Applying this with $C=r_1+H$ and
$x=r_i-r_1$ shows that the $\ell$ quantities $|(r_i+H)\cap\Phi_N|$ differ pairwise by
$o(|\Phi_N|)$. They add up to $|\Phi_N|$, so each of them equals
$|\Phi_N|/\ell+o(|\Phi_N|)$.

(ii) For each \Fol{} sequence $\Phi$ we have
$\udens_\Phi(X\cup Y)\le\udens_\Phi(X)+\udens_\Phi(Y)\le\ubd(X)+\ubd(Y)$. Now take
the supremum over $\Phi$.
\end{proof}

\begin{lemma}[Pigeonhole]\label{G:lem:pigeon}
If $[G:2G]<\infty$ and $B\subset G$ is infinite, then $(B\oplus B)\cap2G\neq\emptyset$.
Consequently, if $B\oplus B+t\subset A$, then $t\in A+2G$. Moreover, if
$B\oplus B+t$ is contained in a single coset $r+2G$, then $t\in r+2G$.
\end{lemma}

\begin{proof}
Since $B$ is infinite and $2G$ has finitely many cosets, there are distinct $b,b'\in B$ in
the same coset $x+2G$, say $b=x+2u$ and $b'=x+2v$. Then $b+b'=2(x+u+v)\in2G$ and
$b+b'\in B\oplus B$. If $B\oplus B+t\subset A$, then $t+b+b'\in A$, so
$t\in A-(b+b')\subset A+2G$. If $B\oplus B+t\subset r+2G$, then $t+b+b'\in r+2G$, so
$t\in r+2G$.
\end{proof}

\begin{lemma}[Moving the translate]\label{G:lem:move}
For $B\subset G$ and $s\in G$ we have $(B-s)\oplus(B-s)=B\oplus B-2s$. Consequently
$T(A)+2G=T(A)$ for every $A\subset G$, i.e.\ $T(A)$ is a union of cosets of $2G$.
\end{lemma}

\begin{proof}
The map $b\mapsto b-s$ is a bijection $B\to B-s$, and $b-s\neq b'-s$ if and only if
$b\neq b'$; this gives the identity. If $B\oplus B+t\subset A$ with $B$ infinite, then
$B-s$ is infinite and $(B-s)\oplus(B-s)+(t+2s)=B\oplus B+t\subset A$. Thus
$t+2s\in T(A)$ for every $s\in G$.
\end{proof}

\subsection{The admissible translates}

\begin{theorem}\label{G:thm:translates}
Let $G$ be a countably infinite abelian group with $[G:2G]<\infty$. For $A\subset G$ put
$R(A)=\{r\in G:\ \ubd(A\cap(r+2G))>0\}$.
\begin{enumerate}[label=\textup{(\alph*)}]
\item\label{G:it:union} $T(A)$ is a union of cosets of $2G$.
\item\label{G:it:sandwich} $R(A)\subset T(A)\subset A+2G$. If $\ubd(A)>0$, then
$R(A)\neq\emptyset$, and hence $T(A)$ contains a full coset of $2G$.
\item\label{G:it:realise} For every nonempty union $U$ of cosets of $2G$ we have
$\ubd(U)>0$ and $T(U)=U$. Consequently, the family
$\{T(A):A\subset G,\ \ubd(A)>0\}$ is exactly the family of nonempty unions of cosets of
$2G$.
\item\label{G:it:suffice} A set $T\subset G$ suffices if and only if $T$ meets every coset of
$2G$. The minimal sets that suffice are exactly the transversals of $2G$ in $G$; they all
have exactly $[G:2G]$ elements.
\end{enumerate}
\end{theorem}

\begin{proof}
\ref{G:it:union} This is \cref{G:lem:move}.

\ref{G:it:sandwich} The inclusion $T(A)\subset A+2G$ is \cref{G:lem:pigeon}. Let
$r\in R(A)$ and put $A'=A\cap(r+2G)$. By \cref{G:thm:CM} applied to $A'$, there are an
infinite $B\subset G$ and $t\in G$ with $B\oplus B+t\subset A'\subset A$. Thus
$t\in T(A)$. Since $B\oplus B+t\subset r+2G$, \cref{G:lem:pigeon} gives $t\in r+2G$. By
part~\ref{G:it:union}, $r\in t+2G\subset T(A)$. Hence $R(A)\subset T(A)$.
If $\ubd(A)>0$, then since $A$ is the union of the finitely many sets $A\cap(r+2G)$,
\cref{G:lem:density}(ii) shows that $\ubd(A\cap(r+2G))>0$ for some $r$, i.e.\
$R(A)\neq\emptyset$. Since $R(A)$ is a union of cosets of $2G$, the last assertion follows.

\ref{G:it:realise} By \cref{G:lem:density}(i), $\ubd(U)\ge1/[G:2G]>0$. For every
$r\in U$ we have $U\cap(r+2G)=r+2G$, which has positive upper Banach density; hence
$U\subset R(U)$. By part~\ref{G:it:sandwich},
$U\subset R(U)\subset T(U)\subset U+2G=U$. The final assertion follows from this and
parts~\ref{G:it:union} and~\ref{G:it:sandwich}.

\ref{G:it:suffice} Suppose $T$ meets every coset of $2G$, and let $\ubd(A)>0$. By
part~\ref{G:it:sandwich}, $T(A)$ contains a coset of $2G$, and this coset meets $T$.
Conversely, suppose $T\cap(r+2G)=\emptyset$ for some $r$. Let $A=r+2G$. Then
$\ubd(A)>0$, and by part~\ref{G:it:realise}, $T(A)=r+2G$, which is disjoint from $T$.
So $T$ does not suffice.
A transversal meets every coset, so it suffices, and removing any of its points leaves
a coset unmet; hence transversals are minimal. If a sufficient set $T$ is not a
transversal, then it contains two points of one coset, and removing one of them leaves a
sufficient set. Hence $T$ is not minimal.
\end{proof}

\Cref{G:thm:translates}\ref{G:it:suffice} answers \cref{G:q:516}, and it confirms the
suspicion in \cref{G:q:53a} in the strongest possible form: the translates can be
restricted to any prescribed set of coset representatives of $2G$, and no smaller set
of translates works. Parts~\ref{G:it:union}--\ref{G:it:realise} concern an individual
set $A$. The admissible translates form a union of cosets of $2G$. This union contains
every coset in which $A$ has positive upper Banach density, and it is contained in the
union of the cosets that $A$ meets. Every nonempty union of cosets arises in this way.
In general neither bound is attained, as the next proposition shows. One always has the
following exact, though essentially tautological, reduction to a single coset:
$r\in T(A)$ if and only if $A\cap(r+2G)\supset B\oplus B+t$ for some infinite $B$ and
some $t$. (If $B\oplus B+r\subset A$, pass to an infinite $B'\subset B$ inside one coset
of $2G$. Then $B'\oplus B'\subset 2G$, so $B'\oplus B'+r\subset A\cap(r+2G)$. The converse
follows from \cref{G:lem:pigeon,G:lem:move}.)

\begin{proposition}\label{G:prop:strict}
Let $G=\Z$.
\begin{enumerate}[label=\textup{(\roman*)}]
\item If $A=2\Z\cup\{3^n:n\in\N\}$, then $T(A)=2\Z$, while $A+2\Z=\Z$.
\item If $A=2\Z\cup X$ with $X=\{2^i+2^j+1:\ 1\le i<j\}$, then $R(A)=2\Z$, while
$T(A)=\Z$.
\end{enumerate}
So both inclusions in \cref{G:thm:translates}\ref{G:it:sandwich} can be strict.
\end{proposition}

\begin{proof}
(i) By \cref{G:thm:translates}\ref{G:it:sandwich}, $2\Z\subset R(A)\subset T(A)$. Since
$3\in A$, we have $A+2\Z=\Z$. Suppose that $t$ is odd and $B\oplus B+t\subset A$ with $B$
infinite. Choose an infinite $B'\subset B$ whose elements all have the same parity, and fix
$b_1<b_2$ in $B'$. For every $b\in B'\setminus\{b_1,b_2\}$ the numbers $b_1+b+t$ and
$b_2+b+t$ are odd elements of $A$, hence powers of $3$, and they differ by $b_2-b_1$. The
equation $3^m-3^n=b_2-b_1$ with $m>n\ge1$ has only finitely many solutions, because
$3^n\le b_2-b_1$ and $n$ determines $m$. But $b_1+b+t$ takes infinitely many values as $b$
ranges over $B'$. This is a contradiction, so $T(A)=2\Z$.

(ii) For $B=\{2^i:i\in\N\}$ we have $B\oplus B+1=X\subset A$, so $1\in T(A)$, and
$1+2\Z\subset T(A)$ by \cref{G:lem:move}. Clearly $2\Z\subset R(A)\subset T(A)$, so $T(A)=\Z$. It remains to show that
$\ubd(X)=0$. We claim that if every interval of $L$ consecutive integers contains at most
$c$ elements of a set $Y\subset\Z$, then $\ubd(Y)\le c/L$. Indeed, for a \Fol{} sequence
$\Phi$ we have
\[
\sum_{x=0}^{L-1}|Y\cap(\Phi_N+x)|=\sum_{y\in\Phi_N}|Y\cap\{y,\dots,y+L-1\}|\le c|\Phi_N|,
\]
and $|Y\cap(\Phi_N+x)|\ge|Y\cap\Phi_N|-|\Phi_N\setminus(\Phi_N+x)|=|Y\cap\Phi_N|-o(|\Phi_N|)$.
Hence $L\,\udens_\Phi(Y)\le c$.

Now let $I$ be an interval of $L\ge2$ consecutive integers. An element $2^i+2^j+1$ of
$X$ with $i<j$ lies in $(2^j,2^{j+1})$. The ranges $[2^j,2^{j+1})$ with $2^j\ge L$ have
length at least $L$, so at most two of them meet $I$. There are at most $\log_2L+1$ ranges
with $2^j<L$. For fixed $j$, the elements of $X\cap I$ correspond to distinct powers $2^i$
in the interval $I-2^j-1$ of length $L$. Such an interval contains at most one power of $2$
that is $\ge L$, and at most $\log_2L+1$ powers of $2$ that are $<L$. Hence
$|X\cap I|\le(\log_2L+3)^2$, and so $\ubd(X)\le(\log_2L+3)^2/L$ for every $L\ge2$. Thus
$\ubd(X)=0$ and $R(A)=2\Z$.
\end{proof}

\formalization{All four parts of \cref{G:thm:translates} are formalized in the file
\leanfile{Q516} as \leanname{thm_translates}, with \cref{G:thm:CM} stated as an explicit
hypothesis. The statements that do not use it, and \cref{G:prop:strict}
(\leanname{prop_strict_i}, \leanname{prop_strict_ii}), are proved unconditionally.}

\begin{remark}\label{G:rem:examples}
We compare \cref{G:thm:translates} with the examples in~\cite[\S5.3]{kra2023problems}.
\begin{enumerate}[label=\textup{(\arabic*)}]
\item For $G=\Z$ we have $2G=2\Z$, and the transversal $\{0,1\}$ suffices. This matches
the case $k=2$ of~\cite[Proposition~2.6]{kra2023problems}, which is stated in $\N$.
Taking $r=0$ in part~\ref{G:it:sandwich}, we get that $t=0$ is admissible whenever
$\ubd(A\cap2G)>0$. This is the group analogue of item~3
of~\cite[Proposition~2.6]{kra2023problems} for $k=2$. That proposition is an equivalence
between statements for all $k$, but its proof works for each $k$ separately. So for $k=2$
its items hold unconditionally, by~\cite[Theorem~2.2]{kra2023problems}.
\item If $G=\bigoplus_{i\in\N}\Z/3\Z$, then $2G=G$, so $T=\{0\}$ suffices, as observed
in~\cite{kra2023problems}.
\item If $G=\bigoplus_{i\in\N}\Z/2\Z$, then $2G=\{0\}$ has infinite index. This group lies
outside the scope of \cref{G:q:516}, because by \cref{G:thm:Ack} Conjecture~5.15 fails
there. In groups with $1<[G:2G]<\infty$, the phenomenon pointed out
in~\cite{kra2023problems} for this group (one cannot always take $t=0$) is explained by
cosets: for $A=r+2G$ with $r\notin2G$, every admissible translate lies in $r+2G$.
\end{enumerate}
\end{remark}

\begin{remark}[Credit]\label{G:rem:credit}
The only deep ingredient is \cref{G:thm:CM}. The two elementary devices in
\cref{G:lem:pigeon,G:lem:move} are not new. For unrestricted sumsets $t+B+B$,
Charamaras, Kousek, Mountakis and Radi\'c~\cite[Lemmas~2.1 and~2.2]{charamaras2025infinite}
use pigeonholing into cosets of $2G$ and write $t=2s+g_i$ with $g_i$ a coset
representative, replacing $B$ by $B+s$. Ackelsberg and Jamneshan~\cite[discussion after
Theorem~1.9, and Corollary~1.10]{ackelsberg2025equidistribution} use the same two
arguments for $B\oplus B\oplus B$ and the cosets of $3\Gamma$, under the assumption
$[\Gamma:6\Gamma]<\infty$. For $G=\Z$ the reduction
appears in~\cite[Proposition~2.6]{kra2023problems}. Our contribution consists of the exact
characterisation of the sets of translates that suffice, the minimality of transversals,
the description of the possible sets $T(A)$, and the two-sided bounds for $T(A)$ for
$B\oplus B$.
\end{remark}


\section{Cartesian products in \texorpdfstring{$\N^2$}{N2}}
\label{G:sec:q510}

In this section we answer Question~5.10 of~\cite{kra2023problems} negatively. We
construct sets $A\subset\N^2$ all of whose fibres have positive density along a fixed
\Fol{} sequence (in one version, every fibre has natural density exactly $1/4$), such that
no translate of $A$ contains the Ramsey configuration
$\{(b_1,b_2):b_1,b_2\in B,\ b_1<b_2\}$ for an infinite $B\subset\N$.

\subsection{The question}
Throughout, $\N=\{1,2,\dots\}$. Following~\cite{kra2023problems}, a \emph{\Fol{} sequence
in $\N$} is a sequence $\Phi=(\Phi_N)$ of finite subsets of $\N$ with
$|(\Phi_N+t)\cap\Phi_N|/|\Phi_N|\to1$ for every $t\in\N$. The upper and lower densities
along $\Phi$ are $\udens_\Phi$ and $\ldens_\Phi$. When they are equal, their common value
is denoted $\dens_\Phi$. For $A\subset\N^2$ and $n\in\N$ the \emph{fibre} of $A$ over $n$ is
$A_n=\{m\in\N:(n,m)\in A\}$. The configuration in question is
\begin{equation}\label{G:eq:BBcart}
\{(b_1,b_2):\ b_1,b_2\in B,\ b_1<b_2\}\subset A-t,
\end{equation}
which is equation~(21) of~\cite{kra2023problems}. It says that
$(b_1+t_1,b_2+t_2)\in A$ for all $b_1<b_2$ in $B$, where $t=(t_1,t_2)$.

\begin{question}{5.10}\label{G:q:510}
Suppose $A\subset\N^2$ satisfies
\begin{equation}\label{G:eq:22}
\liminf_{n\in\N}\dens_\Phi(A_n)>0
\end{equation}
for some \Fol{} sequence $\Phi$ on $\N$. Does there exist an infinite set $B\subset\N$ and
an element $t\in\N^2$ satisfying~\eqref{G:eq:BBcart}?
\end{question}

The survey points out that \eqref{G:eq:BBcart} can only hold if infinitely many fibres
$A_n$ have the property that every finite subfamily has infinite intersection. Example~5.11
of~\cite{kra2023problems}, due to F.~Hern\'andez Castro, takes $A=X\times Y$ with
$X=\bigcup_n[4^n,c4^n)$ and $Y=\bigcup_n[c'4^n,c''4^n)$ (intersected with $\N$), where
$1<c<c'<c''<2$. Then $X\cap(Y-s)$ is finite for every $s\in\Z$, because the blocks of $X$ and
$Y$ are separated by gaps that tend to infinity. Hence no translate of $A$ contains a
configuration~\eqref{G:eq:BBcart}: an element $b$ of $B$ other than the smallest satisfies $b+t_1\in X$ and $b+t_2\in Y$ (compare \cref{G:lem:criterion} below). In that example the fibre $A_n$ is $Y$
for $n\in X$ and empty for $n\notin X$. So it shows that the $\liminf$ in~\eqref{G:eq:22}
cannot be replaced by a $\limsup$. It does not satisfy~\eqref{G:eq:22}. We modify the
example so that \emph{every} fibre has positive density. Two products of this type are
glued along a partition of the rows, and the second coordinate of each product is
offset from its own row set.

\subsection{A criterion}

\begin{lemma}\label{G:lem:criterion}
Let $\N=P_0\sqcup P_1$ be a partition, and let $S_0,S_1\subset\N$ be sets such that
\begin{equation}\label{G:eq:finite}
S_c\cap(P_c+s)\ \text{is finite for every } s\in\Z\text{ and } c\in\{0,1\}.
\end{equation}
Let $A=(P_0\times S_0)\cup(P_1\times S_1)$. Then $A_n=S_c$ for every $n\in P_c$. Moreover,
there is no infinite set $B\subset\N$ together with $t=(t_1,t_2)\in\Z^2$ such that
$(b_1+t_1,b_2+t_2)\in A$ for all $b_1<b_2$ in $B$. In particular, $A$ satisfies
\eqref{G:eq:BBcart} for no infinite $B\subset\N$ and no $t\in\N^2$.
\end{lemma}

\begin{proof}
The identity $A_n=S_c$ for $n\in P_c$ holds because $P_0\cap P_1=\emptyset$. Suppose that
$B$ and $t$ exist. Every $b\in B$ is the smaller element of some pair in $B$, so
$b+t_1\in\N$. By the pigeonhole principle there are $c\in\{0,1\}$ and an infinite
$B'\subset B$ with $b+t_1\in P_c$ for all $b\in B'$. Let $b_0=\min B'$. For every
$b\in B'\setminus\{b_0\}$ we have $(b_0+t_1,b+t_2)\in A$ and $b_0+t_1\in P_c$, so
$b+t_2\in A_{b_0+t_1}=S_c$. Also $b+t_1\in P_c$, so
$b+t_2\in P_c+(t_2-t_1)$. Hence the infinitely many distinct numbers $b+t_2$, for
$b\in B'\setminus\{b_0\}$, lie in $S_c\cap(P_c+(t_2-t_1))$. This contradicts
\eqref{G:eq:finite}.
\end{proof}

\subsection{Two examples}

\begin{example}[Scale $4^k$]\label{G:ex:4k}
Let
\[
P_0=\bigcup_{k\ge0}[4^k,2\cdot4^k),\qquad P_1=\bigcup_{k\ge0}[2\cdot4^k,4^{k+1}),
\]
\[
S_0=\bigcup_{k\ge1}[9\cdot4^{k-1},15\cdot4^{k-1}),\qquad
S_1=\bigcup_{k\ge1}[5\cdot4^{k-1},7\cdot4^{k-1}),
\]
all intervals being intervals of integers. So $\N=P_0\sqcup P_1$. Each block of $S_0$
lies in the middle of the block $[2\cdot4^k,4^{k+1})$ of $P_1$, and each block of $S_1$
lies in the middle of the block $[4^k,2\cdot4^k)$ of $P_0$. Put
$A=(P_0\times S_0)\cup(P_1\times S_1)$.
\end{example}

\begin{example}[Square blocks]\label{G:ex:blocks}
For $k\in\N$ let $J_k=[k^2,(k+1)^2)\cap\N$, so that $\N=\bigsqcup_{k\ge1}J_k$ and
$|J_k|=2k+1$. Let $P'_c=\bigcup_{k\equiv c\ (\mathrm{mod}\ 2)}J_k$ for $c\in\{0,1\}$, and let
$S'_c$ consist of the \emph{middle parts} of the blocks of the opposite parity:
\[
S'_c=\{k^2+r:\ k\in\N,\ k\equiv1-c\ (\mathrm{mod}\ 2),\ r\in\Z,\ k\le2r<3k\}.
\]
Since $k\le 2r<3k$ forces $0\le r\le 2k$, we have $S'_c\cap J_k=\{k^2+r:k\le2r<3k\}$ when
$k\equiv1-c$, and this set has exactly $k$ elements. Put
$A'=(P'_0\times S'_0)\cup(P'_1\times S'_1)$.
\end{example}

\begin{proposition}\label{G:prop:examples}
\begin{enumerate}[label=\textup{(\roman*)}]
\item The sets of \cref{G:ex:4k} satisfy \eqref{G:eq:finite}. More precisely, every element
of $S_c\cap(P_c+s)$ is smaller than $15|s|$. Along the \Fol{} sequence
$\Phi_N=[1,4^N]$ we have $\dens_\Phi(S_0)=1/2$ and $\dens_\Phi(S_1)=1/6$. Hence every fibre
of $A$ has density $1/2$ or $1/6$ along $\Phi$, and $A$ satisfies~\eqref{G:eq:22}.
\item The sets of \cref{G:ex:blocks} satisfy \eqref{G:eq:finite}. More precisely, every
element of $S'_c\cap(P'_c+s)$ is smaller than $(2|s|+1)^2$. Both $S'_0$ and $S'_1$ have
natural density $1/4$. Hence every fibre of $A'$ has natural density $1/4$, and $A'$
satisfies \eqref{G:eq:22} along $\Phi_N=[1,N]$.
\end{enumerate}
\end{proposition}

\begin{proof}
(i) Let $m\in S_0$, say $9\cdot4^{k-1}\le m<15\cdot4^{k-1}$, and suppose $|s|<4^{k-1}$.
Then $2\cdot4^k=8\cdot4^{k-1}<m-s<16\cdot4^{k-1}=4^{k+1}$, so $m-s\in P_1$ and
$m-s\notin P_0$. Similarly, if $m\in S_1$, say $5\cdot4^{k-1}\le m<7\cdot4^{k-1}$, and
$|s|<4^{k-1}$, then $4^k<m-s<2\cdot4^k$, so $m-s\in P_0$ and $m-s\notin P_1$. Hence
every $m\in S_c\cap(P_c+s)$ lies in a block with $4^{k-1}\le|s|$, and so
$m<15\cdot4^{k-1}\le15|s|$. (For $s=0$ the intersection is empty.)

The sets $[1,4^N]$ form a \Fol{} sequence in $\N$, since $|([1,M]+t)\cap[1,M]|=M-t$ for
$0\le t\le M$. For $N\ge2$ the blocks of $S_0$ and $S_1$ with $k\le N-1$ lie in
$[1,4^N]$, since $15\cdot4^{N-2}<4^N$. The blocks with $k\ge N$ start at
$9\cdot4^{N-1}$ and $5\cdot4^{N-1}$, both larger than $4^N$, so they are disjoint from
$[1,4^N]$. Therefore
\[
|S_0\cap[1,4^N]|=\sum_{k=1}^{N-1}6\cdot4^{k-1}=2(4^{N-1}-1),\qquad
|S_1\cap[1,4^N]|=\sum_{k=1}^{N-1}2\cdot4^{k-1}=\tfrac23(4^{N-1}-1),
\]
and dividing by $4^N$ gives the limits $1/2$ and $1/6$. Every fibre $A_n$ equals $S_0$ or
$S_1$ by \cref{G:lem:criterion}, so $\inf_n\dens_\Phi(A_n)=1/6>0$.

(ii) Let $m=k^2+r\in S'_c$ with $k\le2r<3k$ and $k\equiv1-c\pmod2$. Suppose $|s|<k/2$.
Then $r-s\ge r-|s|>r-k/2\ge0$ and $r-s\le r+|s|<3k/2+k/2=2k$. Hence
$m-s=k^2+(r-s)\in J_k\subset P'_{1-c}$, so $m-s\notin P'_c$. Consequently every element
of $S'_c\cap(P'_c+s)$ lies in a block $J_k$ with $k\le2|s|$, and hence is smaller than
$(2|s|+1)^2$.

For the density, let $N\ge1$ and let $K\ge1$ be such that $K^2\le N<(K+1)^2$. The blocks
$J_k$ with $k<K$ lie in $[1,N]$, and $J_K$ is the only block that meets $[1,N]$ without
lying inside it. Hence
\[
|S'_c\cap[1,N]|=\sum_{\substack{1\le k<K\\ k\equiv1-c\ (2)}}k+O(K)
=\frac{K^2}{4}+O(K)=\frac N4+O(\sqrt N).
\]
So $\dens(S'_c)=1/4$. Every fibre of $A'$ is $S'_0$ or $S'_1$, and $[1,N]$ is a \Fol{}
sequence.
\end{proof}

\begin{theorem}\label{G:thm:q510}
The sets $A$ of \cref{G:ex:4k} and $A'$ of \cref{G:ex:blocks} satisfy~\eqref{G:eq:22}. In
fact every fibre has a density along the indicated \Fol{} sequence: $1/2$ or $1/6$
for $A$, and natural density exactly $1/4$ for $A'$. Nevertheless, there is no infinite
$B\subset\N$ and no $t\in\Z^2$ such that $(b_1+t_1,b_2+t_2)$ lies in $A$ (respectively
$A'$) for all $b_1<b_2$ in $B$. In particular, the answer to \cref{G:q:510} is negative.
\end{theorem}

\begin{proof}
Combine \cref{G:prop:examples} with \cref{G:lem:criterion}.
\end{proof}

\formalization{\Cref{G:lem:criterion}, \cref{G:prop:examples} and \cref{G:thm:q510} are
formalized in \leanfile{Q510} (\leanname{criterion}, \leanname{q510_example_4k},
\leanname{q510_example_blocks}, \leanname{q510_negative}).}

\begin{remark}\label{G:rem:q510}
\begin{enumerate}[label=\textup{(\arabic*)}]
\item The set $A'$ shows that the answer is negative already for the standard \Fol{}
sequence $[1,N]$, that is, with natural densities: every fibre has natural density $1/4$.
It does not satisfy~\eqref{G:eq:22} along every \Fol{} sequence. Each $S'_c$ misses the
blocks $J_k$ with $k\equiv c$, and these have length $2k+1\to\infty$. So along
$\Phi_N=J_{2N}$, for instance, the fibres $S'_0$ have density $0$. The hypothesis that
\eqref{G:eq:22} holds along every \Fol{} sequence, and the hypothesis that the fibres have
uniformly positive lower Banach density, $\liminf_n\mathsf d_*(A_n)>0$ (with $\mathsf d_*$ as in
\cref{KR:sec:main}), are not addressed here; compare~(4). The set $A$ of \cref{G:ex:4k} has fibres without natural density. Its fibres have
lower natural densities $2/9$ and $2/15$, which are approached at the integers just
below the left endpoints of the blocks of $S_0$ and $S_1$, respectively.
\item The survey explains that a positive answer to \cref{G:q:510} would imply
Theorem~2.2 and Conjecture~3.10 of~\cite{kra2023problems}. A negative answer has no
bearing on either statement.
\item The survey's Theorem~5.12 shows that a weaker hypothesis gives infinite $B,C$ with
$\{(b,c):b\in B,\ c\in C,\ b<c\}\subset A$. This is consistent with
\cref{G:thm:q510}: our sets contain even the full products $P_0\times S_0$ and
$P'_0\times S'_0$. What fails is the diagonal coupling in~\eqref{G:eq:BBcart}: every
element of $B$, except the smallest, must serve both as a row index and as a column index.
\item The construction in \cref{G:lem:criterion} cannot work if every fibre is required
to be syndetic with a gap bound that does not depend on the fibre. Indeed, suppose that for each $c$ every
$n\in\N$ satisfies $n+s\in S_c$ for some $0\le s\le L$. Then
$P_c\subset\bigcup_{0\le s\le L}(S_c-s)$, and \eqref{G:eq:finite} fails for the class $c$
with $P_c$ infinite. In fact, under this stronger hypothesis the answer to \cref{G:q:510} is
positive, by Ramsey's theorem. Suppose that every $n,q\in\N$ admit $0\le r\le L$ with
$q+r\in A_n$. Colour each pair $x<y$ in $\N$ by the least such $r$ for $n=x+1$ and $q=y+1$. An
infinite set $B$ on which the colouring is constant, say equal to $r$, satisfies
$(b_1+1,b_2+1+r)\in A$ for all $b_1<b_2$ in $B$, that is, \eqref{G:eq:BBcart} with
$t=(1,1+r)\in\N^2$.
\end{enumerate}
\end{remark}


\section{Corners in \texorpdfstring{$G^2$}{G2} and \texorpdfstring{$G^d$}{Gd}}
\label{G:sec:q523}

In this section we answer Question~5.23 of~\cite{kra2023problems} negatively, and also its
$G^d$ version, Question~\S5.4(b), in the two natural formulations given below (the survey does
not write this version out). The negative answer to the question as stated is elementary.
The hypothesis controls only the fibres in the last coordinate, while the conclusion
requires infinitely many points of $A$ on a single line in the first coordinate
direction. So this part of the answer depends on the formulation. We therefore also
treat the natural repaired version in which these points are dropped. Using
\cref{G:sec:q510}, we show that it fails as well for $G=\Z$ (and for $G=\N$). The
repair that adds a positive-density hypothesis on the columns remains open.

\subsection{The question}
Let $(G,\cdot)$ be a group. Following~\cite[\S5.4]{kra2023problems}, a \emph{left \Fol{}
sequence} is a sequence $(\Phi_N)$ of finite nonempty subsets of $G$ with
$|\Phi_N\cap g\Phi_N|/|\Phi_N|\to1$ for every $g\in G$. Austin's theorem on corners in
amenable groups~\cite{austin2016nonconventional}, quoted
in~\cite[Theorem~5.22]{kra2023problems}, motivates the following question.

\begin{question}{5.23}\label{G:q:523}
Let $(G,\cdot)$ be an amenable group and let $N\mapsto\Phi_N$ be a left \Fol{} sequence in
$G$. If $A\subset G^2$ satisfies
\begin{equation}\label{G:eq:23}
\liminf_{g\to\infty}\ \liminf_{M\to\infty}\ \frac1{|\Phi_M|}\sum_{h\in\Phi_M}\one_A(g,h)>0,
\end{equation}
where $g\to\infty$ means leaving all finite sets, do there exist $(t,s)\in G^2$ and an
injective sequence $b\colon\N\to G$ such that
\begin{equation}\label{G:eq:corner}
\{(t,s),\ (b(i)t,s),\ (b(i)t,b(j)s):\ i<j\}\subset A\,?
\end{equation}
\end{question}

The survey adds that if $G=\N$, an affirmative answer to \cref{G:q:523} implies an
affirmative answer to \cref{G:q:510}. Two remarks on the setting are in order. First, a
group with a left \Fol{} sequence is countable. Indeed, let $H$ be the subgroup generated by
$\bigcup_N\Phi_N$, which is countable. If $g\notin H$, then $g\Phi_N\subset gH$ is disjoint
from $\Phi_N\subset H$, so the \Fol{} condition fails for $g$. Second, if $G$ is finite,
there is no injective $b\colon\N\to G$. So we assume, as the survey intends, that $G$ is
countably infinite.

For $A\subset G^2$ and $g\in G$ write $A_g=\{h:(g,h)\in A\}$ for the \emph{row fibre}
and $A^h=\{g:(g,h)\in A\}$ for the \emph{column}. Condition~\eqref{G:eq:23} concerns only
the row fibres. The points $(b(i)t,s)$, $i\in\N$, in \eqref{G:eq:corner} are pairwise
distinct (since $b$ is injective) and lie in the single column $A^s$.

\begin{lemma}\label{G:lem:folnergrow}
If $G$ is an infinite group and $(\Phi_M)$ is a left \Fol{} sequence, then
$|\Phi_M|\to\infty$.
\end{lemma}

\begin{proof}
Suppose that $|\Phi_M|\le K$ for infinitely many $M$. Choose a finite $S\subset G$ with
$|S|>K$. For each $x\in S$ the number $|\Phi_M\setminus x\Phi_M|$ is a nonnegative
integer, and $|\Phi_M\setminus x\Phi_M|/|\Phi_M|\to0$. Along the $M$ with $|\Phi_M|\le K$
we therefore have $|\Phi_M\setminus x\Phi_M|=0$ for all $x\in S$ once $M$ is large. Then
$\Phi_M\subset x\Phi_M$, and since $|x\Phi_M|=|\Phi_M|$, we get $x\Phi_M=\Phi_M$. Pick
such an $M$ and some $y\in\Phi_M$. Then $Sy\subset\Phi_M$, so $|\Phi_M|\ge|S|>K$, a
contradiction.
\end{proof}

\subsection{The question as stated, in \texorpdfstring{$G^2$}{G2} and
\texorpdfstring{$G^d$}{Gd}}

The survey states that \cref{G:q:523} ``can be generalized to $G^d$ for $d>2$ in a
straightforward manner'', but does not write out the generalisation. Modelled on
Austin's $d$-dimensional corners~\cite[Theorem~5.22]{kra2023problems}, the natural
conclusion asks for $t\in G^d$ and an injective $b\colon\N\to G$ such that, for all
$i_1<\dots<i_d$, the $d+1$ points
\begin{equation}\label{G:eq:cornerd}
(t_1,\dots,t_d),\ (b(i_1)t_1,t_2,\dots,t_d),\ (b(i_1)t_1,b(i_2)t_2,t_3,\dots,t_d),\ \dots,\
(b(i_1)t_1,\dots,b(i_d)t_d)
\end{equation}
lie in $A$. The natural analogues of~\eqref{G:eq:23} constrain the fibres in the last
coordinate (here $(x_1,\dots,x_{d-1})\to\infty$ means leaving every finite subset of $G^{d-1}$),
\begin{equation}\label{G:eq:I}
\liminf_{(x_1,\dots,x_{d-1})\to\infty}\ \liminf_{M\to\infty}\
\frac{|\{h\in\Phi_M:(x_1,\dots,x_{d-1},h)\in A\}|}{|\Phi_M|}>0,
\end{equation}
or the row fibres over the first coordinate,
\begin{equation}\label{G:eq:II}
\liminf_{x_1\to\infty}\ \liminf_{M\to\infty}\
\frac{|\{y\in\Phi_M^{d-1}:(x_1,y)\in A\}|}{|\Phi_M|^{d-1}}>0.
\end{equation}

\begin{question}{\S5.4(b)}\label{G:q:54b}
Let $d\ge3$, let $G$ be a countably infinite amenable group with a left \Fol{} sequence
$\Phi$, and suppose $A\subset G^d$ satisfies \eqref{G:eq:I} or \eqref{G:eq:II}. Do there
exist $t\in G^d$ and an injective $b\colon\N\to G$ such that the points
\eqref{G:eq:cornerd} lie in $A$ for all $i_1<\dots<i_d$?
\end{question}

For $d=2$ both \eqref{G:eq:I} and \eqref{G:eq:II} reduce to \eqref{G:eq:23}, and
\cref{G:q:54b} reduces to \cref{G:q:523}.

\begin{theorem}\label{G:thm:q523literal}
Let $G$ be a countably infinite group, enumerated as $G=\{g_1,g_2,\dots\}$ without
repetitions, and let $d\ge2$. Put
\[
E_d=\{x\in G^d:\ x_1=g_n,\ x_d=g_m\ \text{for some } n<m\}.
\]
Let $\Phi$ be any left \Fol{} sequence in $G$. Then:
\begin{enumerate}[label=\textup{(\roman*)}]
\item for every $(x_1,\dots,x_{d-1})\in G^{d-1}$, the last-coordinate fibre
$\{h:(x_1,\dots,x_{d-1},h)\in E_d\}$ is cofinite, and its density along $\Phi$ equals $1$;
\item for every $x_1\in G$, the set $\{y\in G^{d-1}:(x_1,y)\in E_d\}$ has density $1$
along $\Phi^{d-1}=(\Phi_M^{d-1})_M$;
\item $|E_d\cap\Phi_M^d|/|\Phi_M|^d\to1/2$;
\item for every $t\in G^d$, the set $\{g\in G:(gt_1,t_2,\dots,t_d)\in E_d\}$ is finite.
\end{enumerate}
In particular $E_d$ satisfies~\eqref{G:eq:I} and~\eqref{G:eq:II}, and
\eqref{G:eq:23} when $d=2$, with all the fibre-density limits occurring there equal to $1$.
But there is no $t\in G^d$
and no injective $b$ such that $(b(i)t_1,t_2,\dots,t_d)\in E_d$ for infinitely many $i$.
Hence the answer to \cref{G:q:523} is negative, and so is the answer to \cref{G:q:54b} in both
formulations, for every countably infinite amenable group $G$ and every left \Fol{} sequence.
\end{theorem}

\begin{proof}
If $x_1=g_n$, then the last-coordinate fibre is $G\setminus\{g_1,\dots,g_n\}$, whatever
$x_2,\dots,x_{d-1}$ are. Its density along $\Phi$ is at least $1-n/|\Phi_M|$, which tends
to $1$ by \cref{G:lem:folnergrow}. This proves (i). Part (ii) follows similarly: the set in
question is $G^{d-2}\times(G\setminus\{g_1,\dots,g_n\})$, and its relative density in
$\Phi_M^{d-1}$ is at least $1-n/|\Phi_M|$. For (iii), note that among the
$|\Phi_M|^2-|\Phi_M|$ pairs $(x,y)\in\Phi_M^2$ with $x\neq y$, exactly one of $(x,y)$ and
$(y,x)$ has the first entry earlier in the enumeration. Hence
$|E_d\cap\Phi_M^d|=\frac12(|\Phi_M|^2-|\Phi_M|)\,|\Phi_M|^{d-2}$, and (iii) follows from
\cref{G:lem:folnergrow}. For (iv), let $t_d=g_m$. If $(gt_1,t_2,\dots,t_d)\in E_d$, then
$gt_1\in\{g_1,\dots,g_{m-1}\}$, so $g$ lies in the finite set
$\{g_1,\dots,g_{m-1}\}t_1^{-1}$. Since $b$ is injective, $(b(i)t_1,t_2,\dots,t_d)\in E_d$
can hold for at most $m-1$ values of $i$.
\end{proof}

\formalization{\Cref{G:thm:q523literal} and the resulting negative answers to
\cref{G:q:523,G:q:54b} are formalized in \leanfile{Q523} (\leanname{q523literal},
\leanname{q523_negative}). \Cref{G:thm:repairA} below is not formalized.}

For $G=\Z$ with $\Phi_M=[1,M]$, the simpler set $\{(g,h)\in\Z^2:h>|g|\}$ also works. For
$G=\N$ one can take $\{(n,m)\in\N^2:n<m\}$. In the semigroup $\N$, where the survey's
implication to \cref{G:q:510} is stated, this means that the answer to the literal
question is negative for a trivial reason.

\begin{remark}[The answer depends on the formulation]\label{G:rem:formulation}
\Cref{G:thm:q523literal} has nothing to do with sumsets: the conclusion
\eqref{G:eq:corner} contains infinitely many distinct points of a single column, and
\eqref{G:eq:23} does not even force columns to be infinite. The survey's comment that, for
$G=\N$, \cref{G:q:523} implies \cref{G:q:510} shows that the authors are mainly
interested in the points $(b(i)t,b(j)s)$. The points $(t,s)$ and $(b(i)t,s)$ come from
the shape of Austin's corners. There are two natural repairs:
\begin{enumerate}[label=\textup{(\alph*)}]
\item drop the points $(t,s)$ and $(b(i)t,s)$, and ask only for
$\{(b(i)t,b(j)s):i<j\}\subset A$ (and, in $G^d$, only for the top points
$(b(i_1)t_1,\dots,b(i_d)t_d)$ with $i_1<\dots<i_d$);
\item keep \eqref{G:eq:corner} but add to \eqref{G:eq:23} the analogous condition on
the columns, that is,
\[
\liminf_{h\to\infty}\ \liminf_{M\to\infty}\ \frac{|A^h\cap\Phi_M|}{|\Phi_M|}>0 .
\]
\end{enumerate}
In \cref{G:thm:repairA} below we show, using \cref{G:sec:q510}, that repair~(a) still has
a negative answer for $G=\Z$. Repair~(b), with positive-density hypotheses on both the
row fibres and the columns, is to our knowledge open. Neither the
examples of \cref{G:thm:q523literal} nor those of \cref{G:thm:repairA} satisfy it: for
instance, the columns of $\tilde A$ below are empty for $|h|\notin S'_0\cup S'_1$.
\end{remark}

\subsection{The repaired question (a)}

For $G=\N$, repair~(a) of \cref{G:q:523} is answered by \cref{G:thm:q510} directly.
Suppose that $t,s\in\N$ and an injective $b\colon\N\to\N$ satisfy
$(b(i)+t,b(j)+s)\in A'$ for all $i<j$. Since $b$ is injective, it is unbounded, so it has
a strictly increasing subsequence. Its range $B$ then satisfies \eqref{G:eq:BBcart} with
the translate $(t,s)$. This is impossible. Also, $A'$ satisfies \eqref{G:eq:23} with
$\Phi_M=[1,M]$, because every fibre of $A'$ has natural density $1/4$. To treat the group
$\Z$ and the $d$-dimensional version, we transfer $A'$ to $\Z^2$. For $x\in\Z$ let
$|x|_+=\max\{|x|,1\}$, and let
\[
\tilde A=\{(g,h)\in\Z^2:\ h\neq0,\ (|g|_+,|h|)\in A'\},
\]
where $A'=(P'_0\times S'_0)\cup(P'_1\times S'_1)$ is the set of \cref{G:ex:blocks}.

\begin{lemma}\label{G:lem:Ztransfer}
For every $g\in\Z$, the row fibre $\{h:(g,h)\in \tilde A\}$ meets $[1,M]$ in
$|S'_c\cap[1,M]|$ points, where $c$ is the class with $|g|_+\in P'_c$. Hence it has
density $1/4$ along $\Phi_M=[1,M]$. There are no $(t,s)\in\Z^2$ and injective
$b\colon\N\to\Z$ with $(b(i)+t,b(j)+s)\in \tilde A$ for all $i<j$.
\end{lemma}

\begin{proof}
The fibre over $g$ is $\{h\neq0:|h|\in S'_c\}$, and its intersection with $[1,M]$ is
$S'_c\cap[1,M]$. The density statement now follows from
\cref{G:prop:examples}(ii).

Suppose that $t,s,b$ exist. Since $b$ is injective, for every $L$ only finitely many $i$
satisfy $|b(i)|\le L$. Choose $\epsilon\in\{1,-1\}$ such that $\epsilon b(i)>0$ for
infinitely many $i$. Then choose $i_1<i_2<\cdots$ inductively so that the numbers
$\beta_k=\epsilon b(i_k)$ are positive and strictly increasing, with
$\beta_1>|t|+|s|+1$. For $k<l$ we have $(b(i_k)+t,b(i_l)+s)\in \tilde A$. Now
$|b(i_k)+t|_+=|\epsilon\beta_k+t|=\beta_k+\epsilon t$, because $\beta_k>|t|+1$, and similarly
$|b(i_l)+s|=\beta_l+\epsilon s$. So $(\beta_k+\epsilon t,\beta_l+\epsilon s)\in A'$ for all
$k<l$. Then the infinite set $B=\{\beta_k\}\subset\N$ and the translate
$(\epsilon t,\epsilon s)\in\Z^2$ contradict \cref{G:thm:q510}.
\end{proof}

\begin{theorem}\label{G:thm:repairA}
Let $d\ge2$ and $\tilde A^{(d)}=\{x\in\Z^d:(x_1,x_d)\in \tilde A\}$, and let $\Phi_M=[1,M]$.
\begin{enumerate}[label=\textup{(\roman*)}]
\item For every $(x_1,\dots,x_{d-1})\in\Z^{d-1}$, the last-coordinate fibre of $\tilde A^{(d)}$ has
density $1/4$ along $\Phi$. For every $x_1\in\Z$, the fibre
$\{y\in\Z^{d-1}:(x_1,y)\in \tilde A^{(d)}\}$ has density $1/4$ along $\Phi^{d-1}$. In particular
$\tilde A^{(d)}$ satisfies \eqref{G:eq:I} and \eqref{G:eq:II}, and $\tilde A^{(2)}=\tilde A$ satisfies
\eqref{G:eq:23}.
\item There are no $t\in\Z^d$ and injective $b\colon\N\to\Z$ such that
$(b(i_1)+t_1,\dots,b(i_d)+t_d)\in \tilde A^{(d)}$ for all $i_1<\dots<i_d$.
\end{enumerate}
Consequently repair~(a) of \cref{G:q:523} has a negative answer for $G=\Z$, and so does
its analogue for \cref{G:q:54b}, in which only the top points of \eqref{G:eq:cornerd}
are required.
\end{theorem}

\begin{proof}
(i) The last-coordinate fibre over $(x_1,\dots,x_{d-1})$ is the row fibre of $\tilde A$ over
$x_1$. The fibre over $x_1$ is $\Z^{d-2}\times\{h:(x_1,h)\in \tilde A\}$, and its relative
density in $[1,M]^{d-1}$ equals that of the row fibre in $[1,M]$. Now apply
\cref{G:lem:Ztransfer}.

(ii) Suppose that $t$ and $b$ exist. If $i<j$ and $j-i\ge d-1$, then (for $d\ge3$) there are
indices $i<i_2<\dots<i_{d-1}<j$, and so $(b(i)+t_1,b(j)+t_d)\in \tilde A$; for $d=2$ this is
the hypothesis itself. Let
$b'(k)=b((d-1)k)$. Then $b'$ is injective, and for $k<l$ we have
$(d-1)l-(d-1)k\ge d-1$. Hence $(b'(k)+t_1,b'(l)+t_d)\in \tilde A$ for all $k<l$, which
contradicts \cref{G:lem:Ztransfer}.
\end{proof}

To summarise: the configurations~\eqref{G:eq:corner} and~\eqref{G:eq:cornerd} cannot be
forced by the fibre-density hypotheses \eqref{G:eq:23}, \eqref{G:eq:I} or \eqref{G:eq:II}.
For $G=\Z$ (and for $G=\N$) this remains true after the lower corner points are dropped.
A uniform hypothesis on the row fibres does help: for $G=\N$, if the row fibres are syndetic with
a bound on the gaps that does not depend on the fibre, the Ramsey argument of
\cref{G:rem:q510}(4) gives a positive answer to repair~(a). Repair~(b) of
\cref{G:rem:formulation}, which also requires the columns to have positive density, is to our
knowledge open.

\part{Sumsets in the integers}
%

\section{Unrestricted sums \texorpdfstring{$B+B$}{B+B} at the lower density threshold}\label{IA:sec:u32a}

In this section we answer Question~\S3.2(a) of~\cite{kra2023problems} negatively, both as it is
stated and in a natural stronger variant. Throughout this part of the paper $B+B=\{b+b':b,b'\in B\}$ denotes the
\emph{unrestricted} sumset, which contains the diagonal elements $2b$, and
$\N_0=\N\cup\{0\}$.

Kousek and Radi\'c~\cite[Theorem~1.2]{kousek2024infinite} proved that a set $A\subset\N$ contains
$B+B$ for some infinite $B\subset\N$ as soon as $\udens(A)>5/6$ or $\ldens(A)>3/4$, and showed
that both thresholds are sharp~\cite[Propositions~4.1 and~4.2]{kousek2024infinite}. Right after
reporting this, the survey asks the following.

\begin{question}{\S3.2(a)}\label{IA:q:u32a}
What occurs if we only assume, for example, that $|A\cap[N]|>\tfrac34N+\log N$ for infinitely many
$N$? That is: must such a set $A\subset\N$ contain $B+B$ for some infinite set $B\subset\N$?
\end{question}

Read literally (``for infinitely many $N$''), the hypothesis is implied by $\udens(A)>3/4$, so it
lies far below the upper density threshold $5/6$. Motivated by the lower density threshold, we also
answer the stronger variant in which ``for infinitely many $N$'' is replaced by ``for all
sufficiently large $N$''; this is presumably the boundary case the survey has in mind.

\begin{theorem}\label{IA:thm:u32a}
\begin{enumerate}[label=\textup{(\alph*)}]
\item\label{IA:it:u32a-io} There is a set $A\subset\N$ such that for every $C>0$ we have
$|A\cap[N]|>\tfrac34N+C\log N$ for infinitely many $N$, and such that $B+B\not\subset A$ for every
infinite $B\subset\N$.
\item\label{IA:it:u32a-all} There are a set $A\subset\N$ and a non-decreasing unbounded function
$\psi\colon\N\to[0,\infty)$ such that
\[
|A\cap[N]|\ \ge\ \tfrac34N+\psi(N)\log N\qquad\text{for all sufficiently large }N,
\]
and $B+B\not\subset A$ for every infinite $B\subset\N$.
\end{enumerate}
\end{theorem}

Since $\psi\to\infty$, part~\ref{IA:it:u32a-all} gives $|A\cap[N]|>\frac34N+C\log N$ for all large
$N$, for every constant $C$ and every base of the logarithm.

\begin{remark}\label{IA:rem:u32a}
\begin{enumerate}[label=(\roman*)]
\item Part~\ref{IA:it:u32a-io} is essentially contained in~\cite{kousek2024infinite}: the set $A'$
of~\cite[Proposition~4.1]{kousek2024infinite} has upper density $5/6>3/4$ and contains no $B+B$,
so it satisfies the literal hypothesis of \cref{IA:q:u32a}. We give an explicit variant with an exact
count, to make the size of the excess visible. The literal reading is therefore answered by the
known sharpness examples, and the natural question is the one in part~\ref{IA:it:u32a-all}.
\item The set in part~\ref{IA:it:u32a-all} has lower density exactly $3/4$: at least $3/4$ by
construction, and at most $3/4$ by~\cite[Theorem~1.2]{kousek2024infinite}, since it contains no
$B+B$. The proof uses the diagonal elements $2b\in B+B$ in an
essential way (see Step~(iv) in the proof of \cref{IA:prop:rot}).
\item We do not know which functions $g(N)=o(N)$ have the property that $|A\cap[N]|\ge \frac34N+g(N)$
for all large $N$ forces $B+B\subset A$ for some infinite $B$; part~\ref{IA:it:u32a-all} shows that
$g(N)=\psi(N)\log N$ does not, for a suitable $\psi\to\infty$. (Replacing $\psi$ by
$\min(\psi(N),\log N)$, which is still non-decreasing and unbounded, we may assume that this $g$ is
$o(N)$.)
\end{enumerate}
\end{remark}

\subsection{The literal reading}

\begin{proof}[Proof of \cref{IA:thm:u32a}\ref{IA:it:u32a-io}]
Let $S=\bigcup_{n\ge1}\big[4^n,\,2\cdot4^n-2n\big)\cap\Z$ and $A=(2\N-1)\cup(S\cap2\N)$.

\emph{$S$ contains no $D+D$ with $D$ infinite.} Suppose $D+D\subset S$, let $d'=\min D$ and let
$d\in D$ be large. Since $2d\in S$, there is $m$ with $4^m\le 2d<2\cdot4^m-2m$, and $m\to\infty$ as
$d\to\infty$; take $d$ so large that $m\ge d'$. Then $2\cdot 4^{m-1}\le d\le d+d'<4^m-m+d'\le 4^m$.
But $S\cap[2\cdot4^{m-1},4^m)=\varnothing$, because the block of index $m-1$ ends below
$2\cdot4^{m-1}$ and the block of index $m$ starts at $4^m$. This contradicts $d+d'\in S$.

\emph{$A$ contains no $B+B$ with $B$ infinite.} If $B+B\subset A$, choose an infinite $B'\subset B$
whose elements all have the same parity. Then $B'+B'$ consists of even numbers, so
$B'+B'\subset A\cap2\N\subset S$, which is impossible.

\emph{Counting.} Let $N_n=2\cdot4^n-2n-1$. The odd numbers in $[N_n]$ number $(N_n+1)/2=4^n-n$. The
block $[4^m,2\cdot4^m-2m)$ has even length and starts at an even number, so it contains
$4^m/2-m$ even numbers; the blocks with $m\le n$ lie in $[N_n]$ and the block with $m=n+1$ does not
meet it. Hence
\[
|A\cap[N_n]|=4^n-n+\sum_{m=1}^n\Big(\frac{4^m}2-m\Big)=4^n-n+\frac{4^{n+1}-4}{6}-\frac{n(n+1)}2,
\]
and subtracting $\frac34N_n=\frac32 4^n-\frac32n-\frac34$ gives
\[
|A\cap[N_n]|-\tfrac34N_n=\frac{4^n}{6}-\frac{n^2}{2}+\frac1{12}.
\]
This exceeds $C\log N_n=O(n)$ for every fixed $C$ once $n$ is large (for $C=1$ and the natural
logarithm already for $n\ge3$).
\end{proof}

\subsection{Rotation blocks}
For the second reading we first reduce to a statement about a set $E$ of density $1/2$. For
$E\subset\N$ write $2E=\{2e:e\in E\}$.

\begin{lemma}\label{IA:lem:parity}
Let $E\subset\N$ and $A=(2\N-1)\cup2E$. If $B+B\subset A$ for some infinite $B\subset\N$, then there
are an infinite $D\subset\N_0$ and $t\in\{0,1\}$ with $D+D+t\subset E$. Moreover
$|A\cap[N]|=\lceil N/2\rceil+|E\cap[1,\lfloor N/2\rfloor]|$ for every $N\in\N$.
\end{lemma}

\begin{proof}
Choose an infinite $B'\subset B$ of constant parity; then $B'+B'\subset A\cap2\N=2E$. If $B'\subset
2\N$, put $D=\{b/2:b\in B'\}$, so that $2(D+D)=B'+B'$ and $D+D\subset E$. If $B'$ consists of odd
numbers, put $D=\{(b-1)/2:b\in B'\}\subset\N_0$, so that $B'+B'=2(D+D+1)$ and $D+D+1\subset E$. The
counting formula is clear.
\end{proof}

Let $\rho(x)=\one_{[0,1/2)}(\{x\})-\frac12$ for $x\in\R$, so that $\rho$ is $1$-periodic,
right-continuous, has jumps exactly at $\frac12\Z$, and satisfies $\rho(x+\frac12)=-\rho(x)$. For
$n\in\N$ let $\lambda(n)=\lfloor\log_2n\rfloor$ and $\varepsilon(n)=(-1)^{\lambda(n)}$. Note that
\begin{equation}\label{IA:eq:lambda}
\lambda(2n+t)=\lambda(n)+1\qquad(n\in\N,\ t\in\{0,1\}).
\end{equation}

\begin{definition}\label{IA:def:rot}
Let $\alpha,\beta\in\R$ and let $(w_i)_{i\ge1}$ be positive integers. Let
$J_0=\{x\in\T:\{x-\beta\}\in[0,\tfrac12)\}$ and $J_1=\T\setminus J_0$, and put
\[
R=\{n\in\N:\ n\alpha\bmod 1\in J_{\lambda(n)\bmod 2}\},\qquad
G=\bigcup_{i\ge1}\,[2^i-w_i,\,2^i+w_i),\qquad E=R\setminus G .
\]
\end{definition}

So $R$ is a Bohr set on each dyadic block, using the half-circle $J_0$ on blocks
$[2^j,2^{j+1})$ with $j$ even and the complementary half-circle on blocks with $j$ odd; $E$ is
obtained by removing slowly growing neighbourhoods of the powers of~$2$. For $\alpha\in\{0,\frac12\}$
these are block sets of the kind used in~\cite[Section~4]{kousek2024infinite}.

\begin{proposition}\label{IA:prop:rot}
In \cref{IA:def:rot}, assume $w_i\to\infty$. Then there are no infinite $D\subset\N_0$ and
$t\in\{0,1\}$ with $D+D+t\subset E$.
\end{proposition}

\begin{proof}
Suppose $D$ is infinite, $t\in\{0,1\}$ and $D+D+t\subset E$; removing $0$ from $D$ we may assume
$D\subset\N$.
\begin{enumerate}[label=(\roman*),leftmargin=*]
\item For $d\in D$ put $z_d=d\alpha+t\alpha/2\bmod1$. Then $(d+d'+t)\alpha\equiv z_d+z_{d'}$ and
$(2d+t)\alpha\equiv2z_d\pmod 1$.
\item For $k=0,1,2,3$ let $Q_k=\{z\in\T:\{z-\beta/2\}\in[\frac k4,\frac{k+1}4)\}$. If $z,z'\in
Q_k$, write $z\equiv\beta/2+x$, $z'\equiv\beta/2+x'$ with $x,x'\in[\frac k4,\frac{k+1}4)$; then
$z+z'\equiv\beta+(x+x')$ with $x+x'\in[\frac k2,\frac{k+1}2)$, so $z+z'\in J_{k\bmod2}$.
In particular $2z\in J_{k\bmod 2}$ as well.
\item By the pigeonhole principle there are $\sigma\in\{0,1\}$, $k\in\{0,1,2,3\}$ and an infinite
$D'\subset D$ with $\lambda(d)\equiv\sigma\pmod 2$ and $z_d\in Q_k$ for all $d\in D'$.
\item Let $d\in D'$. Since $2d+t\in E\subset R$ and $\lambda(2d+t)=\lambda(d)+1$
by~\eqref{IA:eq:lambda}, we have $(2d+t)\alpha\in J_{1-\sigma}$. By (i) and (ii),
$(2d+t)\alpha\equiv2z_d\in J_{k\bmod 2}$. As $J_0\cap J_1=\varnothing$, this forces $k\equiv1-\sigma
\pmod 2$. (This is the only place where a diagonal element is used.)
\item Let $d_0=\min D'$ and let $d'\in D'$, $d'>d_0$. Then $n:=d_0+d'+t\in E\subset R$ and, by (i)
and (ii), $n\alpha\equiv z_{d_0}+z_{d'}\in J_{k\bmod2}=J_{1-\sigma}$. Since $n\in R$, this means
$\lambda(n)\equiv1-\sigma\not\equiv\lambda(d')\pmod 2$. As $n>d'$, some power $2^i$ lies in
$(d',n]$. Since $n\notin G$ and $n\ge 2^i$, we get $n\ge2^i+w_i$, so
$w_i\le n-2^i<n-d'=d_0+t$.
\item Letting $d'\to\infty$ through $D'$, the index $i$ in (v) tends to infinity (because $2^i>d'$),
while $w_i<d_0+t$ stays bounded. This contradicts $w_i\to\infty$.\qedhere
\end{enumerate}
\end{proof}

\begin{lemma}\label{IA:lem:excess}
In \cref{IA:def:rot}, for every $M\in\N$,
\[
|E\cap[1,M)|\ \ge\ \frac{M-1}{2}+S_{\alpha,\beta}(M)-\sum_{i\ge1:\ 2^i-w_i<M}2w_i,
\]
where $S_{\alpha,\beta}(M)=\sum_{1\le n<M}\varepsilon(n)\,\rho(n\alpha-\beta)$.
\end{lemma}

\begin{proof}
If $\lambda(n)$ is even, then $\one_R(n)=\one_{[0,1/2)}(\{n\alpha-\beta\})$, so
$\one_R(n)-\frac12=\rho(n\alpha-\beta)$; if $\lambda(n)$ is odd, then
$\one_R(n)-\frac12=-\rho(n\alpha-\beta)$. Thus $\one_R(n)-\frac12=\varepsilon(n)\rho(n\alpha-\beta)$
and $|R\cap[1,M)|=\frac{M-1}2+S_{\alpha,\beta}(M)$. Finally $|G\cap[1,M)|$ is at most the sum of the
lengths $2w_i$ of the intervals $[2^i-w_i,2^i+w_i)$ that meet $[1,M)$. (If this sum is infinite,
the bound is trivial; for the $w_i$ used below it is finite.)
\end{proof}

Thus everything hinges on making the twisted discrepancy $S_{\alpha,\beta}(M)$ grow faster than
$\log M$ for \emph{all} large $M$: the gaps cost $\sum 2w_i$, which grows faster than $\log M$
because $w_i\to\infty$.

\begin{proposition}\label{IA:prop:TD}
There are an irrational $\alpha$ and integers $1\le N_1\le N_2\le\cdots$ such that
\[
S_{\alpha,1/6}(M)\ \ge\ i\log M\qquad\text{for all } i\ge1 \text{ and all } M\ge N_i .
\]
\end{proposition}

We prove \cref{IA:prop:TD} in \cref{IA:ssec:TD}, and first deduce part~\ref{IA:it:u32a-all} of the
theorem.

\begin{proof}[Proof of \cref{IA:thm:u32a}\ref{IA:it:u32a-all}]
Let $\alpha$ and $(N_i)$ be as in \cref{IA:prop:TD}; we may assume $N_1\ge2$. Since
$|S_{\alpha,1/6}(M)|\le(M-1)/2$, for each $M\ge2$ the inequality $S_{\alpha,1/6}(M)\ge i\log M$ fails for large
$i$, so $N_i\to\infty$. Let $\psi_0(M)=\max\{i:N_i\le M\}$ for $M\ge N_1$ and $\psi_0(M)=0$ for
$M<N_1$. Then $\psi_0$ is non-decreasing and unbounded, and
$S_{\alpha,1/6}(M)\ge\psi_0(M)\log M$ for $M\ge N_1$. Take $\beta=\frac16$ and
\[
w_i=\max\Big(1,\ \min\big(\lfloor 2^i/4\rfloor,\ \lfloor\psi_0(2^{i-1})/16\rfloor\big)\Big)\qquad(i\ge1),
\]
so that $w_i\to\infty$, and let $E$ be as in \cref{IA:def:rot}.

We bound the gap term of \cref{IA:lem:excess}. We have $w_1=1$. For $i\ge2$ we have $w_i\le2^i/4$, so
$2^i-w_i<M$ implies $\frac34 2^i<M$, hence $2^{i-1}<M$, hence $\psi_0(2^{i-1})\le\psi_0(M)$ and
$w_i\le\max(1,\psi_0(M)/16)$. The indices $i\ge 2$ with $\frac342^i<M$ satisfy
$i<\log_2M+1$, so there are at most $\log_2M+2$ indices in the sum. Hence, once $\psi_0(M)\ge16$,
\[
\sum_{2^i-w_i<M}2w_i\le\frac{\psi_0(M)}{8}\,(\log_2M+2)=\psi_0(M)\Big(\frac{\log M}{8\log 2}+\frac14\Big)
\le\psi_0(M)\big(0.19\log M+0.25\big).
\]
By \cref{IA:lem:excess}, for all large $M$,
\[
|E\cap[1,M)|-\frac{M-1}2\ \ge\ \psi_0(M)\big(0.81\log M-0.25\big)\ \ge\ \tfrac12\psi_0(M)\log M+1 .
\]
Let $A=(2\N-1)\cup2E$. By \cref{IA:lem:parity}, with $M=\lfloor N/2\rfloor+1$,
\[
|A\cap[N]|=\lceil N/2\rceil+|E\cap[1,M)|\ \ge\ \frac N2+\frac{N-1}4+\tfrac12\psi_0(M)\log M+1
\ \ge\ \frac34N+\tfrac12\psi_0(M)\log(N/2)
\]
for all large $N$. With $\psi(N)=\frac14\psi_0(\lfloor N/2\rfloor+1)$, which is non-decreasing and
unbounded, we get $|A\cap[N]|\ge\frac34N+\psi(N)\log N$ for all large $N$. Finally, if $B+B\subset A$
with $B$ infinite, \cref{IA:lem:parity} produces an infinite $D\subset\N_0$ and $t\in\{0,1\}$ with
$D+D+t\subset E$, contradicting \cref{IA:prop:rot}.
\end{proof}

\subsection{A twisted discrepancy estimate}\label{IA:ssec:TD}
For $\theta\in\R$ and $n\in\Z$ let $r_\theta(n)=\rho(n\theta-\frac16)$ and
$S_\theta(M)=\sum_{1\le n<M}\varepsilon(n)r_\theta(n)$, so that $S_\theta=S_{\theta,1/6}$. Let
$\mathcal Q$ be the set of rationals $\theta=a/q$ with
\[
q=3\cdot2^p\ (p\ge1),\qquad 0<a<q,\qquad\gcd(a,q)=1 .
\]
For $\theta=a/q\in\mathcal Q$ (always written in lowest terms) put
\[
\Delta(\theta)=\sum_{2^p\le n<2^{p+1}}r_\theta(n),\qquad G(\theta)=(-1)^p\Delta(\theta),\qquad
\mathcal D(\theta)=\sup_{u\le v}\Big|\sum_{u\le n<v}r_\theta(n)\Big| .
\]
Note that $q$ is even and $a$ is odd, and that $2^p=q/3$.

\begin{lemma}\label{IA:lem:TD0}
Let $\theta=a/q\in\mathcal Q$.
\begin{enumerate}[label=\textup{(\roman*)}]
\item $\sum_{n=u}^{u+q-1}r_\theta(n)=0$ for every $u\in\Z$; in particular $\mathcal D(\theta)\le q/2$.
\item For every $j\ge p$, $\sum_{2^j\le n<2^{j+1}}\varepsilon(n)r_\theta(n)=G(\theta)$.
\item $S_\theta(2^j)=S_\theta(2^p)+(j-p)G(\theta)$ for $j\ge p$, and
$|S_\theta(M)-S_\theta(2^j)|\le\mathcal D(\theta)$ whenever $2^j\le M\le 2^{j+1}$.
\end{enumerate}
\end{lemma}

\begin{proof}
(i) We have $n\theta-\frac16=(na-2^{p-1})/q$ with $2^{p-1}\in\Z$. As $n$ runs through $q$
consecutive integers, $na-2^{p-1}$ runs through all residues modulo $q$, and $\rho(k/q)=\frac12$ for
exactly $q/2$ residues $k$ and $-\frac12$ for the others. Hence the sum vanishes, the function
$D(N)=\sum_{0\le n<N}r_\theta(n)$ is $q$-periodic, every interval sum equals a sum over fewer than $q$
consecutive integers, and $\mathcal D(\theta)\le q/2$.

(ii) For $j\ge p$ let $u_j\in\{1,2\}$ with $u_j\equiv2^{j-p}\pmod 3$; thus $u_j=1$ if $j-p$ is even
and $u_j=2$ otherwise. Since $q=3\cdot2^p$ divides $2^p(2^{j-p}-u_j)$, periodicity gives
$D(2^j)=D(2^pu_j)$. Therefore
$\sum_{2^j\le n<2^{j+1}}r_\theta(n)=D(2^pu_{j+1})-D(2^pu_j)=(-1)^{j-p}\big(D(2^{p+1})-D(2^p)\big)
=(-1)^{j-p}\Delta(\theta)$. On $[2^j,2^{j+1})$ we have $\varepsilon=(-1)^j$, which gives (ii).

(iii) The first claim follows from (ii). If $2^j\le M\le2^{j+1}$, then
\[
S_\theta(M)-S_\theta(2^j)=(-1)^j\sum_{2^j\le n<M}r_\theta(n).
\]
\end{proof}

\begin{lemma}\label{IA:lem:TD1}
Let $\theta=a/q\in\mathcal Q$, let $d\ge1$, $L=2^d$, let $b$ be an odd integer with $0<b<L$, and
suppose that $\theta'=\theta+\frac{b}{qL}=\frac{aL+b}{qL}$ belongs to $\mathcal Q$ (with denominator
$qL$). Then for every set $I$ of at most $qL$ consecutive integers,
\[
\Big|\sum_{n\in I}r_{\theta'}(n)\Big|\le\frac{|I|\,b}{L}+q,\qquad\text{and consequently}\qquad
\mathcal D(\theta')\le q(b+1).
\]
\end{lemma}

\begin{proof}
Let $n_0=\min I$ and $K=\lfloor|I|/q\rfloor$. Split $I$ into the $K$ blocks
$\{n_0+kq+s:0\le s<q\}$, $0\le k<K$, and a remainder of fewer than $q$ integers, which contributes at
most $q/2$ in absolute value. Put $c=b/(2L)\in(0,\frac12)$. Since $a$ is odd,
$\frac q2\theta'=\frac a2+c\equiv\frac12+c\pmod1$, so for every $n$,
\[
r_{\theta'}(n)+r_{\theta'}(n+\tfrac q2)=\rho(x)+\rho(x+\tfrac12+c)=\rho(x)-\rho(x+c),\qquad
x=n\theta'-\tfrac16 .
\]
Within each block pair $s$ with $s+\frac q2$, $0\le s<\frac q2$. Fix $s$, and let
$x_k=(n_0+kq+s)\theta'-\frac16$. Since $q\theta'=a+2c$ with $a\in\Z$, we have
$\rho(x_k)-\rho(x_k+c)=\rho(y_k)-\rho(y_k+c)$ with $y_k=x_0+2ck$. This difference is $0$ unless the
interval $(y_k,y_k+c]$ contains a point of $\frac12\Z$, in which case it is $\pm1$. The intervals
$(y_k,y_k+c]$, $0\le k<K$, are pairwise disjoint and contained in an interval of length $2cK$, which
contains at most $4cK+1$ points of $\frac12\Z$. Summing over the $q/2$ values of $s$, the blocks
contribute at most $\frac q2(4cK+1)=\frac{qKb}{L}+\frac q2\le\frac{|I|b}{L}+\frac q2$. This proves
the first bound. By \cref{IA:lem:TD0}(i) applied to $\theta'$, every interval sum of $r_{\theta'}$
equals a sum over fewer than $qL$ consecutive integers, so $\mathcal D(\theta')\le qL\cdot b/L+q$.
\end{proof}

\begin{lemma}\label{IA:lem:TD2}
Let $\theta=a/q\in\mathcal Q$, $\theta'\ge\theta$ and $n\ge0$ with $n(\theta'-\theta)<1/q$. Then
$r_{\theta'}(n)=r_\theta(n)$.
\end{lemma}

\begin{proof}
$n\theta-\frac16$ lies in $\frac1q\Z$, and $\rho$ is constant on each interval
$[\frac kq,\frac{k+1}q)$ because its jumps lie in $\frac12\Z\subset\frac1q\Z$. The point
$n\theta'-\frac16$ lies in the same interval as $n\theta-\frac16$.
\end{proof}

The next lemma produces, from $\theta\in\mathcal Q$, a nearby $\theta'\in\mathcal Q$ with a much
larger gain $G(\theta')$ per dyadic block.

\begin{lemma}\label{IA:lem:TD3}
Let $\theta=a/q\in\mathcal Q$ with $q=3\cdot2^p$, and put $h=q-a\ge1$. Let $m\ge p$ with $m\equiv
p\pmod 2$, let $b=h2^m+1$ and let $R\ge1$ be the order of $2$ modulo $3b$. For $t\in\N$ put
$d=m+2Rt$, $L=2^d$ and $\theta'=\frac{aL+b}{qL}$. Then $0<b<L$, $\theta'\in\mathcal Q$ with
denominator $qL=3\cdot2^{p+d}$, $p+d$ is even, and
\[
G(\theta')=\Delta(\theta')\ \ge\ \frac{q2^m}{12}-h-q-2 .
\]
\end{lemma}

\begin{proof}
Since $2^R\equiv1\pmod{3b}$ we have $2^R>3b$, hence $L\ge2^{2R}>b$. Moreover
$2^d=2^m(2^R)^{2t}\equiv2^m\pmod{3b}$.
\begin{enumerate}[label=(\arabic*),leftmargin=*]
\item $hL+1\equiv h2^m+1=b\equiv0\pmod b$, so $U=(hL+1)/b$ is an integer.
\item $U-1=h(2^d-2^m)/b$. Here $(2^d-2^m)/b$ is an integer divisible by $3$ (as $3b\mid2^d-2^m$)
and by $2^m$ (as $b$ is odd). Hence $q=3\cdot2^p$ divides $U-1$; in particular $U\equiv1\pmod 6$, say
$U=6v+1$.
\item Put $a'=aL+b$ and $q'=qL$. Since $aU+h\equiv a+h=q\equiv0\pmod q$, the number
$\tau=(aU+h)/q$ is an integer, and
\[
a'U=aLU+bU=aLU+hL+1=L(aU+h)+1=\tau q'+1 .
\]
Hence $\gcd(a',q')=1$ and $\gcd(\tau,U)=1$. Also $0<a'<aL+hL=q'$, so $\theta'=a'/q'\in\mathcal Q$;
the exponent of $2$ in $q'$ is $p+d=p+m+2Rt$, which is even. Moreover
$\theta'=\frac{\tau}{U}+\frac{1}{Uq'}$.
\item Let $N=q'/3=2^{p+d}$, so that $\Delta(\theta')=\sum_{N\le n<2N}r_{\theta'}(n)$. For $n$ in this
range put $z=n/q'\in[\frac13,\frac23)$ and let $m'\in\{0,\dots,U-1\}$ be the residue of $n\tau$
modulo $U$. Then $n\theta'-\frac16\equiv\frac{m'+z}{U}-\frac16\pmod 1$ with $\frac{m'+z}U\in[0,1)$,
so $r_{\theta'}(n)=\frac12$ if and only if $\frac{m'+z}{U}\in[\frac16,\frac23)$, that is,
$m'+z\in[v+\frac16,4v+\frac23)$. Since $z\in[\frac13,\frac23)$, this holds if and only if
$v\le m'\le 4v$. Hence on any $U$ consecutive integers in $[N,2N)$, on which $m'$ takes every value
once (as $\gcd(\tau,U)=1$), there are $3v+1$ values $+\frac12$ and $3v$ values $-\frac12$, and the
sum is $\frac12$.
\item Split $[N,2N)$ into $\lfloor N/U\rfloor$ blocks of $U$ consecutive integers and a remainder of
fewer than $U\le hL+1\le q'$ integers. By \cref{IA:lem:TD1}, the remainder contributes at most
$Ub/L+q=h+\frac1L+q$ in absolute value. Hence
\[
\Delta(\theta')\ge\frac12\Big(\frac NU-1\Big)-h-1-q .
\]
Finally $\frac N{2U}=\frac{qLb}{6(hL+1)}\ge\frac{qb}{6(h+1)}\ge\frac{q2^m}{12}$, using
$b=h2^m+1\ge(h+1)2^{m-1}$. As $p+d$ is even, $G(\theta')=\Delta(\theta')$.\qedhere
\end{enumerate}
\end{proof}

\begin{lemma}\label{IA:lem:TD4}
Let $\theta=a/q\in\mathcal Q$ with $q=3\cdot2^p$, let $c\ge0$, and let $N_0,X_0\in\N$. Suppose that
$G(\theta)>c\log2$ and $S_\theta(M)\ge c\log M$ for all $M\ge N_0$. Let $m,b,R$ be as in
\cref{IA:lem:TD3}, chosen so that $\frac{q2^m}{12}-h-q-2\ge c\log 2$, and let $\theta'_t$ denote
the number $\theta'$ of \cref{IA:lem:TD3} for the parameter $t$. Then for all sufficiently large $t$:
\begin{enumerate}[label=\textup{(\roman*)}]
\item $r_{\theta'_t}(n)=r_\theta(n)$ for $0\le n<X_0$;
\item $S_{\theta'_t}(M)\ge c\log M$ for all $M\ge N_0$.
\end{enumerate}
\end{lemma}

\begin{proof}
Write $\theta'=\theta'_t$, $d=m+2Rt$, $P'=p+d$, and let $\ell=\lceil\log_2b\rceil+1$ and
$J=d-\ell$. For $0\le n<2^J$ we have $n(\theta'-\theta)=\frac{nb}{q2^d}<\frac{b2^{-\ell}}{q}\le
\frac1{2q}$, so $r_{\theta'}(n)=r_\theta(n)$ by \cref{IA:lem:TD2}. In particular (i) holds once
$2^J\ge X_0$, and $S_{\theta'}(M)=S_\theta(M)$ for $M\le2^J$.

Let $B_0=q(b+1)$; by \cref{IA:lem:TD1} (with $\theta'\in\mathcal Q$ by \cref{IA:lem:TD3}),
$\mathcal D(\theta')\le B_0$, and $B_0$ does not depend on $t$. By \cref{IA:lem:TD0}(iii) for
$\theta$,
\[
\Phi(t):=S_\theta(2^p)+(J-p)G(\theta)-(p+\ell)B_0-c(P'+1)\log2-B_0
= d\,\big(G(\theta)-c\log 2\big)+C_0,
\]
where $C_0$ does not depend on $t$. Since $G(\theta)>c\log2$ and $d\to\infty$ as $t\to\infty$, we may
take $t$ so large that $\Phi(t)\ge0$, $J\ge p$ and $2^J\ge\max(X_0,N_0)$. We check (ii) for this $t$.

If $N_0\le M\le 2^J$, then $S_{\theta'}(M)=S_\theta(M)\ge c\log M$.

If $2^J\le M\le 2^{P'}$, the interval $[2^J,M)$ meets at most $P'-J=p+\ell$ dyadic blocks, on each of
which $\varepsilon$ is constant; hence
$S_{\theta'}(M)\ge S_{\theta'}(2^J)-(p+\ell)B_0=S_\theta(2^J)-(p+\ell)B_0
\ge c(P'+1)\log2+B_0\ge c\log M$, using $S_\theta(2^J)=S_\theta(2^p)+(J-p)G(\theta)$ and
$\Phi(t)\ge0$. In particular $S_{\theta'}(2^{P'})\ge c(P'+1)\log2+B_0$.

If $M>2^{P'}$, let $2^j\le M<2^{j+1}$ with $j\ge P'$. By \cref{IA:lem:TD0}(iii) for $\theta'$ (whose
denominator is $3\cdot2^{P'}$) and $G(\theta')\ge c\log2$,
\[
S_{\theta'}(M)\ge S_{\theta'}(2^{P'})+(j-P')G(\theta')-B_0\ge c(P'+1)\log2+(j-P')c\log2=c(j+1)\log2
\ge c\log M .\qedhere
\]
\end{proof}

\begin{proof}[Proof of \cref{IA:prop:TD}]
Let $\theta_1=\frac1{12}$, so $q=12$, $p=2$. For $4\le n<8$ we have
$\frac n{12}-\frac16=\frac{n-2}{12}\in[0,\frac12)$, so $\Delta(\theta_1)=2$ and
$G(\theta_1)=2>2\log2$. By \cref{IA:lem:TD0}(iii), $S_{\theta_1}(M)\ge S_{\theta_1}(4)+2(j-2)-6$ for
$2^j\le M<2^{j+1}$, $j\ge2$, so there is $N_1$ with $S_{\theta_1}(M)\ge\log M$ for $M\ge N_1$.

Suppose that for some $k\ge1$ we have constructed $\theta_1,\dots,\theta_k$ with
$\theta_i=a_i/q_i\in\mathcal Q$, and $N_1\le\dots\le N_k$, such that
\begin{enumerate}[label=(H\arabic*),leftmargin=*]
\item\label{IA:H1} $G(\theta_k)>(k+1)\log2$, and
\item\label{IA:H2} $S_{\theta_k}(M)\ge i\log M$ for all $1\le i\le k$ and all $M\ge N_i$.
\end{enumerate}
(For $k=1$ this was just checked.) Choose $m\ge p_k$ with $m\equiv p_k\pmod2$ and
$\frac{q_k2^m}{12}-h_k-q_k-2>(k+2)\log2$, where $q_k=3\cdot2^{p_k}$, $h_k=q_k-a_k$. Apply
\cref{IA:lem:TD4} with $\theta=\theta_k$, $c=k$ and $N_0=X_0=N_k$; its hypotheses hold by \ref{IA:H1}
and \ref{IA:H2}. Taking $t$ large enough we obtain $\theta_{k+1}=\theta'_t\in\mathcal Q$ such that
$r_{\theta_{k+1}}=r_{\theta_k}$ on $[0,N_k)$, $S_{\theta_{k+1}}(M)\ge k\log M$ for $M\ge N_k$, and in
addition
\begin{equation}\label{IA:eq:eta}
\eta_k:=\theta_{k+1}-\theta_k<\min\Big(\frac1{8q_k^2},\frac{\eta_{k-1}}4\Big)\qquad(\eta_0:=1),
\end{equation}
which is possible because $\theta'_t-\theta_k=b/(q_k2^{m+2Rt})\to0$. By \cref{IA:lem:TD3},
$G(\theta_{k+1})>(k+2)\log2$. Property \ref{IA:H2} for $\theta_{k+1}$ and $i\le k$ holds: for
$N_i\le M\le N_k$ we have $S_{\theta_{k+1}}(M)=S_{\theta_k}(M)\ge i\log M$, and for $M\ge N_k$ we have
$S_{\theta_{k+1}}(M)\ge k\log M\ge i\log M$. For $i=k+1$, \cref{IA:lem:TD0}(iii) and
$G(\theta_{k+1})>(k+2)\log2$ give some $N_{k+1}\ge N_k$ with $S_{\theta_{k+1}}(M)\ge(k+1)\log M$ for
all $M\ge N_{k+1}$.

The sequence $(\theta_k)$ is increasing, and by~\eqref{IA:eq:eta} its limit $\alpha$ satisfies
$0<\alpha-\theta_k<\frac43\eta_k<\frac{1}{6q_k^2}$. Since $q_{k+1}=q_k2^{d}>q_k$, we have
$q_k\to\infty$. If $\alpha=c'/s$ were rational, then $\frac1{sq_k}\le|\alpha-\frac{a_k}{q_k}|
<\frac1{6q_k^2}$ for all $k$, which is impossible for large $k$; so $\alpha$ is irrational. For each
$n\ge1$ the point $n\alpha-\frac16$ is irrational, hence not in $\frac12\Z$, so $\rho$ is continuous
there and $r_{\theta_k}(n)\to r_\alpha(n)$. Consequently $S_{\theta_k}(M)\to S_\alpha(M)$ for each $M$.
Given $i$ and $M\ge N_i$, we have $S_{\theta_k}(M)\ge i\log M$ for all $k\ge i$ by \ref{IA:H2}, and
letting $k\to\infty$ gives $S_\alpha(M)\ge i\log M$.
\end{proof}

\formalization{Both parts of \cref{IA:thm:u32a}, together with
\cref{IA:prop:rot,IA:prop:TD} and the lemmas of \cref{IA:ssec:TD}, are formalized in
\leanfile{S32a} (\leanname{thm_u32a_a}, \leanname{thm_u32a_b}).}


\section{Unrestricted sums \texorpdfstring{$B+B$}{B+B} and lower Banach density}\label{KR:sec:main}

In this section we answer, negatively, both parts of a question of Kousek and
Radi\'c~\cite[Question~5.2]{kousek2024infinite} about lower Banach density thresholds for unrestricted
sumsets $B+B=\{b+b':b,b'\in B\}$; this question is not from~\cite{kra2023problems}. We show that the
lower Banach density thresholds are exactly $\frac12$ (with a shift) and $\frac34$ (without a shift),
the same as the lower density thresholds found in~\cite{kousek2024infinite}. Together with
\cref{IA:sec:u32a}, which concerns what happens exactly at the lower density threshold $\frac34$,
this describes the role of the density hypotheses for unrestricted sumsets $B+B$.

For $A\subset\N$ let
\[
\mathsf d_*(A)=\lim_{H\to\infty}\ \inf_{M\ge1}\frac{|A\cap[M,M+H)|}{H}
\]
denote the lower Banach density (the limit exists by superadditivity; this agrees with the
normalisation $\liminf_{N-M\to\infty}|A\cap[M,N]|/(N-M)$ of~\cite{kousek2024infinite}); $\N_0=\N\cup\{0\}$.
Kousek and Radi\'c proved that $\ldens(A)>\frac34$ implies $B+B\subset A$ for some infinite
$B\subset\N$~\cite[Theorem~1.2]{kousek2024infinite}, and that $\ldens(A)>\frac12$ implies
$B+B+t\subset A$ for some infinite $B$ and some $t\in\{0,1\}$~\cite[Theorem~1.3]{kousek2024infinite}.
They also constructed a syndetic set of lower Banach density $\frac13$ containing no $B+B+t$, and, by
adding the odd numbers, a set of lower Banach density $\frac23$ containing no $B+B$
\cite[Proposition~5.1 and the discussion after it]{kousek2024infinite}, and asked:

\medskip
\noindent\textbf{Question} (Kousek--Radi\'c~\cite[Question~5.2]{kousek2024infinite}). \emph{Let
$A\subset\N$.
\begin{enumerate}[label=\textup{(\arabic*)},leftmargin=*]
\item If $\mathsf d_*(A)>1/3$, does there exist an infinite set $B\subset\N$ and a number $t\in\N$ so
that $B+B+t\subset A$?
\item If $\mathsf d_*(A)>2/3$, does there exist an infinite set $B\subset\N$ so that $B+B\subset A$?
\end{enumerate}}
\medskip

\begin{theorem}\label{KR:thm:main}
Let $\alpha\in\R\setminus\Q$, let $\kappa>0$ satisfy $\kappa\log2\equiv\frac12\pmod1$ (for instance
$\kappa=1/(2\log 2)$), let $0<\delta<\frac14$ and $c=\frac12-2\delta$. Put $\varphi(n)=n\alpha-\kappa\log n$
and
\[
E=\{n\in\N:\ \{\varphi(n)\}\in[0,c)\},\qquad A'=(2\N-1)\cup2E .
\]
\begin{enumerate}[label=\textup{(\alph*)}]
\item\label{KR:it:dens} $E$ has uniform density $c$: for every $\eps>0$ there is $H_0$ such that
$\big||E\cap[M,M+H)|-cH\big|\le\eps H$ for all $M\ge1$ and $H\ge H_0$. In particular
$\mathsf d_*(E)=\ubd(E)=\frac12-2\delta$, and similarly $\mathsf d_*(A')=\ubd(A')=\frac34-\delta$.
\item\label{KR:it:E} For every $t\in\Z$ and every infinite $B\subset\N_0$, $B+B+t\not\subset E$.
\item\label{KR:it:A} For every infinite $B\subset\N$, $B+B\not\subset A'$.
\end{enumerate}
\end{theorem}

\begin{corollary}\label{KR:cor:main}
Both parts of Question~5.2 of~\cite{kousek2024infinite} have negative answers (take
$0<\delta<\frac1{12}$). Moreover
\begin{align*}
\sup\{\mathsf d_*(A):\ B+B+t\not\subset A\ \text{for all } t\ge0 \text{ and all infinite } B\}&=\tfrac12,\\
\sup\{\mathsf d_*(A):\ B+B\not\subset A\ \text{for all infinite } B\}&=\tfrac34 .
\end{align*}
\end{corollary}

\begin{proof}
The lower bounds come from \cref{KR:thm:main}. For the upper bounds note that
$\ldens(A)\ge\mathsf d_*(A)$ (take $M=1$ in the definition). Hence $\mathsf d_*(A)>\frac12$ gives
$B+B+t\subset A$ for some infinite $B$ and $t\in\{0,1\}$ by~\cite[Theorem~1.3]{kousek2024infinite},
and $\mathsf d_*(A)>\frac34$ gives $B+B\subset A$ by~\cite[Theorem~1.2]{kousek2024infinite}.
Since $B+B\subset A$ implies $C+C+2\subset A$ for the infinite set $C=\{b-1:b\in B,\ b\ge2\}$, the
first supremum is the same whether $t$ ranges over $t\ge0$ or over $t\in\N$, as in the question.
\end{proof}

\begin{remark}\label{KR:rem:credit}
The obstruction is the logarithmic-scale mechanism used in~\cite[Example~3.6]{kra2023problems} and
in~\cite[Proposition~5.1]{kousek2024infinite}: for $b'$ fixed and $b\to\infty$, the numbers $2b+t$ and
$b+b'+t$ differ by a factor close to $2$, so the phase $\kappa\log n$ of the two differs by about
$\frac12$. We combine it with an irrational rotation $n\alpha$, which makes the set uniformly dense
(the logarithmic obstruction alone has lower Banach density $0$), while the rotation phases of $2b+t$ and
$b+b'+t$ still nearly agree along a subsequence of $B$ on which $b\alpha$ converges, because both have
total coefficient $2$ on $B$. The proof uses the diagonal element $2b+t$. The sharpness examples
of~\cite[Proposition~4.2]{kousek2024infinite} for lower density contain arbitrarily long gaps, so they
have lower Banach density $0$; \cref{KR:cor:main} shows that lower Banach density gives no improvement
over lower density. We do not know whether the boundary values are attained, that is, whether some
$A$ with $\mathsf d_*(A)=\frac12$ contains no $B+B+t$, or some $A$ with $\mathsf d_*(A)=\frac34$ contains
no $B+B$.
\end{remark}

\begin{lemma}\label{KR:lem:rot}
For $L\in\N$ let
\[
D_L=\sup_{\gamma\in\R}\ \sup_{0\le a\le b\le1}\ \frac1L\Big|\#\{0\le j<L:\ \{j\alpha+\gamma\}\in[a,b)\}-(b-a)L\Big| .
\]
Then $D_L\to0$ as $L\to\infty$.
\end{lemma}

\begin{proof}
Let $K\ge2$. By Weyl's equidistribution theorem, for each of the $K$ arcs $I_i=[\frac iK,\frac{i+1}K)$
of $\T$ we have $\#\{0\le j<L:j\alpha\bmod1\in I_i\}=L/K+o(L)$, so there is $L_0$ such that for $L\ge
L_0$ each of these counts lies in $[(1-\frac1K)\frac LK,(1+\frac1K)\frac LK]$. Now $\{j\alpha+\gamma\}\in
[a,b)$ if and only if $j\alpha\bmod1$ lies in the arc $J=[a-\gamma,b-\gamma)$ of length $b-a$. The arc
$J$ contains at least $K(b-a)-2$ and is covered by at most $K(b-a)+2$ of the arcs $I_i$. Hence for
$L\ge L_0$ the count lies between $(K(b-a)-2)(1-\frac1K)\frac LK$ and $(K(b-a)+2)(1+\frac1K)\frac LK$,
so $D_L\le 4/K$.
\end{proof}

\begin{proof}[Proof of \cref{KR:thm:main}\ref{KR:it:dens}]
Let $0<\eps<\min(c,1-c)$, choose $L$ with $D_L\le\eps/4$, put $K=\lceil4\kappa/\eps\rceil$, and let
$H\ge H_0=8(K+1)L/\eps$. Let $W=[M,M+H)$ with $M\ge1$. At most $KL\le\eps H/8$ integers of $W$ are
smaller than $KL$. Split the rest of $W$ into consecutive blocks $\{m,\dots,m+L-1\}$, $m\ge KL$, and a
remainder of fewer than $L\le\eps H/8$ integers. For $n=m+j$, $0\le j<L$, we have
\[
\varphi(n)=j\alpha+\gamma_m-\eta_n,\qquad \gamma_m=m\alpha-\kappa\log m,\qquad
\eta_n=\kappa\log\frac{m+j}m\in\Big[0,\frac{\kappa L}m\Big)\subset\Big[0,\frac\eps4\Big).
\]
For $x\in\R$ and $0\le\eta<\eps/4$ we have, since $c+\eps/4<1$,
\[
\{x\}\in[\tfrac\eps4,c)\ \Longrightarrow\ \{x-\eta\}\in[0,c)\ \Longrightarrow\ \{x\}\in[0,c+\tfrac\eps4).
\]
Applying this with $x=j\alpha+\gamma_m$ and using \cref{KR:lem:rot}, each block contains between
$L(c-\frac\eps2)$ and $L(c+\frac\eps2)$ elements of $E$. If $P$ denotes the number of integers of $W$ lying in
blocks, then $(1-\frac\eps4)H\le P\le H$, and therefore
\[
(c-\tfrac\eps2)(1-\tfrac\eps4)H\ \le\ |E\cap W|\ \le\ (c+\tfrac\eps2)H+\tfrac\eps4H,
\]
so $\big||E\cap W|-cH\big|\le\eps H$. For $A'$: a window $[M,M+H)$ contains $\frac H2+O(1)$ odd numbers,
and its even elements of $A'$ are the numbers $2n$ with $n$ in $E\cap[\frac M2,\frac{M+H}2)$, a window of
length $\frac H2+O(1)$; so by the above they number $\frac{cH}2+o(H)$, uniformly in $M$. Hence $A'$ has
uniform density $\frac{1+c}2=\frac34-\delta$.
\end{proof}

\begin{proof}[Proof of \cref{KR:thm:main}\ref{KR:it:E}]
Let $\eta=\frac12-c=2\delta$, and suppose that $t\in\Z$ and $B\subset\N_0$ is infinite with
$B+B+t\subset E$. By compactness of $\T$ there are an infinite $B_1\subset B$ and $\beta\in\T$ with
$b\alpha\to\beta$ in $\T$ as $b\to\infty$ in $B_1$. Choose $b'\in B_1$ with $\|b'\alpha-\beta\|<\eta/4$, where
$\|\cdot\|$ is the distance to the nearest integer. Since $(2b+t)/(b+b'+t)\to2$ as $b\to\infty$, we can
choose $b\in B_1$, $b>b'$, with
\[
\|b\alpha-\beta\|<\frac\eta4\qquad\text{and}\qquad
|r|<\frac\eta2,\quad\text{where } r=\kappa\Big(\log\frac{2b+t}{b+b'+t}-\log2\Big).
\]
The numbers $m_1=2b+t=b+b+t$ and $m_2=b+b'+t$ belong to $B+B+t\subset E$ (here the diagonal element
$b+b$ is used). Let $x=\{\varphi(m_1)\}$ and $y=\{\varphi(m_2)\}$. Then $x,y\in[0,c)$, so
$x-y\in(-c,c)$ and
\[
\operatorname{dist}\big(x-y,\tfrac12+\Z\big)\ \ge\ \tfrac12-c\ =\ \eta .
\]
On the other hand,
\[
\varphi(m_1)-\varphi(m_2)=(b-b')\alpha-\kappa\log\frac{2b+t}{b+b'+t}=(b-b')\alpha-\kappa\log2-r,
\]
and $\kappa\log2\in\frac12+\Z$. Since $x-y\equiv\varphi(m_1)-\varphi(m_2)\pmod1$, we get
\[
\operatorname{dist}\big(x-y,\tfrac12+\Z\big)\le\|(b-b')\alpha\|+|r|
\le\|b\alpha-\beta\|+\|\beta-b'\alpha\|+|r|<\eta,
\]
a contradiction.
\end{proof}

\begin{proof}[Proof of \cref{KR:thm:main}\ref{KR:it:A}]
Suppose $B\subset\N$ is infinite and $B+B\subset A'$. Choose $\sigma\in\{0,1\}$ such that
$B_\sigma=B\cap(2\Z+\sigma)$ is infinite, and let $C=\{(b-\sigma)/2:b\in B_\sigma\}\subset\N_0$. For
$c_1,c_2\in C$ the number $(2c_1+\sigma)+(2c_2+\sigma)=2(c_1+c_2+\sigma)$ is an even element of $A'$,
so $c_1+c_2+\sigma\in E$. Thus $C+C+\sigma\subset E$, contradicting part~\ref{KR:it:E}.
\end{proof}

\formalization{\Cref{KR:thm:main} and the lower bounds in \cref{KR:cor:main} are formalized in
\leanfile{KR52} (\leanname{main}, \leanname{q52_negative}, \leanname{sup_lower_bounds}). The
upper bounds in \cref{KR:cor:main} are Theorems~1.2 and~1.3 of~\cite{kousek2024infinite} and
are not formalized.}

%

\section{Ordered polynomial sumsets}\label{IA:sec:ordered}

In this section we answer two unnumbered questions from~\cite[\S3.3]{kra2023problems}. In
\cref{IA:sec:u33a} we determine exactly which ordered polynomial sums
$\{f(b_1)+g(b_2):b_1<b_2\}$ appear, up to a shift, in every set of positive density: those with
$\deg f\le1$ and $\deg g\le1$. In \cref{IA:sec:u33e} we observe that the natural generalisations of
Question~3.18 of~\cite{kra2023problems} have negative answers. Both answers rely on work of
Ackelsberg~\cite{ackelsberg2024counterexamples} and Kousek~\cite{kousek2025asymmetric}; our
contribution is small and we say precisely what it is.

\subsection{Which ordered polynomial sums are unavoidable}\label{IA:sec:u33a}
For functions $f,g$ and $B\subset\N$ write
\[
f(B)\oplus_<g(B)=\{f(b_1)+g(b_2):\ b_1,b_2\in B,\ b_1<b_2\}.
\]
The survey considers these ``ordered sums'' for polynomials $f,g$ ``over $\N$''
\cite[(5)]{kra2023problems}, notes that they are partition regular by Ramsey's theorem, and asks:

\begin{question}{\S3.3(a)}\label{IA:q:u33a}
Let $f$ and $g$ be polynomials over $\N$. Can instances of $f(B)\oplus_<g(B)$ be found, up to a
shift, in every set with positive density? That is, does every $A\subset\N$ of positive density
contain $f(B)\oplus_<g(B)+t$ for some infinite $B\subset\N$ and some $t\ge0$?
\end{question}

The survey does not specify the density; its default is upper Banach density. Special cases are
Theorem~2.2 of~\cite{kra2023problems} ($f=g=x$, positive, in the restricted form), Conjecture~3.10
($f,g$ linear), Question~3.13 ($f=x^2$, $g=x$) and Example~3.14 ($g=x^2$). We interpret ``polynomials
over $\N$'' in the widest sensible way and let
\[
\mathcal P=\{f\in\Q[x]:\ f(\N)\subset\N_0\};
\]
this contains all polynomials with coefficients in $\N_0$. We use the convention
$\deg0=-\infty$, so that ``$\deg f\le1$'' means that $f$ is constant or linear.

\begin{theorem}\label{IA:thm:u33a}
Let $f,g\in\mathcal P$.
\begin{enumerate}[label=\textup{(\alph*)}]
\item\label{IA:it:u33a-pos} If $\deg f\le1$ and $\deg g\le1$, then for every $A\subset\N$ with
$\ubd(A)>0$ there are an infinite $B\subset\N$ and $t\ge0$ with $f(B)\oplus_<g(B)+t\subset A$.
\item\label{IA:it:u33a-neg} If $\max(\deg f,\deg g)\ge2$, then for every $\eps>0$ there is a set
$A\subset\N$ with $\ldens(A)>1-\eps$ such that $f(B)\oplus_<g(B)+t\not\subset A$ for every
infinite $B\subset\N$ and every $t\in\Z$. Moreover, one set $A$ works simultaneously for all such
pairs $(f,g)$.
\end{enumerate}
\end{theorem}

The hypothesis in~\ref{IA:it:u33a-pos} is positive upper Banach density, and the set in
\ref{IA:it:u33a-neg} has lower density close to $1$, so its upper and lower natural densities, its
upper and lower logarithmic densities and its upper Banach density are all positive. Hence the classification
is the same for each of these notions of density in \cref{IA:q:u33a}.

\begin{remark}[Attribution]\label{IA:rem:u33a-credit}
Part~\ref{IA:it:u33a-pos} is a direct consequence of Kousek's theorem~\cite[Theorem~1.2]{kousek2025asymmetric}
(which proved Conjecture~3.10 of~\cite{kra2023problems}); we only add the bookkeeping for constant
terms and zero coefficients. Part~\ref{IA:it:u33a-neg} was known for $\deg g\ge2$ in the reversed or
sparse forms of~\cite[Example~3.14]{kra2023problems}, and for $g(x)=x$ by Ackelsberg's
theorem~\cite[Theorem~1.5]{ackelsberg2024counterexamples}. Our proof of~\ref{IA:it:u33a-neg} is
Ackelsberg's argument with $P(x)+x$ replaced by $f(x)+g(x)$; to our knowledge, the cases with
$\deg f\ge2$ and $g(x)=mx+c$, $m\ne1$, were not covered explicitly in the literature.
\end{remark}

\begin{remark}[Scope]\label{IA:rem:u33a-scope}
\begin{enumerate}[label=(\roman*)]
\item The argument for~\ref{IA:it:u33a-neg} does not use the order $b_1<b_2$: the same set $A$
also avoids $\{f(b_1)+g(b_2):b_1>b_2\}+t$ and $\{f(b_1)+g(b_2):b_1\neq b_2\}+t$.
\item Integer polynomials that are not $\N_0$-valued on $\N$ are outside \cref{IA:q:u33a}, and the
classification does not extend to them. For instance, for $f(x)=-x$ and $g(x)=x$ the configuration is
the infinite set $\{b_2-b_1:b_1<b_2\}$ and $f+g=0$; since $|R_k(0)|=1$, the proof
of~\ref{IA:it:u33a-neg} below shows that it can be avoided by sets of lower density $>1-\eps$, although
both polynomials are linear.
\item The survey remarks that it is immaterial whether one uses $f(B)\oplus_<g(B)$ or
$\{r+s:r\in f(B),s\in g(B),r<s\}$. This is not accurate in general. If $g$ is non-constant and
$f(b)\ge g(b)$ for all large $b$, the two existence problems are equivalent, after passing to an
infinite subset of $B$: refining $B$ we may assume
$f(b_1)<g(b_2)$ for all $b_1<b_2$ in $B$, and also that $f$ is non-decreasing and $f\ge g$ on $B$, so
that $f(b_1)<g(b_2)$ forces $b_1<b_2$. If instead $f(b)<g(b)$ for all large $b$, the second set
contains the diagonal elements $f(b)+g(b)$, and the answer can change even for linear $f,g$: for
$f(x)=x$, $g(x)=2x$ the second set contains $\{b_1+2b_2:b_1\le b_2\}$, and
Kousek~\cite[Proposition~5.1]{kousek2025asymmetric} (with $m=1$, $\ell=2$) gives a set of positive
upper density containing no shift $\{b_1+2b_2:b_1\le b_2\}+t$, $t\in\N$, with $B$ infinite (the
case $t=0$ reduces to $t=3$ by replacing $B$ with $\{b-1:b\in B,\ b\ge2\}$). If $g\equiv c$ and $\deg f\ge2$, the
second set is finite, while $f(B)\oplus_<g(B)=f(B)+c$ is infinite (every element of $B$ has a
larger one).
\Cref{IA:thm:u33a} uses the formulation~\cite[(5)]{kra2023problems}.
\item In $\Q$ the situation is different: Ackelsberg~\cite[Theorem~1.4]{ackelsberg2025infinite}
proved that if $P\in\Q[x]$ and $P(x)+x$ are non-constant, then every $A\subset\Q$ of positive upper
Banach density contains $\{b_i,\ P(b_i)+b_j:i<j\}+t$ for some infinite $\{b_n\}\subset\Q$ and some
$t\in\Q$.
\end{enumerate}
\end{remark}

We collect elementary facts about $\mathcal P$.

\begin{lemma}\label{IA:lem:P}
Let $f,g\in\mathcal P$.
\begin{enumerate}[label=\textup{(\roman*)}]
\item $f(\Z)\subset\Z$, and there is $D\in\N$ with $Df\in\Z[x]$. For such $D$ and every $k\in\N$,
$n\equiv n'\pmod{Dk}$ implies $f(n)\equiv f(n')\pmod k$.
\item If $f$ is non-constant, its leading coefficient is positive. Consequently
$\deg(f+g)=\max(\deg f,\deg g)$.
\item If $\deg f\le1$, then $f(x)=mx+c$ with $m\in\N_0$ and $c\in\Z$.
\end{enumerate}
\end{lemma}

\begin{proof}
(i) Write $f(x)=\sum_{i}\Delta^if(1)\binom{x-1}{i}$ (Newton's forward difference formula); the
numbers $\Delta^if(1)$ are integers because $f(\N)\subset\Z$, and $\binom{x-1}i\in\Z$ for $x\in\Z$.
A common denominator $D$ of the coefficients exists since $f\in\Q[x]$. If $n-n'=Dks$, then
$D(f(n)-f(n'))=(n-n')Q$ for some $Q\in\Z$ (as $Df\in\Z[x]$), so $f(n)-f(n')=ksQ$.
(ii) A non-constant polynomial with negative leading coefficient tends to $-\infty$. If
$\deg f=\deg g$, the leading coefficients are both positive and do not cancel.
(iii) $m=f(2)-f(1)\in\Z$, $c=f(1)-m\in\Z$, and $m\ge0$ by (ii).
\end{proof}

For an integer-valued $h\in\Q[x]$ and $k\in\N$ let $R_k(h)=\{h(n)\bmod k:n\in\Z\}\subset\Z/k\Z$.

\begin{lemma}[Ackelsberg]\label{IA:lem:Rk}
Let $h$ be a polynomial in $\Q[x]$ of degree at least $2$ such that $h(\Z)\subset\Z$. Then
$\inf_{k\in\N}|R_k(h)|/k=0$.
\end{lemma}

\begin{proof}
Let $\P_h$ be the set of primes $p$ such that $h(\Z)$ meets every residue class modulo $p$. By
Ackelsberg's Lemma~4.3 in~\cite{ackelsberg2024counterexamples}, $\sum_{p\in\P\setminus\P_h}1/p=\infty$,
so $\prod_{p\in\P\setminus\P_h}(1-1/p)=0$. If $p_1,\dots,p_s\in\P\setminus\P_h$ are distinct and
$k=p_1\cdots p_s$, then by the Chinese remainder theorem every element of $R_k(h)$ reduces modulo
$p_i$ to an element of $R_{p_i}(h)$, so $|R_k(h)|\le\prod_i|R_{p_i}(h)|\le\prod_i(p_i-1)$ and
$|R_k(h)|/k\le\prod_i(1-1/p_i)$. This is the argument of the proof
of~\cite[Theorem~1.5]{ackelsberg2024counterexamples}, applied to $h$ instead of $P(x)+x$.
\end{proof}

\begin{lemma}\label{IA:lem:subadd}
Let $E_1,E_2,\ldots\subset\N$ have natural densities $\delta_1,\delta_2,\dots$. There is
$S\subset\N$ such that $E_j\setminus S$ is finite for every $j$ and $\udens(S)\le\sum_j\delta_j$.
\end{lemma}

\begin{proof}
This is the natural-density version of~\cite[Lemma~3.1]{ackelsberg2024counterexamples}. Choose
$M_1<M_2<\cdots$ such that $|E_i\cap[N]|\le(\delta_i+2^{-J})N$ for all $J$, all $i\le J$ and all
$N\ge M_J$, and let $S=\bigcup_j\big(E_j\cap(M_j,\infty)\big)$. Then $E_j\setminus S\subset[M_j]$.
If $M_J\le N<M_{J+1}$, then $S\cap[N]\subset\bigcup_{j\le J}(E_j\cap[N])$, so
$|S\cap[N]|\le(\sum_j\delta_j+J2^{-J})N$.
\end{proof}

\begin{proof}[Proof of \cref{IA:thm:u33a}\ref{IA:it:u33a-neg}]
The set of pairs $(f,g)\in\mathcal P^2$ with $\max(\deg f,\deg g)\ge2$ is countable; enumerate the
triples $(f,g,t)$ with such $(f,g)$ and $t\in\Z$ as $(f_j,g_j,t_j)_{j\ge1}$. Put $h_j=f_j+g_j$. By
\cref{IA:lem:P}, $h_j$ is integer-valued on $\Z$ and $\deg h_j\ge2$, so by \cref{IA:lem:Rk} there is
$k_j$ with $|R_{k_j}(h_j)|/k_j<\eps2^{-j}$. Let
\[
E_j=\{x\in\N:\ (x-t_j)\bmod k_j\in R_{k_j}(h_j)\},
\]
a union of residue classes of natural density $\delta_j=|R_{k_j}(h_j)|/k_j$. Let $S$ be given by
\cref{IA:lem:subadd} and $A=\N\setminus S$. Then $\ldens(A)=1-\udens(S)\ge1-\sum_j\delta_j>1-\eps$.

Suppose that $f_j(B)\oplus_<g_j(B)+t_j\subset A$ for some $j$ and some infinite $B$. Let $k=k_j$ and
let $D$ be such that $Df_j,Dg_j\in\Z[x]$. By the pigeonhole principle there is $n$ such that
$B'=B\cap(n+Dk\Z)$ is infinite. By \cref{IA:lem:P}(i), $f_j(b_1)+g_j(b_2)\equiv f_j(n)+g_j(n)=h_j(n)
\pmod k$ for all $b_1,b_2\in B'$, so $f_j(B')\oplus_<g_j(B')+t_j\subset A\cap E_j\subset E_j\setminus
S$, which is finite. But $f_j(B')\oplus_<g_j(B')$ is infinite: if $g_j$ is non-constant, fix
$b_1=\min B'$ and let $b_2\to\infty$; if $g_j$ is constant, then $\deg f_j\ge2$ and $f_j(b_1)\to\infty$
as $b_1\to\infty$ in $B'$. This contradiction proves~\ref{IA:it:u33a-neg} for every $(f,g)$ and every
$t\in\Z$ simultaneously.
\end{proof}

\begin{proof}[Proof of \cref{IA:thm:u33a}\ref{IA:it:u33a-pos}]
By \cref{IA:lem:P}(iii), $f(x)=mx+c$ and $g(x)=\ell x+c'$ with $m,\ell\in\N_0$ and $c,c'\in\Z$. Let
$\ubd(A)>0$; in particular $A$ is infinite.

If $m,\ell\ge1$, Kousek's theorem~\cite[Theorem~1.2]{kousek2025asymmetric} gives an infinite $B$ and
$t\in\N$ with $\{mb_1+\ell b_2:b_1,b_2\in B,\,b_1<b_2\}+t\subset A$. For $j\in\N$ let
$B_j=\{b-j:b\in B,\,b>j\}$. Then
\[
f(B_j)\oplus_<g(B_j)+t_j'=\{mb_1+\ell b_2:\ b_1,b_2\in B,\ j<b_1<b_2\}+t\subset A,
\qquad t_j'=t-c-c'+(m+\ell)j,
\]
and $t'_j\ge0$ for $j$ large.

If $m=0<\ell$, then $f(B)\oplus_<g(B)=\{\ell b+c+c':b\in B\setminus\{\min B\}\}$. Choose $r$ such
that $A\cap(r+\ell\Z)$ is infinite, then $r'\equiv r\pmod\ell$ with $r'\ge c+c'$, and put $t=r'-c-c'$
and $B=\{(x-r')/\ell:\ x\in A,\ x\equiv r\ (\mathrm{mod}\ \ell),\ x>r'\}$. Then $\ell b+c+c'+t\in A$ for
all $b\in B$. The case $\ell=0<m$ is the same, since then
$f(B)\oplus_<g(B)=\{mb+c+c':b\in B\}$ (every element of an infinite $B$ has a larger one). If
$m=\ell=0$, the configuration is the single point $c+c'$; take $t=x-c-c'$ for some $x\in A$ with
$x\ge c+c'$.
\end{proof}

\formalization{\Cref{IA:thm:u33a} is formalized in \leanfile{S33a}. Part~\ref{IA:it:u33a-neg}
(\leanname{thm_u33a_neg}) is unconditional; there Ackelsberg's Lemma~4.3, used in the proof of
\cref{IA:lem:Rk}, is proved by a different argument (Hermite's criterion and Dirichlet's
theorem). Part~\ref{IA:it:u33a-pos} (\leanname{thm_u33a_pos}) takes Kousek's
theorem~\cite[Theorem~1.2]{kousek2025asymmetric} as an explicit hypothesis.}

\subsection{Generalisations of Question~3.18}\label{IA:sec:u33e}
The survey asks whether shifts of
$\{b_i^2+b_j:i<j\}$ appear in every set of positive upper Banach density
\cite[Question~3.13]{kra2023problems}, and generalises this by adding the base point $b_j$
\cite[Question~3.15]{kra2023problems} and by replacing $x^2$ with arbitrary polynomials
$p_1,\dots,p_k\in\Z[x]$ with $p_\ell(0)=0$ \cite[Question~3.16]{kra2023problems}. Its
Remark~3.17 records that the answers to Questions~3.15 and~3.16 are negative by Ackelsberg's work:
the configuration of Question~3.15 contains that of Question~3.13, and Question~3.16 specialises to
it for the one-element family $\{x^2\}$; Ackelsberg showed that this configuration can be
avoided~\cite[Theorem~1.5]{ackelsberg2024counterexamples}. Then it asks:

\begin{question}{3.18}\label{IA:q:318}
If $A\subset\N$ has positive upper Banach density, must there exist an infinite set
$B=\{b_1<b_2<\cdots\}$ and $t\ge0$ such that $\{b_ib_j+b_k:\ i<j<k\}$ is contained in $A-t$?
\end{question}

\begin{question}{\S3.3(e)}\label{IA:q:u33e}
Similar to the way Question~3.13 is generalised in Questions~3.15 and~3.16, one can inquire about
analogous generalisations of Question~3.18.
\end{question}

Two natural literal interpretations are:
\begin{itemize}
\item[(F1)] (base point, as in Question~3.15) must every $A$ with $\ubd(A)>0$ contain
$\{b_k,\ b_ib_j+b_k:\ i<j<k\}+t$ for some infinite $B$ and $t\ge0$?
\item[(F2)] (polynomial family, as in Question~3.16) given $p_1,\dots,p_r\in\Z[x,y]$ with
$p_\ell(0,0)=0$, must every $A$ with $\ubd(A)>0$ contain
$\{p_\ell(b_i,b_j)+b_k:\ 1\le\ell\le r,\ i<j<k\}+t$ for some infinite $B$ and $t\ge0$?
\end{itemize}
As in Remark~3.17 of~\cite{kra2023problems}, (F2) is a question about all such families.

\begin{proposition}\label{IA:prop:u33e}
Both (F1) and (F2) have negative answers. More precisely, for every $\eps>0$ there is $A\subset\N$
with $\ubd(A)>1-\eps$ such that for no $t\in\Z$ and no infinite $B=\{b_1<b_2<\cdots\}\subset\N$ does
$A-t$ contain the configuration in (F1), or the configuration in (F2) for the one-element family
$p_1(x,y)=xy$.
\end{proposition}

\begin{proof}
Ackelsberg~\cite[Theorem~1.7]{ackelsberg2024counterexamples} proved that for every $\eps>0$ there is
$A\subset\N$ with $\ubd(A)>1-\eps$ such that whenever $t\in\Z$ and
$\{b_1b_2+b_3+t:\ b_1,b_2,b_3\in B,\ b_1<b_2<b_3\}\subset A$, the set $B$ is finite; this answers
\cref{IA:q:318} negatively. The configuration in (F1), and the configuration in (F2) for $p_1=xy$,
both contain $\{b_ib_j+b_k:i<j<k\}$, so this set $A$ has the required property.
\end{proof}

\formalization{\Cref{IA:prop:u33e} and the remark following it are formalized in
\leanfile{S33e} (\leanname{prop_u33e}, \leanname{F2_sq}), including proofs of the two quoted
theorems of Ackelsberg in the form used here.}

For (F2) one does not even need Question~3.18: for the one-element family $p_1(x,y)=x^2$, if
$\{b_i^2+b_k:i<j<k\}+t\subset A$, then $C=\{b_1,b_3,b_5,\dots\}$ satisfies
$\{c^2+c':c,c'\in C,\ c<c'\}+t\subset A$ (between $b_{2i-1}$ and $b_{2i'-1}$, $i<i'$, lies $b_{2i}$),
which is the configuration of Question~3.13 of~\cite{kra2023problems}; by Ackelsberg's
theorem~\cite[Theorem~1.5]{ackelsberg2024counterexamples}, recorded in~\cite{kra2023problems},
there are sets of upper Banach density arbitrarily close to $1$ in which no translate contains it.

\begin{remark}\label{IA:rem:u33e}
This answer is easy: the generalisations contain the configuration of Question~3.18 itself, which
was already answered negatively by Ackelsberg~\cite[Theorem~1.7]{ackelsberg2024counterexamples}
(the second arXiv version of~\cite{kra2023problems} does not record this), in exactly the way that
Questions~3.15 and~3.16 are answered by Remark~3.17 of~\cite{kra2023problems}. For (F2) the survey
itself already suffices, through the family $\{x^2\}$. We include it for completeness. A more
interesting question is which single polynomials $p\in\Z[x,y]$ have the property that every set of
positive upper Banach density contains a translate of $\{p(b_i,b_j)+b_k:i<j<k\}$ for some infinite $B$;
we do not know the answer, for instance, for $p(x,y)=x(y-x)$.
\end{remark}

\section{Sumsets of the form \texorpdfstring{$B+mB$}{B+mB}}\label{IA:sec:BmB}

This section concerns Subsection~3.4 of~\cite{kra2023problems}. For $B\subset\N$ and $m\in\N$ the
survey writes $B+mB=\{b+mb':b,b'\in B\}$; this unrestricted sumset contains the diagonal elements
$(m+1)b$. We write
\[
B\oplus mB=\{b+mb':\ b,b'\in B,\ b\ne b'\}
\]
for the restricted version. We answer Question~3.20 negatively, for two and for three or more
colours, even after excluding $m=1$ (\cref{IA:sec:q320}); Question~3.23 negatively as stated, the
answer depending on the diagonal (\cref{IA:sec:q323}); and Question~3.24 negatively
(\cref{IA:sec:q324}), where a counterexample is in fact already implicit in the work of Kousek and
Radi\'c~\cite{kousek2024infinite}.

\subsection{Colourings and Question~3.20}\label{IA:sec:q320}
A finite colouring of $\N$ is a map $\chi\colon\N\to\{0,1,\dots,r-1\}$. The survey gives, for each
$m$, a $3$-colouring with no monochromatic $B+mB$, $B$ infinite~\cite[Example~3.9]{kra2023problems}.
Two colours already suffice: Hindman~\cite[Theorem~2.11]{hindman1979partitions} proved that for every
$m\ge2$ the $2$-colouring $n\mapsto\lfloor\log_mn\rfloor\bmod2$ admits no infinite $B$ with
$\{mx+y:x,y\in B,\,x\ne y\}$ monochromatic (see also~\cite[Remark~1.5]{huang2026affirmative}).
Observing that such colourings depend heavily on $m$, the survey asks:

\begin{question}{3.20}\label{IA:q:320}
Is it true that for any finite colouring of $\N$ there exists $q\in\N$ such that for all $m\in\N$
with $\gcd(m,q)=1$ there is an infinite set $B\subset\N$ such that $B+mB$ is monochromatic?
\end{question}

The motivation given is that ``the set of problematic $m$'' might be ``small in a multiplicative
sense''. Since $\gcd(1,q)=1$ for every $q$, a positive answer for a colouring $\chi$ would require an
infinite $B$ with $B+B$ monochromatic. Hindman~\cite[Theorem~2.4]{hindman1979partitions} constructed
a $3$-colouring of $\N$ with no such $B$, so \cref{IA:q:320} has a negative answer as stated, for
three or more colours. This hinges on the literal formulation: for two colours, $m=1$ is harmless by
the recent solution of Owings' problem~\cite{owings1974problem} by Huang, Lian, Shao, Xiao, Xu and
Zhang~\cite[Theorem~1.1]{huang2026affirmative}, and the heuristic behind the question concerns
$m\ge2$. We show that the answer is negative also for $m\ge2$ and for two colours (for two colours,
with the unrestricted $B+mB$ of the survey). To our knowledge, what is new is that a \emph{single} colouring is bad for
infinitely many $m$ simultaneously: for three or more colours a set of bad $m$ of positive
logarithmic density containing infinitely many primes, and for two colours infinitely many bad
primes; for each individual $m$, Hindman's theorem already provides a $2$-colouring for which $m$ is
strongly bad.

For a colouring $\chi$ call $m\in\N$ \emph{bad} if no infinite $B\subset\N$ has $B+mB$
monochromatic, and \emph{strongly bad} if no infinite $B$ has $B\oplus mB$ monochromatic.
Strongly bad integers are bad.

\begin{theorem}\label{IA:thm:320A}
Let $r\ge3$ and $\chi_r(n)=\lfloor\log_2n\rfloor\bmod r$. Every integer $m$ with $2^k<m<2^{k+1}$ for
some $k\ge1$ with $k\not\equiv0,-1\pmod r$ is strongly bad for $\chi_r$. The set of these $m$ has
logarithmic density $1-2/r$ and contains infinitely many primes.
\end{theorem}

\begin{theorem}\label{IA:thm:320B}
For $\rho>1$ let $\chi^{(\rho)}(n)=\lfloor\log_\rho n\rfloor\bmod 2$.
\begin{enumerate}[label=\textup{(\roman*)}]
\item\label{IA:it:lucas} Let $\varphi=\frac{1+\sqrt5}2$ and let $L_1=1$, $L_2=3$,
$L_{n+1}=L_n+L_{n-1}$ be the Lucas numbers. For every odd $n\ge3$, $L_n$ is bad for
$\chi^{(\varphi)}$. For every $q\in\N$ there are infinitely many odd $n\ge3$ with
$L_n\equiv1\pmod q$.
\item\label{IA:it:primes} There are $\rho\in(\frac32,2)$ and infinitely many primes that are bad for
$\chi^{(\rho)}$.
\end{enumerate}
\end{theorem}

\begin{corollary}\label{IA:cor:320}
For every $r\ge2$ there is an $r$-colouring of $\N$ such that for every $q\in\N$ some prime $m\ge2$
with $\gcd(m,q)=1$ is bad. In particular \cref{IA:q:320} has a negative answer for every number
$r\ge2$ of colours, also when the condition is imposed only for $m\ge2$ or only for prime $m$.
\end{corollary}

\begin{proof}
For $r\ge3$ use \cref{IA:thm:320A}: infinitely many primes are bad, and $q$ has finitely many prime
factors. For $r=2$ use \cref{IA:thm:320B}\ref{IA:it:primes} (or, for not necessarily prime $m$,
\cref{IA:thm:320B}\ref{IA:it:lucas}, where $L_n\equiv1\pmod q$ gives $\gcd(L_n,q)=1$).
\end{proof}

By \cref{IA:thm:320A}, the set of bad $m$ need not be small in any multiplicative sense: for $r$
colours it can contain a set of logarithmic density $1-2/r$. Question~3.19 of~\cite{kra2023problems}, whether
for every finite colouring some $m$ is not bad, remains open for $r\ge3$ colours; for $r=2$ it is
answered positively with $m=1$ by~\cite[Theorem~1.1]{huang2026affirmative}.

\begin{remark}[The diagonal and two colours]\label{IA:rem:320diag}
For $m\ge2$, $B+mB$ is the union of the diagonal $\{(m+1)b:b\in B\}$ and the two ordered pieces
$\{mx+y:x,y\in B,\,x<y\}$ and $\{mx+y:x,y\in B,\,x>y\}$. Huang et
al.~\cite[Theorem~1.6]{huang2026affirmative} proved that for every $2$-colouring of $\N$ and all
$m,\ell\in\N$ there is an infinite $B$ with $(m+\ell)B\cup\{mx+\ell y:x,y\in B,\,x<y\}$ monochromatic.
Applied with $(m,1)$ and with $(1,m)$, this shows that for two colours the diagonal together with
either ordered piece can always be made monochromatic. So a $2$-colouring for which some $m$ is bad
must use both ordered pieces; the proof of \cref{IA:thm:320B} uses one element from each of the three
parts. Hindman's theorem~\cite[Theorem~2.11]{hindman1979partitions} shows that each single $m\ge2$
is strongly bad for some $2$-colouring, namely $\lfloor\log_mn\rfloor\bmod2$ (this is the colouring
$\chi^{(\rho)}$ of \cref{IA:thm:320B} with $\rho=m$). Whether a single $2$-colouring can make
infinitely many pairwise coprime $m$ strongly bad, that is, whether Hindman's theorem has a version
uniform in $m$ (equivalently, the version of \cref{IA:q:320} with $B\oplus mB$ and two colours),
remains open. For related $2$- and $3$-colour obstructions to weighted ordered sums see
also~\cite[Propositions~5.1 and~5.2]{huang2026affirmative}.
\end{remark}

\begin{remark}
Di Nasso, Luperi Baglini, Mamino, Mennuni and Ragosta~\cite[\S1]{dinasso2025ramsey} answer negatively
the shift-free versions of several questions of Subsections~3.3 and~3.4 of~\cite{kra2023problems}
(in the numbering of its published version) in a strengthened ``Ramsey partition regular'' form.
These are different statements from \cref{IA:q:320}.
\end{remark}

For $\rho>1$ and $n\in\N$ put $\lambda_\rho(n)=\lfloor\log_\rho n\rfloor$, so that $\lambda_\rho(n)=k$
if and only if $\rho^k\le n<\rho^{k+1}$.

\begin{lemma}\label{IA:lem:floor}
If $x,y\in\R$ and $y-x\in(k,k+1)$ for some $k\in\Z$, then $\lfloor y\rfloor-\lfloor x\rfloor\in\{k,k+1\}$.
\end{lemma}

\begin{proof}
$\lfloor y\rfloor-\lfloor x\rfloor$ is an integer in $(y-x-1,y-x+1)\subset(k-1,k+2)$.
\end{proof}

\begin{lemma}\label{IA:lem:asym}
Let $\rho>1$, $m\ge2$ and $b,b'\in\N$. Put $\delta=\log_\rho(1+mb/b')$ and
$\delta'=\log_\rho(1+b/(mb'))$. Then $\log_\rho(b'+mb)=\log_\rho b'+\delta$,
$\log_\rho(mb'+b)=\log_\rho b'+\log_\rho m+\delta'$, and $0<\delta'<\delta$; moreover $\delta\to0$
as $b'\to\infty$ with $b$ fixed.
\end{lemma}

\begin{proof}
The identities are immediate, and $b/(mb')<mb/b'$ because $m\ge2$.
\end{proof}

\begin{proposition}\label{IA:prop:320A}
Let $\rho>1$, $r\ge3$ and $\chi(n)=\lambda_\rho(n)\bmod r$. Let $m\ge2$ be such that
$L=\log_\rho m\notin\Z$ and $k=\lfloor L\rfloor$ satisfies $k\not\equiv0$ and $k+1\not\equiv0
\pmod r$. Then $m$ is strongly bad for $\chi$.
\end{proposition}

\begin{proof}
Let $B$ be infinite and fix $b\in B$. By \cref{IA:lem:asym} there is $b'\in B$, $b'\ne b$, with
$\delta<L-k$. Put $X=b'+mb$ and $Y=b+mb'$, both in $B\oplus mB$. By \cref{IA:lem:asym},
$\log_\rho Y-\log_\rho X=L+\delta'-\delta\in(L-\delta,L)\subset(k,k+1)$, so by \cref{IA:lem:floor}
$\lambda_\rho(Y)-\lambda_\rho(X)\in\{k,k+1\}$. Neither number is divisible by $r$, so
$\chi(X)\ne\chi(Y)$.
\end{proof}

\begin{proof}[Proof of \cref{IA:thm:320A}]
For $2^k<m<2^{k+1}$ we have $\log_2m\notin\Z$ and $\lfloor\log_2m\rfloor=k$, so
\cref{IA:prop:320A} with $\rho=2$ applies. Let $S$ be the set of these $m$. For each $k$,
$\sum_{2^k<m<2^{k+1}}1/m=\log2+O(2^{-k})$, and among any $r$ consecutive values of $k$ exactly $r-2$
are admissible. Hence $\sum_{m\in S,\,m<2^{K+1}}1/m=\frac{r-2}r K\log2+O(1)$, and for $2^K\le
N<2^{K+1}$ the partial sum up to $N$ differs from this by $O(1)$. So $S$ has logarithmic density
$1-2/r$. By Bertrand's postulate each interval $(2^k,2^{k+1})$, $k\ge1$, contains a prime.
\end{proof}

\begin{proposition}\label{IA:prop:320B}
Let $m\ge2$ and $\rho>1+\frac1m$, and suppose there is an odd integer $j$ with $m<\rho^j<m+1$. Then
$m$ is bad for $\chi(n)=\lambda_\rho(n)\bmod2$.
\end{proposition}

\begin{proof}
Let $L=\log_\rho m$ and $L'=\log_\rho(m+1)$. Then $L<j<L'$ and $L'-L=\log_\rho(1+\frac1m)<1$. Let
$\eta=\min(j-L,L'-j)>0$. Let $B$ be infinite, fix $b\in B$, and by \cref{IA:lem:asym} choose $b'\in
B$, $b'\ne b$, with $\delta<\eta$. Consider the elements
\[
X=b'+mb,\qquad Y=b+mb',\qquad Z=b'+mb'=(m+1)b'
\]
of $B+mB$, and let $u=\log_\rho b'$, $x=\log_\rho X$, $y=\log_\rho Y$, $z=\log_\rho Z$. By
\cref{IA:lem:asym}, $x=u+\delta$, $y=u+L+\delta'$ and $z=u+L'$ with $0<\delta'<\delta<\eta$. Hence
$y<u+L+\eta\le u+j<x+j$ and $z=u+L'\ge u+j+\eta>x+j$, so
\begin{equation}\label{IA:eq:sandwich}
y<x+j<z,\qquad 0<z-y<L'-L<1 .
\end{equation}
Suppose $X,Y,Z$ have the same colour. Since $0<z-y<1$, the integer $\lambda_\rho(Z)-\lambda_\rho(Y)$ is
$0$ or $1$, and it is even; so $\lambda_\rho(Y)=\lambda_\rho(Z)=:N$. Then $N\le y<z<N+1$, and
by~\eqref{IA:eq:sandwich} $N<x+j<N+1$, so $N=\lfloor x\rfloor+j=\lambda_\rho(X)+j$. As $j$ is odd,
$\chi(X)\ne\chi(Y)$, a contradiction.
\end{proof}

\begin{lemma}\label{IA:lem:nested}
There are $\rho\in(\frac32,2)$, odd integers $j_1<j_2<\cdots$ and primes $p_1<p_2<\cdots$ with
$p_k<\rho^{j_k}<p_k+1$ for all $k$.
\end{lemma}

\begin{proof}
We construct closed intervals $[\frac32,2]=I_0\supset I_1\supset\cdots$, odd $j_k$ and primes $p_k$
with $p_k<s^{j_k}<p_k+1$ for all $s\in I_k$; any $\rho\in\bigcap_kI_k$ then works. Let $p_0=1$,
$j_0=0$, and suppose $I_{k-1}=[\alpha,\beta]$ with $\frac32\le\alpha<\beta$ is constructed. Choose an
odd $j>j_{k-1}$ so large that $x:=\alpha^j\ge p_{k-1}+2$ and $\beta^j\ge4x$, put $n=\lceil x\rceil$,
and let $p$ be a prime with $n<p<2n$ (Bertrand's postulate). Then $\alpha^j\le n<p$, $p>p_{k-1}$,
and $p+1\le2n\le2x+2<4x\le\beta^j$. Let $I_k=[(p+\frac13)^{1/j},(p+\frac23)^{1/j}]\subset(\alpha,\beta)$,
$j_k=j$ and $p_k=p$.
\end{proof}

\begin{proof}[Proof of \cref{IA:thm:320B}]
(i) Let $\psi=\frac{1-\sqrt5}2=-\varphi^{-1}$, so $L_n=\varphi^n+\psi^n$. For odd $n$,
$\psi^n\in(-1,0)$, so $L_n<\varphi^n<L_n+1$. For $n\ge3$, $L_n\ge4$ and $\varphi>1+\frac1{L_n}$, so
\cref{IA:prop:320B} applies with $m=L_n$, $\rho=\varphi$, $j=n$. For the second claim, $(L_n\bmod q)$
satisfies an invertible linear recurrence over $\Z/q\Z$ and is therefore purely periodic, with some
period $P$; then $L_n\equiv L_1=1\pmod q$ for all $n\equiv1\pmod{2P}$, and these $n$ are odd.

(ii) Take $\rho$, $(j_k)$, $(p_k)$ from \cref{IA:lem:nested}. Then $\rho>\frac32\ge1+\frac1{p_k}$ and
$p_k<\rho^{j_k}<p_k+1$ with $j_k$ odd, so \cref{IA:prop:320B} applies to $m=p_k$.
\end{proof}

\formalization{\Cref{IA:thm:320A,IA:thm:320B}, \cref{IA:cor:320} and the propositions and
lemmas they rely on are formalized in \leanfile{Q320} (\leanname{thm320A}, \leanname{thm320B},
\leanname{cor320}). Hindman's $3$-colouring, quoted for $m=1$, is not formalized.}

\subsection{Logarithmic density and Question~3.23}\label{IA:sec:q323}
The survey shows that a density version of Question~3.19 needs care: there is a thick set of positive
upper density containing no $B\oplus mB+t$ with $m\ge2$ and $B$ infinite
\cite[Example~3.21]{kra2023problems}. That set has zero upper logarithmic density, which motivates:

\begin{question}{3.23}\label{IA:q:323}
Does every set $A\subset\N$ with $\udlog(A)>0$ contain a sumset $B+mB+t$ for some $m\ge2$, $t\ge0$
and some infinite set $B\subset\N$?
\end{question}

As written, $B+mB$ is the unrestricted sumset: the survey writes $B\oplus B$ or states $b_1\ne b_2$
explicitly when it means the restricted one, and calls $B+B$ the unrestricted sumset. The answer to
\cref{IA:q:323} as stated is negative.

For each \emph{fixed} pair $(\ell,m)$, unrestricted sums $mB+\ell B$ have been studied by
Kousek~\cite{kousek2025asymmetric}. He showed that $\ldens(A)>1/2$ forces
$\{mb_1+\ell b_2:b_1\le b_2\}+t\subset A$, and that a suitable upper density condition depending on
$m/\ell$ does too~\cite[Theorems~1.6 and~1.7]{kousek2025asymmetric}; these are sharp, and for distinct
$\ell,m$ there are sets of positive lower density containing no shift of
$mB+\ell B=\{mb_1+\ell b_2:b_1,b_2\in B\}$ with $B$ infinite~\cite[Propositions~5.1, 6.5, 7.6
and~7.7]{kousek2025asymmetric}; see also his Questions~7.8 and~7.9. Those examples depend on
$(\ell,m)$, whereas \cref{IA:q:323} allows $m$ to depend on $A$. To our knowledge, what is new below is a single set of
positive logarithmic density that avoids $\ell B+mB+t$ for all $\ell,m\ge1$ and $t\in\Z$ at once. As in
Kousek's results (and as in \cref{IA:rem:320diag} for colourings), the diagonal together with one
ordered piece cannot be avoided by sets of lower density $>1/2$, and our proof uses the diagonal and
both ordered pieces.

Fix $\eta\in(0,\frac16)$ and let $U_\eta=(\frac13+\eta,\frac23-\eta)$. Let $T_k=2^{2^k}$,
$X_k=kT_k$ and $Y_k=T_{k+1}/k=T_k^2/k$ for $k\ge1$; note $T_k\le X_k<Y_k\le T_{k+1}$. Let
\[
R_k=\{s\in\{0,\dots,k-1\}:\ s/k\in U_\eta\},\qquad
A_\eta=\bigcup_{k\ge1}\{n\in\N:\ X_k\le n<Y_k,\ n\bmod k\in R_k\}.
\]
Thus on the block $[X_k,Y_k)$ the set $A_\eta$ is the periodic set $\{n:\{n/k\}\in U_\eta\}$, and
the transition intervals $[Y_{k-1},X_k)$ are empty.

\begin{theorem}\label{IA:thm:323}
\begin{enumerate}[label=\textup{(\alph*)}]
\item\label{IA:it:323dens} The logarithmic density of $A_\eta$ exists and equals $\frac13-2\eta$.
\item\label{IA:it:323avoid} For all integers $\ell,m\ge1$, all $t\in\Z$ and every infinite
$B\subset\N$, we have $\ell B+mB+t\not\subset A_\eta$. More precisely, for every $c\in B$ only
finitely many $b\in B$ satisfy
\[
\ell b+mc+t,\quad \ell c+mb+t,\quad (\ell+m)b+t\ \in A_\eta .
\]
\end{enumerate}
In particular (with $\ell=1$) $A_\eta$ is a negative answer to \cref{IA:q:323}, simultaneously for
all $m$.
\end{theorem}

\begin{remark}\label{IA:rem:323}
\begin{enumerate}[label=(\roman*)]
\item The answer hinges on the diagonal: the contradiction comes from the element $(\ell+m)b+t$.
The example motivating \cref{IA:q:323}, \cite[Example~3.21]{kra2023problems}, concerns the
restricted sumset, and for it $A_\eta$ is no counterexample: for every odd $m\ge3$ there is an
infinite $B$ with $B\oplus mB\subset A_\eta$. Indeed, let $\eps=\frac12(\frac16-\eta)$, so that
$[\frac12,\frac12+\eps)\subset U_\eta$. Choose $b_1<b_2<\cdots$ inductively: let $b_1\in\N$ be
arbitrary, and given $b_1,\dots,b_{j-1}$,
choose an even $k$ so large that $(m+1)b_{j-1}<\eps k$ and $Y_k/(m+1)-X_k>k$, and choose
$b_j\in[X_k,Y_k/(m+1))$ with $b_j\equiv k/2\pmod k$. For $i<j$ the numbers $b_j+mb_i$ and
$mb_j+b_i$ lie in $[X_k,Y_k)$ and are congruent modulo $k$ to $\frac k2+mb_i$ and to
$\frac k2+b_i$ respectively (as $m$ is odd, $\frac{mk}2\equiv\frac k2$), both of which lie in
$[\frac k2,\frac k2+\eps k)$. Hence both belong to $A_\eta$.
\item The restricted version of \cref{IA:q:323} (does every $A$ with $\udlog(A)>0$ contain
$B\oplus mB+t$ for some $m\ge2$, $t\ge0$ and infinite $B$?) remains open, to our knowledge.
\end{enumerate}
\end{remark}

\begin{lemma}\label{IA:lem:sumfree}
Let $k\ge1$, $s_1,s_2,s_3\in R_k$ and $s\in\Z$ with $|s|<\eta k$. Then $s_1+s_2-s_3\not\equiv s\pmod k$.
\end{lemma}

\begin{proof}
From $s_i/k\in U_\eta$ we get $s_1+s_2-s_3-s\in\big(k(\frac23+2\eta)-k(\frac23-\eta)-\eta k,\
k(\frac43-2\eta)-k(\frac13+\eta)+\eta k\big)=(2\eta k,\,k(1-2\eta))\subset(0,k)$, which contains no
multiple of $k$.
\end{proof}

\begin{lemma}\label{IA:lem:blocksep}
If $n,n'\in A_\eta$ lie in the blocks of indices $k<k'$, then $n'/n>k(k+1)$. If $n\in A_\eta$ and
$n\ge T_{K+1}$, then the block of $n$ has index at least $K+1$.
\end{lemma}

\begin{proof}
$n<Y_k=T_{k+1}/k$ and $n'\ge X_{k+1}=(k+1)T_{k+1}$. The block of index $k$ lies in $[T_k,T_{k+1})$.
\end{proof}

\begin{lemma}\label{IA:lem:perblock}
Let $k\in\N$, $R\subset\{0,\dots,k-1\}$, $\rho=|R|/k$, and $1\le P<Q$ real. Let
$g\colon\Z\cap[P,Q)\to\{0,1\}$ be constant on each of $J$ consecutive subintervals partitioning
$[P,Q)$. Then
\[
\Big|\sum_{\substack{P\le n<Q\\ n\bmod k\in R}}\frac{g(n)}n-\rho\sum_{P\le n<Q}\frac{g(n)}n\Big|\le\frac{2Jk}P .
\]
\end{lemma}

\begin{proof}
It suffices to show that the error on one subinterval $[P_1,Q_1)$, $P_1\ge P$, on which $g\equiv1$,
is at most $2k/P$. Let $n_0=\lceil P_1\rceil$ and split the integers of $[P_1,Q_1)$ into complete
periods $\{n_0+ik,\dots,n_0+(i+1)k-1\}$, $0\le i<I$, and a remainder of fewer than $k$ integers. On
the $i$-th period both sums lie in $[\rho k/(n_0+(i+1)k),\,\rho k/(n_0+ik)]$, so the errors telescope
to at most $\rho k/n_0\le k/P$. On the remainder both sums lie in $[0,k/n_0]$.
\end{proof}

\begin{proof}[Proof of \cref{IA:thm:323}\ref{IA:it:323avoid}]
Fix $\ell,m\ge1$, $t\in\Z$ and $c\in B$, and put $s=(\ell+m)c+t$ and $C=\ell+m+1$. For $b\in B$ let
\[
u=\ell b+mc+t,\qquad v=\ell c+mb+t,\qquad w=(\ell+m)b+t,
\]
all three elements of $\ell B+mB+t$, and note the identity $u+v-w=s$. If
$b\ge2(|t|+(\ell+m)c)$, then $u,v,w\in[b/2,Cb]$, so their pairwise ratios are at most $2C$. Choose
$K$ with $K(K+1)>2C$ and $K\ge|s|/\eta$, and suppose $b\ge\max\big(2(|t|+(\ell+m)c),\,2T_{K+1}\big)$ and
$u,v,w\in A_\eta$. By \cref{IA:lem:blocksep} the blocks of $u,v,w$ have indices $\ge K+1$, and they
coincide, since otherwise two of the numbers would have ratio $>(K+1)(K+2)>2C$. Let $k>K$ be the common
index. The residues of $u,v,w$ modulo $k$ lie in $R_k$ and satisfy $(u\bmod k)+(v\bmod k)-(w\bmod k)
\equiv s\pmod k$ with $|s|<\eta k$, contradicting \cref{IA:lem:sumfree}. So only finitely many $b\in B$
have $u,v,w\in A_\eta$, and as $B$ is infinite, $\ell B+mB+t\not\subset A_\eta$.
\end{proof}

\begin{proof}[Proof of \cref{IA:thm:323}\ref{IA:it:323dens}]
All logarithms are natural. Let $H(P,Q)=\sum_{P\le n<Q}1/n$, so that
$H(P,Q)=\log(Q/P)+O(1/P)$ for $1\le P<Q$. Let $\rho_k=|R_k|/k$ and $\mu=|U_\eta|=\frac13-2\eta$. The open interval
$(k(\frac13+\eta),k(\frac23-\eta))$ has length $k\mu$, so $|\rho_k-\mu|\le1/k$.

Let $T_K\le N<T_{K+1}$ and $S(N)=\sum_{n\le N,\,n\in A_\eta}1/n$. By \cref{IA:lem:perblock} with
$J=1$,
\[
S(N)=\sum_{k<K}\rho_kH(X_k,Y_k)+\rho_K H\big(X_K,\min(N+1,Y_K)\big)+O\Big(\sum_k\frac{2k}{X_k}\Big),
\]
where the middle term is read as $0$ if $N<X_K$, and $\sum_k2k/X_k=\sum_k2/T_k<\infty$. We have
$H(X_k,Y_k)=2^k\log2-2\log k+O(1/T_k)$, and
$\log N=2\log 2+\sum_{1\le k<K}2^k\log2+\log(N/T_K)$. Moreover the middle term differs from
$\mu\log(N/T_K)$ by at most $|\rho_K-\mu|2^K\log 2+O(\log K)$: if $N<X_K$ then $\log(N/T_K)<\log K$;
if $X_K\le N<Y_K$ then $H(X_K,N+1)=\log(N/T_K)-\log K+O(1)$; and if $N\ge Y_K$ then both
$H(X_K,Y_K)$ and $\log(N/T_K)$ lie within $2\log K+O(1)$ of $2^K\log 2$. Altogether
\[
\big|S(N)-\mu\log N\big|\le\sum_{k\le K}\frac{2^k\log2}{k}+O(K\log K)=O(2^K/K),
\]
while $\log N\ge2^K\log2$. Hence $S(N)/\log N\to\mu$.
\end{proof}

\formalization{Both parts of \cref{IA:thm:323} are formalized in \leanfile{Q323}
(\leanname{thm323}); \cref{IA:rem:323} is not.}

\subsection{Logarithmic equidistribution and Question~3.24}\label{IA:sec:q324}
All examples of the survey of sets of positive density avoiding some unrestricted sumset $B+B+t$
are built at a logarithmic scale; to test whether this is the only obstruction, the survey asks:

\begin{question}{3.24}\label{IA:q:324}
If $A\subset\N$ has $\dlog(A)>0$ and for every interval $(a,b)\subset[0,1)$ and every $\theta>0$
\begin{equation}\label{IA:eq:324}
\dlog\big(\{n\in A:\ \{\theta\log n\}\in(a,b)\}\big)=\dlog(A)\cdot(b-a),
\end{equation}
does $A$ contain a sumset $B+B+t$ for some $t\ge0$ and some infinite set $B\subset\N$?
\end{question}

The answer is negative, and a counterexample is already implicit in the work of Kousek and
Radi\'c~\cite{kousek2024infinite}: the set from the first part of~\cite[Proposition~4.2]{kousek2024infinite},
which has natural density $\frac12$ and contains no $B+B+t$, also satisfies~\eqref{IA:eq:324}, as we
check in \cref{IA:prop:KR} (this observation is not made in~\cite{kousek2024infinite}). By
\cite[Theorem~1.3]{kousek2024infinite}, every set of lower density $>\frac12$ contains some $B+B+t$,
so the value $\frac12$ is optimal among sets having a natural density; we do not know whether
$\dlog(A)>\frac12$ together with~\eqref{IA:eq:324} forces $B+B+t\subset A$. That example couples a
logarithmic phase with residue classes modulo $2$. In \cref{IA:thm:324} we give an alternative
example with no residue-class coupling at all, in which a single logarithmic phase drifts slowly;
it shows that the test~\eqref{IA:eq:324}, which involves only linear phases $\theta\log n$, does
not detect logarithmic-scale obstructions. Recall that by Kra, Moreira, Richter
and Robertson~\cite[Theorem~1.2]{kra2022proof} every set of positive upper Banach density contains
$B\oplus B+t$ for some infinite $B$ and $t\in\N$; so all these examples necessarily exploit the
diagonal elements $2b$.

\begin{lemma}\label{IA:lem:logsum}
Let $F\subset[0,\infty)$ be a union of pairwise disjoint intervals with left endpoints $\alpha_i$
such that $\#\{i:\alpha_i\le L\}\le C(L+1)$ for all $L\ge0$. Let $k\in\N$ and $j\in\Z$. Then, as
$N\to\infty$,
\[
\sum_{\substack{n\le N\\ \log_2n\in F}}\frac1n=\log2\cdot\big|F\cap[0,\log_2N]\big|+O(1)
\qquad\text{and}\qquad
\sum_{\substack{n\le N,\ n\equiv j\ (k)\\ \log_2n\in F}}\frac1n
=\frac1k\sum_{\substack{n\le N\\ \log_2n\in F}}\frac1n+O(1),
\]
where $|\cdot|$ is Lebesgue measure.
\end{lemma}

\begin{proof}
If $1\le x\le y$ and $I$ is an interval with endpoints $x,y$, then $|\sum_{n\in I}1/n-\log(y/x)|\le
1/x$, by comparison with $\int dt/t$. Each component of $F\cap[0,\log_2N]$ with left endpoint
$\alpha_i$ corresponds to such an interval $I$ with $x=2^{\alpha_i}$, and $\log(y/x)=\log2$ times
the length of the component. The total error is at most $\sum_i2^{-\alpha_i}\le\sum_{L\ge0}
C(L+1)2^{1-L}<\infty$. The second claim follows in the same way from \cref{IA:lem:perblock} ($J=1$),
applied to each component, with total error at most $\sum_i 2k\,2^{-\alpha_i}<\infty$.
\end{proof}

\begin{proposition}[set of Kousek--Radi\'c]\label{IA:prop:KR}
Let $A^{(1)}=\bigcup_{n\ge1}[4^n,(2-\frac1n)4^n)$, $A^{(2)}=\bigcup_{n\ge1}[(2+\frac1n)4^n,4^{n+1})$
and $A_{\mathrm{KR}}=(A^{(1)}\cap2\N)\cup(A^{(2)}\cap(2\N+1))$ (intersected with $\N$). Then
$A_{\mathrm{KR}}$ has natural density $\frac12$, contains no $B+B+t$ with $t\ge0$ and $B$ infinite,
and satisfies~\eqref{IA:eq:324} with $\dlog(A_{\mathrm{KR}})=\frac12$.
\end{proposition}

\begin{proof}
The first two assertions are~\cite[Proposition~4.2]{kousek2024infinite}, stated there for $t\in\N$;
the case $t=0$ follows, because if $B+B\subset A_{\mathrm{KR}}$ then $B'=\{b-1:b\in B,\ b\ge2\}$
satisfies $B'+B'+2\subset A_{\mathrm{KR}}$.

For~\eqref{IA:eq:324}, let $V_0=[0,\frac12)$, $V_1=[\frac12,1)$ and
$A^*=\{n\in2\N:\{\frac12\log_2n\}\in V_0\}\cup\{n\in2\N+1:\{\frac12\log_2n\}\in V_1\}$. Since
$\{\frac12\log_2n\}\in V_0$ if and only if $n\in[4^j,2\cdot4^j)$ for some $j\ge0$, the symmetric
difference of $A_{\mathrm{KR}}$ and $A^*$ is contained in $[1,4)\cup W$, where
$W=\bigcup_{n\ge1}[(2-\frac1n)4^n,(2+\frac1n)4^n]$. For $4^K\le N<4^{K+1}$ we have
$\sum_{n\le N,\,n\in W}\frac1n\le\sum_{n\le K}\big(\log\frac{2+1/n}{2-1/n}+4^{-n}\big)=O(\log K)
=o(\log N)$, so it suffices to verify~\eqref{IA:eq:324} for $A^*$. Fix $\theta>0$ and
$(a,b)\subset[0,1)$, put $\theta'=\theta\log2$, so that $\{\theta\log n\}=\{\theta'\log_2n\}$, and
let $F_i=\{x\ge0:\{x/2\}\in V_i,\ \{\theta'x\}\in(a,b)\}$ for $i\in\{0,1\}$. These sets satisfy the
hypothesis of \cref{IA:lem:logsum}, and $F_0\sqcup F_1=F:=\{x\ge0:\{\theta'x\}\in(a,b)\}$. By
\cref{IA:lem:logsum},
\[
\sum_{\substack{n\le N,\ n\in A^*\\ \{\theta\log n\}\in(a,b)}}\frac1n
=\sum_{i=0}^1\ \sum_{\substack{n\le N,\ n\equiv i\ (2)\\ \log_2n\in F_i}}\frac1n
=\frac12\sum_{i=0}^1\ \sum_{\substack{n\le N\\ \log_2n\in F_i}}\frac1n+O(1)
=\frac12\log 2\,|F\cap[0,\log_2N]|+O(1),
\]
and $|F\cap[0,L]|=(b-a)L+O(1)$. Dividing by $\log N=\log2\cdot\log_2N$ gives the limit
$\frac12(b-a)$. The same computation without the condition on $\{\theta\log n\}$ (that is, with
$F=[0,\infty)$) gives $\dlog(A^*)=\frac12$.
\end{proof}

The colour class $(A_0\cap3\N)\cup(A_1\cap(3\N+1))\cup(A_2\cap(3\N+2))$, with
$A_i=\{n:\{\frac12\log_2n\}\in[\frac i3,\frac{i+1}3)\}$, from the remark after
Proposition~5.1 of~\cite{kousek2024infinite} is handled in the same way (with residues modulo $3$)
and gives a counterexample of logarithmic density $\frac13$.

\begin{theorem}\label{IA:thm:324}
Let $0<\eta<\frac12$, $f(x)=\frac x2+\sqrt x$ for $x\ge0$, and
\[
A=\big\{n\in\N:\ \{f(\log_2n)\}\in[0,\tfrac12-\eta)\big\}.
\]
Then $\dlog(A)=\frac12-\eta$, $A$ satisfies~\eqref{IA:eq:324}, and $A$ contains no $B+B+t$ with
$t\in\Z$ and $B\subset\N$ infinite.
\end{theorem}

\begin{lemma}\label{IA:lem:oscint}
Let $\theta'>0$ and $(h,k)\in\Z^2\setminus\{(0,0)\}$. Then
$\frac1L\int_0^L\e\big(hf(x)+k\theta'x\big)\,dx\to0$ as $L\to\infty$, where $\e(y)=e^{2\pi iy}$.
\end{lemma}

\begin{proof}
The phase is $\psi(x)=\gamma x+h\sqrt x$ with $\gamma=\frac h2+k\theta'$. If $h=0$, then $k\ne0$, so
$\gamma\ne0$ and $\int_0^L\e(\gamma x)\,dx=O(1/|\gamma|)$. If $h\ne0$ and $\gamma\ne0$, let
$x_0=(h/\gamma)^2$. On $[x_0,L]$ the derivative $\psi'(x)=\gamma+\frac h{2\sqrt x}$ is monotone and
$|\psi'|\ge|\gamma|/2$. Integrating by parts,
\[
\int_{x_0}^L\e(\psi)\,dx=\Big[\frac{\e(\psi)}{2\pi i\psi'}\Big]_{x_0}^L
-\int_{x_0}^L\e(\psi)\,d\Big(\frac1{2\pi i\psi'}\Big),
\]
and since $1/\psi'$ is monotone with $|1/\psi'|\le2/|\gamma|$, each of the two terms is at most
$\frac{1}{2\pi}\cdot\frac{4}{|\gamma|}$.
The integral over $[0,x_0]$ is at most $x_0$. If $\gamma=0$, then $h\ne0$ (otherwise $k\theta'=0$, so
$k=0$), and $\int_0^L\e(h\sqrt x)\,dx=\int_0^{\sqrt L}2y\,\e(hy)\,dy=O(\sqrt L/|h|)$ by one
integration by parts. In both cases the integral is $o(L)$.
\end{proof}

\begin{proof}[Proof of \cref{IA:thm:324}]
Let $c=\frac12-\eta$, fix $\theta>0$, put $\theta'=\theta\log2$, and let
$\Psi(x)=(f(x)\bmod1,\ \theta'x\bmod1)\in\T^2$. By \cref{IA:lem:oscint} and linearity,
$\frac1L\int_0^LP(\Psi(x))\,dx\to\int_{\T^2}P$ for every trigonometric polynomial $P$, hence, by
uniform density of trigonometric polynomials in $C(\T^2)$, for every continuous $P$. If
$\Omega\subset\T^2$ is a box, then for each $\delta>0$ there are continuous $P_-\le\one_\Omega\le P_+$
with $\int(P_+-P_-)<\delta$, and therefore $\frac1L|\{x\in[0,L]:\Psi(x)\in\Omega\}|\to|\Omega|$. We
apply this to $\Omega=[0,c)\times(a,b)$ and to $\Omega=[0,c)\times\T$.

Since $f$ is continuous and strictly increasing with $f(0)=0$, the sets
$F=\{x\ge0:\{f(x)\}\in[0,c)\}=\bigcup_{j\ge0}[f^{-1}(j),f^{-1}(j+c))$ and
$F'=F\cap\{x:\{\theta'x\}\in(a,b)\}$ are unions of disjoint intervals; the endpoints of their
components in $[0,L]$ are among the at most $2(f(L)+1)+2(\theta'L+1)$ points where $f(x)$ or
$\theta'x$ hits $\Z\cup(\Z+c)$ or $(\Z+a)\cup(\Z+b)$, so \cref{IA:lem:logsum} applies. Since
$A=\{n:\log_2n\in F\}$ and $\{n\in A:\{\theta\log n\}\in(a,b)\}=\{n:\log_2n\in F'\}$, we obtain
$\dlog(A)=\lim|F\cap[0,L]|/L=c$ and the logarithmic density of $\{n\in A:\{\theta\log n\}\in(a,b)\}$
equals $\lim|F'\cap[0,L]|/L=c(b-a)$. This proves~\eqref{IA:eq:324}.

Now suppose $B$ is infinite, $t\in\Z$ and $B+B+t\subset A$. Let $b'=\min B$ and, for $b\in B$ large,
put $n_1=b+b'+t$, $n_2=2b+t$ (both in $B+B+t$) and $x_i=\log_2n_i$. As $b\to\infty$,
$x_2-x_1=\log_2\frac{2b+t}{b+b'+t}\to1$ and $|\sqrt{x_2}-\sqrt{x_1}|\le\frac{|x_2-x_1|}{2\sqrt{\min(x_1,x_2)}}
\to0$, so $\Delta_b:=f(x_2)-f(x_1)-\frac12\to0$. Take $b$ with $|\Delta_b|<\eta/2$. Since $n_1\in A$,
$\{f(x_1)\}\in[0,c)$, so $f(x_2)\equiv\{f(x_1)\}+\frac12+\Delta_b$ with
$\{f(x_1)\}+\frac12+\Delta_b\in[\frac12-\frac\eta2,1-\frac\eta2)$. Hence
$\{f(x_2)\}\notin[0,c)$, that is, $n_2\notin A$, a contradiction.
\end{proof}

\formalization{\Cref{IA:prop:KR}, including the two assertions quoted
from~\cite{kousek2024infinite}, and \cref{IA:thm:324} are formalized in \leanfile{Q324}
(\leanname{prop_KR}, \leanname{thm324}).}

\begin{remark}\label{IA:rem:324}
\begin{enumerate}[label=(\roman*)]
\item The obstruction in \cref{IA:thm:324} is still of logarithmic scale: the phase $f(\log_2n)$
increases by $\frac12+o(1)$ when $n$ is doubled. Equidistribution against all linear phases
$\theta\log n$ does not rule out such phases, because the term $\sqrt{\log_2n}$ makes $f$
asymptotically independent of every linear phase, including the resonant frequency
$\theta=1/(2\log2)$.
\item The set $A_\eta$ of \cref{IA:sec:q323} is a further counterexample to \cref{IA:q:324}: it
contains no $B+B+t$ by \cref{IA:thm:323}\ref{IA:it:323avoid} with $\ell=m=1$, and it
satisfies~\eqref{IA:eq:324}. To see the latter, let $g(n)=\one[\{\theta\log n\}\in(a,b)]$. On the
block $[X_k,Y_k)$, $g$ is constant on $J_k\le2\theta\log(Y_k/X_k)+3=O(2^k)$ subintervals, so by
\cref{IA:lem:perblock} the sum of $g(n)/n$ over $A_\eta\cap[X_k,Y_k)$ equals $\rho_k$ times the sum
over $[X_k,Y_k)$ up to $O(2^k/T_k)$, which is summable. The rest of the proof of
\cref{IA:thm:323}\ref{IA:it:323dens} applies verbatim with weights $g(n)$ (the transition intervals
contribute $O(K\log K)$ and the errors $|\rho_k-\mu|$ contribute $O(2^K/K)$), giving
$\sum_{n\le N,\,n\in A_\eta}g(n)/n=\mu\sum_{n\le N}g(n)/n+o(\log N)$; and
$\sum_{n\le N}g(n)/n=(b-a)\log N+O(1)$ by \cref{IA:lem:logsum}. Here the obstruction comes from scale-separated blocks with periodic structure inside each block,
rather than from a phase of $\log n$; the empty transition intervals $[Y_{k-1},X_k)$ have
logarithmic density zero and are invisible to~\eqref{IA:eq:324}.
\end{enumerate}
\end{remark}


\section{Sums of blocks from a density Ramsey family}
\label{IB:sec:q331}

In this section we answer Question~3.31 of~\cite{kra2023problems} and its
$k$-fold analogue, Question~\S3.6(a), positively, with ordered (hence disjoint) blocks and
simultaneously for all sum lengths $1\le m\le k$. The proof is a short modification of the proof of the Density Finite
Sums Theorem of Kra, Moreira, Richter and Robertson~\cite{kra2025density}: what
is needed is the quantitative (positive lower average) form of the recurrence
property that is actually established in~\cite{kra2025density}, combined with a
version of Bergelson's intersectivity lemma along \Fol{} sequences.

Throughout, $\N=\{1,2,\dots\}$, and for $A\subset\N$ and $t\in\Z$ we write
$A-t=\{n\in\N: n+t\in A\}$. The upper Banach density is
\[
\ubd(A)=\limsup_{N\to\infty}\ \sup_{M\ge 0}\frac{|A\cap\{M+1,\dots,M+N\}|}{N}.
\]

\begin{definition}[{\cite[Definition~3.29]{kra2023problems}}]
A collection $\mathcal D\subset\{F\subset\N: 0<|F|<\infty\}$ has the
\emph{density Ramsey property} if every $A\subset\N$ with $\ubd(A)>0$ contains
some $X\in\mathcal D$.
\end{definition}

By Szemer\'edi's theorem, the family of $\ell$-term arithmetic progressions has
the density Ramsey property for every $\ell$; trivially so does the family of
singletons. Theorem~3.30 of~\cite{kra2023problems} states that for such a family
$\mathcal D$ and $\ubd(A)>0$ there are pairwise disjoint $B_1,B_2,\ldots\in\mathcal D$
and pairwise disjoint $C_1,C_2,\ldots\in\mathcal D$ with $B_i+C_j\subset A$ for all
$i,j$. The survey observes that this does not contain the theorem that $A-t\supset B\oplus B$
for some infinite $B$ and some $t$, where $B\oplus B=\{b+b':b,b'\in B,\ b\ne b'\}$
(\cite[Theorem~1.2]{kra2022proof}, Theorem~2.2 of~\cite{kra2023problems}), and asks:

\begin{question}{3.31}
\label{IB:q:331}
If $\mathcal D\subset\{F\subset\N: 0<|F|<\infty\}$ has the density Ramsey property,
is it true that for any $A\subset\N$ with $\ubd(A)>0$ there exist $t\in\N$ and
disjoint $B_1,B_2,\ldots\in\mathcal D$ such that $B_i+B_j\subset A-t$ for all
$i\neq j\in\N$?
\end{question}

Immediately afterwards the survey adds: ``Similarly, one can formulate the analog
for the $k$-fold sum of the sets.'' We read this as follows.

\begin{question}{\S3.6(a)}
\label{IB:q:36a}
Let $\mathcal D$ have the density Ramsey property, $k\in\N$ and $\ubd(A)>0$. Do there
exist $t\in\N$ and disjoint $B_1,B_2,\ldots\in\mathcal D$ such that
$B_{i_1}+\dots+B_{i_k}\subset A-t$ for all pairwise distinct
$i_1,\dots,i_k\in\N$?
\end{question}

The indices must be distinct: with $\mathcal D$ the singletons and repeated indices
allowed, already $k=2$ would require $B+B+t\subset A$, which fails in general by
Example~2.3 of~\cite{kra2023problems}.

\begin{theorem}
\label{IB:thm:dR}
Let $\mathcal D\subset\{F\subset\N: 0<|F|<\infty\}$ have the density Ramsey property,
let $k\in\N$ and let $A\subset\N$ with $\ubd(A)>0$. Then there exist $t\in\N$ and
$B_1,B_2,\ldots\in\mathcal D$ with $\max B_i<\min B_{i+1}$ for all $i$, such that
\[
B_{i_1}+B_{i_2}+\dots+B_{i_m}\subset A-t
\qquad\text{for all } 1\le m\le k \text{ and all } i_1<i_2<\dots<i_m .
\]
\end{theorem}

\begin{corollary}
\label{IB:cor:331}
Question~\ref{IB:q:331} and Question~\ref{IB:q:36a} have positive answers. In
Question~\ref{IB:q:331} one may moreover require $B_i\subset A-t$ for all $i$ and
$\max B_i<\min B_{i+1}$.
\end{corollary}

\begin{proof}
Take $k=2$, respectively the given $k$, in \cref{IB:thm:dR}; sums over distinct
indices may be written with increasing indices.
\end{proof}

\begin{remark}
\label{IB:rem:dR}
Taking $\mathcal D$ to be the singletons, \cref{IB:thm:dR} is the Density Finite
Sums Theorem~\cite[Theorem~1.1]{kra2025density} (with the shift $t\ge 1$); for
$k=2$ it contains the theorem $B\oplus B\subset A-t$ of~\cite{kra2022proof} (only distinct
summands occur, as they must). Taking $\mathcal D$ to be the $\ell$-term
progressions gives a common extension of these results and Szemer\'edi's theorem,
which is the unification sought in~\cite[\S3.6]{kra2023problems}. Concerning
novelty: the proof below combines the construction of nested boxes from support
points in the proof of~\cite[Proposition~3.2]{kra2025density} with the chain argument
of~\cite[Lemma~2.2]{kra2025density}, with single return times replaced by blocks
from $\mathcal D$; all
dynamical input is taken from~\cite{kra2025density}. Our contribution is the
observation that the quantitative recurrence established there, together with
\cref{IB:lem:intersect}, suffices.
\end{remark}

\subsection{Two lemmas on \Fol{} sequences}

A \emph{\Fol{} sequence} in $\N$ is a sequence $\Phi=(\Phi_N)$ of finite nonempty
subsets of $\N$ with $|(\Phi_N+1)\cap\Phi_N|/|\Phi_N|\to 1$. For $K\subset\N$ put
$\udens_\Phi(K)=\limsup_N|K\cap\Phi_N|/|\Phi_N|$.

\begin{lemma}
\label{IB:lem:folner}
Let $\Phi$ be a \Fol{} sequence in $\N$. Then $|\Phi_N|\to\infty$, and
$\ubd(K)\ge\udens_\Phi(K)$ for every $K\subset\N$.
\end{lemma}

\begin{proof}
For a finite nonempty $F\subset\N$ we have $\max F+1\notin F$, so
$|(F+1)\cap F|\le|F|-1$; hence the \Fol{} property forces $|\Phi_N|\to\infty$.
Decompose $\Phi_N$ into its $r_N$ maximal blocks of consecutive integers. If
$x,x+1\in\Phi_N$ then $x$ and $x+1$ lie in the same block, so
$|(\Phi_N+1)\cap\Phi_N|=|\Phi_N|-r_N$ and therefore $r_N=o(|\Phi_N|)$. Fix $L\in\N$ and
let $\delta_L=\sup\{|K\cap I|/|I|\}$, the supremum over all intervals $I\subset\N$
of length at least $L$. Blocks of length $<L$ contain at most $Lr_N$ points in
total, and a block $I$ of length $\ge L$ contains at most $\delta_L|I|$ points of
$K$. Hence $|K\cap\Phi_N|\le Lr_N+\delta_L|\Phi_N|$ and
$\udens_\Phi(K)\le\delta_L$. Finally $\delta_L=\sup_{L'\ge L}a_{L'}$ where
$a_{L'}=\sup_M|K\cap\{M+1,\dots,M+L'\}|/L'$, so $\delta_L\to\limsup_{L'}a_{L'}=\ubd(K)$.
\end{proof}

The next lemma is a version of Bergelson's intersectivity
lemma~\cite{bergelson1985sets} in which density is measured along a \Fol{}
sequence; working directly along $\Phi$ avoids the passage through an
enumeration of the index set that appears in the proof of Theorem~3.30
of~\cite{kra2023problems}.

\begin{lemma}
\label{IB:lem:intersect}
Let $(\Omega,\nu)$ be a probability space, $\Phi$ a \Fol{} sequence in $\N$, and
$R\subset\N$ with $\udens_\Phi(R)\ge\delta>0$. For $n\in R$ let $O_n\subset \Omega$ be
measurable with $\nu(O_n)\ge\kappa>0$. Then there is $K\subset R$ with
$\ubd(K)\ge\delta\kappa$ such that $\nu\bigl(\bigcap_{n\in F}O_n\bigr)>0$ for every
finite nonempty $F\subset K$.
\end{lemma}

\begin{proof}
Let $\mathcal F_0$ be the (countable) set of finite nonempty $F\subset R$ with
$\nu(\bigcap_{n\in F}O_n)=0$, and let $Z=\bigcup_{F\in\mathcal F_0}\bigcap_{n\in F}O_n$; then
$\nu(Z)=0$. Put
$f_N=|\Phi_N|^{-1}\sum_{n\in R\cap\Phi_N}\one_{O_n\setminus Z}$. Then $0\le f_N\le 1$
and $\int f_N\,d\nu\ge\kappa\,|R\cap\Phi_N|/|\Phi_N|$. Choose $N_1<N_2<\cdots$ with
$|R\cap\Phi_{N_i}|/|\Phi_{N_i}|\to\udens_\Phi(R)$. By the reverse Fatou lemma,
$\int\limsup_i f_{N_i}\,d\nu\ge\limsup_i\int f_{N_i}\,d\nu\ge\delta\kappa$, so there is
$\omega\in\Omega$ with $\limsup_i f_{N_i}(\omega)\ge\delta\kappa$. Let $K=\{n\in R: \omega\in O_n\setminus Z\}$.
Then $f_N(\omega)=|K\cap\Phi_N|/|\Phi_N|$, so $\udens_\Phi(K)\ge\delta\kappa$ and
$\ubd(K)\ge\delta\kappa$ by \cref{IB:lem:folner}. If $F\subset K$ is finite and
nonempty, then $\omega\in\bigcap_{n\in F}O_n$ and $\omega\notin Z$, so
$\nu(\bigcap_{n\in F}O_n)>0$.
\end{proof}

\subsection{The dynamical input}

For a topological system $(Y,T)$ (a compact metric space with a homeomorphism) we
write $T_\Delta=T\times\dots\times T$ ($k+1$ factors) on $Y^{k+1}$ and $M(Y^{k+1})$ for
the Borel probability measures on $Y^{k+1}$.

\begin{proposition}[Kra--Moreira--Richter--Robertson]
\label{IB:prop:dyn}
Let $A\subset\N$ with $\ubd(A)>0$ and $k\in\N$. There exist a topological system
$(Y,T)$, a point $a\in Y$, an open set $\tilde E\subset Y$, a \Fol{} sequence $\Phi$
in $\N$ such that $a$ is generic along $\Phi$ for an ergodic $T$-invariant Borel probability
measure $\tilde\mu$ on $Y$, a measure $\sigma\in M(Y^{k+1})$ and an integer $t\ge 0$ such that
\begin{enumerate}[label=\textup{(\roman*)}]
\item\label{IB:dyn1} $A=\{n\in\N: T^na\in\tilde E\}$;
\item\label{IB:dyn2} $\sigma(\{a\}\times Y^k)=1$;
\item\label{IB:dyn3} $\sigma\bigl(Y\times(T^{-t}\tilde E)^k\bigr)>0$;
\item\label{IB:dyn4} for all open $U_1,\dots,U_k\subset Y$ such that
$V=U_1\times\dots\times U_k$ satisfies $\sigma(Y\times V)>0$,
\[
\liminf_{N\to\infty}\frac1{|\Phi_N|}\sum_{n\in\Phi_N}
\sigma\bigl((Y\times V)\cap T_\Delta^{-n}(V\times Y)\bigr)>0 .
\]
\end{enumerate}
\end{proposition}

All of this is contained in~\cite{kra2025density}; we indicate where. Property
\ref{IB:dyn4} is stronger than the statement that $\sigma$ is ``progressive''
(\cite[Definition~3.1]{kra2025density}, which only asks for infinitely many such
$n$), but it is exactly what is proved in~\cite{kra2025density}, in the proof that
Theorems~6.1 and~6.2 imply Theorem~4.9; we repeat that argument.

\begin{proof}
By the correspondence principle~\cite[Theorem~2.3]{kra2025density} there are a
measure preserving system $(X,\mu,S)$, a point $x_0\in X$ whose forward orbit
closure has full measure, and an open $E\subset X$ with $\mu(E)>0$ and
$A=\{n: S^nx_0\in E\}$. As in the proof that Theorem~2.4 implies Theorem~1.2
in~\cite{kra2025density}, there are an ergodic $S$-invariant $\mu'$ with $\mu'(E)\ge\mu(E)$
and a \Fol{} sequence along which $x_0$ is generic for $\mu'$ (there $\mu'$ comes from the ergodic
decomposition, with $\mu'$ giving full measure to the orbit closure of $x_0$, and the \Fol{}
sequence from a result on generic points cited in~\cite{kra2025density}). As in the proof that Theorems~4.8 and~4.9 imply Theorem~3.3
in~\cite{kra2025density}, which relies on~\cite[Proposition~5.7 and
Lemma~5.8]{kra2022infinite}, there are an ergodic system $(Y,\tilde\mu,T)$ with
topological pronilfactors, a continuous factor map $\pi\colon Y\to X$, a \Fol{}
sequence $\Phi$ in $\N$ and a point $a\in Y$ generic for $\tilde\mu$ along $\Phi$ with
$\pi(a)=x_0$. (Lemma~5.8 of~\cite{kra2022infinite} is stated for a transitive
point; as in the proof that Theorem~5.9 implies Theorem~4.4 there, one first replaces $X$ by the orbit
closure of $x_0$, which has full $\mu'$-measure, and $E$ by its intersection with
it.) A continuous factor map satisfies $\pi\circ T=S\circ\pi$ \emph{everywhere}
\cite[Definition~2.2]{kra2022infinite}; in~\cite[Lemma~5.8]{kra2022infinite} it is
the projection to the first coordinate of an orbit closure of a product map.
Hence with $\tilde E=\pi^{-1}(E)$, which is open, $T^na\in\tilde E$ if and only if
$S^nx_0\in E$, which is \ref{IB:dyn1}; moreover $\tilde\mu(\tilde E)=\mu'(E)>0$.

Let $\sigma$ be the measure of~\cite[Definition~4.5]{kra2025density} for
$(Y,\tilde\mu,T)$, $k$, $\Phi$ and $a$. Formula~(4.3) there shows
$\sigma=\delta_a\otimes\sigma'$ for some $\sigma'$, which gives \ref{IB:dyn2}, and
\cite[Theorem~4.8]{kra2025density} gives $t\ge0$ with \ref{IB:dyn3}. For
\ref{IB:dyn4}, let $V$ be as there. Since $V$ is open and $\sigma(Y\times V)>0$,
regularity and Urysohn's lemma give a continuous $G\colon Y^k\to[0,1]$ with
$G\le\one_V$ and $\int 1\otimes G\,d\sigma>0$. Pointwise,
$\one_{(Y\times V)\cap T_\Delta^{-n}(V\times Y)}\ge(1\otimes G)\cdot\bigl((G\otimes1)\circ T_\Delta^n\bigr)$.
By~\cite[Theorem~6.2]{kra2025density}, which applies to this $\sigma$, $\Phi$ and $a$ and to every
continuous $G\colon Y^k\to\mathbb C$,
\[
\lim_{N\to\infty}\Big\|\frac1{|\Phi_N|}\sum_{n\in\Phi_N}\bigl((G\otimes1)\circ T_\Delta^n-(1\otimes G)\circ
T_\Delta^n\bigr)\Big\|_{L^2(\sigma)}=0;
\] as
$|1\otimes G|\le 1$, the Cauchy--Schwarz inequality gives
\[
\liminf_{N}\frac1{|\Phi_N|}\sum_{n\in\Phi_N}\sigma\bigl((Y\times V)\cap T_\Delta^{-n}(V\times Y)\bigr)
\ \ge\ \liminf_N\frac1{|\Phi_N|}\sum_{n\in\Phi_N}\int_{Y^k} G\cdot\bigl(G\circ(T\times\dots\times T)^n\bigr)\,d\nu ,
\]
where $\nu$ is the projection of $\sigma$ to the last $k$ coordinates. By
\cite[Lemma~4.7]{kra2025density}, $\sigma$ is invariant under
$\mathrm{Id}\times T\times T^2\times\dots\times T^k$, so $\nu$ is invariant under
$T\times T^2\times\dots\times T^k$. Since $\int G\,d\nu>0$, the right-hand side is positive by
\cite[Theorem~6.1]{kra2025density}, which states: if $(X,T)$ is a topological system, $\Phi$ a
\Fol{} sequence, $\nu$ a $T\times T^2\times\dots\times T^k$-invariant measure on $X^k$ and
$G\colon X^k\to[0,1]$ continuous with $\int G\,d\nu>0$, then
$\liminf_N|\Phi_N|^{-1}\sum_{n\in\Phi_N}\int G\cdot T_\Delta^nG\,d\nu>0$.
\end{proof}

\subsection{Proof of \texorpdfstring{\cref{IB:thm:dR}}{the theorem}}

We first prove the statement with $t\ge0$ instead of $t\in\N$. Let $(Y,T)$, $a$,
$\tilde E$, $\Phi$, $\sigma$, $t$ be as in \cref{IB:prop:dyn} and put
$U=T^{-t}\tilde E$, an open set. By \ref{IB:dyn1}, for $n\in\N$,
\begin{equation}\label{IB:eq:U}
T^na\in U\iff n\in A-t .
\end{equation}
Since $\{a\}\times Y^k$ is a closed set of full measure, $\supp\sigma\subset\{a\}\times Y^k$.

By a \emph{box} we mean a set $V=V^{(1)}\times\dots\times V^{(k)}\subset Y^k$ with all
$V^{(j)}$ open and nonempty. We construct boxes $V_0,V_1,\ldots$ and sets
$B_1,B_2,\ldots\in\mathcal D$ such that for all $i\ge1$:
\begin{enumerate}[label=\textup{(I\arabic*)}]
\item\label{IB:I1} $\sigma(Y\times V_i)>0$ and $V_i^{(j)}\subset V_{i-1}^{(j)}$ for all $j$;
\item\label{IB:I2} for every $n\in B_i$ and every $v=(v_1,\dots,v_k)\in V_i$ we have
$T^na\in V_{i-1}^{(1)}$ and $T^nv_j\in V_{i-1}^{(j+1)}$ for $1\le j\le k-1$;
\item\label{IB:I3} $\min B_i>\max B_{i-1}$ (with $\max B_0:=0$).
\end{enumerate}
Let $V_0=U^k$; by \ref{IB:dyn3}, $\sigma(Y\times V_0)>0$ (in particular $U\neq\emptyset$).

Suppose $V_i$ has been constructed with $\sigma(Y\times V_i)>0$. For $n\in\N$ let
$O_n=(Y\times V_i)\cap T_\Delta^{-n}(V_i\times Y)$, an open subset of $Y^{k+1}$. By
\ref{IB:dyn4}, $c:=\liminf_N|\Phi_N|^{-1}\sum_{n\in\Phi_N}\sigma(O_n)>0$. Let
$R=\{n>\max B_i:\sigma(O_n)\ge c/2\}$. As $\sigma(O_n)\le1$,
\[
\frac1{|\Phi_N|}\sum_{n\in\Phi_N}\sigma(O_n)\le\frac c2+\frac{|R\cap\Phi_N|}{|\Phi_N|}
+\frac{\max B_i}{|\Phi_N|},
\]
and the last term tends to $0$ by \cref{IB:lem:folner}; hence $\udens_\Phi(R)\ge c/2$.
\Cref{IB:lem:intersect} (with $\nu=\sigma$, $\delta=\kappa=c/2$) gives $K\subset R$ with
$\ubd(K)>0$ such that $\sigma(\bigcap_{n\in F}O_n)>0$ for every finite nonempty
$F\subset K$. By the density Ramsey property there is $B_{i+1}\in\mathcal D$ with
$B_{i+1}\subset K$; then \ref{IB:I3} holds for $i+1$.

The set $O=\bigcap_{n\in B_{i+1}}O_n$ is open with $\sigma(O)>0$, so it meets
$\supp\sigma$, say in $(a,w)$. By definition of the product topology there is a
box $W\ni w$ with $\{a\}\times W\subset O$. Put $V_{i+1}=W$. Since $Y\times W$ is an
open neighbourhood of a point of $\supp\sigma$, $\sigma(Y\times W)>0$. Since
$\{a\}\times W\subset O_n\subset Y\times V_i$ and both $W$ and $V_i$ are boxes with nonempty
factors, $W^{(j)}\subset V_i^{(j)}$ for all $j$; this is \ref{IB:I1}. If $n\in B_{i+1}$
and $v\in W$, then $(a,v)\in O_n\subset T_\Delta^{-n}(V_i\times Y)$, that is,
$(T^na,T^nv_1,\dots,T^nv_{k-1})\in V_i$; this is \ref{IB:I2}.

Now fix $1\le m\le k$, indices $i_1<\dots<i_m$ and $b_l\in B_{i_l}$. We claim that for
$h=0,1,\dots,m-1$,
\begin{equation}\label{IB:eq:chain}
y_h:=T^{\,b_{m-h}+\dots+b_m}a\in V^{(h+1)}_{i_{m-h}-1}.
\end{equation}
For $h=0$ this is \ref{IB:I2} for the block $B_{i_m}$. Suppose \eqref{IB:eq:chain}
holds for some $h\le m-2$. Since $i_{m-h}-1\ge i_{m-h-1}$, \ref{IB:I1} gives
$y_h\in V^{(h+1)}_{i_{m-h-1}}$. As $V_{i_{m-h-1}}$ is a box with nonempty factors, there
is $v\in V_{i_{m-h-1}}$ with $v_{h+1}=y_h$. Applying \ref{IB:I2} to the block
$B_{i_{m-h-1}}$, the element $n=b_{m-h-1}$ and $j=h+1\le m-1\le k-1$, we get
$y_{h+1}=T^{b_{m-h-1}}y_h\in V^{(h+2)}_{i_{m-h-1}-1}$, which is \eqref{IB:eq:chain} for $h+1$.
For $h=m-1$ we obtain $T^{b_1+\dots+b_m}a\in V^{(m)}_{i_1-1}\subset V^{(m)}_0=U$, so
$b_1+\dots+b_m\in A-t$ by \eqref{IB:eq:U}. Together with \ref{IB:I3} this proves the
theorem with $t\ge0$.

Finally, to get $t\in\N$ apply the above to $A'=A-1=\{n\in\N:n+1\in A\}$, which has
$\ubd(A')=\ubd(A)>0$. If all the required sums lie in $A'-t'$ with $t'\ge0$, then
they lie in $A-(t'+1)$. \qed

\formalization{\Cref{IB:thm:dR,IB:cor:331}, \cref{IB:lem:folner,IB:lem:intersect} and the
construction above are formalized in \leanfile{Q331} (\leanname{thm_dR}, \leanname{cor331}),
with the conclusion of \cref{IB:prop:dyn} stated as an explicit hypothesis; the dynamical
input from~\cite{kra2025density} is not formalized.}

\section{Sumsets in the sets \texorpdfstring{$A_x=\{n:\{n^\gamma x\}\le 1/\log n\}$}{A\_x}}
\label{IB:sec:q413}

In this section we answer Question~4.13 of~\cite{kra2023problems} completely: the
answer is ``exactly the non-integers $\gamma$''. For $t\in\R$ let $\{t\}$ be its
fractional part and $\|t\|$ its distance to $\Z$.

\begin{question}{4.13}
\label{IB:q:413}
Fix $\gamma>0$ and, for each $x\in\R$, let
$A_x:=\{n\in\N:\{n^\gamma x\}\le1/\log n\}$. For which values of $\gamma$ is it true
that for almost every $x\in\R$, the set $A_x$ contains a sumset $B+C$ for
infinite sets $B,C\subset\N$?
\end{question}

(For $n=1$ the expression $1/\log n$ is undefined; we put $1\in A_x$, say. Since every
element of $B+C$ is at least $2$, this choice is irrelevant.) The survey notes that $A_x$ is thick for almost
every $x$ when $\gamma<1$, and that for integer $\gamma$ and every irrational $x$ the
set $A_x$ contains no infinite sumset.

\begin{theorem}
\label{IB:thm:413}
Let $\gamma>0$.
\begin{enumerate}[label=\textup{(\alph*)}]
\item\label{IB:413a} If $\gamma\notin\Z$, then for almost every $x\in\R$ the set $A_x$ is
thick, and hence contains $B+C$ for some infinite $B,C\subset\N$.
\item\label{IB:413b} If $\gamma=d\in\N$ and $x\notin\Q$, then $A_x$ contains no set $B+F$
with $B\subset\N$ infinite and $F\subset\Z$, $|F|=d+1$; in particular $A_x$ contains no
$B+C$ with $B,C$ infinite.
\end{enumerate}
Consequently, $A_x$ contains an infinite sumset $B+C$ for almost every $x$ if and
only if $\gamma\notin\N$.
\end{theorem}

For integer $\gamma=d$ and rational $x=p/q$ one has $q\N\subset A_x$, so $A_x\supset B+C$;
this does not affect the almost-everywhere statement. Part~\ref{IB:413a} is an
application of the standard second-moment method of metric Diophantine
approximation; part~\ref{IB:413b} is a short argument with divided differences
supplying the proof omitted in~\cite{kra2023problems}.

\begin{lemma}
\label{IB:lem:thick}
If $A\subset\N$ is thick, then $B+C\subset A$ for some infinite $B,C\subset\N$.
\end{lemma}

\begin{proof}
Since $A$ contains arbitrarily long intervals, for all $R,L$ it contains an interval of length
$>L$ whose left endpoint exceeds $R$ (remove the first $R$ points of an interval of length
$>R+L$). Choose $b_1<c_1<b_2<c_2<\cdots$ inductively (with $c_0=0$): $b_j>c_{j-1}$ is the left
endpoint of an interval of $A$ of length $>c_{j-1}$, and $c_j>b_j$ is the left
endpoint of an interval of $A$ of length $>b_j$. For $i<j$, $b_j+c_i\in[b_j,b_j+c_{j-1}]$
lies in the interval starting at $b_j$; for $i\le j$, $c_j+b_i\in[c_j,c_j+b_j]$ lies in
the interval starting at $c_j$. Thus $\{b_j:j\ge1\}+\{c_j:j\ge1\}\subset A$.
\end{proof}

\subsection{Proof of \texorpdfstring{\cref{IB:thm:413}\ref{IB:413b}}{part (b)}}

Let $F=\{f_0,\dots,f_d\}$ be distinct integers and suppose $B+F\subset A_x$ with $B$
infinite. For fixed $b$, the polynomial $P_b(f)=(b+f)^dx$ has degree $d$ and leading
coefficient $x$, so its $d$-th divided difference at the nodes $f_0,\dots,f_d$ equals
$x$:
\[
\sum_{i=0}^d\lambda_i(b+f_i)^dx=x,\qquad \lambda_i=\prod_{j\ne i}(f_i-f_j)^{-1}\in\Q .
\]
Let $M\in\N$ be a common denominator of the $\lambda_i$ and $e_i=M\lambda_i\in\Z$; then
$\sum_ie_i(b+f_i)^dx=Mx$, so $Mx\equiv\sum_ie_i\{(b+f_i)^dx\}\pmod 1$ and, for all
$b\in B$ so large that $b+\min F\ge2$,
\[
\|Mx\|\le\sum_i|e_i|\,\{(b+f_i)^dx\}\le\frac{\sum_i|e_i|}{\log(b+\min F)} .
\]
Letting $b\to\infty$ in $B$ gives $\|Mx\|=0$, i.e.\ $x\in\Q$. If $B+C\subset A_x$ with $B,C$
infinite, apply this to a subset $F\subset C$ with $|F|=d+1$. \qed

\subsection{Proof of \texorpdfstring{\cref{IB:thm:413}\ref{IB:413a}}{part (a)}}

Let $k=\lfloor\gamma\rfloor\ge0$, $\varsigma=\gamma-k\in(0,1)$ and
$c_j(n)=\binom{\gamma}{j}n^{\gamma-j}$ for $0\le j\le k$. Since $\gamma\notin\Z$, every
$\binom\gamma j$ ($j\ge0$) and every $\gamma-j$ is nonzero. Constants below may depend on
$\gamma$, $L$ and the cut-off $\psi$ introduced later, but never on $N$, $n$ or $x$.

\emph{Taylor expansion.} For $n\ge1$ and $0\le h<L$, Taylor's formula with
Lagrange remainder gives
\begin{equation}\label{IB:eq:taylor}
(n+h)^\gamma=\sum_{j=0}^kc_j(n)h^j+r_n(h),\qquad
|r_n(h)|=\Bigl|\tbinom{\gamma}{k+1}(n+\xi)^{\varsigma-1}h^{k+1}\Bigr|\le C_1n^{\varsigma-1},
\end{equation}
for some $\xi\in[0,h]$, because $\varsigma-1<0$.

Fix $L\in\N$ and $X_0>0$. For $N\in\N$ put
\[
\eta=\eta_N=\frac1{\log 4N},\qquad \rho=\frac{\eta}{16(k+1)L^k},
\]
where $N$ ranges over powers of $2$; for $N\le n<2N$ let $E_n$ be the set of $x\in\R$ with
$\{c_0(n)x\}\in[\eta/4,\eta/2]$ and $\|c_j(n)x\|\le\rho$ for $1\le j\le k$.

\begin{lemma}
\label{IB:lem:En}
There is $N_0$ such that for $N\ge N_0$, $N\le n<2N$ and $x\in E_n$ with $|x|\le X_0$ we
have $\{n,n+1,\dots,n+L-1\}\subset A_x$.
\end{lemma}

\begin{proof}
Write $c_j(n)x=m_j+\theta_j$ with $m_j\in\Z$, $\theta_0\in[\eta/4,\eta/2]$ and $|\theta_j|\le\rho$
for $j\ge1$. By \eqref{IB:eq:taylor},
$(n+h)^\gamma x=\sum_jm_jh^j+\theta_0+\sum_{j\ge1}\theta_jh^j+xr_n(h)$. Here
$|\sum_{j\ge1}\theta_jh^j|\le kL^k\rho\le\eta/16$, and $|xr_n(h)|\le X_0C_1N^{\varsigma-1}\le\eta/16$
once $N$ is large, since $\eta\asymp1/\log N$. Hence
$\{(n+h)^\gamma x\}\in[\eta/8,5\eta/8]$, and $\eta\le1/\log(n+h)$ because $2\le n+h<4N$ (for $N\ge L$).
\end{proof}

\emph{Smooth windows.} Fix $\Phi_0\in C_c^\infty(-\frac12,\frac12)$ with
$0\le\Phi_0\le1$ and $\Phi_0=1$ on $[-\frac14,\frac14]$, and fix an exponent
$P\ge 4k+10$. For $\alpha\in\R$ and $0<w<\frac12$ let
$\varphi_{\alpha,w}(t)=\sum_{m\in\Z}\Phi_0\bigl((t-\alpha-m)/(2w)\bigr)$. Then
$0\le\varphi_{\alpha,w}\le\one_{[\alpha-w,\alpha+w]+\Z}$, $\int_0^1\varphi_{\alpha,w}\ge w$, and
$\hat\varphi_{\alpha,w}(\ell)=2w\,\e(-\ell\alpha)\hat\Phi_0(2w\ell)$, where
$\e(t)=e^{2\pi it}$ and $|\hat\Phi_0(u)|\le C_P(1+|u|)^{-P}$. Comparing sums with
integrals one gets, uniformly in $\alpha$ and $w$,
\begin{equation}\label{IB:eq:window}
\sum_{\ell\in\Z}|\hat\varphi_{\alpha,w}(\ell)|\le C_2,\qquad
\sum_{|\ell|>Q}|\hat\varphi_{\alpha,w}(\ell)|\le C_2(wQ)^{1-P}\quad(wQ\ge1).
\end{equation}
Let $\varphi_0=\varphi_{3\eta/8,\eta/8}$ and $\varphi_j=\varphi_{0,\rho}$ for $1\le j\le k$, and
\[
F_n(x)=\prod_{j=0}^k\varphi_j(c_j(n)x),\qquad Z_N=\sum_{N\le n<2N}F_n .
\]
Then $0\le F_n\le\one_{E_n}$. Expanding in Fourier series (absolutely convergent by
\eqref{IB:eq:window}),
\[
F_n(x)=\sum_{\ell\in\Z^{k+1}}a(\ell)\,\e\bigl(x\,\Xi_n(\ell)\bigr),\qquad
a(\ell)=\prod_{j}\hat\varphi_j(\ell_j),\qquad \Xi_n(\ell)=\sum_j\ell_jc_j(n),
\]
with $\sum_\ell|a(\ell)|\le C_2^{k+1}$ and
$\mu:=a(0)=\prod_j\int_0^1\varphi_j\ge(\eta/8)\rho^k\ge c_3(\log N)^{-(k+1)}$.
Put $Q=T=(\log N)^3$. Since both window widths are $\gg1/\log N$,
\eqref{IB:eq:window} gives
\begin{equation}\label{IB:eq:tail}
\sum_{\|\ell\|_\infty>Q}|a(\ell)|\le C_4(\log N)^{-2(P-1)}=o(\mu^2).
\end{equation}

Fix $\psi\in C_c^\infty(\R)$ with $0\le\psi\le1$, $\supp\psi\subset[-X_0,X_0]$ and
$m_\psi=\int\psi>0$, and let $\check\psi(\xi)=\int\psi(x)\e(\xi x)\,dx$, so that
$|\check\psi(\xi)|\le C_P(1+|\xi|)^{-P}$ and $|\check\psi|\le m_\psi$. Write
$\mathbf E_\psi[f]=m_\psi^{-1}\int f\psi\,dx$. In what follows $N\to\infty$ and $o(1)$ is
uniform in $n,n'\in[N,2N)$.

\begin{lemma}
\label{IB:lem:freq}
If $N$ is large, $n\in[N,2N)$ and $0\ne\ell\in\Z^{k+1}$ with $\|\ell\|_\infty\le Q$, then
$|\Xi_n(\ell)|\ge c_5N^{\varsigma}$.
\end{lemma}

\begin{proof}
Let $j_0=\min\{j:\ell_j\ne0\}$. Then
$|\Xi_n(\ell)|\ge|\binom\gamma{j_0}|n^{\gamma-j_0}-Qk\max_j|\binom\gamma j|\,n^{\gamma-j_0-1}
\ge\frac12|\binom\gamma{j_0}|n^{\gamma-j_0}$ for $N$ large, and $\gamma-j_0\ge\varsigma$.
\end{proof}

\emph{First moment.} $\mathbf E_\psi F_n=m_\psi^{-1}\sum_\ell a(\ell)\check\psi(\Xi_n(\ell))$. The
term $\ell=0$ is $\mu$; the terms with $0<\|\ell\|_\infty\le Q$ contribute
$O(C_2^{k+1}N^{-\varsigma P})$ by \cref{IB:lem:freq}; the others contribute $o(\mu^2)$ by
\eqref{IB:eq:tail}. Hence $\mathbf E_\psi F_n=\mu(1+o(1))$ and
$\mathbf E_\psi Z_N=N\mu(1+o(1))$.

\emph{Second moment.} For $n\ne n'$,
$\mathbf E_\psi F_nF_{n'}=m_\psi^{-1}\sum_{\ell,\ell'}a(\ell)a(\ell')\check\psi\bigl(\Xi_n(\ell)+\Xi_{n'}(\ell')\bigr)$.
The pair $(0,0)$ gives $\mu^2$. Pairs with $\ell$ or $\ell'$ outside the box
$\|\cdot\|_\infty\le Q$ contribute $o(\mu^2)$ by \eqref{IB:eq:tail}. Pairs with exactly one
of $\ell,\ell'$ equal to $0$ and the other in the box contribute
$O(\mu N^{-\varsigma P})=o(\mu^2)$ by \cref{IB:lem:freq}. Call $(n,n')$ \emph{good} if
$|\Xi_n(\ell)+\Xi_{n'}(\ell')|\ge T$ for all nonzero $\ell,\ell'$ in the box, and
\emph{bad} otherwise. For good pairs the remaining terms contribute
$O(T^{-P})=o(\mu^2)$, so $\mathbf E_\psi F_nF_{n'}\le\mu^2(1+o(1))$.

\begin{lemma}
\label{IB:lem:bad}
The number of bad pairs is at most $C_6(\log N)^{6k+9}N^{2-\varsigma}$.
\end{lemma}

\begin{proof}
Fix $n$ and nonzero $\ell,\ell'$ in the box, let $j_1=\min\{j:\ell'_j\neq0\}$ and
$g(s)=\Xi_s(\ell')$ for real $s\in[N,2N]$. Then
\[
|g'(s)|\ge\bigl|\tbinom\gamma{j_1}(\gamma-j_1)\bigr|s^{\gamma-j_1-1}-C_7Qs^{\gamma-j_1-2}
\ge c_8N^{\varsigma-1}
\]
for $N$ large, using $\gamma-j_1\ne0$ and $\gamma-j_1-1\ge\varsigma-1$ (if the exponent is
negative, evaluate at $s=2N$). So $g$ is strictly monotone on $[N,2N]$, and the set
of $s$ with $|g(s)+\Xi_n(\ell)|<T$ is an interval of length at most $2T/(c_8N^{\varsigma-1})$,
containing at most $2TN^{1-\varsigma}/c_8+1$ integers $n'$. Summing over the $N$ values of $n$
and the at most $(2Q+1)^{2k+2}$ pairs $(\ell,\ell')$ gives the claim.
\end{proof}

For bad pairs we use $F_nF_{n'}\le F_n$, and $F_n^2\le F_n$ on the diagonal. Thus
\[
\mathbf E_\psi Z_N^2\le N\mu(1+o(1))+N^2\mu^2(1+o(1))+C_6(\log N)^{6k+9}N^{2-\varsigma}\mu(1+o(1)).
\]
Since $N\mu\to\infty$ and $(\log N)^{6k+9}N^{-\varsigma}/\mu\to0$, we get
$\mathbf E_\psi Z_N^2\le(\mathbf E_\psi Z_N)^2(1+o(1))$, and Chebyshev's inequality yields
$\mathbf P_\psi(Z_N=0)\le\operatorname{Var}_\psi(Z_N)/(\mathbf E_\psi Z_N)^2\to0$, where $\mathbf P_\psi$ is the
probability measure $m_\psi^{-1}\psi\,dx$.

Let $G_r=\{x:Z_{2^r}(x)>0\}$. Then $\mathbf P_\psi(\limsup_rG_r)\ge\limsup_r\mathbf P_\psi(G_r)=1$. If
$x\in G_r$ then $F_n(x)>0$, hence $x\in E_n$, for some $n\in[2^r,2^{r+1})$. By
\cref{IB:lem:En}, for $\mathbf P_\psi$-almost every $x$ there are infinitely many $n$ with
$\{n,\dots,n+L-1\}\subset A_x$. Taking $\psi=1$ on $[-X_0+1,X_0-1]$, and then intersecting
over integers $X_0\ge2$ and $L\in\N$, we conclude that $A_x$ is thick for almost every $x\in\R$.
\Cref{IB:lem:thick} completes the proof. \qed

\formalization{Part~\ref{IB:413b} of \cref{IB:thm:413} and \cref{IB:lem:thick} are formalized
in \leanfile{Q413} (\leanname{thm413b}, \leanname{thm413b_sumset}, \leanname{lem_thick}).
Part~\ref{IB:413a} and the final equivalence are formalized unconditionally:
\leanfile{Q413a} contains \cref{IB:lem:En} and the deduction of part~\ref{IB:413a} from the
second-moment estimate, and \leanfile{Q413SecondMoment} proves that estimate
(\leanname{secondMomentBound}, \leanname{thm413a}, \leanname{thm413_iff}). The formal proof
follows the second-moment method above, with Fej\'er kernels in place of the smooth windows and
the indicator of $[-X_0,X_0]$ in place of $\psi$.}

\section{Sumsets in random sets}
\label{IB:sec:u44a}

In this section we answer the question raised in \S4.4 of~\cite{kra2023problems},
which we call Question~\S4.4(a). Let $p=(p_n)_{n\in\N}\in[0,1]^\N$ and let $A\subset\N$ be the
random set in which the events $\{n\in A\}$ are independent and $\mathbf P(n\in A)=p_n$.

\begin{question}{\S4.4(a)}
\label{IB:q:44a}
For which sequences $(p_n)_{n\in\N}$ does the set $A$ contain almost surely sumsets?
\end{question}

We take ``sumsets'' to mean $B+C$ with $B,C$ infinite, and we also treat $B\oplus B$, $B+B$ and
the finite configurations $B+F$.

The survey shows: if $p$ is decreasing and $p_nn^\eps\to\infty$ for every $\eps>0$, then
$A$ is almost surely thick and hence contains $B+C$; if $p_nn^\eps\to0$ for some
$\eps>0$, then for some $k$ almost surely $A$ contains no $B+F$ with $B$ infinite and
$|F|=k$. Our answer has two parts.
\begin{itemize}
\item For monotone $p$ (indeed for $p$ comparable to a monotone sequence) the
answer is completely explicit (\cref{IB:thm:monotone}): $A$ almost surely contains
$B+C$ if and only if $A$ is almost surely thick, if and only if $\sum_np_n^k=\infty$
for all $k$, if and only if $p_nn^\eps\not\to0$ for every $\eps>0$. In particular the
gap left between the two conditions above is entirely on the positive side.
\item For general $p$ there is no Borel characterization: the set of $p$ for which $A$
almost surely contains $B+C$ (or $B\oplus B$, or $B+B$) is $\Sigma^1_1$-complete, hence
not Borel, even within $(0,1)^\N$; for $B+C$ this persists within
$\{p\in(0,1)^\N:p_n\to0\}$ (\cref{IB:thm:nonborel}). On the positive side we give the
exact criterion for finite configurations $B+F$ (\cref{IB:thm:finite}), an exact
``strategy'' criterion with a $0$--$1$ law for $B+C$, $B\oplus B$ and $B+B$
(\cref{IB:thm:strategy,IB:thm:strategy2}), and simple sufficient and necessary
conditions, with examples showing that all implications between them are strict
(\cref{IB:thm:hierarchy}).
\end{itemize}

The strategy criteria are close to a reformulation of the event in question;
by \cref{IB:thm:nonborel}, no Borel condition on $p$ can replace them.
\Cref{IB:thm:finite} is standard. The non-Borel part of
\cref{IB:thm:nonborel} for $B\oplus B$ in the deterministic setting is due to
Mazurkiewicz and \.Zeberski~\cite[Theorem~2.6]{mazurkiewicz2023ideal}, who proved
that the ideal of sets containing no $B\oplus B$ (their $\mathcal H_2$) is $\Pi^1_1$-complete (their method,
reducing ill-founded trees to additive families, is the one we use); for $B+C$ and
$B+B$, and for the random versions, we did not find these statements in the
literature.

\subsection{Notation and statements}

We set $p_m=0$ for integers $m\le0$. For finite $U\subset\Z$ and $x\in\Z$ put
$w_U(x)=\prod_{u\in U}p_{x+u}$ (so $w_\emptyset\equiv1$), and for finite nonempty $F\subset\Z$
put $S_F=\sum_{n\in\Z}w_F(n)\in[0,\infty]$; clearly $S_{F+r}=S_F$ for $r\in\Z$, and
$S_{F'}\ge S_F$ if $F'\subset F$. Write $B\oplus B=\{b+b':b,b'\in B,\ b\ne b'\}$. Consider the
events
\begin{itemize}
\item $\mathcal E_{BC}$: there are infinite $B,C\subset\N$ with $B+C\subset A$;
\item $\mathcal E_\oplus$, $\mathcal E_+$: there is an infinite $B\subset\N$ with $B\oplus B\subset A$,
respectively $B+B\subset A$;
\item $\mathcal E_F$ ($F\subset\Z$ finite nonempty): there is an infinite $B\subset\N$ with $B+F\subset A$;
\item $\mathcal E_k$: $\mathcal E_F$ holds for some $F\subset\N$ with $|F|=k$; and
$\mathcal E_{<\omega}=\bigcap_k\mathcal E_k$.
\end{itemize}
We identify subsets of $\N$ with points of $2^\N$. Each of these events is an
analytic subset of $2^\N$ (a projection of a Borel subset of $2^\N\times2^\N\times2^\N$),
hence universally measurable (Lusin's theorem, see~\cite[Theorem~21.10]{kechris1995classical}); $\mathbf P=\mathbf P_p$
denotes the completed product measure.
The events $\mathcal E_{BC},\mathcal E_\oplus,\mathcal E_+,\mathcal E_F$ are invariant under
finite modifications of $A$: adding points is harmless, and if $E$ is finite and
$B+C\subset A$, then $(B\cap(\max E,\infty))+C\subset A\setminus E$; similarly
$(B\cap(\max E-\min F,\infty))+F\subset A\setminus E$ if $B+F\subset A$, and the cases $B\oplus B$ and
$B+B$ are handled by $B\cap(\max E,\infty)$.

\begin{theorem}[Monotone sequences]
\label{IB:thm:monotone}
Suppose there are a non-increasing sequence $(q_n)$ and constants $0<\alpha\le\beta$ with
$\alpha q_n\le p_n\le\beta q_n$ for all $n$. Then the following are equivalent:
\begin{enumerate}[label=\textup{(\arabic*)}]
\item\label{IB:m1} $A$ almost surely contains $B+C$ for some infinite $B,C$;
\item\label{IB:m2} $A$ is almost surely thick;
\item\label{IB:m3} $\sum_np_n^k=\infty$ for every $k\in\N$;
\item\label{IB:m4} for every $\eps>0$, $p_nn^\eps$ does not tend to $0$;
\item\label{IB:m5} almost surely, for every $k$ there are an infinite $B$ and $F\subset\N$ with
$|F|=k$ and $B+F\subset A$.
\end{enumerate}
If they fail, there is $k$ such that almost surely $A$ contains no $B+F$ with $B$
infinite and $|F|=k$. In each case the probability of the event in \ref{IB:m1},
\ref{IB:m2}, \ref{IB:m5} is $0$ or $1$.
\end{theorem}

\begin{example}
\label{IB:ex:gap}
Let $N_1=10$, $N_{j+1}=N_j^{\,j+1}$, $p_n=1$ for $n<N_1$ and $p_n=N_j^{-1/2}$ for
$N_j\le n<N_{j+1}$. Then $p$ is non-increasing. At $n=N_j$, $p_nn^\eps=N_j^{\eps-1/2}\to0$
for $\eps<\frac12$, so the first condition of the survey fails; at $n=N_{j+1}-1$,
$p_nn^\eps\ge N_j^{-1/2}(N_j^{j+1}/2)^\eps\to\infty$, so $p_nn^\eps\not\to0$ for every $\eps$
and the second condition fails as well. By \cref{IB:thm:monotone}, $A$ is almost
surely thick.
\end{example}

\begin{theorem}[Finite configurations]
\label{IB:thm:finite}
For finite nonempty $F\subset\Z$, $\mathbf P(\mathcal E_F)=1$ if $S_F=\infty$ and $\mathbf P(\mathcal E_F)=0$
otherwise. Consequently $\mathbf P(\mathcal E_k)=1$ if $S_F=\infty$ for some $F\subset\N$ with
$|F|=k$, and $\mathbf P(\mathcal E_k)=0$ otherwise; $\mathbf P(\mathcal E_{<\omega})\in\{0,1\}$ with value $1$
if and only if this holds for every $k$; and $A$ is almost surely thick if
$S_{\{0,\dots,L-1\}}=\infty$ for all $L$, and almost surely not thick otherwise.
\end{theorem}

\begin{definition}
\label{IB:def:strategy}
A \emph{$BC$-strategy} for $p$ is a set $\mathcal S$ of pairs $(U,V)$ of finite subsets
of $\N$ such that $(\emptyset,\emptyset)\in\mathcal S$ and for every $(U,V)\in\mathcal S$
\[
\text{(a)}\ \sum_{x\in X_{\mathcal S}(U,V)}w_U(x)=\infty,\qquad
\text{(b)}\ \sum_{y\in Y_{\mathcal S}(U,V)}w_V(y)=\infty,
\]
where $X_{\mathcal S}(U,V)=\{x\in\N\setminus V:(U,V\cup\{x\})\in\mathcal S\}$ and
$Y_{\mathcal S}(U,V)=\{y\in\N\setminus U:(U\cup\{y\},V)\in\mathcal S\}$.
An \emph{$\oplus$-strategy} (resp.\ \emph{$+$-strategy}) is a set $\mathcal S$ of finite
subsets of $\N$ with $\emptyset\in\mathcal S$ such that for every $U\in\mathcal S$,
$\sum_{x\in X_{\mathcal S}(U)}w^\ast_U(x)=\infty$, where
$X_{\mathcal S}(U)=\{x\in\N\setminus U:U\cup\{x\}\in\mathcal S\}$ and $w^\ast_U=w_U$ (resp.\
$w^\ast_U(x)=p_{2x}w_U(x)$).
\end{definition}

\begin{theorem}
\label{IB:thm:strategy}
$\mathbf P(\mathcal E_{BC})\in\{0,1\}$, and $\mathbf P(\mathcal E_{BC})=1$ if and only if $p$ has a
$BC$-strategy.
\end{theorem}

\begin{theorem}
\label{IB:thm:strategy2}
$\mathbf P(\mathcal E_\oplus)\in\{0,1\}$, and $\mathbf P(\mathcal E_\oplus)=1$ if and only if $p$ has an
$\oplus$-strategy; likewise for $\mathcal E_+$ and $+$-strategies.
\end{theorem}

For $p\in\{0,1\}^\N$, i.e.\ $A=D$ deterministic, a strategy exists if and only if $D$
contains the configuration in question (see the end of \cref{IB:ss:proofsAB}); the
two criteria in \cref{IB:thm:strategy2} differ, as $D=\{2^i+2^j:i<j\}$ shows
(\cref{IB:ex:oplus}).

To state the non-Borel results, recall that a \emph{tree} is a set
$\mathcal T\subset\N_0^{<\N_0}$ of finite sequences of elements of $\N_0=\{0,1,2,\dots\}$ closed
under initial segments; the trees form a closed subset $\mathrm{Tr}$ of
$2^{\N_0^{<\N_0}}$, and $\mathrm{IF}\subset\mathrm{Tr}$ denotes the set of trees with an infinite
branch. A subset $Z$ of a Polish space $W$ is \emph{$\Sigma^1_1$-complete} if it is
analytic and for every analytic subset $Z'$ of a zero-dimensional Polish space $W'$
there is a continuous $f\colon W'\to W$ with $Z'=f^{-1}(Z)$. The set $\mathrm{IF}$ is
$\Sigma^1_1$-complete (see~\cite[Theorem~27.1]{kechris1995classical}); hence an analytic set
$Z$ is $\Sigma^1_1$-complete as soon as $\mathrm{IF}=g^{-1}(Z)$ for some continuous
$g\colon\mathrm{Tr}\to W$, and $\Sigma^1_1$-complete sets are not Borel.

\begin{theorem}
\label{IB:thm:nonborel}
\begin{enumerate}[label=\textup{(\alph*)}]
\item\label{IB:nb1} The families
$\mathfrak D_{BC}=\{D\subset\N:\exists\text{ infinite }B,C,\ B+C\subset D\}$,
$\mathfrak D_\oplus=\{D:\exists\text{ infinite }B,\ B\oplus B\subset D\}$ and
$\mathfrak D_+=\{D:\exists\text{ infinite }B,\ B+B\subset D\}$ are $\Sigma^1_1$-complete subsets of $2^\N$.
\item\label{IB:nb2} For $\ast\in\{BC,\oplus,+\}$ let
$\mathfrak A_\ast=\{p\in[0,1]^\N:\mathbf P_p(\mathcal E_\ast)=1\}$. Then $\mathfrak A_\ast$ is a
$\Sigma^1_1$-complete subset of $[0,1]^\N$, and $\mathfrak A_\ast\cap(0,1)^\N$ is a
$\Sigma^1_1$-complete subset of $(0,1)^\N$.
\item\label{IB:nb3} Let $\mathfrak C=\{p\in(0,1)^\N:p_n\to0\}$. There is a continuous
$g\colon\mathrm{Tr}\to\mathfrak C$ with $g^{-1}(\mathfrak A_{BC})=\mathrm{IF}$. In particular there is
no Borel set $W\subset[0,1]^\N$ with $W\cap\mathfrak C=\mathfrak A_{BC}\cap\mathfrak C$.
\end{enumerate}
\end{theorem}

We do not know whether the analogue of \ref{IB:nb3} holds for $\mathcal E_\oplus$ and $\mathcal E_+$.

\begin{theorem}
\label{IB:thm:hierarchy}
Consider the conditions
\begin{itemize}
\item[(F)] there are infinite $B_0,C_0\subset\N$ such that $\sum_{c\in C_0}w_U(c)=\infty$ for every
finite $U\subset B_0$ and $\sum_{b\in B_0}w_V(b)=\infty$ for every finite $V\subset C_0$;
\item[(N)] there are infinite $B_0,C_0\subset\N$ such that $S_G=\infty$ for every finite nonempty
$G\subset B_0$ and for every finite nonempty $G\subset C_0$.
\end{itemize}
Then $\mathrm{(F)}\Rightarrow\mathbf P(\mathcal E_{BC})=1\Rightarrow\mathrm{(N)}\Rightarrow\mathbf P(\mathcal E_{<\omega})=1$,
and none of the three implications can be reversed
(\cref{IB:ex:tree,IB:ex:half,IB:ex:rigid}).
\end{theorem}

Condition (F) holds, for instance, with $B_0=C_0=\N$ if $S_G=\infty$ for every finite
nonempty $G\subset\N$ (as in the situation of \cref{IB:thm:monotone}), or if some level set
$\{n:p_n\ge\delta\}$, $\delta>0$, contains $B_0+C_0$; in particular, for $p=q\one_D$ with
$0<q\le1$ we have $\mathbf P(\mathcal E_{BC})=1$ if and only if $D\in\mathfrak D_{BC}$. In (N) the
two families of finite sets are separate: we do not require $S_G=\infty$ for finite
$G\subset B_0\cup C_0$.

\subsection{Proofs of \texorpdfstring{\cref{IB:thm:finite,IB:thm:strategy,IB:thm:strategy2}}{the criteria}}
\label{IB:ss:proofsAB}

\begin{lemma}
\label{IB:lem:BC}
Let $(K_x)_{x\in I}$ be a countable family of finite subsets of $\Z$ \textup{(}recall $p_m=0$
for $m\le0$, so $\mathbf P(K_x\subset A)=0$ unless $K_x\subset\N$\textup{)} such that each
$K_x$ meets $K_{x'}$ for at most $M$ indices $x'\ne x$. If
$\sum_x\mathbf P(K_x\subset A)=\infty$, then almost surely $K_x\subset A$ for infinitely many $x$.
\end{lemma}

\begin{proof}
Enumerate $I$ and colour each $x$ with the least colour in $\{1,\dots,M+1\}$ not used by
an earlier $x'$ with $K_{x'}\cap K_x\ne\emptyset$. Some colour class $I_0$ satisfies
$\sum_{x\in I_0}\mathbf P(K_x\subset A)=\infty$. The sets $K_x$, $x\in I_0$, are pairwise disjoint,
so the events $\{K_x\subset A\}$, $x\in I_0$, are independent, and the second
Borel--Cantelli lemma applies.
\end{proof}

We use this with $K_x=x+U$, where $K_x\cap K_{x'}\ne\emptyset$ forces $x'-x\in U-U$ (so
$M\le|U|^2$), and with $K_x=\{2x\}\cup(x+U)$, where it forces $x'\in x+(U-U)$,
$2x'\in x+U$ or $2x\in x'+U$ (so $M\le|U|^2+2|U|$).

\begin{proof}[Proof of \cref{IB:thm:finite}]
$\mathcal E_F$ holds if and only if $n+F\subset A$ for infinitely many $n\in\Z$: if $B+F\subset A$
take $n\in B$; conversely such $n$ satisfy $n\ge1-\min F$, so infinitely many of them
lie in $\N$. Since $\mathbf P(n+F\subset A)=w_F(n)$, the first Borel--Cantelli lemma handles
$S_F<\infty$ and \cref{IB:lem:BC} (with $K_n=n+F$) handles $S_F=\infty$. Next,
$\mathcal E_k$ is the countable union of the $\mathcal E_F$ with $F\subset\N$, $|F|=k$, each of
probability $0$ or $1$, and $\mathcal E_{<\omega}$ is a countable intersection. Finally, $A$ is
thick if and only if for every $L$ there are infinitely many $n$ with
$\{n,\dots,n+L-1\}\subset A$ (an interval of length $L'\ge L$ contains $L'-L+1$ such
blocks), i.e.\ if and only if $\mathcal E_{\{0,\dots,L-1\}}$ holds for all $L$.
\end{proof}

\begin{proof}[Proof of \cref{IB:thm:strategy}]
\emph{If $\mathcal S$ is a $BC$-strategy, then $\mathbf P(\mathcal E_{BC})=1$.} As $\mathcal S$ is countable,
(a), (b) and \cref{IB:lem:BC} give an event $\Omega_0$ of probability $1$ on which, for
every $(U,V)\in\mathcal S$, there are infinitely many $x\in X_{\mathcal S}(U,V)$ with $x+U\subset A$
and infinitely many $y\in Y_{\mathcal S}(U,V)$ with $y+V\subset A$. (For $U=\emptyset$, (a) just
says that $X_{\mathcal S}(\emptyset,V)$ is infinite.) Fix an outcome in $\Omega_0$. Starting
from $(U_0,V_0)=(\emptyset,\emptyset)$, choose $x\in X_{\mathcal S}(U_r,V_r)$ with
$x>\max(U_r\cup V_r)$ (with $\max\emptyset=0$) and $x+U_r\subset A$, put $V_{r+1}=V_r\cup\{x\}$; then choose
$y\in Y_{\mathcal S}(U_r,V_{r+1})$ with $y>\max(U_r\cup V_{r+1})$ and $y+V_{r+1}\subset A$, and put
$U_{r+1}=U_r\cup\{y\}$. Inductively $(U_r,V_r)\in\mathcal S$ and $U_r+V_r\subset A$, since the new
sums are exactly $x+U_r$ and $y+V_{r+1}$. Then $B=\bigcup_rU_r$ and $C=\bigcup_rV_r$ are
infinite and $B+C\subset A$.

\emph{If $\mathbf P(\mathcal E_{BC})>0$, then $p$ has a $BC$-strategy.} For finite $U,V\subset\N$ let
$\mathcal E(U,V)$ be the event that there are infinite $B',C'\subset\N$ with $U+C'\subset A$,
$B'+V\subset A$ and $B'+C'\subset A$. It is analytic, and $\mathcal E(\emptyset,\emptyset)=\mathcal E_{BC}$. Let
$\mathcal S^\ast=\{(U,V):\mathbf P(\mathcal E(U,V))>0\}$, which contains $(\emptyset,\emptyset)$. We show that
$\mathcal S^\ast$ satisfies (a); (b) is symmetric. Let $(U,V)\in\mathcal S^\ast$, write
$X^\ast=X_{\mathcal S^\ast}(U,V)$, and suppose $\sum_{x\in X^\ast}w_U(x)<\infty$. On
$\mathcal E(U,V)$, with witnesses $B',C'$, every $x\in C'\setminus V$ satisfies $x+U\subset A$, and
$\mathcal E(U,V\cup\{x\})$ holds with witnesses $B',C'\setminus\{x\}$ (because $B'+x\subset B'+C'\subset A$).
Hence $\mathcal E(U,V)\subset N_1\cup N_2$, where $N_1$ is the union of the null events
$\mathcal E(U,V\cup\{x\})$ over $x\in\N\setminus(V\cup X^\ast)$, and $N_2$ is the event that $x+U\subset A$ for
infinitely many $x\in X^\ast$, which is null by the first Borel--Cantelli lemma. Thus
$\mathbf P(\mathcal E(U,V))=0$, a contradiction. Combined with the first part this also proves
the $0$--$1$ law.
\end{proof}

\begin{proof}[Proof of \cref{IB:thm:strategy2}]
The proof is the same, with single sets. Given a strategy, choose new elements
$x\in X_{\mathcal S}(U_r)$, $x>\max U_r$, with $x+U_r\subset A$ (and also $2x\in A$ in the
$+$ case); for $x\notin U$ we have $2x\notin x+U$, so the probability of this event is
exactly $w^\ast_U(x)$, and \cref{IB:lem:BC} applies with $K_x=x+U$, resp.\
$K_x=\{2x\}\cup(x+U)$. Conversely, let $\mathcal E^\ast(U)$ be the event that there is an infinite
$B'$ with $U+B'\subset A$ and $B'\oplus B'\subset A$ (and $2b\in A$ for $b\in B'$ in the $+$ case).
On $\mathcal E^\ast(U)$, every $x\in B'\setminus U$ satisfies $x+U\subset A$ (and $2x\in A$), and
$\mathcal E^\ast(U\cup\{x\})$ holds with witness $B'\setminus\{x\}$, since
$x+(B'\setminus\{x\})\subset B'\oplus B'$. The rest is as before.
\end{proof}

\emph{The deterministic case.} If $p=\one_D$, then $\sum_{x\in X}w_U(x)=\infty$ means that
$x+U\subset D$ for infinitely many $x\in X$. If $D\supset B+C$ with $B,C$ infinite, the
set of pairs of finite subsets of $B$ and of $C$ is a $BC$-strategy (in (a),
$X_{\mathcal S}(U,V)\supset C\setminus V$ and all these terms equal $1$); conversely a strategy
yields $B+C\subset D$ by the greedy construction. The same holds for $\oplus$ and $+$
with the finite subsets of $B$.

\begin{example}
\label{IB:ex:oplus}
Let $D=\{2^i+2^j:0\le i<j\}$. Then $B=\{2^i:i\ge0\}$ satisfies $B\oplus B\subset D$. But no
infinite $B$ satisfies $B+B\subset D$: from $2b\in D$, every $b\in B$ has exactly two
binary digits $1$. Fix $b=2^a+2^c\in B$ ($a<c$). If $b'\in B$ has lowest binary digit
above position $c+1$, then $b+b'$ has four binary digits $1$; so infinitely many
$b'\in B$ have lowest digit at most $c+1$, and two of them,
$b_1'=2^\alpha+2^{c_1}$ and $b_2'=2^\alpha+2^{c_2}$, share the lowest digit $\alpha$ and have
$\alpha+1<c_1<c_2$. Then $b_1'+b_2'=2^{\alpha+1}+2^{c_1}+2^{c_2}\notin D$.
\end{example}

\subsection{Rigid sets}

Fix $\Lambda=\{\lambda_1<\lambda_2<\cdots\}\subset\N$ with $\lambda_{j+1}>8(\lambda_1+\dots+\lambda_j)$ for all $j$ (for
instance $\lambda_j=9^j$).

\begin{lemma}
\label{IB:lem:rigid}
If $\sum_je_j\lambda_j=0$ with $e_j\in\Z$ and $\sum_j|e_j|\le8$, then all $e_j=0$. Consequently,
every $n\in\Lambda+\Lambda$ has a unique representation $n=\lambda+\lambda'$ as an unordered pair
(multiset) $P(n)=\{\lambda,\lambda'\}$.
\end{lemma}

\begin{proof}
If $J$ is the largest index with $e_J\ne0$, then
$|e_J\lambda_J|\ge\lambda_J>8\sum_{j<J}\lambda_j\ge|\sum_{j<J}e_j\lambda_j|$. Unique representation is
the case of a relation with $\sum|e_j|\le4$.
\end{proof}

\begin{lemma}
\label{IB:lem:R}
Let $X,\Gamma\subset\Z$ with $|X|\ge2$ and $|\Gamma|\ge2$, and suppose that $X+\Gamma\subset\Lambda+\Lambda$. Then
there is a unique $t=t(X,\Gamma)\in\Z$ such that $x-t\in\Lambda$ for all $x\in X$, $c+t\in\Lambda$ for all
$c\in\Gamma$, and $P(x+c)=\{x-t,c+t\}$ for all $x\in X$, $c\in\Gamma$. If $X'\subset X$ and
$\Gamma'\subset\Gamma$ with $|X'|,|\Gamma'|\ge2$, then $t(X',\Gamma')=t(X,\Gamma)$.
\end{lemma}

\begin{proof}
\emph{Uniqueness.} If $t\neq t'$ both work and $x\in X$, then for $c\in\Gamma$,
$\{x-t,c+t\}=\{x-t',c+t'\}$ forces $x-t=c+t'$; this determines $c$, contradicting
$|\Gamma|\ge2$. The last assertion follows, since $t(X,\Gamma)$ works for $(X',\Gamma')$.

\emph{Two rows.} Let $x_1\ne x_2$ in $X$ and write $P_{ic}=P(x_i+c)$. First, some $c_0\in\Gamma$ has
$P_{1c_0}\cap P_{2c_0}\neq\emptyset$. Otherwise, take $c\ne c'$ in $\Gamma$. The identity
$(x_1+c)+(x_2+c')=(x_1+c')+(x_2+c)$ is a relation among at most eight elements of $\Lambda$, so by
\cref{IB:lem:rigid} the multisets $P_{1c}\uplus P_{2c'}$ and $P_{1c'}\uplus P_{2c}$ coincide. As
$P_{1c}$ is disjoint from $P_{2c}$, it is contained in $P_{1c'}$ as a multiset, so $P_{1c}=P_{1c'}$
and $c=c'$, a contradiction. Write $x_1+c_0=u+v$ and $x_2+c_0=u'+v$ with $u,u',v\in\Lambda$; then
$u\ne u'$ because $x_1\ne x_2$. Put $t=x_1-u$, so that $x_1-t=u$ and $x_2-t=u'$ lie in $\Lambda$.
Let $c\in\Gamma$. Since $x_1-x_2=u-u'$, the identity $(x_1+c)+u'=(x_2+c)+u$ is a relation among
six elements of $\Lambda$, so $P_{1c}\uplus\{u'\}=P_{2c}\uplus\{u\}$. As $u\ne u'$, we get $u\in P_{1c}$,
that is, $x_1+c=u+w$ with $w\in\Lambda$. Then $c+t=w\in\Lambda$, and $P(x_i+c)=\{x_i-t,c+t\}$ for
$i=1,2$ by unique representation.

\emph{General $X$.} Fix $x_1\ne x_2$ in $X$ and let $t=t(\{x_1,x_2\},\Gamma)$. For $x\in X$,
$x\neq x_1$, the two-row case gives $t'=t(\{x_1,x\},\Gamma)$; both $t$ and $t'$ work for
$(\{x_1\},\Gamma)$ in the sense of the uniqueness argument, which only used one row,
so $t'=t$. Hence $x-t\in\Lambda$, and $P(x+c)=\{x-t,c+t\}$ by unique representation.
\end{proof}

\begin{corollary}
\label{IB:cor:R}
If $B,C\subset\Z$ are infinite and $B+C\subset\Lambda+\Lambda$, there is $t\in\Z$ with $B-t\subset\Lambda$,
$C+t\subset\Lambda$ and $P(b+c)=\{b-t,c+t\}$ for all $b\in B$, $c\in C$.
\end{corollary}

\subsection{Proof of \texorpdfstring{\cref{IB:thm:nonborel}}{the non-Borel theorem}}

Fix injective labellings $s\mapsto\beta_s$, $s\mapsto\gamma_s$ of $\N_0^{<\N_0}$ by elements of $\Lambda$
with disjoint images. Comparability of finite sequences means that one is an
initial segment of the other.

\emph{Upper bounds.} $\mathfrak D_\ast$ is the projection of a Borel subset of
$2^\N\times2^\N\times2^\N$, hence analytic. For $\mathfrak A_{BC}$, by \cref{IB:thm:strategy} it is
the projection to $[0,1]^\N$ of the set of pairs $(\mathcal S,p)\in2^{\mathcal P_{\rm fin}(\N)^2}\times[0,1]^\N$
such that $\mathcal S$ is a $BC$-strategy for $p$; this set is Borel, since each partial
sum of the series in (a), (b) is a continuous function of $(\mathcal S,p)$ and
``$=\infty$'' means ``$\forall M\,\exists N$: partial sum $>M$''. The same applies to $\mathfrak A_\oplus$,
$\mathfrak A_+$, and to the intersections with the Polish space $(0,1)^\N$.

\emph{\ref{IB:nb1} for $B+C$.} For a tree $\mathcal T$ let
$D_{\mathcal T}=\{\beta_s+\gamma_{s'}: s,s'\in\mathcal T\text{ comparable}\}$. By \cref{IB:lem:rigid}, whether
$n\in D_{\mathcal T}$ depends on at most two coordinates of $\mathcal T$, so $\mathcal T\mapsto D_{\mathcal T}$ is
continuous. If $z$ is an infinite branch of $\mathcal T$, then $B=\{\beta_{z|n}\}$,
$C=\{\gamma_{z|n}\}$ give $B+C\subset D_{\mathcal T}$. Conversely let $B+C\subset D_{\mathcal T}$ with $B,C$
infinite. By \cref{IB:cor:R}, $\{b-t,c+t\}=\{\beta_s,\gamma_{s'}\}$ with $s,s'\in\mathcal T$
comparable. Fixing $c$, if $c+t$ is a $\gamma$-label then every $b-t$ is a $\beta$-label, and
then every $c+t$ is a $\gamma$-label; so either $B-t$ consists of $\beta$-labels and $C+t$ of
$\gamma$-labels, or vice versa. In either case we get infinite sets $Q_1,Q_2\subset\mathcal T$ such
that every element of $Q_1$ is comparable with every element of $Q_2$. A node has
finitely many initial segments, so for $q\in Q_1$ all but finitely many elements of
$Q_2$ properly extend $q$, and vice versa. Choosing alternately we obtain
$q_1\subsetneq q_1'\subsetneq q_2\subsetneq q_2'\subsetneq\cdots$ in $\mathcal T$, which defines an infinite branch.
So $\mathcal T\mapsto D_{\mathcal T}$ reduces $\mathrm{IF}$ to $\mathfrak D_{BC}$.

\emph{\ref{IB:nb1} for $B\oplus B$ and $B+B$.} Let
$D^\oplus_{\mathcal T}=\{\gamma_s+\gamma_{s'}:s\ne s'\in\mathcal T\text{ comparable}\}$ and
$D^+_{\mathcal T}=\{\gamma_s+\gamma_{s'}:s,s'\in\mathcal T\text{ comparable}\}$; both maps are continuous, and a
branch $z$ gives $B=\{\gamma_{z|n}\}$. Conversely, let $B$ be infinite with
$B\oplus B\subset D^\oplus_{\mathcal T}$ or $B+B\subset D^+_{\mathcal T}$. Split $B$ into disjoint infinite sets
$B_1,B_2$; then $B_1+B_2\subset B\oplus B\subset\Lambda+\Lambda$, and \cref{IB:cor:R} gives $t$ with
$P(b+b')=\{b-t,b'+t\}$ for $b\in B_1$, $b'\in B_2$. Since $b+b'=\gamma_s+\gamma_{s'}$ with $s,s'\in\mathcal T$
comparable, $b-t$ and $b'+t$ are $\gamma$-labels of comparable nodes of $\mathcal T$. So the nodes
labelling $B_1-t$ and those labelling $B_2+t$ form infinite sets $Q_1,Q_2\subset\mathcal T$ such that every
element of $Q_1$ is comparable with every element of $Q_2$, and the argument for $B+C$ gives an
infinite branch.

\emph{\ref{IB:nb2}.} The embedding $D\mapsto\one_D$ of $2^\N$ into $[0,1]^\N$ is continuous, and
$\mathbf P_{\one_D}(\mathcal E_\ast)=1$ if and only if $D\in\mathfrak D_\ast$; so the first claim follows
from \ref{IB:nb1}. For $(0,1)^\N$, let $p^D_n=\frac12$ for $n\in D$ and $p^D_n=2^{-n-2}$ for $n\notin D$;
$D\mapsto p^D$ is continuous. Under $\mathbf P_{p^D}$, $A\setminus D$ is almost surely finite (first
Borel--Cantelli lemma), so by invariance under finite modifications
$\mathbf P_{p^D}(\mathcal E_\ast)=\mathbf P(A\cap D\in\mathcal E_\ast)$, where $A\cap D$ is a random subset of $D$
containing each element independently with probability $\frac12$. If $D\notin\mathfrak D_\ast$
this probability is $0$. If $D\supset B_0+C_0$, the pairs of finite subsets of $B_0$
and $C_0$ form a $BC$-strategy for $\frac12\one_D$ (all weights are $2^{-|U|}$ or
$2^{-|V|}$, with infinitely many candidates), so the probability is $1$ by
\cref{IB:thm:strategy}. Likewise, if $D\supset B_0\oplus B_0$ (resp.\ $B_0+B_0$), the finite
subsets of $B_0$ form an $\oplus$- (resp.\ $+$-) strategy with weights $\ge2^{-|U|-1}$.
Hence $\mathbf P_{p^D}(\mathcal E_\ast)=1$ if and only if $D\in\mathfrak D_\ast$, and composing with the
reductions in \ref{IB:nb1} proves the second claim.

\emph{\ref{IB:nb3}.} For $s\in\N_0^{<\N_0}$ let $\ell(s)=|s|+\sum_is_i$, and let $\mathcal N$ be the
set of pairs $x=(s,\eta)$ with $s\in\N_0^{<\N_0}$ and $\eta\in\{0,1\}^{\ell(s)}$, partially ordered
by $(s,\eta)\preceq(s',\eta')$ if $s\subset s'$ and $\eta$ is an initial segment of $\eta'$. Write
$\ell(x)=\ell(s)$. Then: if $s\subsetneq s'$ then $\ell(s)<\ell(s')$; every $x\in\mathcal N$ has
finitely many predecessors; for each $L$ there are finitely many $x$ with $\ell(x)=L$;
and two distinct elements with the same first coordinate are incomparable. Fix
injective labellings $x\mapsto\beta_x$, $x\mapsto\gamma_x$ of $\mathcal N$ by elements of $\Lambda$ with
disjoint images, and let $\theta_L=1/(L+2)$. For a tree $\mathcal T$ define $q^{\mathcal T}\in[0,\frac12]^\N$ by
$q^{\mathcal T}_n=\theta_{\max(\ell(x),\ell(y))}$ if $n=\beta_x+\gamma_y$ with $x,y$ comparable and with
first coordinates in $\mathcal T$, and $q^{\mathcal T}_n=0$ otherwise; and let $g(\mathcal T)=p^{\mathcal T}$, where
$p^{\mathcal T}_n=q^{\mathcal T}_n$ if $q^{\mathcal T}_n>0$ and $p^{\mathcal T}_n=2^{-n-2}$ otherwise. Each coordinate of
$p^{\mathcal T}$ depends on at most two coordinates of $\mathcal T$ (\cref{IB:lem:rigid}), so $g$ is
continuous, and $p^{\mathcal T}\in(0,1)^\N$. Moreover $p^{\mathcal T}_n\to0$: $q^{\mathcal T}_n\ge\eps$ forces
$\max(\ell(x),\ell(y))\le1/\eps$, which happens for finitely many pairs $(x,y)$ only.

Suppose $z$ is an infinite branch of $\mathcal T$ and let $\ell_m=\ell(z|m)$, a strictly
increasing sequence. Let $\mathcal N_z=\{(z|m,\eta)\in\mathcal N:m\ge0\}$ and let $\mathcal S$ consist of the
pairs $(\{\beta_x\}_{x\in P},\{\gamma_y\}_{y\in Q})$ with $P,Q\subset\mathcal N_z$ finite and $P\cup Q$ a chain.
We check (a) for $q^{\mathcal T}$; (b) is symmetric. If $P\cup Q=\emptyset$, (a) holds since
$U=\emptyset$ and there are infinitely many candidates. Otherwise let $\delta=(z|d,\eta_\delta)$ be the
largest element of $P\cup Q$. For $m>d$ each of the $2^{\ell_m-\ell_d}$ elements
$y=(z|m,\eta)\succ\delta$ extends the chain, and $w_U(\gamma_y)=\prod_{x\in P}q^{\mathcal T}_{\beta_x+\gamma_y}=\theta_{\ell_m}^{|P|}$.
Hence the sum in (a) is at least $\sum_{m>d}2^{\ell_m-\ell_d}(\ell_m+2)^{-|P|}=\infty$. So $\mathcal S$ is a
$BC$-strategy for $q^{\mathcal T}$, hence for $p^{\mathcal T}\ge q^{\mathcal T}$, and $\mathbf P_{p^{\mathcal T}}(\mathcal E_{BC})=1$.

Conversely suppose $\mathbf P_{p^{\mathcal T}}(\mathcal E_{BC})>0$. Almost surely $A\setminus\supp q^{\mathcal T}$ is finite,
so with positive probability $A\cap\supp q^{\mathcal T}$ contains some $B+C$; in particular
$\supp q^{\mathcal T}\supset B+C$ for some infinite $B,C$. As $\supp q^{\mathcal T}\subset\Lambda+\Lambda$, the argument
in \ref{IB:nb1} (which only used that every element has finitely many predecessors)
yields an infinite chain $x_1\prec x_2\prec\cdots$ in $\mathcal N$ whose first coordinates lie
in $\mathcal T$. Distinct comparable elements have distinct first coordinates, so the
first coordinates form a strictly increasing sequence in $\mathcal T$, i.e.\ an infinite
branch. Thus $g^{-1}(\mathfrak A_{BC})=\mathrm{IF}$. If $W$ were Borel with
$W\cap\mathfrak C=\mathfrak A_{BC}\cap\mathfrak C$, then $\mathrm{IF}=g^{-1}(W)$ would be Borel, which it is not.
\qed

\subsection{Proof of \texorpdfstring{\cref{IB:thm:hierarchy}}{the hierarchy theorem} and the examples}

\emph{(F) implies $\mathbf P(\mathcal E_{BC})=1$:} the pairs of finite subsets of $B_0$ and $C_0$
form a $BC$-strategy, since removing the finitely many terms $x\in V$ from a divergent
series of terms in $[0,1]$ keeps it divergent. For the level-set condition,
$w_U(c)\ge\delta^{|U|}$ for $c\in C_0$; for $p=q\one_D$ necessity is clear since $A\subset D$.

\emph{$\mathbf P(\mathcal E_{BC})=1$ implies (N):} by the first Borel--Cantelli lemma, almost
surely every finite $G$ with $S_G<\infty$ has only finitely many translates in $A$.
Fix an outcome in this event and in $\mathcal E_{BC}$, with $B+C\subset A$. Every finite
$G\subset C$ has the infinitely many translates $b+G\subset A$, $b\in B$, so $S_G=\infty$; similarly
for $G\subset B$. Take $B_0=B$, $C_0=C$.

\emph{(N) implies $\mathbf P(\mathcal E_{<\omega})=1$:} by \cref{IB:thm:finite} applied to $k$-subsets of
$C_0$.

\begin{example}[$\mathbf P(\mathcal E_{BC})=1$ but (F) fails]
\label{IB:ex:tree}
Label the binary tree $\{0,1\}^{<\N_0}$ injectively by $\beta_\sigma,\gamma_\sigma\in\Lambda$ (disjoint images), let
$\theta_j=1/(j+1)$, and let $p_{\beta_\sigma+\gamma_\tau}=\theta_{\max(|\sigma|,|\tau|)}$ for comparable $\sigma,\tau$, and $p_n=0$
otherwise. Exactly as in the proof of \cref{IB:thm:nonborel}\ref{IB:nb3}, the pairs
$(\{\beta_\sigma\}_{\sigma\in P},\{\gamma_\tau\}_{\tau\in Q})$ with $P\cup Q$ a finite chain form a $BC$-strategy:
below the deepest node, at depth $d$, there are $2^{j-d}$ candidates at depth $j$, each
of weight $\theta_j^{|P|}$, and $\sum_j2^{j-d}(j+1)^{-|P|}=\infty$. So $\mathbf P(\mathcal E_{BC})=1$.

Suppose $B_0,C_0$ satisfy (F). Fix $b_1\ne b_2$ in $B_0$ and let
$\Gamma=\{c\in C_0:w_{\{b_1,b_2\}}(c)>0\}$, which is infinite by (F). As $\supp p\subset\Lambda+\Lambda$, let
$t=t(\{b_1,b_2\},\Gamma)$ (\cref{IB:lem:R}). For $b\in B_0$, the set
$\Gamma_b=\{c\in C_0:w_{\{b_1,b_2,b\}}(c)>0\}\subset\Gamma$ is infinite, and
$t(\{b_1,b_2,b\},\Gamma_b)=t(\{b_1,b_2\},\Gamma_b)=t$; hence $B_0-t\subset\Lambda$. For each $c\in\Gamma$,
$P(b_i+c)=\{b_i-t,c+t\}$ contains one $\beta$- and one $\gamma$-label, so $b_1-t$ and $b_2-t$
(and then, by the same argument, all of $B_0-t$) have the same type; as $p$ is
symmetric under exchanging the roles of $\beta$ and $\gamma$, we may assume
$B_0-t=\{\beta_\sigma:\sigma\in\Sigma_B\}$. Let $c\in C_0$. By (F) with $V=\{c\}$ there are distinct
$b=\beta_{\sigma_1}+t$, $b'=\beta_{\sigma_2}+t$ in $B_0$ with $b+c=\beta_{\sigma_1'}+\gamma_{\tau_1}$ and
$b'+c=\beta_{\sigma_2'}+\gamma_{\tau_2}$. Eliminating $c$ gives a relation of six elements of $\Lambda$,
so by \cref{IB:lem:rigid} the multisets $\{\beta_{\sigma_1'},\gamma_{\tau_1},\beta_{\sigma_2}\}$ and
$\{\beta_{\sigma_2'},\gamma_{\tau_2},\beta_{\sigma_1}\}$ coincide; hence $\tau_1=\tau_2$, and since $\sigma_1\ne\sigma_2$ we
get $\beta_{\sigma_1'}=\beta_{\sigma_1}$, i.e.\ $c+t=\gamma_{\tau_1}$. So $C_0+t=\{\gamma_\tau:\tau\in\Sigma_C\}$, and
$b+c\in\supp p$ (for $b=\beta_\sigma+t$, $c=\gamma_\tau-t$) if and only if $\sigma,\tau$ are comparable.
By (F) with $V=\{c,c'\}$, infinitely many $\sigma\in\Sigma_B$ are comparable with both $\tau$ and
$\tau'$; if $\tau,\tau'$ were incomparable, such $\sigma$ would be common initial segments of
$\tau$ and $\tau'$, of which there are finitely many. So $\Sigma_C$ is a chain, with at most
one node at each depth. Then for $U=\{b_1,b_2\}$,
$\sum_{c\in C_0}w_U(c)\le\sum_{\tau\in\Sigma_C}\theta_{|\tau|}^2\le\sum_j(j+1)^{-2}<\infty$, contradicting (F).
\end{example}

\begin{example}[(N) holds but $\mathbf P(\mathcal E_{BC})=0$]
\label{IB:ex:half}
Label $\{\beta_{k,m}:k,m\in\N\}\cup\{\gamma_i:i\in\N\}$ injectively by elements of $\Lambda$ and let
$p=\one_D$ with $D=\{\beta_{k,m}+\gamma_i:i\le k,\ m\in\N\}$. Condition (N) holds with
$B_0=C_0=\{\gamma_i\}$: a finite $G\subset\{\gamma_1,\dots,\gamma_k\}$ has the translates $\beta_{k,m}+G\subset D$,
$m\in\N$. But $D\notin\mathfrak D_{BC}$: if $B+C\subset D$, then by \cref{IB:cor:R}
$\{b-t,c+t\}=\{\beta_{k,m},\gamma_i\}$ with $i\le k$. If the $b-t$ are $\beta$-labels, fix $b_1$ with
$b_1-t=\beta_{k_1,m_1}$; then every $c+t$ lies in $\{\gamma_1,\dots,\gamma_{k_1}\}$, so $C$ is finite.
Otherwise, symmetrically, $B$ is finite. In particular $\mathcal E_{<\omega}$ holds but
$\mathcal E_{BC}$ fails, for a deterministic set.
\end{example}

\begin{example}[$\mathbf P(\mathcal E_{<\omega})=1$ but (N) fails]
\label{IB:ex:rigid}
Label $\{\beta_{k,m},\gamma_{k,i}:k,m\in\N,\ 1\le i\le k\}$ injectively by elements of $\Lambda$ and let
$p=\one_D$ with $D=\{\beta_{k,m}+\gamma_{k,i}:1\le i\le k,\ m\in\N\}$. Then
$\{\beta_{k,m}:m\in\N\}+\{\gamma_{k,1},\dots,\gamma_{k,k}\}\subset D$, so $\mathcal E_k$ holds for every $k$.
Suppose $C_0$ is infinite and $S_G=\infty$ (i.e.\ $G$ has infinitely many translates in
$D$) for all $G\subset C_0$ with $|G|\le3$. Fix $c_1\ne c_2$ in $C_0$, let $\Gamma$ be the
(infinite) set of $n$ with $n+\{c_1,c_2\}\subset D$, and $t=t(\{c_1,c_2\},\Gamma)$. For each
$n\in\Gamma$, $n+t$ has one type and $c_1-t$, $c_2-t$ have the other. If $c_1-t=\beta_{k_1,m_1}$ and
$c_2-t=\beta_{k_2,m_2}$, then $n+t\in\{\gamma_{k_1,i}\}\cap\{\gamma_{k_2,i}\}$, a finite set, for all
$n\in\Gamma$, which is impossible. Hence $c_j-t=\gamma_{k,i_j}$ and $n+t=\beta_{k,m_n}$ with the same
$k$. For any $c\in C_0$, applying the same reasoning to $\{c_1,c_2,c\}$ and its
infinite set $\Gamma_c\subset\Gamma$ of translates, \cref{IB:lem:R} gives the same $t$, and $c-t$ is a
$\gamma$-label paired with $n+t=\beta_{k,m_n}$, so $c-t\in\{\gamma_{k,1},\dots,\gamma_{k,k}\}$. Thus $C_0$
is finite, a contradiction; so (N) fails.
\end{example}

\subsection{Proof of \texorpdfstring{\cref{IB:thm:monotone}}{the monotone theorem}}

\ref{IB:m3}$\Rightarrow$\ref{IB:m2}: Fix $L$. The events $\{[jL,(j+1)L)\subset A\}$, $j\ge1$, are
independent, with probabilities
$\prod_{i=jL}^{(j+1)L-1}p_i\ge\alpha^Lq_{(j+1)L}^L$. Since $q_n\le q_{(j+1)L}$ for
$(j+1)L\le n<(j+2)L$, we have $\sum_{j\ge1}q_{(j+1)L}^L\ge L^{-1}\sum_{n\ge2L}q_n^L=\infty$, as
$p_n^L\le\beta^Lq_n^L$ and $\sum_np_n^L=\infty$. By the second Borel--Cantelli lemma $A$ almost
surely contains infinitely many intervals of length $L$; intersecting over $L$, $A$
is almost surely thick.

\ref{IB:m2}$\Rightarrow$\ref{IB:m1} is \cref{IB:lem:thick}, and
\ref{IB:m1}$\Rightarrow$\ref{IB:m5} is \cref{IB:thm:hierarchy}.

\ref{IB:m5}$\Rightarrow$\ref{IB:m3}: Suppose $\sum_np_n^k<\infty$ for some $k$, so $\sum_nq_n^k<\infty$. For
$F\subset\N$ with $|F|=k$ and $\min F=f_0$, monotonicity gives
$S_F=\sum_nw_F(n)\le\beta^k\sum_{n}q_{n+f_0}^k<\infty$, so $\mathbf P(\mathcal E_F)=0$ by \cref{IB:thm:finite}.
Taking the union over the countably many $F$, almost surely $A$ contains no $B+F$
with $B$ infinite and $|F|=k$. This also proves the statement following the list.

\ref{IB:m3}$\Leftrightarrow$\ref{IB:m4}: Both conditions are unchanged when $p$ is replaced
by $q$. If $\sum_nq_n^k=S<\infty$, then $nq_n^k\le\sum_{m\le n}q_m^k\le S$, so
$q_n\le(S/n)^{1/k}$ and $q_nn^\eps\to0$ for $\eps<1/k$. Conversely, if $q_nn^\eps\to0$, then
$q_n\le n^{-\eps}$ for large $n$ and $\sum_nq_n^k<\infty$ for $k>1/\eps$.

The $0$--$1$ statements follow from \cref{IB:thm:finite,IB:thm:strategy}. \qed

\formalization{\Cref{IB:thm:monotone,IB:thm:finite} and \cref{IB:ex:gap} are formalized in
\leanfile{S44a} (\leanname{thm_monotone}, \leanname{thm_finite_int}, \leanname{ex_gap}), with
the random set modelled by an infinite product of Bernoulli measures.
\Cref{IB:thm:strategy,IB:thm:strategy2,IB:thm:hierarchy}, \cref{IB:lem:BC,IB:lem:rigid,IB:lem:R}
and \cref{IB:ex:oplus,IB:ex:tree,IB:ex:half,IB:ex:rigid} are formalized in \leanfile{S44b}
(\leanname{thm_strategy}, \leanname{thm_strategy2_oplus}, \leanname{thm_strategy2_plus},
\leanname{thm_hierarchy}, \leanname{thm_hierarchy_strict}). For \cref{IB:thm:nonborel} the
same file proves analyticity, the continuity of the reductions and the reduction property on
trees (\leanname{thm_nonborel_a_BC}, \leanname{thm_nonborel_b_BC}, \leanname{thm_nonborel_c}
and their variants); the classical facts that $\mathrm{IF}$ is $\Sigma^1_1$-complete and not
Borel~\cite[Theorems~14.2 and~27.1]{kechris1995classical} are stated as explicit hypotheses.}

\section{Ordered sumsets in the primes with summands in \texorpdfstring{$\P-1$}{P-1}}
\label{IB:sec:c45}

In this section we prove Conjecture~4.5 of~\cite{kra2023problems}. It turns out
to be a short corollary of the theorem of Tao and Ziegler on ascending chains of
prime-producing tuples~\cite{tao2023infinite}.

\begin{surveyconj}{4.5}
\label{IB:q:c45}
There exist infinite sets $B\subset\N$ and $C\subset\P-1$ such that
$\P\supset\{b+c:b\in B,\ c\in C;\ b<c\}$.
\end{surveyconj}

The survey remarks that the analogous statement with $B\subset\P-1$ instead of
$C\subset\P-1$ follows from the results of~\cite{tao2023infinite} together with the
existence of an infinite admissible subset of $\P-1$. Following~\cite[Definition~1.1]{tao2023infinite},
a tuple $(h_1,\dots,h_k)$ of natural numbers is \emph{prime-producing} if there are
infinitely many $n$ such that $n+h_1,\dots,n+h_k$ are simultaneously prime.

\begin{theorem*}[{Tao--Ziegler~\cite[Theorem~1.5]{tao2023infinite}}]
There exists an infinite sequence $h_1<h_2<\dots$ of natural numbers, such that the
$k$-tuple $(h_1,\dots,h_k)$ is prime-producing for every $k$.
\end{theorem*}

\begin{theorem}
\label{IB:thm:c45}
For all integers $s,t$ there are infinite sets $B,C\subset\N$ such that $C+t\subset\P$ and
$b+c+s\in\P$ for all $b\in B$, $c\in C$ with $b<c$. In particular
\textup{(}$s=0$, $t=1$\textup{)} Conjecture~4.5 of~\cite{kra2023problems} holds, and
\textup{(}$s=t=1$\textup{)} there are infinite $B,C$ with $C\subset\P-1$ and
$\{b+c:b\in B,\ c\in C,\ b<c\}\subset\P-1$.
\end{theorem}

\begin{proof}
Let $h_1<h_2<\cdots$ be as in the Tao--Ziegler theorem and put
$\beta_i=h_{i+1}-h_1+t-s$; fix $i_0$ with $\beta_i\ge1$ for $i\ge i_0$. For $k\ge i_0$ let
\[
M_k=\{m\in\N:\ m+t\in\P\ \text{and}\ m+s+\beta_i\in\P\ \text{for}\ i_0\le i\le k\}.
\]
Each $M_k$ is infinite: there are infinitely many $n$ with $n+h_1,\dots,n+h_{k+1}$ all
prime; they satisfy $n\ge2-h_1$, so infinitely many of them exceed $t-h_1$, and for
those $m=n+h_1-t\in\N$ satisfies $m+t=n+h_1\in\P$ and $m+s+\beta_i=n+h_{i+1}\in\P$.

Now choose $i_1=i_0$ and inductively $c_j\in M_{i_j}$ with $c_j>\beta_{i_j}$, and then
$i_{j+1}>i_j$ with $\beta_{i_{j+1}}>c_j$ (possible since $\beta_i\to\infty$). Then
$\beta_{i_1}<c_1<\beta_{i_2}<c_2<\cdots$. Let $B=\{\beta_{i_j}:j\in\N\}$ and $C=\{c_j:j\in\N\}$. Then
$C+t\subset\P$. If $b=\beta_{i_l}$ and $c=c_j$ with $b<c$, the interlacing gives $l\le j$, so
$i_0\le i_l\le i_j$, and $c_j\in M_{i_j}$ yields $b+c+s=c_j+s+\beta_{i_l}\in\P$.
\end{proof}

\formalization{\Cref{IB:thm:c45} is formalized in \leanfile{C45} (\leanname{c45_of_TZ},
\leanname{conjecture45_of_TZ}), with the Tao--Ziegler theorem stated as an explicit
hypothesis.}

\begin{remark}
\label{IB:rem:c45}
\Cref{IB:thm:c45} is a short consequence of~\cite[Theorem~1.5]{tao2023infinite}, and
all the substance lies in the Tao--Ziegler theorem. The only observation is that
being prime-producing is invariant under translation, so the chain can be
translated to contain $1$; then every witness $m$ satisfies $m+1\in\P$. The
variant mentioned in the survey, with $B\subset\P-1$ in place of $C\subset\P-1$, appears to
follow in the same way from the argument of~\cite{tao2023infinite}. There,
Proposition~3.1 (the Maynard sieve input) is stated for an arbitrary admissible
tuple, and in the proof of Theorem~1.5 the odd squares serve only to make every
finite subtuple admissible. Running the same proof with the odd squares replaced
by an infinite admissible subset of $\P-1$ gives a chain inside $\P-1$, as
indicated in the introduction of~\cite{tao2023infinite}; Corollary~1.6 there and
the interlacing above then give the variant. We have not checked every step of this adaptation;
Conjecture~4.5 uses only the statement of Theorem~1.5, whereas the variant would re-run its
proof.

Neither argument gives the two conditions at once. To get $B,C\subset\P-1$ with
$b+c\in\P$ for $b<c$ in the same way, one would need a prime-producing chain with
entries in $\P-1$ whose translate by $1-h_1$ also has its entries in $\P-1$, and
the argument of~\cite{tao2023infinite} does not obviously provide this. The
statement that seems intended, the ordered analogue of~\cite[Conjecture~4.3]{kra2023problems}
with $k=2$, also needs this: it asks for $B\cup C\cup\{b+c:b<c\}\subset\P-1$, which is the
case $s=t=1$ of \cref{IB:thm:c45} with the additional requirement $B\subset\P-1$. We do not
know whether these strengthenings hold.
\end{remark}

\part{Sums and products}
\section{Sums and products in \texorpdfstring{$\N$}{N} and \texorpdfstring{$\Q$}{Q}}
\label{SP:sec}

This section is about three questions from~\cite[\S5.1]{kra2023problems} that ask for sums and
products of infinite sets at the same time.
In \cref{SP:sec:u51a} we answer Question~\S5.1(a) positively: the two explicit sets named
in~\cite[\S5.1]{kra2023problems}, namely Beatty sequences and the set of integers with an even
number of prime factors, both contain $B+C$ and $BC$ for some infinite $B,C\subset\N$.
In \cref{SP:sec:q54} we answer Question~5.4 of~\cite{kra2023problems} negatively. The counterexample
is a set built at logarithmic scale; it has logarithmic density $3/10$ and it is aperiodic.
The same set answers Question~5.2 of~\cite{kra2023problems} negatively if ``positive density''
is read as positive upper or lower density (\cref{SP:rem:q52}).
In \cref{SP:sec:q56} we answer Question~5.6 of~\cite{kra2023problems} negatively, as it is literally
stated. For each double \Fol{} sequence $\Phi$ we construct a counterexample, which depends on
$\Phi$: it combines the sign constraint at the real place with the removal of a $\Phi$-null set
that makes every bounded part of the set finite.

\smallskip
\noindent\emph{Notation.}
For $x\in\R$ we write $\{x\}\in[0,1)$ for the fractional part of $x$ and $\|x\|$ for the distance
from $x$ to the nearest integer.
We identify $\T=\R/\Z$ with $[0,1)$, and we write $\|x-y\|$ for the distance in $\T$.
For $X,Y\subset\N$ (or $\Q$) we write $X+Y=\{x+y:x\in X,\,y\in Y\}$, $XY=\{xy:x\in X,\,y\in Y\}$,
$X+t=\{x+t:x\in X\}$ and $dX=\{dx:x\in X\}$.
Recall from~\cite[eq.~(19)]{kra2023problems} that $A\subset\N$ is \emph{aperiodic} if
\begin{equation}\label{SP:eq:aperiodic}
\lim_{N\to\infty}\frac1N\sum_{n=1}^N\one_A(n)\,\e(\alpha n)=0\qquad\text{for all }\alpha\in\Q\setminus\Z,
\end{equation}
where $\e(x)=e^{2\pi i x}$.

\subsection{The explicit examples}
\label{SP:sec:u51a}

Question~5.2 of~\cite{kra2023problems} asks whether every aperiodic $A\subset\N$ of positive
density contains $B+C$ and $BC$ for some infinite sets $B,C\subset\N$. Just before that question
the survey names two typical aperiodic sets: the set of all numbers with an even number of prime
factors, and the Beatty sequences $\{\lfloor n\alpha+\gamma\rfloor:n\in\N\}$, where $\alpha>0$ is
irrational and $\gamma\in\R$. It then remarks: ``Even for the explicit examples we have mentioned,
we consider the answer to be interesting.''

\begin{question}{\S5.1(a)}\label{SP:q:u51a}
Let $A$ be either a Beatty sequence $\{\lfloor n\alpha+\gamma\rfloor:n\in\N\}$, with $\alpha>0$
irrational and $\gamma\in\R$, or the set of all $n\in\N$ with an even number of prime factors. Are there infinite sets
$B,C\subset\N$ with $B+C\subset A$ and $BC\subset A$?
\end{question}

The survey does not say whether prime factors are counted with multiplicity. We answer the
question for both readings. The answer is yes in all cases (\cref{SP:cor:beatty,SP:thm:liouville}).

\subsubsection{Bohr sets and Beatty sequences}

\begin{theorem}\label{SP:thm:bohr}
Let $\vartheta\in\R\setminus\Q$, $\delta\in\R$, let $J\subset\T$ be a nonempty open set and let
$M_0\in\N$. Suppose that
\[
A\supseteq\{m\in\N:\ m\ge M_0,\ m\vartheta+\delta\bmod 1\in J\}.
\]
Then there is an infinite set $Z\subset\N$ such that $z+z'\in A$ and $zz'\in A$ for all distinct
$z,z'\in Z$. In particular, there are disjoint infinite sets $B,C\subset\N$ with $B+C\subset A$
and $BC\subset A$.
\end{theorem}

\begin{proof}
Since $J$ is open and nonempty, there are $\xi_0\in\T$ and $\ell\in(0,1)$ such that every
$x\in\T$ with $\|x-\xi_0\|<\ell/2$ lies in $J$. Fix $q\in\N$ with $q\ge 4/\ell$.
The points $2r/q$ ($r\in\Z$) form a subgroup of $\T$ in which consecutive points are at distance
at most $2/q$. Hence we can choose $r\in\Z$ with $\|2r/q+\delta-\xi_0\|\le 1/q\le\ell/4$.
Put $\xi=2r/q+\delta$. Then
\begin{equation}\label{SP:eq:xiJ}
\|x-\xi\|<\ell/4\ \Longrightarrow\ \|x-\xi_0\|<\ell/2\ \Longrightarrow\ x\in J .
\end{equation}

We choose $z_1<z_2<\cdots$ recursively. Put $z_0=1$, and suppose $z_0,\dots,z_{n-1}$ have been
chosen. Let $\eps_n=\ell/(8z_{n-1})$. We look for $z_n$ of the form $z_n=qk+2$ with $k\in\N$.
Then $z_n\vartheta-r/q=k(q\vartheta)+(2\vartheta-r/q)$. Since $q\vartheta$ is irrational, the
sequence $(kq\vartheta\bmod1)_{k>K}$ is dense in $\T$ for every $K$. So we can choose $k$ such that
$z_n=qk+2>\max(z_{n-1},M_0)$ and $\|z_n\vartheta-r/q\|<\eps_n$.
Thus every $n\ge1$ satisfies
\[
z_n\equiv 2\pmod q,\qquad z_n\vartheta=K_n+\frac rq+e_n,\quad K_n\in\Z,\ |e_n|<\eps_n=\frac{\ell}{8z_{n-1}}\le\frac\ell8 .
\]
Let $1\le i<j$. Write $z_i=qk_i+2$ with $k_i\in\Z$. We compute the two phases modulo $1$.
\begin{itemize}
\item Sums: $(z_i+z_j)\vartheta\equiv 2r/q+e_i+e_j$, and $|e_i+e_j|<\ell/4$.
\item Products: $z_iz_j\vartheta=z_iK_j+z_ir/q+z_ie_j$. Here $z_iK_j\in\Z$ and
$z_ir/q=k_ir+2r/q\equiv 2r/q$. Moreover $|z_ie_j|<z_i\ell/(8z_{j-1})\le\ell/8$, because
$z_i\le z_{j-1}$. So $z_iz_j\vartheta\equiv 2r/q+z_ie_j$.
\end{itemize}
Hence both $(z_i+z_j)\vartheta+\delta$ and $z_iz_j\vartheta+\delta$ lie within distance
$\ell/4$ of $\xi$ in $\T$, so they lie in $J$ by \eqref{SP:eq:xiJ}. Since
$z_i+z_j$ and $z_iz_j$ are at least $M_0$, both lie in $A$.
Take $Z=\{z_n:n\ge1\}$, $B=\{z_1,z_3,z_5,\dots\}$ and $C=\{z_2,z_4,z_6,\dots\}$.
\end{proof}

\begin{corollary}\label{SP:cor:beatty}
Let $\alpha>0$ be irrational, let $\gamma\in\R$, and let
$A=\{\lfloor n\alpha+\gamma\rfloor:n\in\N\}$. Then there are infinite sets $B,C\subset\N$ with
$B+C\subset A$ and $BC\subset A$.
\end{corollary}

\begin{proof}
Suppose first that $\alpha<1$. Then consecutive terms of the nondecreasing sequence
$\lfloor n\alpha+\gamma\rfloor$ differ by $0$ or $1$, and the sequence tends to infinity. So $A$
contains every integer $m\ge\lfloor\alpha+\gamma\rfloor$, and any infinite sets $B,C$ of
sufficiently large integers work. (Here and below $A$ may contain non-positive integers; only
its positive part matters.)

Now suppose $\alpha>1$ and put $\beta=1/\alpha\in(0,1)$. An integer $m$ lies in $A$ if and only
if some $n\in\N$ satisfies $m\le n\alpha+\gamma<m+1$, that is, $y_m\le n<y_m+\beta$, where
$y_m=(m-\gamma)\beta$. The least integer that is at least $y$ is $y+\{-y\}$. So the interval
$[y_m,y_m+\beta)$ contains an integer if and only if $\{-y_m\}<\beta$. When $m>\gamma$, this
integer is at least $y_m>0$, so it lies in $\N$. Therefore, for $m>\gamma$,
\[
m\in A\iff \{m\vartheta+\delta\}\in[0,\beta),\qquad\text{where }\vartheta=-\beta,\ \delta=\gamma\beta .
\]
Here $\vartheta$ is irrational. Apply \cref{SP:thm:bohr} with $J=(0,\beta)$ and any
$M_0>\gamma$.
\end{proof}

\begin{remark}\label{SP:rem:bohr}
\Cref{SP:thm:bohr} does not use any density assumption. It applies to every set that contains an
inhomogeneous Bohr set with an irrational frequency, up to finitely many elements; this includes
all Beatty sequences. The argument is elementary, and we expect the statement to be folklore. We
did not find it in the literature.
\end{remark}

\subsubsection{The Liouville function}

Let $\Omega(n)$ and $\omega(n)$ be the numbers of prime factors of $n$ counted with and
without multiplicity, respectively. Let $\lambda(n)=(-1)^{\Omega(n)}$ be the Liouville function.
It is completely multiplicative: $\lambda(mn)=\lambda(m)\lambda(n)$ for all $m,n\in\N$. Put
\[
L_+=\{n\in\N:\Omega(n)\text{ is even}\}=\{n:\lambda(n)=1\},\qquad
L'_+=\{n\in\N:\omega(n)\text{ is even}\}.
\]

\begin{theorem}\label{SP:thm:liouville}
For each $E\in\{L_+,L'_+\}$ there are infinite sets $B,C\subset\N$ with $B+C\subset E$ and
$BC\subset E$.
\end{theorem}

\begin{proof}
Let $S$ be the set of primes $p\equiv5\pmod6$. There are infinitely many such primes: if
$p_1,\dots,p_r$ were all of them, then $6p_1\cdots p_r-1\equiv5\pmod6$ would be coprime
to $6p_1\cdots p_r$; so all its prime factors would be $\equiv1\pmod6$, which is impossible. Let $E\in\{L_+,L'_+\}$, and colour each
two-element subset $\{s,s'\}$ of $S$ according to whether $s+s'\in E$. By the infinite Ramsey
theorem for pairs~\cite{ramsey1930problem} there is an infinite set $S'\subset S$ such that either
$s+s'\in E$ for all distinct $s,s'\in S'$, or $s+s'\notin E$ for all distinct $s,s'\in S'$. Split
$S'$ into two disjoint infinite sets $B_0$ and $C_0$, and let $b_0\in B_0$, $c_0\in C_0$ be arbitrary. Then
$b_0\neq c_0$ are primes different from $2$ and $3$, and $b_0+c_0$ is even and $\equiv1\pmod3$.
Also $\Omega(b_0c_0)=\omega(b_0c_0)=2$, so $b_0c_0\in L_+\cap L'_+$.

If $s+s'\in E$ for all distinct $s,s'\in S'$, then $B=B_0$ and $C=C_0$ work. Otherwise we put
$B=dB_0$ and $C=dC_0$, where $d=2$ if $E=L_+$ and $d=6$ if $E=L'_+$.
\begin{itemize}
\item $E=L_+$: $\Omega(2(b_0+c_0))=\Omega(b_0+c_0)+1$ is even, and
$\Omega(4b_0c_0)=4$.
\item $E=L'_+$: since $2\mid b_0+c_0$ and $3\nmid b_0+c_0$, we get
$\omega(6(b_0+c_0))=\omega(b_0+c_0)+1$, which is even. Since $b_0,c_0\notin\{2,3\}$,
$\omega(36b_0c_0)=4$.
\end{itemize}
In both cases $b+c\in E$ and $bc\in E$ for all $b=db_0\in B$ and $c=dc_0\in C$.
\end{proof}

For $L_+$, the same idea with Hindman's theorem in place of Ramsey's theorem gives the following
generalization.

\begin{proposition}\label{SP:prop:finitegroup}
Let $G$ be a finite abelian group, and let $f\colon\N\to G$ be completely multiplicative, i.e.\
$f(mn)=f(m)f(n)$ for all $m,n\in\N$. Then there are infinite sets $B,C\subset\N$ with
$f(b+c)=f(bc)=1_G$ for all $b\in B$ and $c\in C$.
\end{proposition}

\begin{proof}
Hindman's theorem~\cite[Theorem~3.1]{hindman1974finite}, applied to the finite colouring
$n\mapsto f(n)$, gives $h\in G$ and a sequence $(x_i)_{i\in\N}$ in $\N$ such that $f$ takes the
value $h$ on $\FS((x_i))=\{x_{i_1}+\dots+x_{i_k}:k\in\N,\ i_1<\dots<i_k\}$. We may assume that
$x_1<x_2<\cdots$. Indeed, if $(x_i)$ is unbounded we pass to a strictly increasing subsequence,
whose finite sums form a subset of $\FS((x_i))$. Otherwise some value $x$ occurs infinitely
often, so $\FS((x_i))$ contains every multiple of $x$, and we replace $(x_i)$ by $(2^ix)$. Let $B_0=\{x_1,x_3,\dots\}$ and $C_0=\{x_2,x_4,\dots\}$. Then $f$ equals $h$
on $B_0$, on $C_0$ and on $B_0+C_0$. The image $f(\N)$ is a nonempty subset of the finite group
$G$ that is closed under multiplication, hence a subgroup. So there is $d\in\N$ with
$f(d)=h^{-1}$. Put $B=dB_0$ and $C=dC_0$. For $b_0\in B_0$ and $c_0\in C_0$ we get
$f(db_0+dc_0)=f(d)f(b_0+c_0)=h^{-1}h=1_G$ and
$f(db_0\cdot dc_0)=f(d)^2f(b_0)f(c_0)=h^{-2}h^2=1_G$.
\end{proof}

For instance, for every $k\ge2$ the set $\{n:\Omega(n)\equiv0\pmod k\}$ contains $B+C$ and $BC$
for some infinite $B,C\subset\N$.

\formalization{\Cref{SP:thm:bohr,SP:cor:beatty,SP:thm:liouville} are formalized in
\leanfile{S51a} (\leanname{bohr}, \leanname{beatty_sum_prod}, \leanname{liouville}).
\Cref{SP:prop:finitegroup} is not formalized.}

\begin{remark}\label{SP:rem:liouville}
\Cref{SP:cor:beatty,SP:thm:liouville} answer Question~\S5.1(a) for both examples named in the
survey. Question~5.2 of~\cite{kra2023problems} itself remains open, to our knowledge, when ``positive density'' means
that the natural density exists and is positive (see also \cref{SP:rem:q52}).
\end{remark}

\subsection{Sums and shifted products at logarithmic scale}
\label{SP:sec:q54}

Logarithmic density is defined in~\cite[Definition~3.22]{kra2023problems}: for $A\subset\N$,
\[
\udlog(A)=\limsup_{N\to\infty}\frac1{\log N}\sum_{n=1}^N\frac{\one_A(n)}{n},
\]
with the lower version defined by $\liminf$. When the two agree, their common value is the
logarithmic density $\dlog(A)$. Question~5.4 of~\cite{kra2023problems} reads as follows.

\begin{question}{5.4}\label{SP:q:54}
Suppose $A\subset\N$ has positive upper logarithmic density. Are there infinite sets
$B,C\subset\N$ with $B+C\subset A$ and $BC+1\subset A$?
\end{question}

The survey introduces the shift $BC+1$ because the unshifted question fails for congruence
(local) reasons: the odd numbers contain no $B+C$ together with $BC$, since $b+c$ and $bc$ cannot
both be odd. The survey also notes that Question~5.4 generalizes an earlier question of Moreira
about singletons $B=\{x\}$, $C=\{y\}$. Taken literally, the singleton version is trivial ($x=1$,
$y=a-1$ with $a\in A$, $a\ge2$); with $x,y\ge2$ it is Moreira's conjecture, which Richter proved for every
set of positive upper logarithmic density~\cite[Theorem~1.5 with $Q(y)=y-1$]{richter2025sums}.
Indeed, that theorem, applied to $(A-1)\cap\N$, which has the same upper logarithmic density as $A$,
gives $x,y\ge2$ with $x+y-1\in A-1$ and $xy\in A-1$, i.e.\ $x+y\in A$ and $xy+1\in A$. Richter also relates his
results to Questions~5.2 and~5.4 of~\cite{kra2023problems}. We show that the infinite version
fails, because of an obstruction of a different kind: an archimedean one. Our set is aperiodic in
the sense of \eqref{SP:eq:aperiodic}, so it has no congruence obstruction.

Throughout this subsection let $I=[\tfrac{7}{20},\tfrac{13}{20})$, let
$\psi(n)=\log_2 n\bmod 1\in\T$ for $n\in\N$, and let
\[
A=\{n\in\N:\{\log_2n\}\in I\}=\bigcup_{k\ge0}\big[2^{k+7/20},\,2^{k+13/20}\big)\cap\N .
\]
Membership is decided in exact arithmetic: $n\in A$ if and only if
$2^{20k+7}\le n^{20}<2^{20k+13}$, where $k=\lfloor\log_2n\rfloor$.

\begin{proposition}\label{SP:prop:Aprops}
The set $A$ has the following properties.
\begin{enumerate}[label=\textup{(\roman*)}]
\item Its logarithmic density exists and $\dlog(A)=3/10$.
\item It is aperiodic, i.e.\ it satisfies \eqref{SP:eq:aperiodic}.
\item $\udens(A)=2-2^{7/10}\approx0.3755$ and $\ldens(A)=2^{3/10}-1\approx0.2311$. In particular,
the natural density of $A$ does not exist.
\end{enumerate}
\end{proposition}

\begin{proof}
Write $a=7/20$ and $b=13/20$. For real $1\le u<v$, comparing with $\int_u^v dx/x$ gives
$\sum_{u\le n<v}1/n=\log(v/u)+O(1/u)$. So the $k$-th block contributes
$(b-a)\log2+O(2^{-k})$ to $\sum\one_A(n)/n$.

(i) Let $2^K\le N<2^{K+1}$. Summing over the blocks with $k<K$, and bounding the contribution of
the block $k=K$ by $(b-a)\log 2+O(2^{-K})$, we get $\sum_{n\le N}\one_A(n)/n=\tfrac{3}{10}K\log2+O(1)$.
Since $\log N=K\log 2+O(1)$, the claim follows.

(ii) Let $\alpha\in\Q\setminus\Z$. For integers $m\le m'$ we have
$|\sum_{n=m}^{m'}\e(\alpha n)|\le 1/|\sin\pi\alpha|$. The set $A\cap[1,N]$ is a union of at most
$\log_2N+1$ intervals of consecutive integers. Hence
$|\frac1N\sum_{n\le N}\one_A(n)\e(\alpha n)|\le(\log_2N+1)/(N|\sin\pi\alpha|)\to0$.

(iii) Let $N=2^{K+s}$ with $K\in\N$ and $s\in[0,1)$. Counting the blocks with $k<K$ and the part
of block $K$ below $N$ gives
\[
|A\cap[1,N]|=2^K\big(2^b-2^a+g(s)\big)+O(K),\qquad g(s)=\max\big(0,\min(2^s,2^b)-2^a\big).
\]
So $|A\cap[1,N]|/N=\rho(s)+o(1)$, uniformly in $s$, where $\rho(s)=(2^b-2^a+g(s))/2^s$.
On $[0,a]$ the function $\rho(s)=(2^b-2^a)2^{-s}$ decreases. On $[a,b]$ it equals
$1-(2^{a+1}-2^b)2^{-s}$, which increases because $2^{a+1}>2^b$. On $[b,1)$ it equals
$2(2^b-2^a)2^{-s}$, which decreases. Therefore $\min\rho=\rho(a)=2^{b-a}-1=2^{3/10}-1$ and
$\sup\rho=\rho(b)=2-2^{a+1-b}=2-2^{7/10}$. Taking $N=2^{K+a}$ and $N=2^{K+b}$ (rounded down) as
$K\to\infty$ shows that these values are $\ldens(A)$ and $\udens(A)$.
\end{proof}

Let $\eps=\log_2(1+1/58)$. Since $\log(1+x)\le x$ and $\log2>0.693$, we have
$\eps\le1/(58\log2)<0.0249$. The constant $58$ is the least $M\in\N$ with
$2\log_2(1+1/M)<1/20$; this is the margin needed in \cref{SP:thm:q54}(a).

\begin{lemma}\label{SP:lem:shiftphase}
If $x,y\in\N$ satisfy $y\ge58x$ and $x+y\in A$, then
$\{\log_2y\}\in[\tfrac7{20}-\eps,\tfrac{13}{20})$.
\end{lemma}

\begin{proof}
We have $\log_2(x+y)=\log_2y+\eps'$ with $0<\eps'=\log_2(1+x/y)\le\eps<\tfrac7{20}$. Write
$\log_2(x+y)=k+\theta$ with $k\in\Z$ and $\theta=\{\log_2(x+y)\}\in I$. Then
$\log_2y=k+(\theta-\eps')$, and $0\le\theta-\eps'<1$ because $\theta\ge\tfrac7{20}>\eps'$. Hence
$\{\log_2y\}=\theta-\eps'\in[\tfrac7{20}-\eps,\tfrac{13}{20})$.
\end{proof}

\begin{theorem}\label{SP:thm:q54}
\leavevmode
\begin{enumerate}[label=\textup{(\alph*)}]
\item There are no $b_0,c_0,b,c\in\N$ and $\tau\in\{0,1\}$ with $b\ge58c_0$ and $c\ge58b_0$
such that
\[
b_0+c\in A,\qquad c_0+b\in A,\qquad bc+\tau\in A .
\]
\item Let $t_1,t_2\in\Z$. There are no infinite sets $B,C\subset\N$ with $b+c+t_1\in A$ and
$bc+t_2\in A$ for all $b\in B$ and $c\in C$.
\end{enumerate}
In particular, \cref{SP:q:54} has a negative answer: $\dlog(A)=3/10$, but there are no infinite
sets $B,C\subset\N$ with $B+C\subset A$ and $BC+1\subset A$.
\end{theorem}

\begin{proof}
(a) By \cref{SP:lem:shiftphase}, applied to $(x,y)=(b_0,c)$ and to $(x,y)=(c_0,b)$,
\[
\{\log_2c\},\ \{\log_2b\}\in\big[\tfrac7{20}-\eps,\ \tfrac{13}{20}\big).
\]
Write $\log_2(bc+\tau)=\log_2b+\log_2c+\eta$, where $\eta=\log_2(1+\tau/(bc))$. Since
$bc\ge58^2$, we have $0\le\eta\le1/(58^2\log2)<0.001$. Let $u=\{\log_2b\}+\{\log_2c\}+\eta$.
Then $\log_2(bc+\tau)-u\in\Z$, so $\{\log_2(bc+\tau)\}=\{u\}$, and
$u\in[\tfrac{7}{10}-2\eps,\ \tfrac{13}{10}+\eta)$.
If $u<1$, then $\{u\}=u\ge\tfrac7{10}-2\eps>\tfrac{13}{20}$, because $2\eps<0.05$.
If $u\ge1$, then $\{u\}=u-1<\tfrac3{10}+\eta<\tfrac7{20}$.
In both cases $\{\log_2(bc+\tau)\}\notin I$, so $bc+\tau\notin A$.

(b) Suppose such $B$ and $C$ exist. Fix $b_0\in B$. For $c\in C$, $c\to\infty$, the distance in
$\T$ between $\psi(c+b_0+t_1)$ and $\psi(c)$ is at most $|\log_2(1+(b_0+t_1)/c)|$, which tends to $0$.
Since $\psi(c+b_0+t_1)\in I$, every limit point of $\psi(c)$ as $c\to\infty$ in $C$ lies in the
closure $\bar I=[\tfrac7{20},\tfrac{13}{20}]$. By symmetry the same holds for $\psi(b)$ as $b\to\infty$ in $B$.
Choose $b_n\in B$ and $c_n\in C$ with $b_n,c_n\to\infty$. Since $\T$ is compact, after passing
to a subsequence we may assume that $\psi(b_n)\to\beta$ and $\psi(c_n)\to\gamma$, where
$\beta,\gamma\in\bar I$. Now
$\psi(b_nc_n+t_2)=\psi(b_n)+\psi(c_n)+\log_2(1+t_2/(b_nc_n))$ in $\T$, and this tends to $\beta+\gamma$. But
$\beta+\gamma\in[\tfrac7{10},\tfrac{13}{10}]$, and modulo $1$ this set is
$[\tfrac7{10},1)\cup[0,\tfrac3{10}]$, whose distance from $\bar I$ is $\tfrac1{20}$. On the other hand
$\psi(b_nc_n+t_2)\in I$ for all $n$, so $\beta+\gamma\in\bar I$. This is a contradiction.

For the final statement, recall $\dlog(A)=3/10$ from \cref{SP:prop:Aprops}(i). If $B,C\subset\N$
were infinite with $B+C\subset A$ and $BC+1\subset A$, we could fix $b_0\in B$ and $c_0\in C$ and
then pick $b\in B$ with $b\ge58c_0$ and $c\in C$ with $c\ge58b_0$, contradicting (a) with
$\tau=1$. (It also follows from (b) with $t_1=0$, $t_2=1$.)
\end{proof}

\formalization{\Cref{SP:lem:shiftphase}, both parts of \cref{SP:thm:q54} and
\cref{SP:prop:Aprops}(i) are formalized in \leanfile{Q54} (\leanname{no_ratio58_config},
\leanname{no_shifted_config}, \leanname{q54_negative}); parts (ii) and (iii) of
\cref{SP:prop:Aprops} are not.}

\begin{remark}\label{SP:rem:q54}
(i) The obstruction is asymptotic, and it needs both $B$ and $C$ to be infinite. If $B=\{1\}$, then
$B+C=BC+1=C+1$, so $C=(A-1)\cap\N$ works. By Richter's theorem quoted above, every set of positive
upper logarithmic density, in particular $A$, contains $\{b+c,bc+1\}$ with $b,c\ge2$. Larger finite
configurations also occur in $A$: for $B=\{1,2,4,8,16\}$ and $C=\{1306,1307,\dots,1310\}$, all
$25$ sums $b+c$ and all $25$ numbers $bc+1$ lie in $A$. This is a direct check of $50$ memberships
with the exact test above. Part~(a) of \cref{SP:thm:q54} is a quantitative form of the
obstruction: it rules out a four-element configuration in which $b,c$ are large compared with
$c_0,b_0$.

(ii) By \cref{SP:thm:q54}(b), shifting the sums or the products by any fixed integers does not
remove the obstruction. Unlike the congruence obstruction to the unshifted question exhibited by the odd numbers,
it is not local: $A$ is aperiodic (\cref{SP:prop:Aprops}(ii)). It is also thick (the blocks have
unbounded length), so $\ubd(A)=1$. Similar constructions at logarithmic scale appear throughout
\S\S2--3 of the survey~\cite{kra2023problems}, where they serve as counterexamples for sumsets
$B+B+t$; Richter's Example~1.6 in~\cite{richter2025sums} is another block-type construction of
this kind, a set of positive upper density without $\{x+y-1,xy\}$, $x,y\ge2$. Here the same idea is applied
to infinite sums and shifted products. To our knowledge this counterexample is new.

(iii) The set $A$ is not syndetic. The survey also asks Question~5.4 for syndetic $A$, and
Question~5.5 of~\cite{kra2023problems} (whether every syndetic set contains $BC$ with $B,C$ infinite);
both remain open, to our knowledge.
\end{remark}

Recall Question~5.2 of~\cite{kra2023problems}.

\begin{question}{5.2}\label{SP:q:52}
Fix $A\subset\N$ that has positive density and is aperiodic. Are there infinite sets
$B,C\subset\N$ with $B+C\subset A$ and $BC\subset A$?
\end{question}

\begin{remark}\label{SP:rem:q52}
Taking $\tau=0$ in \cref{SP:thm:q54}(a), or $t_1=t_2=0$ in (b), shows that $A$ contains no $B+C$
together with $BC$ for infinite $B,C\subset\N$. By \cref{SP:prop:Aprops}, $A$ is aperiodic and has
positive lower density (hence positive upper, logarithmic and upper Banach density), but its natural
density does not exist. So $A$ gives a negative answer to \cref{SP:q:52} if ``positive density''
there is read as positive lower or upper density (or positive upper Banach or logarithmic
density). This reading is
consistent with the survey's text, which says that (19) is equivalent to equidistribution in
residue classes ``if the density of $A$ exists''. It does not settle \cref{SP:q:52} when the
natural density is required to exist; to our knowledge that case remains open.
\end{remark}

\subsection{Sums and products in \texorpdfstring{$\Q$}{Q}}
\label{SP:sec:q56}

Following~\cite[\S5.1]{kra2023problems}, where the notion is attributed
to~\cite[Definition~1.3]{bergelson2017ergodic}, we use the formulation
of~\cite[Definition~1.7]{kousek2024revisiting}: a \emph{double \Fol{} sequence} in $\Q$ is a
sequence $\Phi=(F_N)_{N\in\N}$ of nonempty finite subsets of $\Q$ such that
\[
\lim_{N\to\infty}\frac{|F_N\cap(x+F_N)|}{|F_N|}=1\quad(x\in\Q)
\qquad\text{and}\qquad
\lim_{N\to\infty}\frac{|F_N\cap xF_N|}{|F_N|}=1\quad(x\in\Q^*).
\]
The survey describes $\Phi$ as simultaneously a \Fol{} sequence in $(\Q,+)$ and in
$(\Q^*,\times)$; strictly speaking the latter requires $F_N\subset\Q^*$. This makes no
difference below: $0$ is a single point and $|F_N|\to\infty$ (\cref{SP:lem:folner}(i)), so all
densities are unchanged if $0$ is removed from every $F_N$.
For $A\subset\Q$ we write $\udens_\Phi(A)=\limsup_N|A\cap F_N|/|F_N|$. When the limit exists we
denote it by $\dens_\Phi(A)$, as in~\cite[\S1]{kra2023problems}.

\begin{question}{5.6}\label{SP:q:56}
Let $A\subset\Q$ have $\dens_\Phi(A)>0$ for some double \Fol{} sequence $\Phi$. Are there infinite
sets $B,C\subset\Q$ such that $B+C\subset A$ and $BC\subset A$?
\end{question}

We show that the answer is negative: for every double \Fol{} sequence $\Phi$ there is a
counterexample $A=A_\Phi$.

\begin{lemma}\label{SP:lem:folner}
Let $\Phi=(F_N)$ be a double \Fol{} sequence in $\Q$.
\begin{enumerate}[label=\textup{(\roman*)}]
\item $|F_N|\to\infty$.
\item For every bounded set $X\subset\Q$ we have $|F_N\cap X|/|F_N|\to0$.
\item $|F_N\cap\Q_{<0}|/|F_N|\to1/2$.
\end{enumerate}
\end{lemma}

\begin{proof}
(i) Suppose $|F_N|\le M$ for infinitely many $N$. For those $N$ that are large enough,
$|F_N\cap(1+F_N)|>(1-\frac1{M+1})|F_N|\ge|F_N|-1$, hence $F_N\subset1+F_N$. Both sets have the same
cardinality, so $F_N=1+F_N$. This is impossible for a finite nonempty $F_N\subset\Q$, since
$\max(1+F_N)=1+\max F_N>\max F_N$.

(ii) Choose $R\in\N$ with $X\subset[-R,R]$, fix $k\in\N$, and let $t_j=j(2R+1)$ for
$0\le j<k$. The sets $X+t_j$ are pairwise disjoint. For each $t\in\Q$,
\[
|F_N\cap(X+t)|=|(F_N-t)\cap X|\ge|F_N\cap X|-|F_N\setminus(F_N-t)|=|F_N\cap X|-o(|F_N|),
\]
since $|F_N\setminus(F_N-t)|=|F_N|-|F_N\cap(-t+F_N)|$. Summing over $j$ gives
$k|F_N\cap X|\le|F_N|+o(|F_N|)$, so $\limsup_N|F_N\cap X|/|F_N|\le1/k$. Now let $k\to\infty$.

(iii) Let $p_N=|F_N\cap\Q_{>0}|$ and $n_N=|F_N\cap\Q_{<0}|$. The map $y\mapsto-y$ gives
$n_N=|(-F_N)\cap\Q_{>0}|$, and this differs from $p_N$ by at most
$|(-F_N)\setminus F_N|=|F_N|-|F_N\cap(-1)F_N|=o(|F_N|)$. Moreover $p_N+n_N\ge|F_N|-1$. By (i),
$n_N/|F_N|\to1/2$.
\end{proof}

\begin{lemma}\label{SP:lem:noconf}
Let $A\subset\Q_{<0}$ be such that $A\cap[-R,R]$ is finite for every $R>0$. Then there are no
infinite sets $B,C\subset\Q$ with $B+C\subset A$ and $BC\subset A$.
\end{lemma}

\begin{proof}
Suppose such $B,C$ exist. Every product $bc$ is negative, so $0\notin B\cup C$, and every
$b\in B$ and $c\in C$ have opposite signs. If $B$ contained a positive and a negative element,
no $c\in C$ could have the sign opposite to both. So $B$ has constant sign, and $C$ has the
opposite constant sign. By symmetry of the hypotheses in $B$ and $C$ we may assume that
$B\subset\Q_{<0}$ and $C\subset\Q_{>0}$. Fix $b^*\in B$. For $c\in C$ we have $b^*+c<0$, so
$0<c<|b^*|$ and $b^*c\in(-(b^*)^2,0)$. Since $c\mapsto b^*c$ is injective and $C$ is infinite,
$b^*C$ is an infinite subset of $A\cap[-(b^*)^2,(b^*)^2]$, which is a contradiction.
\end{proof}

\begin{theorem}\label{SP:thm:q56}
Let $\Phi=(F_N)$ be any double \Fol{} sequence in $\Q$. There is a set $A\subset\Q_{<0}$ such
that $\dens_\Phi(A)=\lim_{N}|A\cap F_N|/|F_N|=1/2$, and such that there are no infinite sets
$B,C\subset\Q$ with $B+C\subset A$ and $BC\subset A$. In particular, \cref{SP:q:56} has a
negative answer.
\end{theorem}

\begin{proof}
For $R\in\N$ write $[-R,R]$ for $\Q\cap[-R,R]$, and let $\eps_R(N)=|F_N\cap[-R,R]|/|F_N|$. By
\cref{SP:lem:folner}(ii), $\eps_R(N)\to0$ as $N\to\infty$ for each $R$. So we can choose
$N_1<N_2<\cdots$ such that $\eps_R(N)<1/R$ for all $N\ge N_R$. Put
\[
K_R=\bigcup_{N<N_{R+1}}F_N\ \ (\text{a finite set}),\qquad
Z=\bigcup_{R\in\N}\big([-R,R]\setminus K_R\big),\qquad A=\Q_{<0}\setminus Z .
\]
\emph{Bounded parts of $A$ are finite.} Let $R\in\N$. If $x\in A$ and $|x|\le R$, then
$x\notin[-R,R]\setminus K_R$, so $x\in K_R$. Hence $A\cap[-R,R]\subset K_R$ is finite. By
\cref{SP:lem:noconf}, $A$ contains no $B+C$ together with $BC$ for infinite $B,C$.

\emph{$Z$ has density zero.} Let $N\ge N_1$, and let $R$ be such that $N_R\le N<N_{R+1}$.
If $R'\ge R$, then $N<N_{R'+1}$, so $F_N\subset K_{R'}$ and $F_N$ does not meet
$[-R',R']\setminus K_{R'}$. If $R'<R$, then $[-R',R']\subset[-R,R]$. Hence
$F_N\cap Z\subset F_N\cap[-R,R]$, and so $|F_N\cap Z|\le\eps_R(N)|F_N|<|F_N|/R$. As
$N\to\infty$, also $R\to\infty$, so $|F_N\cap Z|/|F_N|\to0$.

\emph{Density of $A$.} By \cref{SP:lem:folner}(iii),
$|A\cap F_N|/|F_N|=|F_N\cap\Q_{<0}|/|F_N|-O(|F_N\cap Z|/|F_N|)\to1/2$.
\end{proof}

\formalization{\Cref{SP:lem:folner,SP:lem:noconf} and \cref{SP:thm:q56} are formalized in
\leanfile{Q56} (\leanname{theoremN}), with double \Fol{} sequences as in
\leanfile{DoubleFolnerQ}.}

\begin{remark}\label{SP:rem:q56easy}
We record some easy facts that put \cref{SP:thm:q56} in context. Let $\Phi$ be a double \Fol{}
sequence and let $A\subset\Q$ satisfy $\udens_\Phi(A)>0$.
\begin{enumerate}[label=\textup{(\roman*)}]
\item Each of the two conditions can be met separately. Moreira, Richter and Robertson proved the
following~\cite[Theorem~1.3]{moreira2019proof}: if $G$ is a countable group, $\Psi$ is a two-sided
\Fol{} sequence on $G$, and $E\subset G$ has positive upper density along $\Psi$, then $E\supset BC$
for some infinite $B,C\subset G$. Applied to $G=(\Q,+)$ and $\Psi=\Phi$, this gives infinite $B,C$
with $B+C\subset A$. Next apply it to $G=(\Q^*,\cdot)$ and $\Psi=(F_N\setminus\{0\})_{N\ge N_0}$, where $N_0$ is such that these sets are nonempty; this is a
\Fol{} sequence in $\Q^*$, since removing one point changes $|F_N|$ and $|F_N\cap xF_N|$ by at most $1$
and $|F_N|\to\infty$ by \cref{SP:lem:folner}(i); likewise the upper density of $A\setminus\{0\}$
along it equals $\udens_\Phi(A)$. This gives infinite $B,C\subset\Q^*$ with $BC\subset A$. (Both
groups are abelian, so one-sided and two-sided \Fol{} sequences coincide.)
\item The set $\Q_{<0}$ itself contains the configuration: take $B=\{-n:n\ge2\}$ and
$C=\{1/m:m\in\N\}$. So the deletion of the null set $Z$ in \cref{SP:thm:q56} is essential.
\item The question requires two sets $B,C$. The version with $B=C$ fails even for restricted
products: if $B\subset\Q$ is infinite, then $B\setminus\{0\}$ has two distinct elements of the
same sign. So $\{b_1b_2:b_1,b_2\in B,\ b_1\neq b_2\}$ always contains a nonnegative number, and it
is not contained in $\Q_{<0}$. On the other hand, $\dens_\Phi(\Q_{<0})=1/2$ by
\cref{SP:lem:folner}(iii).
\end{enumerate}
\end{remark}

\begin{remark}\label{SP:rem:q56literal}
\Cref{SP:thm:q56} answers \cref{SP:q:56} as stated. The question assumes only that $A$ has positive
density along some double \Fol{} sequence, and for any prescribed one we construct a set with
density $1/2$ along it. The obstruction is the sign of $x$, i.e.\ the real place of $\Q$. It plays the same role as
congruence obstructions in $\N$, which is why the survey adds the aperiodicity hypothesis to
Question~5.2. The point of \cref{SP:thm:q56} is that, once a set of density zero has been removed,
this obstruction cannot be detected by density along a double \Fol{} sequence. The $\Phi$-dependent deletion of finite pieces is similar to the device used
in~\cite[Example~5.9]{kra2023problems} (compare also the proof
of~\cite[Proposition~3.2]{kra2023problems}). The question that is
presumably intended remains open, to our knowledge: does every $A\subset\Q$ that is large and has no such local
obstruction contain $B+C$ and $BC$? Possible formalizations are to require $\ldens_\Phi(A)>0$
for \emph{every} double \Fol{} sequence $\Phi$, which our construction (it depends on $\Phi$) does
not address, or to impose a hypothesis on $A$ analogous to aperiodicity. We do not know the answer
in either case.
\end{remark}

\section*{Acknowledgements}
This work was carried out with substantial assistance from AI systems (OpenAI's GPT models and
Anthropic's Claude), including the development of the proofs, the checking of the arguments, and
the Lean formalization. The authors take full responsibility for the content.

\bibliographystyle{amsplain}
\bibliography{refs}

@article{kra2023problems,
  author = {Kra, Bryna and Moreira, Joel and Richter, Florian K. and Robertson, Donald},
  title = {Problems on infinite sumset configurations in the integers and beyond},
  journal = {Bull. Amer. Math. Soc. (N.S.)},
  volume = {62},
  number = {4},
  year = {2025},
  pages = {537--574},
  doi = {10.1090/bull/1868},
  note = {arXiv:2311.06197v2; question numbers refer to this version}
}

@article{charamaras2024finding,
  author = {Charamaras, Dimitrios and Mountakis, Andreas},
  title = {Finding product sets in some classes of amenable groups},
  journal = {Forum Math. Sigma},
  volume = {13},
  pages = {Paper No. e10},
  year = {2025},
  note = {arXiv:2402.07779; theorem numbers refer to arXiv:2402.07779v3}
}

@article{ackelsberg2024counterexamples,
  author = {Ackelsberg, Ethan},
  title = {Counterexamples to generalizations of the {E}rd{\H{o}}s {$B+B+t$} problem},
  journal = {Electron. J. Combin.},
  volume = {33},
  number = {3},
  pages = {Paper No. P3.89},
  year = {2026},
  note = {arXiv:2404.17383; theorem numbers refer to arXiv:2404.17383v1}
}

@article{charamaras2025infinite,
  author = {Charamaras, Dimitrios and Kousek, Ioannis and Mountakis, Andreas and Radi{\'c}, Trist{\'a}n},
  title = {Infinite unrestricted sumsets in subsets of abelian groups with large density},
  journal = {Int. Math. Res. Not. IMRN},
  volume = {2026},
  number = {20},
  pages = {Paper No. rnag216},
  year = {2026},
  note = {arXiv:2504.08649; lemma numbers refer to arXiv:2504.08649v1}
}

@article{ackelsberg2025equidistribution,
  author = {Ackelsberg, Ethan and Jamneshan, Asgar},
  title = {Equidistribution in 2-nilpotent {P}olish groups and triple restricted sumsets},
  journal = {J. Lond. Math. Soc. (2)},
  volume = {114},
  number = {3},
  pages = {Paper No. e70693},
  year = {2026},
  note = {arXiv:2504.07865; theorem numbers refer to arXiv:2504.07865v3}
}

@article{austin2016nonconventional,
  author = {Austin, Tim},
  title = {Non-conventional ergodic averages for several commuting actions of an amenable group},
  journal = {J. Anal. Math.},
  volume = {130},
  year = {2016},
  pages = {243--274}
}

@article{kousek2024infinite,
  author = {Kousek, Ioannis and Radi{\'c}, Trist{\'a}n},
  title = {Infinite unrestricted sumsets of the form {$B+B$} in sets with large density},
  journal = {Bull. Lond. Math. Soc.},
  volume = {57},
  number = {1},
  year = {2025},
  pages = {48--68},
  note = {arXiv:2404.12201}
}

@article{kousek2025asymmetric,
  author = {Kousek, Ioannis},
  title = {Asymmetric infinite sumsets in large sets of integers},
  journal = {Forum Math. Sigma},
  volume = {14},
  pages = {Paper No. e7},
  year = {2026},
  note = {arXiv:2502.03112; theorem numbers refer to the arXiv version}
}

@article{hindman1979partitions,
  author = {Hindman, Neil},
  title = {Partitions and sums of integers with repetition},
  journal = {J. Combin. Theory Ser. A},
  volume = {27},
  number = {1},
  year = {1979},
  pages = {19--32}
}

@misc{huang2026affirmative,
  author = {Huang, Wen and Lian, Zhengxing and Shao, Song and Xiao, Rongzhong and Xu, Leiye and Zhang, Shuhao},
  title = {An affirmative answer to {O}wings's sumset question},
  year = {2026},
  eprint = {2607.17333},
  archivePrefix = {arXiv},
  note = {arXiv:2607.17333}
}

@misc{dinasso2025ramsey,
  author = {Di Nasso, Mauro and Luperi Baglini, Lorenzo and Mamino, Marcello and Mennuni, Rosario and Ragosta, Mariaclara},
  title = {Ramsey's witnesses},
  year = {2025},
  eprint = {2503.09246},
  archivePrefix = {arXiv},
  note = {arXiv:2503.09246}
}

@article{owings1974problem,
  author = {Owings, Jr., J. C.},
  title = {Problem {E}2494},
  journal = {Amer. Math. Monthly},
  volume = {81},
  number = {8},
  year = {1974},
  pages = {902}
}

@article{kra2022proof,
  author = {Kra, Bryna and Moreira, Joel and Richter, Florian K. and Robertson, Donald},
  title = {A proof of {E}rd{\H{o}}s's {$B+B+t$} conjecture},
  journal = {Commun. Amer. Math. Soc.},
  volume = {4},
  year = {2024},
  pages = {480--494},
  note = {arXiv:2206.12377}
}

@misc{ackelsberg2025infinite,
  author = {Ackelsberg, Ethan},
  title = {Infinite polynomial patterns in large subsets of the rational numbers},
  year = {2025},
  eprint = {2506.19667},
  archivePrefix = {arXiv},
  note = {arXiv:2506.19667}
}

@article{kra2025density,
  author = {Kra, Bryna and Moreira, Joel and Richter, Florian K. and Robertson, Donald},
  title = {The density finite sums theorem},
  journal = {Invent. Math.},
  volume = {243},
  number = {1},
  pages = {1--31},
  year = {2026},
  note = {arXiv:2504.06424; theorem numbers refer to arXiv:2504.06424v3}
}

@article{kra2022infinite,
  author = {Kra, Bryna and Moreira, Joel and Richter, Florian K. and Robertson, Donald},
  title = {Infinite sumsets in sets with positive density},
  journal = {J. Amer. Math. Soc.},
  volume = {37},
  number = {3},
  year = {2024},
  pages = {637--682},
  note = {arXiv:2206.01786; theorem numbers refer to arXiv:2206.01786v2}
}

@article{bergelson1985sets,
  author = {Bergelson, Vitaly},
  title = {Sets of recurrence of {${\bf Z}^m$}-actions and properties of sets of differences in {${\bf Z}^m$}},
  journal = {J. London Math. Soc. (2)},
  volume = {31},
  number = {2},
  year = {1985},
  pages = {295--304}
}

@article{tao2023infinite,
  author = {Tao, Terence and Ziegler, Tamar},
  title = {Infinite partial sumsets in the primes},
  journal = {J. Anal. Math.},
  volume = {151},
  number = {1},
  pages = {375--389},
  year = {2023},
  note = {arXiv:2301.10303; theorem numbers refer to arXiv:2301.10303v4}
}

@article{mazurkiewicz2023ideal,
  author = {Mazurkiewicz, {\L}ukasz and {\.Z}eberski, Szymon},
  title = {Ideal analytic sets},
  journal = {Math. Logic Quart.},
  volume = {72},
  number = {1},
  pages = {Paper No. e70012},
  year = {2026},
  note = {arXiv:2310.07693}
}

@book{kechris1995classical,
  author = {Kechris, Alexander S.},
  title = {Classical Descriptive Set Theory},
  series = {Graduate Texts in Mathematics},
  volume = {156},
  publisher = {Springer-Verlag, New York},
  year = {1995}
}

@article{moreira2019proof,
  author  = {Moreira, Joel and Richter, Florian K. and Robertson, Donald},
  title   = {A proof of a sumset conjecture of {E}rd{\H{o}}s},
  journal = {Ann. of Math. (2)},
  volume  = {189},
  number  = {2},
  year    = {2019},
  pages   = {605--652},
  note    = {arXiv:1803.00498}
}

@article{hindman1974finite,
  author  = {Hindman, Neil},
  title   = {Finite sums from sequences within cells of a partition of {$N$}},
  journal = {J. Combinatorial Theory Ser. A},
  volume  = {17},
  year    = {1974},
  pages   = {1--11}
}

@article{ramsey1930problem,
  author  = {Ramsey, F. P.},
  title   = {On a problem of formal logic},
  journal = {Proc. London Math. Soc. (2)},
  volume  = {30},
  number  = {4},
  year    = {1930},
  pages   = {264--286}
}

@article{bergelson2017ergodic,
  author  = {Bergelson, Vitaly and Moreira, Joel},
  title   = {Ergodic theorem involving additive and multiplicative groups of a field and {$\{x+y,xy\}$} patterns},
  journal = {Ergodic Theory Dynam. Systems},
  volume  = {37},
  number  = {3},
  year    = {2017},
  pages   = {673--692}
}

@misc{richter2025sums,
  author        = {Richter, Florian K.},
  title         = {Sums and products in sets of positive density},
  year          = {2025},
  eprint        = {2507.00515},
  archivePrefix = {arXiv},
  note          = {arXiv:2507.00515v2}
}

@article{kousek2024revisiting,
  author  = {Kousek, Ioannis},
  title   = {Revisiting sums and products in countable and finite fields},
  journal = {Ergodic Theory Dynam. Systems},
  volume  = {46},
  number  = {2},
  year    = {2026},
  pages   = {543--574},
  note    = {arXiv:2407.03304}
}

@misc{kra2026short,
  author = {Kra, Bryna and Moreira, Joel and Richter, Florian K. and Robertson, Donald},
  title = {A short proof of {E}rd{\H{o}}s's {$B+C$} conjecture},
  year = {2026},
  eprint = {2603.27258},
  archivePrefix = {arXiv},
  note = {arXiv:2603.27258v2}
}

@phdthesis{griesmer2009ergodic,
  author = {Griesmer, John T.},
  title = {Ergodic averages, correlation sequences, and sumsets},
  school = {The Ohio State University},
  year = {2009}
}

@article{furstenberg1977ergodic,
  author = {Furstenberg, Harry},
  title = {Ergodic behavior of diagonal measures and a theorem of {S}zemer{\'e}di on arithmetic progressions},
  journal = {J. Analyse Math.},
  volume = {31},
  year = {1977},
  pages = {204--256}
}

@article{host2009uniformity,
  author = {Host, Bernard and Kra, Bryna},
  title = {Uniformity seminorms on $\ell^\infty$ and applications},
  journal = {J. Anal. Math.},
  volume = {108},
  year = {2009},
  pages = {219--276}
}

\end{document}